\pdfoutput=1
\documentclass[11pt]{article}
\usepackage{amsthm}
\usepackage{amssymb}
\usepackage{amsmath}
\usepackage{hyperref}
\usepackage{graphicx}
\usepackage{caption}
\usepackage{array}
\usepackage{enumitem}
\usepackage{mathtools}
\usepackage{etoolbox}
\usepackage[backend=biber, style=alphabetic, maxbibnames=9, maxcitenames=9, maxalphanames=9]{biblatex}
\bibliography{G2paper}
\usepackage[all,cmtip]{xy}

\pretocmd{\section}{%
  }{}{}

\numberwithin{table}{section}

\newtheorem{theorem}{Theorem}[subsection]
\newtheorem*{theorem*}{Theorem}
\newtheorem{proposition}[theorem]{Proposition}
\newtheorem*{proposition*}{Proposition}
\newtheorem{corollary}[theorem]{Corollary}
\newtheorem{lemma}[theorem]{Lemma}
\newtheorem*{lemma*}{Lemma}
\newtheorem{conjecture}[theorem]{Conjecture}
\newtheorem{introthm}{Theorem}[section]

\theoremstyle{definition}
\newtheorem{definition}[theorem]{Definition}

\newtheorem*{exercise*}{Exercise}
\newtheorem{remark}[theorem]{Remark}
\newtheorem{assumption}[theorem]{Assumption}

\numberwithin{equation}{subsection}
\numberwithin{figure}{section}

\let\hom\relax
\let\det\relax
\let\L\relax

\newcommand{\mb}{\mathbb}
\newcommand{\mbf}{\mathbf}
\newcommand{\mc}{\mathcal}
\newcommand{\mf}{\mathfrak}
\newcommand{\mr}{\mathrm}
\newcommand\sm[1]{{\tiny\arraycolsep=0.3\arraycolsep\ensuremath{\begin{pmatrix}#1\end{pmatrix}}}}
\newcommand\pmat[1]{\begin{pmatrix}#1\end{pmatrix}}
\newcommand\L[1]{\prescript{L}{}{#1}}

\newcommand{\Z}{\mathbb{Z}}
\newcommand{\Q}{\mathbb{Q}}
\newcommand{\R}{\mathbb{R}}
\newcommand{\C}{\mathbb{C}}
\newcommand{\F}{\mathbb{F}}

\newcommand{\A}{\mathbb{A}}

\renewcommand{\sl}{\mathfrak{sl}}

\newcommand{\GL}{\mathrm{GL}}
\newcommand{\SL}{\mathrm{SL}}

\newcommand{\Sset}[2]{\left\lbrace{#1}\,\,\middle|\,\,{#2}\right\rbrace}
\newcommand{\sset}[2]{\lbrace{#1}\,\,|\,\,{#2}\rbrace}

\DeclareMathOperator{\ad}{ad}
\DeclareMathOperator{\Ad}{Ad}

\DeclareMathOperator{\chars}{char}
\DeclareMathOperator{\cl}{cl}

\DeclareMathOperator{\crys}{crys}
\DeclareMathOperator{\cusp}{cusp}
\DeclareMathOperator{\cyc}{cyc}

\DeclareMathOperator{\det}{det}
\DeclareMathOperator{\diag}{diag}

\DeclareMathOperator{\disc}{disc}

\DeclareMathOperator{\dR}{dR}
\DeclareMathOperator{\Eis}{Eis}

\DeclareMathOperator{\ext}{Ext}
\DeclareMathOperator{\Fil}{Fil}
\DeclareMathOperator{\Frac}{Frac}
\DeclareMathOperator{\Frob}{Frob}

\DeclareMathOperator{\hom}{Hom}
\newcommand{\id}{\mathrm{id}}

\DeclareMathOperator{\Ind}{Ind}

\DeclareMathOperator{\Lie}{Lie}
\DeclareMathOperator{\MF}{MF}
\newcommand{\modulo}[1]{\,\,(\mathrm{mod}\,\,{#1})}

\DeclareMathOperator{\ord}{ord}

\DeclareMathOperator{\re}{Re}

\DeclareMathOperator{\sss}{ss}
\DeclareMathOperator{\st}{st}
\DeclareMathOperator{\std}{Std}

\DeclareMathOperator{\Sym}{Sym}

\DeclareMathOperator{\tr}{Tr}

\DeclareMathOperator{\vol}{Vol}

\begin{document}
\title{Eisenstein series for $G_2$ and the symmetric cube Bloch--Kato conjecture}
\author{Sam Mundy}
\date{}
\maketitle
\begin{abstract}
Let $F$ be a cuspidal eigenform of even weight and trivial nebentypus, let $p$ be a prime not dividing the level of $F$, and let $\rho_F$ be the $p$-adic Galois representation attached to $F$. Assume that the $L$-function attached to $\Sym^3(\rho_F)$ vanishes to odd order at its central point. Then under some mild hypotheses, and conditional on certain consequences of Arthur's conjectures, we construct a nontrivial element in the Bloch--Kato Selmer group of an appropriate twist of $\Sym^3(\rho_F)$, in accordance with the Bloch--Kato conjectures.

Our technique is based on the method of Skinner and Urban. We construct a class in the appropriate Selmer group by $p$-adically deforming Eisenstein series for the exceptional group $G_2$ in a generically cuspidal family and then studying a lattice in the corresponding family of $G_2$-Galois representations. We also make a detailed study of the specific conjectures used and explain how one might try to prove them.
\end{abstract}
\tableofcontents

\section*{Introduction}
Let $F$ be a cuspidal eigenform with weight $k$, level $N$, and trivial nebentypus. Fix a prime $p$. Then attached to $F$ is its $p$-adic Galois representation $\rho_F$, which for simplicity, we view as a representation $G_\Q\to GL_2(\overline\Q_p)$; here $G_\Q$ denotes the absolute Galois group of $\Q$, and $\rho_F$ is normalized so that for $\ell\nmid N$, the trace of $\rho_F$ on an arithmetic Frobenius element is the $\ell$th Fourier coefficient of $F$. In this paper, we study the symmetric cube $\Sym^3(\rho_F)$ of $\rho_F$; in particular, we will be interested in the relationship, predicted by the famous conjectures of Bloch and Kato, between the central value of the $L$-function $L(s,\Sym^3(\rho_F))$ and the Bloch--Kato Selmer group of (the appropriate twist of) $\Sym^3(\rho_F)$.

In fact, let us consider the Tate twist $\Sym^3(\rho_F)(\tfrac{4-3k}{2})$. (Note that $k$ must be even since there are no modular forms of odd weight and trivial nebentypus; hence $\tfrac{4-3k}{2}$ is an integer.) The significance of this is that, writing $R$ for this twisted representation, we have $R\cong R^\vee(1)$. The Bloch--Kato conjecture then predicts that
\[\ord_{s=s_0}L(s,\Sym^3(\rho_F))=\dim_{\overline\Q_p}H_f^1(\Q,\Sym^3(\rho_F)(-s_0)),\]
where $s_0=\tfrac{3k-4}{2}$ is the central value of the $L$-function on the left hand side, the $L$-function itself is defined using geometric Frobenius elements, and $H_f^1$ denotes as usual the Bloch--Kato Selmer group.

Except in certain circumstances, this equality seems out of reach at the moment. One could then try instead to prove the implication
\[L(s_0,\Sym^3(\rho_F))=0\quad\Longrightarrow\quad H_f^1(\Q,\Sym^3(\rho_F)(-s_0))\ne 0,\]
which would be implied by the equality above. This turns out to be more reasonable, and in fact, the purpose of this paper is to prove this implication under some mild hypotheses, assuming the order of vanishing of the $L$-function on the left hand side is odd, under Arthur's conjectures. (Actually, only some very specific consequences of Arthur's conjectures are needed, and we comment on this below.) Here is a precise statement.

\begin{introthm}
\label{intthmmain}
Assume that the weight $k$ of $F$ is at least $4$ and that the level $N$ is not divisible by $p$. In addition to this, assume:
\begin{itemize}
\item $L(s,\Sym^3(\rho_F))$ vanishes to odd order at the central value $s=\tfrac{3k-4}{2}$
\item $F$ is not CM;
\item $4\nmid N$ and $9\nmid N$;
\item The Hecke polynomial of $F$ at $p$ has simple roots.
\end{itemize}
Then under Conjectures \ref{conjmult} (a) and \ref{conjliftings} below, the Bloch--Kato Selmer group
\[H_f^1(\Q,\Sym^3(\rho_F)(\tfrac{4-3k}{2}))\]
is nontrivial.
\end{introthm}

The overarching technique used to prove this theorem is the Skinner--Urban method, and below we will outline the proof and explain what parts of it are responsible for each of the hypotheses of the theorem. But first we would like to briefly explain this method in general as well as give some of its history.

The Skinner--Urban method is a method by which one can try to construct nontrivial elements in the Bloch--Kato Selmer group of an automorphic Galois representation $\rho$, assuming the $L$-function of $\rho$ vanishes at its central point and that this point is critical. It has three main steps. If $\rho$ is attached in some way via the Langlands correspondence to an automorphic representation $\pi$ of some reductive group $M$ defined over $\Q$, then the first step is to embed $M$ as the Levi of a parabolic subgroup $P=MN$ in a reductive group $G$ over $\Q$ and study a functorial lift $\Pi$ of $\pi$ from $M$ to $G$. This is usually done via a construction involving Eisenstein series. One uses the vanishing hypothesis on the $L$-function of $\rho$ here in the course of studying this functorial lift.

The second step is to use the information obtained from the first step to deform (a critical $p$-stabilization of) $\Pi$ in a $p$-adic family $\mc{E}$ of automorphic representations of $G$ which is generically cuspidal. Assuming there is a suitable Langlands correspondence for $G$ as well, one should then construct a corresponding family $\rho_{\mc{E}}$ of Galois representations into the Langlands dual group $G^\vee$ of $G$.

The third step is to construct a certain lattice $\mc{L}$ in the family $\rho_{\mc{E}}$ and specialize it at the point in the family corresponding to $\Pi$. If this is done correctly, one should obtain a $p$-adic Galois representation $\rho_{\overline{\mc{L}}}$ which is related to $\rho$ up to semisimplification, but which is not semisimple. The difference between $\rho_{\overline{\mc{L}}}$ and $\rho$ will be in some way measured by a nontrivial cocycle $\sigma$ which lives in the Bloch--Kato Selmer group for $\rho$. We summarize this situation in a metaphorical diagram below.

\begin{figure}[h]
\centering
\includegraphics[scale=.115]{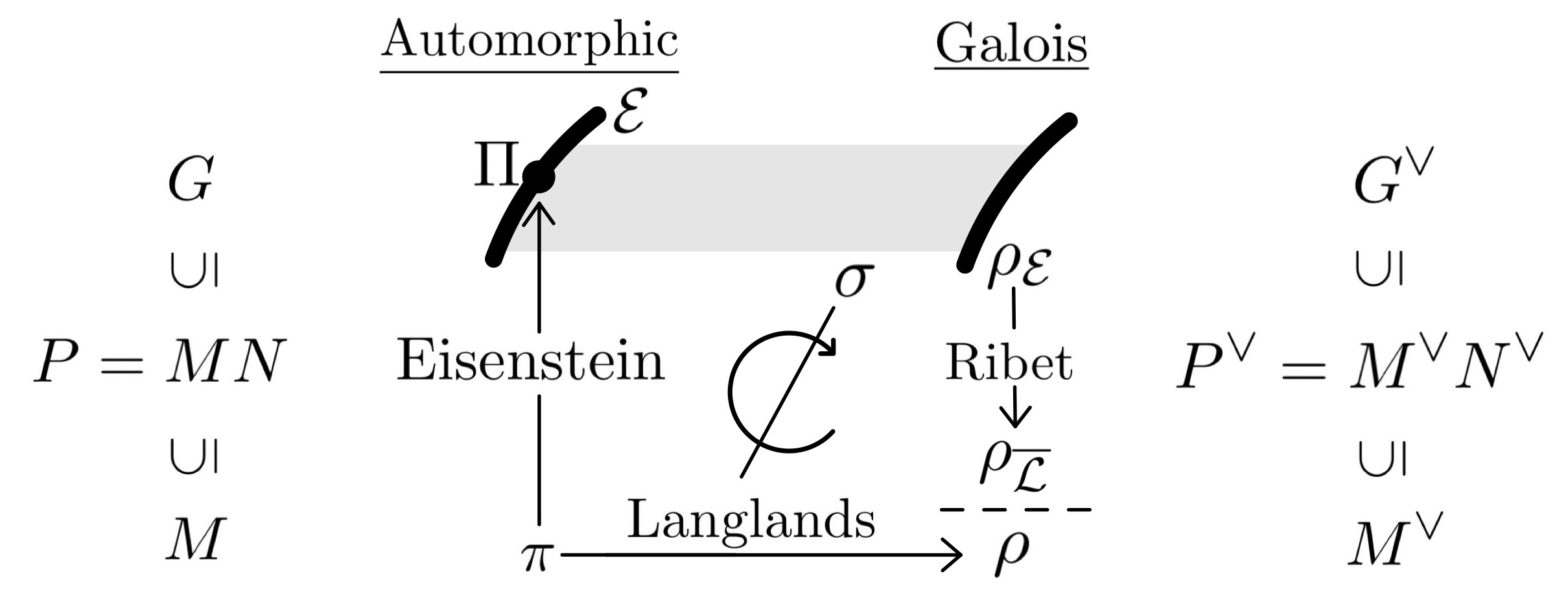}
\captionsetup{labelformat=empty}
\caption{The Skinner--Urban method}
\label{diagram}
\end{figure}

There are certain pieces of numerology concerning all the objects at play here which must hold in order for this method to have a chance of succeeding. Some of these concern the parabolic subgroup $P^\vee$ in $G^\vee$ corresponding to $P$, as well as its unipotent radical $N^\vee$, and we will return to this in our case in a moment after some historical remarks.

The arrow labelled ``Ribet'' in the diagram signifies the formation of the lattice $\mc{L}$. The first time the idea to use a lattice in this way to construct nice Galois cohomology classes appears in Ribet's paper \cite{ribet} where he proves the converse to Herbrand's theorem. This idea was then taken quite far and upgraded by Mazur--Wiles \cite{MWMC} and Wiles \cite{wilesmc} to prove the cyclotomic Iwasawa Main Conjecture for $GL_1$ over $\Q$ and then over totally real fields, respectively.

The Skinner--Urban method is a different upgrade of Ribet's general argument. It was first used by Skinner and Urban in \cite{SUgsp4} to prove that
\[H_f^1(\Q,\rho_F(\tfrac{2-k}{2}))\ne 0,\]
under the hypothesis that the eigenform $F$ is ordinary at $p$, that $k\geq 4$, and that the $L$-function $L(s,F)$ vanishes to odd order at the central point $\tfrac{k}{2}$. In \textit{loc. cit.}, the group $G$ of the method is $GSp_4$ and $P=MN$ is the Siegel parabolic.

The method also appears in \cite{SUunitary} where Skinner and Urban explain how to use this method to construct nontrivial elements in the Selmer groups over imaginary quadratic fields $K$ for Galois representations attached to certain automorphic representations of unitary groups of mixed signature $(a,b)$ defined with respect to $K$. Then in this case the big group $G$ becomes a unitary group of signature $(a+1,b+1)$.

As mentioned above, this paper is concerned with using the Skinner--Urban method to construct Selmer classes for the symmetric cube $\Sym^3(\rho_F)$. To do this, we pass from $M=GL_2$ to $G=G_2$, the split exceptional of rank $2$ over $\Q$. This group is defined by its root system, which has two simple roots, one long and one short, and we write $\alpha$ for the long root and $\beta$ for the short root. Then we let $P=P_\alpha$ be the maximal parabolic with Levi $M=M_\alpha$ containing the unipotent group corresponding to $\alpha$. Then indeed $M_\alpha\cong GL_2$.

We now describe the contents of the sections of this paper by explaining how they each contribute to the Skinner--Urban method for the symmetric cube. There are four sections and an appendix. Section \ref{secg2} sets up the necessary background on the group $G_2$ itself.

In Section \ref{seceiscoh} we begin the first step of the method. Let $\pi_F$ be the unitary automorphic representation of $GL_2(\A)$ attached to the eigenform $F$. The functorial lift we are interested in will be the Langlands quotient, denoted $\mc{L}_\alpha(\pi_F,1/10)$, of the unitary parabolic induction
\[\iota_{P_\alpha(\A)}^{G_2(\A)}(\pi_F\otimes\vert\det\vert^{1/2}).\]
The number $1/10$ is in this notation because, under the identification $M_\alpha\cong GL_2$, we have that $\vert\det\vert^{1/2}=\delta_{P_\alpha(\A)}^{1/10}|_{M_\alpha(\A)}$, where $\delta_{P_\alpha(\A)}$ is the modulus character of $P_\alpha(\A)$.

Let $\A_f$ be the ring of finite adeles. As we will explain below, we will need to locate every instance of the finite part $\mc{L}_\alpha(\pi_F,1/10)_f$ of our Langlands quotient, as a $G_2(\A_f)$-module, in the cohomology of the locally symmetric spaces attached to $G_2$ with respect to the local system defined by a particular irreducible algebraic representation of $G_2(\C)$. This task can be separated into two parts, namely we can separately locate $\mc{L}_\alpha(\pi_F,1/10)_f$ in cuspidal cohomology and in Eisenstein cohomology. Much of the appendix is devoted to doing the former, and we will return to this momentarily; Section \ref{seceiscoh} does the latter.

Let us state the main result on the location of $\mc{L}_\alpha(\pi_F,1/10)_f$ in Eisenstein cohomology. Let $\mf{g}_2$ be the complexified Lie algebra of $G_2$ and $K_\infty$ a maximal compact subgroup of $G_2(\R)$. If $E$ is an irreducible algebraic representation of $G_2(\C)$, then we can define the Eisenstein cohomology
\[H_{\Eis}^*(\mf{g}_2,K_\infty;E).\]
(See Definition \ref{defeiscoh} below.) Let $\lambda_0$ be the dominant weight
\[\lambda_0=\frac{k-4}{2}(2\alpha+3\beta),\]
and write $E_{\lambda_0}$ for the irreducible, finite dimensional representation of $G_2(\C)$ of highest weight $\lambda_0$. Finally, recall that one says two irreducible admissible $G_2(\A_f)$-modules are nearly equivalent if their local components at places $v$ are isomorphic for all but finitely many $v$. Then we have the following theorem, which is Theorem \ref{thmlqineiscoh} below.

\begin{introthm}
Assume that $k\geq 4$ and that
\[L(1/2,\pi_F,\Sym^3)=0.\]
Then there is a unique direct summand of the Eisenstein cohomology $H_{\Eis}^*(\mf{g}_2,K_\infty;E)$ which is isomorphic to the unitary induction
\[\iota_{P_\alpha(\A_f)}^{G_2(\A_f)}(\pi_{F,f}\otimes\vert\det\vert^{1/2}),\]
where $\pi_{F,f}$ denotes the finite part of $\pi_F$. It appears in middle degree $4$. Moreover, any subquotient of $H_{\Eis}^*(\mf{g}_2,K_\infty;E)$ nearly equivalent to $\mc{L}_\alpha(\pi_F,1/10)_f$ can only appear as a subquotient of this summand.
\end{introthm}

In order to show that one has located every instance of $\mc{L}_\alpha(\pi_F,1/10)_f$ in Eisenstein cohomology, one must not only produce every subquotient of Eisenstein cohomology isomorphic to $\mc{L}_\alpha(\pi_F,1/10)_f$, but also show that all other irreducible subquotients of Eisenstein cohomology are not isomorphic to $\mc{L}_\alpha(\pi_F,1/10)_f$. Both of these tasks will be carried out in Section \ref{subseclqincoh}. Before this, however, we will require a substantial amount of background on cohomology and automorphic forms.

The cohomological methods in this paper are entirely automorphic; the locally symmetric spaces themselves never play any direct role here, and we interpret all the cohomology spaces that occur in terms of $(\mf{g},K)$-cohomology. We can do this because results of Franke \cite{franke} have shown that the cohomology of such locally symmetric spaces can be described entirely in terms of the $(\mf{g},K)$-cohomology of the space of automorphic forms. Furthermore, the decomposition theorem of Franke--Schwermer \cite{FS} make it feasible to study the cohomology of the space of automorphic forms rather explicitly. Therefore in section \ref{subsecfsdecomp} we review the structural results of Franke and Franke--Schwermer, and in Section \ref{subsecindcoh} we carefully describe the cohomology of the pieces of the Franke--Schwermer decomposition. These two subsections are set up to work for any reductive group.

Let us assume from now on that $k\geq 4$ and $p\nmid N$. Fix a root $\alpha_p$ of the Hecke polynomial of $F$ at $p$. Having located $\mc{L}_\alpha(\pi_F,1/10)_f$ in Eisenstein cohomology (and, in the Appendix, in the cuspidal spectrum) we continue to the second step of the Skinner--Urban method in our case and, in Section \ref{secdeform}, $p$-adically deform a critical $p$-stabilization $\Pi_F^{(p)}$ of $\mc{L}_\alpha(\pi_F,1/10)_f$ in a generically cuspidal family of $p$-stabilizations of cohomological automorphic representations of $G_2(\A)$ of varying weight. In the absence of a $G_2$-Shimura variety, our options for doing this are limited to the methods present in the paper \cite{urbanev} of Urban on eigenvarieties for reductive groups with discrete series, of which $G_2$ is one.

Making a $p$-adic deformation of a noncritical $p$-stabilization of an automorphic representation is not too difficult with Urban's methods, but when the $p$-stabilization is instead critical, this becomes significantly harder. And we do need to use a critical $p$-stabilization of $\mc{L}_\alpha(\pi_F,1/10)_f$ in order to have in place certain pieces of numerology involving slopes on the Galois side later.

Now the techniques in Urban's paper which allow us to make this $p$-adic deformation go through his theory of multiplicities. Urban defines certain local systems on the locally symmetric spaces of a reductive group $G$ with discrete series whose cohomology contains a subspace which can be considered a space whose constituents are overconvergent $p$-adic automorphic representations of $G$. If a $p$-stabilization of an automorphic representation of $G(\A)$ appears in this space with a nonzero multiplicity, then it can be $p$-adically deformed.

There is furthermore a variant of this notion of multiplicity which allows us to detect when a $p$-stabilization of an automorphic representation of $G(\A)$ deforms in a generically cuspidal family. On the one hand, this cuspidal overconvergent multiplicity, as we call it, can be expressed as a difference between the overconvergent multiplicity just described and other overconvergent ``Eisenstein'' multiplicities coming from smaller Levi subgroups. On the other hand, Urban also relates the ``classical'' multiplicity of an automorphic representation $\pi$, namely that of $\pi$ in the cohomology of the locally symmetric spaces attached to $G$, to the (noncuspidal) overconvergent multiplicity of a $p$-stabilization of $\pi$ and various ``Weyl twists'' of this $p$-stabilization.

The results of Section \ref{seceiscoh}, along with the location of $\mc{L}_\alpha(\pi_F,1/10)_f$ in the cuspidal spectrum of $G_2$ as in the appendix, give the classical multiplicity of $\Pi_F^{(p)}$, which we relate to the overconvergent multiplicities just mentioned. Then we will relate these overconvergent multiplicities to certain cuspidal overconvergent multiplicities of $\Pi_F^{(p)}$ by computing explicitly the Eisenstein multiplicities. Compiling these computations gives that the cuspidal overconvergent multiplicity of $\Pi_F^{(p)}$ is at least $3$, as long as the sign $\epsilon(1/2,\pi_F,\Sym^3)$ of the symmetric cube functional equation equals $-1$, and as long as the weight $k$ is sufficiently large with respect to the $p$-adic valuation $v_p(\alpha_p)$ of the number $\alpha_p$.

Let us comment on these two hypotheses which appear at the end of this last statement. The first hypothesis, which is equivalent to for the hypothesis of odd order vanishing in Theorem \ref{intthmmain}, is there since if $\epsilon(1/2,\pi_F,\Sym^3)=-1$, then Conjecture \ref{conjmult} (b) shows that $\mc{L}_\alpha(\pi_F,1/10)_f$ appears as the finite part of a cuspidal automorphic representation exactly once in the cuspidal spectrum of $G_2$, and that this representation is in the discrete series at infinity. Thus it contributes to this multiplicity. Otherwise (see Conjecture \ref{conjmult} (a)) if $\epsilon(1/2,\pi_F,\Sym^3)=+1$, then $\mc{L}_\alpha(\pi_F,1/10)_f$ appears in the cuspidal spectrum still exactly once, but as the finite part of a cuspidal automorphic representation which is nontempered at infinity, and appears in cohomology in degrees $3$ and $5$. Thus the classical multiplicity, and hence also the overconvergent cuspidal multiplicity, of $\Pi_F^{(p)}$ drops by $3$, and we can no longer conclude that it is positive. We expect that our methods are in general insufficient for $p$-adically deforming $\Pi_F^{(p)}$ when the symmetric cube $L$-function of $\pi_F$ vanishes to even order.

The second hypothesis, that $k$ is sufficiently large with respect to $\alpha_p$, is there to ensure that the slope of $\Pi_F^{(p)}$, although critical, is minimally so, and in such a way that the computation of multiplicities is not impossible. Fortunately we do have methods to overcome this hypothesis. We do this by putting $F$ in a Coleman family, hence also deforming $\Pi_F^{(p)}$ along the axis of weights which are multiples of the weight $\lambda_0$ fixed above. We show that the sign of the symmetric cube functional equation is locally constant in this family, and hence that the higher weight members of this family give rise to Langlands quotients with $p$-stabilizations analogous to our representation $\Pi_F^{(p)}$, but with cuspidal overconvergent multiplicity at least $3$. This implies (by Lemma \ref{lemintsummand} below) that the cuspidal overconvergent multiplicity of $\Pi_F^{(p)}$ is also positive. Thus we get the following; See Theorem \ref{thmpadicdef} for the precise meaning of this.

\begin{introthm}
\label{intthmpadicdef}
Assume $k\geq 4$, $p\nmid N$, and $\epsilon(1/2,\pi_F,\Sym^3)=-1$. Then under Conjecture \ref{conjmult} (b), the $p$-stabilization $\Pi_F^{(p)}$ admits a $p$-adic deformation in a generically cuspidal family of cohomological automorphic representations of $G_2(\A)$.
\end{introthm}

We then use the theory of types to study the local properties of this deformation. This study is essential to showing that, on the Galois side, the family of Galois representations we obtain is no more ramified at places $\ell|N$ than $\rho_F$ itself. This is where the assumptions that $4\nmid N$ and $9\nmid N$ come in to play; by results of Fintzen \cite{fintzen}, we have a satisfactory theory of types for $G_2(\Q_\ell)$, but only at primes $\ell$ not dividing the order of the Weyl group of $G_2$, which is $12$. However, if $N$ is divisible exactly once by a prime $\ell$, then $\pi_F$ is an unramified twist of Steinberg at $\ell$, and we have a way to circumvent having to use the general theory of types in this case.

We take this opportunity to point out that a similar calculation of multiplicities to this one was sketched in \cite[\S 5.5]{urbanev} for a situation analogous to this one, but for $GSp_4$ in place of $G_2$. There are two mistakes that calculation, however, and they cancel each other. The first is that the multiplicities of representations for $GSp_4$ which Urban was working with should be defined with respect to the connected component of the identity of the maximal compact subgroup for $GSp_4(\R)$, and the Siegel Eisenstein multiplicities should be defined with respect to the full maximal compact subgroup of $GL_2(\R)$. The second is that there is no Eisenstein class in degree $2$, as is claimed in \textit{loc. cit.} In order to make these calculations rigorous, one also needs to make a computation analogous to those in Section \ref{thmlqineiscoh} of this paper. We wish to assure the reader that this is all doable.

We also point out that there is an error in \cite{urbanev} with the general definition of the cuspidal overconvergent multiplicities. We explain this error in Section \ref{subsecev}, and we explain how it affects (or, more accurately, does not affect) the case of $G_2$, and why the main results of \textit{loc. cit.} are still valid. Forthcoming work of Urban and the author of the present article also corrects this error in general, building off work of Gulotta \cite{gulotta}.

Next, in Section \ref{secgalreps}, we proceed to the third step of the Skinner--Urban method in our case and use Theorem \ref{intthmpadicdef} to construct a family of Galois representations. The group $G_2$ is self dual and has a $7$-dimensional representation $R_7$, so we expect Langlands functoriality to allow us to lift certain automorphic representations from $G_2$ to $GL_7$. So here we assume Conjecture \ref{conjliftings}, which will give us these functorial lifts to $GL_7$ of the cuspidal members of our $p$-adic family with regular weight. We show using results of Chenevier \cite{chen} that the Galois representations attached to these cuspidal automorphic representations factor through $G_2(\overline\Q_p)$, and we glue these together using the theory of pseudocharacters of V. Lafforgue \cite{laff}.

The result is that we get the following objects:
\begin{itemize}
\item A $G_2$-pseudocharacter $\Theta$ over the ring $\mc{O}(\mf{Z})$ of analytic functions on an affinoid rigid analytic curve $\mf{Z}$;
\item A Galois representation $\rho_{\Theta}$ into $G_2(\overline{F(\mf{Z})})$, where $F(\mf{Z})=\Frac(\mc{O}(\mf{Z}))$;
\item The composition
\[R_7\circ\rho_{\Theta}:G_\Q\to GL_7(\mc{O}(\mf{Z})),\]
which is a continuous Galois representation.
\end{itemize}

The curve $\mf{Z}$ sits over a curve in the $p$-adic family coming from Theorem \ref{intthmpadicdef} and contains a point $y_0$ corresponding to $\Pi_F^{(p)}$. The semisimplification of the specialization of $R_7\circ\rho_{\Theta}$ at $y_0$ is given by
\[(R_7\circ\rho_{\Theta})^{\sss}\cong\rho_F(-\tfrac{k-2}{2})\oplus\Ad(\rho_F)\oplus\rho_F(-\tfrac{k}{2}),\]

We then construct a lattice $\mc{L}$ in the Galois representation $R_7\circ\rho_{\Theta}$. This construction differs is an essential way from the lattice constructions in \cite{SUgsp4} and \cite{SUunitary}; if we were to follow these works, we would construct $\mc{L}$ so that its specialization $\overline{\mc{L}}$ at the point $y_0$ has unique irreducible quotient $\Ad(\rho_F)$. There are two issues with this. The first is that we don't know at this point that $R_7\circ\rho_{\Theta}$ is generically irreducible. The second is that, even if we did, a successful construction of $\overline{\mc{L}}$ could only give us a nontrivial extension of $\Ad(\rho_F)$ by $\rho_F(-\tfrac{k-2}{2})$, and we have
\[\ext_{G_\Q}^1(\rho_F(-\tfrac{k-2}{2}),\Ad(\rho_F))=H^1(\Q,\Sym^3(\rho_F)(\tfrac{4-3k}{2}))\oplus H^1(\Q,\rho_F(-\tfrac{k-2}{2})).\]
Thus it may be the case that this construction only gives us a class in the Selmer group of $\rho_F(-\tfrac{k-2}{2})$.

So instead we construct the lattice $\mc{L}$ so that $\overline{\mc{L}}$ has unique irreducible quotient $\rho_F(-\tfrac{k}{2})$ if $R_7\circ\rho_{\Theta}$ is generically irreducible, and otherwise so that it has irreducible quotients $\rho_F(-\tfrac{k}{2})$ and $\Ad(\rho_F)$. We show that the only possibilities for the matricial shape of $\overline{\mc{L}}$ are:
\begin{equation*}
\overline{\mc{L}}\sim\pmat{\rho_F(-\tfrac{k-2}{2}) & *_3 & *_2\\
0 & \Ad(\rho_F) & *_1\\
0 & 0 & \rho_F(-\tfrac{k}{2})},
\end{equation*}
with $*_1$ and $*_2$ nontrivial, or
\begin{equation*}
\overline{\mc{L}}\sim\pmat{\Ad(\rho_F) & *_3 & *_2\\
0 & \rho_F(-\tfrac{k-2}{2}) & *_1\\
0 & 0 & \rho_F(-\tfrac{k}{2})},
\end{equation*}
again with $*_1$ and $*_2$ nontrivial, or
\begin{equation*}
\overline{\mc{L}}\sim\pmat{\rho_F(-\tfrac{k-2}{2}) & 0 & *_2\\
0 & \Ad(\rho_F) & 0\\
0 & 0 & \rho_F(-\tfrac{k}{2})},
\end{equation*}
with $*_2$ nontrivial. The first two possibilities can occur only when $R_7\circ\rho_{\Theta}$ is generically irreducible, and the last can only occur when it is not. Here we use that $F$ is not CM so that $\rho_F$ has big image.

At this point we need two things to happen, namely that the first of these three possibilities occurs, and that $\overline{\mc{L}}$ factors through $G_2(\overline\Q_p)$. Then $*_1$ (or equivalently, $*_3$) would give the desired extension, and the factorization through $G_2(\overline\Q_p)$ would further rule out the possibility that $*_1$ (or $*_3$) could contribute to the Selmer group of $\rho_F(-\tfrac{k-2}{2})$. We show both of these two things at the same time, and we do this by studying an alternating trilinear form $\langle\cdot,\cdot,\cdot\rangle$ on $\overline{\mc{L}}$ which comes from the fact that $\rho_\Theta$ is a representation into $G_2$, and that the $G_2$-subgroups of $GL_7$ are those preserving sufficiently nondegenerate alternating trilinear forms. The key point here is that, if $\langle\cdot,\cdot,\cdot\rangle$ were degenerate then we could use $*_2$ to construct a nontrivial extension of $\overline\Q_p$ by $\overline\Q_p(1)$. This extension is unramified at $\ell\ne p$ by the local study of the $p$-adic family we make at the end of Section \ref{secdeform}, and it is crystalline at $p$ by a lemma of Kisin on the interpolation of crystalline periods. This would therefore give a nontrivial Selmer element in
\[H_f(\Q,\overline\Q_p(1)),\]
which is the trivial group, contradiction. 

Thus we can use $*_3$ to construct the desired Selmer group element in $H_f^1(\Q,\Sym^3(\rho_F)(\tfrac{4-3k}{2}))$ To show that this element satisfies the Selmer conditions, we use the same results as cited above, and we also prove the following result, which we wish to highlight here:

\begin{introthm}
Assuming $k\geq 4$, $p\nmid N$, and that the Hecke polynomial of $F$ at $p$ has simple roots, we have that any extension $E$ of Galois representations,
\[0\to \rho_F(1)\to E\to \rho_F\to 0\]
is semistable at $p$.
\end{introthm}

See Section \ref{subsecpadichodge} for the proof of this theorem. This gets us started in showing the crystallinity of $*_3$

Showing the crystallinity of $*_3$ also requires a trick of switching the root $\alpha_p$ of the Hecke polynomial of $F$ at $p$, and this is another place where we assume this polynomial has simple roots. We remark that this assumption is conjectured to always be true. See, for example, \cite{CE}, where this assumption is discussed and shown to follow from the Tate conjecture.

This completes our sketch of the proof of Theorem \ref{intthmmain}. Now we briefly discuss the Conjectures \ref{conjmult} and \ref{conjliftings}. The appendix of this paper is devoted to justifying our believe in these conjectures, and in particular showing how they follow from Arthur's conjectures. 

As mentioned above, Conjecture \ref{conjmult} is concerned with the $\mc{L}_\alpha(\pi_F,1/10)_f$-isotypic component of the cuspidal spectrum of $G_2$. As explained in \cite{gangurl} and recalled in Section \ref{seccuspmult}, there is a global Arthur parameter $\psi_F$ for $G_2$ whose associated local packets at any finite place $v$ contain the $v$-component $\mc{L}_\alpha(\pi_F,1/10)_v$ of $\mc{L}_\alpha(\pi_F,1/10)$. If $\pi_F$ is unramified at $v$, the local packet at $v$ is a singleton, and otherwise it has two elements.

As we will explain in Section \ref{seccuspmult}, Arthur's multiplicity formula then implies that $\mc{L}_\alpha(\pi_F,1/10)_f$ occurs as the finite component of exactly one automorphic representation $\Pi$ in the discrete spectrum of $G_2$, and that the archimedean component $\Pi_\infty$ of $\Pi$ is an element of the local Arthur packet at $\infty$ associated with the component $\psi_{F,\infty}$ of $\psi_F$ at $\infty$. Moreover, if $L(1/2,\pi_F,\Sym^3)=0$, then $\Pi$ is not residual and is thus cuspidal.

Now the local packet associated with $\psi_{F,\infty}$ should contain two elements, call them $\Pi_\infty^+$ and $\Pi_\infty^-$, and the multiplicity formula tells us in this case that
\[\Pi_\infty=\Pi_\infty^{\epsilon(1/2,\pi_F,\Sym^3)}.\]
What is left to do then is to compute this packet and show that $\Pi_\infty^+$ is nontempered and cohomological of weight $\lambda_0$ in degrees $3$ and $5$, and that $\Pi_\infty^-$ is discrete series of weight $\lambda_0$.

The work of Adams--Johnson \cite{adjo} explains how to compute such packets using cohomological induction. We review this in Section \ref{subsecadjo} and in Section \ref{subsecajpacket} we prove the following theorem.

\begin{introthm}
The Arthur packet associated with $\psi_{F,\infty}$, as defined via cohomological induction by Adams--Johnson \cite{adjo}, consists of $\Pi_\infty^+=\mc{L}_\alpha(\pi_F,1/10)_\infty$ and another representation $\Pi_\infty^-$ which is the quaternionic discrete series of weight $k/2$ (see \cite{ggs}).
\end{introthm}

We remark that this result has already found use in work of R. Dalal on counting quaternionic $G_2$-automorphic forms; see \cite{dalal}.

Now the other conjecture on which Theorem \ref{intthmmain} depends, namely Conjecture \ref{conjliftings}, predicts the existence of functorial lifts of cohomological, cuspidal automorphic representations $\Pi$ of $G_2(\A)$ of regular weight to cohomological, discrete automorphic representations of $GL_7(\A)$ of regular weight. In section \ref{subsecliftings}, we explain how the existence of such lifts follow from Arthur's conjectures.

The main point is that, if $\psi$ is an Arthur parameter for such a $\Pi$, then the fact that $\Pi$ is cohomological of regular weight forces the restriction of $\psi$ to $SL_2(\C)$ to be trivial. We show this by classifying the Arthur parameters for $G_2(\R)$ whose packets can contain cohomological representations in Section \ref{subseccohparams}. Then we consider $R_7\circ\psi$, which therefore just comes from a tempered Langlands parameter for $GL_7$, and we take our lifting to be the one defined by this Langlands parameter.

We remark that Conjecture \ref{conjmult} seems to be within reach, and should follow (nontrivially) from the forthcoming work of Baki\'c and Gan on theta correspondence for $U_3\times G_2$. It is also plausible that Conjecture \ref{conjliftings} is within reach; as suggested by M. Harris, one could try to show that the theta lift of such a $\Pi$ as above to $PGSp_6$ is nonvanishing directly by computing a particular Fourier coefficient of it in terms of $L$-functions. But this is certainly more speculative at this point.

One other remark about Theorem \ref{intthmmain} that we would like to make is that, if the level of $F$ is $1$, then actually most of the hypotheses on $F$ in that theorem are vacuous; indeed, in this case we have $k\geq 12$, $\epsilon(1/2,\pi_F,\Sym^3)=-1$ always, $F$ is not CM, and the Hecke polynomial of $F$ at $p$ has simple roots (because $a_p$ is a rational integer in this case). Thus, for $F$ a cuspidal eigenform with trivial nebentypus, we get nonvanishing of the symmetric cube Selmer group assuming only Conjectures \ref{conjmult} (b) and \ref{conjliftings} (a)-(c). (We do not need (d) and (e) of the latter conjecture in this case since those points are used to study liftings at bad places.)

Finally, we note that there has been a lot of interest lately in symmetric cube Selmer groups. For example, there is the recent work of Haining Wang \cite{wang} and the work of Loeffler--Zerbes \cite{LZSym3}. Both of these papers work in the Euler system direction, establishing upper bounds on the ranks of the symmetric cube Selmer groups that they study, as opposed to this paper which works in the opposite direction.

\subsection*{Acknowledgements}
The present article grew out of, and expands upon, the work done in my Ph.D. thesis \cite{thesis}. My thesis advisor was Eric Urban, and a great debt of gratitude is owed to him for his advice and suggestions.

Many thanks are also due to Jeff Adams, who originally suggested that the archimedean Arthur packet studied in Section \ref{subsecajpacket} could be constructed via cohomological induction.

In addition to these two people, I would also like to thank Rapha\"el Beuzart-Plessis, Johan de Jong, Wee Teck Gan, Dan Gulotta, Michael Harris, Herv\'e Jacquet, Chao Li, Shizhang Li, Stephen D. Miller, Aaron Pollack, Chris Skinner, and Shuai Wang for helpful conversations.

This material is based upon work supported by the National Science Foundation under Award DMS-2102473.

\section*{Notation and conventions}
The following conventions will be used throughout this entire paper.

\subsubsection*{Groups and Lie algebras}
In general, our convention is to use uppercase roman letters to denote groups over $\Q$, such as $G$, and to use the corresponding lowercase fraktur letters to denote complex Lie algebras. So for example, $\mf{g}$ will always denote the complexified Lie algebra of the $\Q$-group $G$. The only exceptions to this convention occur in the appendix, specifically in Section \ref{subsecadjo}; see that section for the notation used there.

When working with a group $G$, we will often fix a parabolic subgroup $P$ of $G$ along with a Levi decomposition $P=MN$. In this decomposition, $M$ will always denote the Levi factor and $N$ the unipotent radical. If we have another parabolic subgroup with fixed Levi decomposition, then we use subscripts on the notation for its fixed Levi factor and its unipotent radical to distinguish them from those of $P$; so if $Q$ is another parabolic subgroup, we will write $Q=M_Q N_Q$ for its Levi decomposition.

For any parabolic $Q$ as above, the notation $A_Q$ will denote the maximal $\Q$-split torus in the center of the Levi $M_Q$ of $Q$. This applies in particular to $P$ and $G$; we use $A_G$ to denote the maximal $\Q$-split torus in the center of $G$, and $A_P$ that of $M$.

Now we have the complexified Lie algebras $\mf{g}$, $\mf{p}$, $\mf{q}$, $\mf{m}$, $\mf{m}_Q$, $\mf{n}$, $\mf{n}_Q$, $\mf{a}_P$, and $\mf{a}_Q$ of, respectively, $G$, $P$, $Q$, $M$, $M_Q$, $N$, $N_Q$, $A_P$, and $A_Q$. We let $\mf{g}_0=[\mf{g},\mf{g}]$, the self-commutator of $\mf{g}$, and more generally, we write $\mf{m}_{Q,0}=[\mf{m}_Q,\mf{m}_Q]$, or $\mf{m}_0=[\mf{m},\mf{m}]$. We also write $\mf{q}_0=\mf{q}\cap\mf{g}_0$ and $\mf{a}_{Q,0}=\mf{a}_Q\cap\mf{g}_0$, and similarly for $\mf{p}_0$ and $\mf{a}_{P,0}$. Then there are decompositions
\[\mf{q}=\mf{m}_{Q,0}\oplus\mf{a}_Q\oplus\mf{n}_Q,\]
and
\[\mf{q}_0=\mf{m}_{Q,0}\oplus\mf{a}_{Q,0}\oplus\mf{n}_Q.\]

We will always write $\rho_Q$ for the character $\rho_Q:\mf{a}_{Q,0}\to\C$ given by
\[\rho_Q(X)=\tr(\ad(X)|\mf{n}_Q),\qquad X\in\mf{a}_{Q,0},\]
and similarly for $\rho_P$.

Most of the time we will be working with the group $G_2$, which we introduce in Section \ref{subsecg2str}. The objects associated with this group have various pieces of notation attached to them as well, and we refer to that section for those notations.

\subsubsection*{Points of groups}
When $v$ is a place of $\Q$, we write $\Q_v$ for the completion of $\Q$ at $v$. Then $\R=\Q_\infty$. The group of $\Q_v$-points of any affine algebraic group over $\Q$ is always given the usual topology induced from $\Q_v$.

We write $\A$ for the adeles of $\Q$ and $\A_f$ for the finite adeles. If $v$ is a fixed finite place of $\Q$, then $\A_f^v$ will denote the finite adeles away from $v$. The groups of $\A$-points, $\A_f$-points, or $\A_f^v$-points of any affine algebraic group over $\Q$ are given their standard topologies.

When a parabolic $P=MN$ in a group $G$ is fixed as above, we will often consider the associated height function $H_P$. This is a function
\[H_P:G(\A)\to\mf{a}_{P,0}.\]
To define it, we must fix a maximal compact subgroup $K\subset G(\A)$. We assume $K=K_{f,\mr{max}} K_\infty$ where $K_\infty$ is a fixed maximal compact subgroup of $G(\R)$, where $K_{f,\mr{max}}=\prod_{v<\infty} K_v$, and where the groups $K_v$ are maximal compact subgroups of $G(\Q_v)$. We moreover assume $K$ to be in good position with respect to a fixed minimal parabolic inside $P$. In particular, the Iwasawa decomposition holds for $P(\A)$ and $K$.

Write $\langle\cdot,\cdot\rangle$ for the natural pairing
\[\langle\cdot,\cdot\rangle:\mf{a}_{P,0}\times\mf{a}_{P,0}^\vee\to\C\]
given by evaluation, where $\mf{a}_{P,0}^\vee=\hom_{\C}(\mf{a}_{P,0},\C)$. Write $X^*(M)$ for the group of algebraic characters of $M$. Then $H_P$ is defined first on the subgroup $M(\A)$ by requiring
\[e^{\langle H_P(m),d\Lambda\rangle}=\vert\Lambda(m)\vert,\qquad m\in M(\A),\,\,\Lambda\in X^*(M),\]
where $d\Lambda$ denotes the restriction to $\mf{a}_{P,0}$ of the differential at the identity of the restriction of $\Lambda$ to $A_P(\R)$, and $\vert\cdot\vert$ is usual the adelic absolute value. Then $H_P$ is defined in general by declaring it to be left invariant with respect to $N(\A)$ and right invariant with respect to $K$.

If $R$ is one of the rings $\Q_v$, $\A$, or $\A_f$, we use the notation $\delta_{P(R)}$ to denote the modulus character of $P(R)$, and similarly for other parabolics.

\subsubsection*{Automorphic representations}
When $G$ is a reductive $\Q$-group, we take the point of view that an ``automorphic representation" of $G(\A)$ is (among other things) an irreducible object in the category of admissible $G(\A_f)\times(\mf{g},K_\infty)$-modules, where $K_\infty$ is, like above, a maximal compact subgroup in $G(\R)$. We often even view automorphic representations as $G(\A_f)\times(\mf{g}_0,K_\infty)$-modules by restriction. We let $\mc{A}(G)$ denote the space of all automorphic forms on $G(\A)$.

If $\Pi$ is an automorphic representation of $G(\A)$ and $v$ is a place of $\Q$, we will denote by $\Pi_v$ the local component of $\Pi$ at $v$. If $v$ is finite, then this is an irreducible admissible representation of $G(\Q_v)$, and if $v=\infty$, then this is an irreducible admissible $(\mf{g},K_\infty)$-module. We also let $\Pi_f$ denote the associated representation of $G(\A_f)$, so that $\Pi\cong\Pi_f\otimes\Pi_\infty$.

If $P-MN$ is a parabolic subgroup of $G$ and $\pi$ an automorphic representation of $M(\A)$, then we denote the nonunitary parabolic induction of $\pi$ to $G$ along $P$ by $\Ind_{P(\A)}^{G(\A)}(\pi)$, and the unitary parabolic induction by $\iota_{P(\A)}^{G(\A)}(\pi)$. So
\[\iota_{P(\A)}^{G(\A)}(\pi)=\Ind_{P(\A)}^{G(\A)}(\pi\otimes\delta_{P(\A)}^{1/2}).\]
More generally, if $\lambda\in\mf{a}_P^\vee$, we write
\[\iota_{P(\A)}^{G(\A)}(\pi,\lambda)=\iota_{P(\A)}^{G(\A)}(\pi\otimes e^{\langle H_P(\cdot),\lambda\rangle})=\Ind_{P(\A)}^{G(\A)}(\pi\otimes e^{\langle H_P(\cdot),\lambda+\rho_P\rangle})\]
We similarly write $\Ind_{P(\A_f)}^{G(\A_f)}$ and $\iota_{P(\A_f)}^{G(\A_f)}$ for the corresponding functors on smooth admissible representations of $M(\A_f)$, and $\Ind_{P(\Q_v)}^{G(\Q_v)}$ and $\iota_{P(\Q_v)}^{G(\Q_v)}$ for their local analogues.

\subsubsection*{Galois theory}
We fix a prime $p$ throughout this paper, and an isomorphism $\overline\Q_p\cong\C$. Though we note that $p$ will only begin to play a role in Section \ref{secdeform}.

The $p$-adic absolute value on $\overline{\Q}_p$ will always be denoted by $\vert\cdot\vert$ and normalized so that $\vert p\vert=p^{-1}$.

We will write $G_\Q$ for the absolute Galois group of $\Q$, and for any place $v$ of $\Q$, we will similarly write $G_{\Q_v}$ for the absolute Galois group of $\Q_v$. If $v$ is finite, we always view $G_{\Q_v}$ as a subgroup of $G_\Q$ via by fixing a decomposition group at $v$.

For us, $p$-adic Galois representation will always be into the $\overline\Q_p$-points of a fixed algebraic group. Moreover, they will always be continuous. If a fixed Galois representation, like the one attached to an eigenform, has, a priori, values over a finite extension of $\Q_p$, then it will be our convention to change the base to $\overline\Q_p$.

We will consider Fontaine's functors $D_{\dR}$, $D_{\st}$, and $D_{\crys}$ of, respectively, de Rham, semistable, and crystalline periods. Correspondingly to the above convention about Galois representations, all $p$-adic Hodge theoretic constructions we consider in this paper will be considered as $\overline\Q_p$-linear objects. Therefore, given a Galois representation $V$ of $G_{\Q_p}$ over $\overline\Q_p$, the spaces $D_{\dR}(V)$, $D_{\st}(V)$, and $D_{\crys}(V)$ will be considered as $\overline\Q_p$-vector spaces with extra structure.

The conventions we use in this paper are geometric. So for a prime $\ell$, $\Frob_\ell$ denotes a geometric Frobenius element of the Galois group $G_\Q$. The Hodge--Tate weight of the cyclotomic character is $-1$. Given a semistable representation $V$ of $G_{\Q_p}$ over $\overline\Q_p$, the crystalline Frobenius $\phi$ on $D_{\st}(V)$ will also be geometric. The filtrations on $D_{\dR}(V)$, $D_{\st}(V)$, and $D_{\crys}(V)$ do not change.

For example, Let $V$ be the $2$-dimensional Galois representation attached to a modular eigenform of weight $k$ and level $N$; for a prime $\ell\nmid N$, the trace of $\Frob_\ell^{-1}$ on $V$ is the $\ell$th Fourier coefficient of the eigenform. If $p\nmid N$, then $V$ is crystalline at $p$ and $D_{\crys}(V)$ has Hodge--Tate weights $0$ and $-(k-1)$. The filtration $\Fil^i$ for $D_{\dR}(V)$ falls at $i=0$ and $i=k-1$. Both the Newton and Hodge polygons lie on or below the horizontal axis.

\subsubsection*{Rings of analytic functions}
Given a rigid analytic space $\mf{V}$ over $\Q_p$ or a finite extension thereof, we let $\mc{O}(\mf{V})$ denote the ring of analytic functions on $\mf{V}$. If $\mf{V}$ is affinoid, we let $\mc{O}(\mf{V})^\circ$ denote the subring of $\mc{O}(\mf{V})$ of analytic functions whose evaluations at all points in $\mf{V}$ are bounded above in absolute value by $1$. If $\mf{V}$ is affinoid and reduced, we view $\mc{O}(\mf{V})$ with its usual $\Q_p$-Banach space topology. Then $\mc{O}(\mf{V})^\circ$ is an open ball in $\mc{O}(\mf{V})$.

\subsubsection*{Duals}
We use the symbol $(\cdot)^\vee$ in various ways to denote duality. If $\mf{a}$ is an abelian Lie algebra, we write $\mf{a}^\vee=\hom_\C(\mf{a},\C)$. If $R$ is a complex representation of a group, then $R^\vee$ is the usual dual representation over $\C$. Similarly, if $\rho$ is an $\ell$-adic Galois representation, then $\rho^\vee$ is the usual dual representation over $\overline\Q_p$. If $G$ is our reductive $\Q$-group, then $G^\vee(\C)$ or $G^\vee(\overline\Q_p)$ will denote the dual group over either of the algebraically closed fields $\C$ or $\overline\Q_p$, respectively.

\section{The group $G_2$}
\label{secg2}
We begin by collecting various facts about the group $G_2$ itself and consolidating them here for the convenience of the reader. Section \ref{subsecg2str} explains various structural aspects of $G_2$ involving its root system, its parabolic subgroups, and its real points. Section \ref{subsecalt3forms} briefly studies the connection between $G_2$ and generic alternating $3$-forms; the material of that section will only play a role in Sections \ref{subseclattice} and \ref{subsecmainthm}.

\subsection{Structure of the group $G_2$}
\label{subsecg2str}
We define $G_2$ to be the split simple group over $\Q$ with Dynkin diagram as in Figure \ref{figg2dynkin}. Fixing a maximal $\Q$-split torus $T$ in $G_2$, we choose a long simple root $\alpha$ and a short simple root $\beta$, as notated in the Dynkin diagram. The group $G_2$ has trivial center.
\begin{figure}[h]
\centering
\includegraphics[scale=.2]{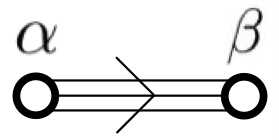}
\caption{The Dynkin diagram of $G_2$}
\label{figg2dynkin}
\end{figure}

It is worth noting that $G_2$ does not have a very nice matricial definition, at least not one that is as nice as for, say, the group $Sp_4$. There is, however, a faithful representation of $G_2$ into $GL_7$ that we will make use of, and it is possible to characterize the image of that representation, up to conjugation, in terms of the preservation of certain alternating $3$-forms, as we will do in Section \ref{subsecalt3forms}. But it is hard to make that characterization explicit in terms of matrices. Consequently, we will mostly study $G_2$ from the point of view of its root system, which we discuss now.

\subsubsection*{The root lattice}
The root lattice of $G_2$ looks as in Figure \ref{figg2chamber}. There, the dominant chamber is shaded.
\begin{figure}[h]
\centering
\includegraphics[scale=.25]{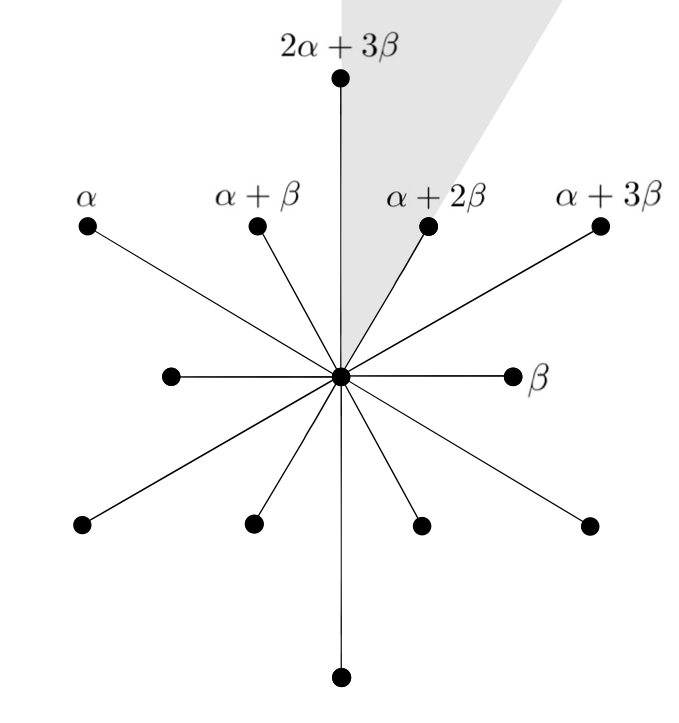}
\caption{The root lattice of $G_2$}
\label{figg2chamber}
\end{figure}

Write $\Delta$ for the set of roots of $T$ in $G_2$, and write $\Delta^+$ for the subset of positive roots. So we have
\[\Delta^+=\{\alpha,\beta,\alpha+\beta,\alpha+2\beta,\alpha+3\beta,2\alpha+3\beta\}.\]

One nice feature of $G_2$ is that the $\Z$-span of the root lattice equals the character group of $T$:
\[X^*(T)=\Z\alpha\oplus\Z\beta.\]
Since the Cartan matrix of $G_2$ has determinant $1$, an analogous fact holds for the cocharacter group.

\subsubsection*{Parabolic subgroups}
Let $B$ denote the standard Borel subgroup of $G_2$, defined with respect to $\Delta^+$. We write $B=TU$ for its Levi decomposition. Besides $B$, there are two other proper standard parabolic subgroups, and they are maximal. Let $P_\alpha$ denote the standard parabolic subgroup whose Levi contains $\alpha$, and write $P_\alpha=M_\alpha N_\alpha$ for its Levi decomposition. Similarly define $P_\beta=M_\beta N_\beta$.

For $\gamma\in\Delta$ a root, write
\[\mathbf{x}_\gamma:\mb{G}_a\to G_2\]
for the corresponding root group homomorphism, where $\mb{G}_a$ is the additive group scheme. The Levis $M_\alpha$ and $M_\beta$ are both isomorphic to $GL_2$. We write
\[i_\alpha: GL_2\to M_\alpha\quad\textrm{and}\quad i_\beta: GL_2\to M_\beta\]
for the isomorphisms which send the upper triangular matrix $\sm{1&a\\ 0&1}$ in $GL_2$ to the element $\mathbf{x}_\alpha(a)$ and $\mathbf{x}_\beta(a)$, respectively. We also often write
\[\pmat{a&b \\ c&d}_\gamma=i_\gamma\left(\pmat{a&b \\ c&d}\right),\qquad\gamma\in\{\alpha,\beta\}\]
for elements in the image of these maps. We also write
\[\det_\gamma=\det\circ i_\gamma^{-1},\qquad\gamma\in\{\alpha,\beta\}.\]

\subsubsection*{The standard representation}
The smallest fundamental weight of $G_2$ is $\alpha+2\beta$, and the representation attached to it is seven dimensional. We denote it by $R_7$ and call it the \textit{standard representation} of $G_2$; it is the representation one naturally gets when defining $G_2$ through its action on traceless split octonions.

Let $V_7$ be the space of $R_7$. This representation contains weight vectors for the seven weights given by the six short roots together with the zero weight; see Figure \ref{figr7weights}.
\begin{figure}[h]
\centering
\includegraphics[scale=.25]{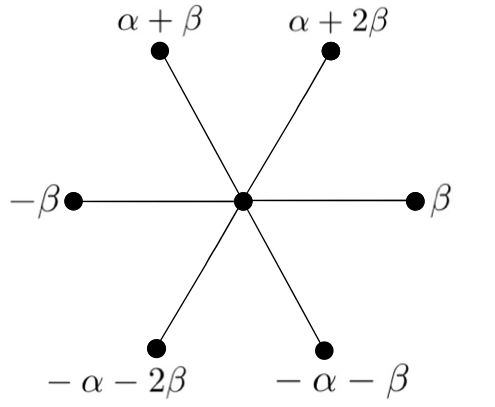}
\caption{The weights of $R_7$}
\label{figr7weights}
\end{figure}

For such a weight $\lambda$, choose a nonzero vector $v_\lambda\in V_7$ corresponding to that weight, and order these seven vectors as follows:
\[v_{-\alpha-2\beta},v_{-\alpha-\beta},v_{-\beta},v_0,v_\beta,v_{\alpha+\beta},v_{\alpha+2\beta}.\]
Then using the above list as an ordered basis represents $G_2$ as $7\times 7$ matrices acting on the linear span of these seven weight vectors. We then have the following matrix representations of the standard maximal Levi subgroups of $G_2$:
\begin{equation}
\label{eqr7alpha}
R_7\circ i_\alpha=\pmat{\det^{-1}&&&&\\ &\std^\vee &&&\\ &&1&&\\ \ &&&\std &\\ &&&&\det},
\end{equation}
where $\std$ is the standard representation of $GL_2$, and
\begin{equation}
\label{eqr7beta}
R_7\circ i_\beta=\pmat{\std^\vee &&\\ &\Ad &\\ &&\std},
\end{equation}
where $\Ad=\Sym^2(\std)\otimes\det^{-1}$ is the (three dimensional) adjoint representation of $\GL_2$. These can be seen by looking at strings in the directions of $\alpha$ and $\beta$ in the weight diagram as in Figure \ref{figr7levis}.
\begin{figure}[h]
\centering
\includegraphics[scale=.1666]{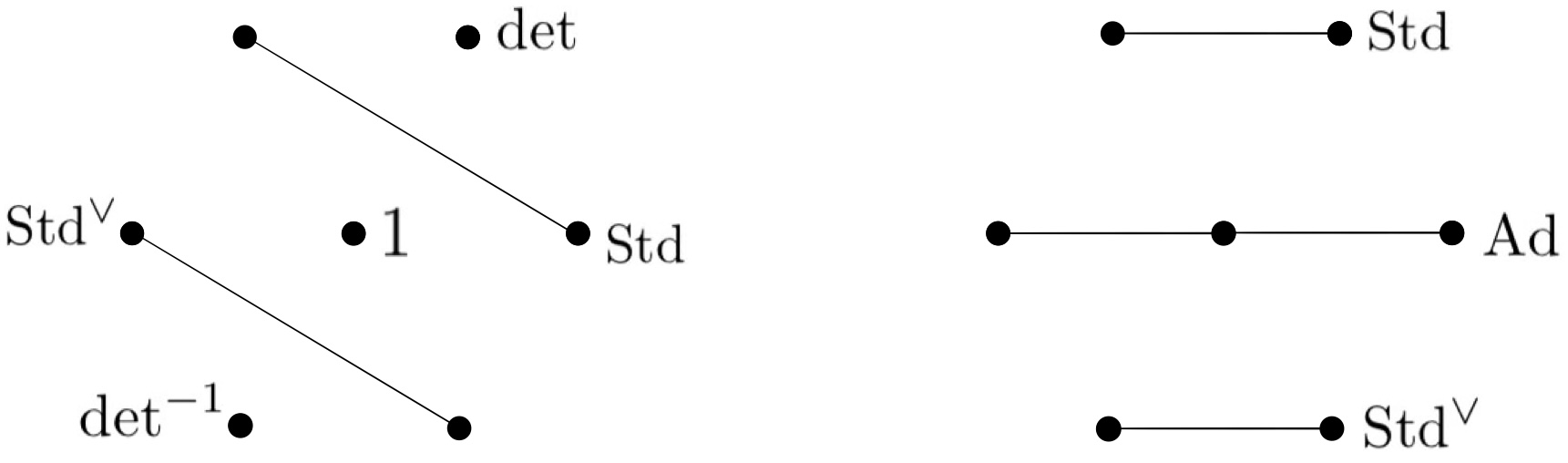}
\caption{The standard maximal Levis of $G_2$ under $R_7$}
\label{figr7levis}
\end{figure}

\subsubsection*{Duality}
The group $G_2$ is self dual, and identifying $G_2$ with its dual group switches the long and short simple roots. More explicitly, fix identifications $GL_2^\vee\cong GL_2$ and $G_2\cong G_2^\vee$ so that positive coroots correspond on the dual side to positive roots. Identify $M_\alpha$ and $M_\beta$ with $GL_2$ via the maps $i_\alpha$ and $i_\beta$ introduced above. Then $M_\alpha^\vee$ and $M_\beta^\vee$ are identified with $GL_2^\vee$, and we have commuting diagrams
\begin{equation}
\label{eqphialphag2}
\xymatrix{
GL_2^\vee\ar[r]^-\sim \ar[d]^\sim & M_\alpha^\vee \ar@{^{(}->}[r] \ar[d]^\sim & G_2^\vee\ar[d]^\sim \\
GL_2 \ar[r]^-{i_\beta} & M_\beta \ar@{^{(}->}[r] & G_2,}
\end{equation}
and
\begin{equation}
\label{eqphibetag2}
\xymatrix{
GL_2^\vee\ar[r]^-\sim \ar[d]^\sim & M_\beta^\vee \ar@{^{(}->}[r] \ar[d]^\sim & G_2^\vee\ar[d]^\sim \\
GL_2 \ar[r]^-{i_\alpha} & M_\alpha \ar@{^{(}->}[r] & G_2.}
\end{equation}

\subsubsection*{The Weyl group}
Let $W=W(T,G_2)$ be the Weyl group of $G_2$. The group $W$ is isomorphic to the dihedral group $D_6$ with $12$ elements acting naturally on the root lattice.

For $\gamma\in\Delta$, let $w_\gamma$ be the reflection about the line perpendicular to $\gamma$. Then $W$ is generated by the simple reflections $w_\alpha$ and $w_\beta$. We use the following notation for amalgamations of such elements: Write $w_{\alpha\beta}=w_\alpha w_\beta$, $w_{\alpha\beta\alpha}=w_\alpha w_\beta w_\alpha$, and so on. Then
\[W=\{1,w_\alpha,w_\beta,w_{\alpha\beta},w_{\beta\alpha}, w_{\alpha\beta\alpha},w_{\beta\alpha\beta},w_{\alpha\beta\alpha\beta},w_{\beta\alpha\beta\alpha},w_{\alpha\beta\alpha\beta\alpha},w_{\beta\alpha\beta\alpha\beta},w_{-1}\}.\]
The elements above are written minimally in terms of products of the simple reflections $w_\alpha$ and $w_\beta$, except for the final element $w_{-1}$. This is the element that acts by negation on the root lattice, and it of length $6$, equal to both $w_{\alpha\beta\alpha\beta\alpha\beta}$ and $w_{\beta\alpha\beta\alpha\beta\alpha}$.

For $P=MN$ one of the standard parabolic subgroups of $G_2$, we write as usual
\[W^P=\sset{w\in W}{w^{-1}\gamma>0\textrm{ for all positive roots }\gamma\textrm{ in }M}\]
for the set of representatives for the quotient $W_M\backslash W$ of minimal length, where $W_M=W(T,M)$ is the Weyl group of $T$ in $M$. Then
\[W^{P_\alpha}=\{1,w_\beta,w_{\beta\alpha},w_{\beta\alpha\beta},w_{\beta\alpha\beta\alpha},w_{\beta\alpha\beta\alpha\beta}\},\qquad W^{P_\beta}=\{1,w_\alpha,w_{\alpha\beta},w_{\alpha\beta\alpha},w_{\alpha\beta\alpha\beta},w_{\alpha\beta\alpha\beta\alpha}\},\]
and $W^B=W$.

\subsubsection*{The group $G_2(\R)$}
The real Lie group $G_2(\R)$ is connected and has discrete series. Fix a maximal compact torus $T_c$ in $G_2(\R)$. Then $T_c$ is $2$-dimensional and lies in a maximal compact subgroup of $G_2(\R)$, which we denote by $K_\infty$. Then $K_\infty$ is connected and $6$-dimensional. In fact
\[K_\infty\cong(SU(2)\times SU(2))/\mu_2,\]
where $\mu_2=\{\pm 1\}$ is diagonally embedded in $SU(2)\times SU(2)$.

Let $\mf{t}_c$ be the complexified Lie algebra of $T_c$, and $\mf{k}$ that of $K_\infty$. We abuse notation and write $\Delta=\Delta(\mf{t}_c,\mf{g}_2)$ for the roots of $\mf{t}_c$ in $\mf{g}_2$. Let $\Delta_c=\Delta(\mf{t}_c,\mf{k})$ denote the set of compact roots. There are four roots in $\Delta_c$ consisting of a pair of short roots and a pair of long roots. The short compact roots are orthogonal to the long ones.

Again, abusing notation, choose two simple roots $\alpha,\beta$ of $\mf{t}_c$ in $\mf{g}_2$ with $\alpha$ long and $\beta$ short, and choose them so that $\beta$ is compact. Then
\[\Delta_c=\{\pm\beta,\pm(2\alpha+3\beta)\}.\]

The compact Weyl group $W_c=W(\mf{t}_c,\mf{k})$ has four elements and is isomorphic to $(\Z/2\Z)\oplus(\Z/2\Z)$. In fact, we have
\[W_c=\{1,w_\beta,w_{\alpha\beta\alpha\beta\alpha},w_{-1}\},\]
and $w_{\alpha\beta\alpha\beta\alpha}$ equals the reflection across the line perpendicular to $2\alpha+3\beta$. It follows from the theory of Harish-Chandra parameters that the discrete series representations of $G_2(\R)$ are parameterized by integral weights in the union of the three chambers between $\beta$ and $2\alpha+3\beta$ which are far enough from the walls of those chambers.

\subsection{Alternating trilinear forms and $G_2$}
\label{subsecalt3forms}
Let $k$ be an algebraically closed field, and say for simplicity that $k$ is of characteristic zero. Let $V$ be the space of the standard representation of $GL_7$ over $k$. It is a classical fact that the group $G_2$ over $k$ is the stabilizer of any alternating trilinear form which is \textit{generic}, meaning that the orbit of this form is Zariski open in the space of alternating $3$-forms on $V$. Let $e_1,\dotsc,e_7$ be a basis for $V$, and $e_1^\vee,\dotsc,e_7^\vee$ the dual basis for $V^\vee$. Then a standard example of such a trilinear form is given by
\begin{equation}
\label{eqfirstgen3form}
e_1^\vee\wedge e_2^\vee\wedge e_3^\vee+e_4^\vee\wedge e_5^\vee\wedge e_6^\vee+e_1^\vee\wedge e_4^\vee\wedge e_7^\vee+e_2^\vee\wedge e_5^\vee\wedge e_7^\vee+e_3^\vee\wedge e_6^\vee\wedge e_7^\vee.
\end{equation}
(See, for example, \cite{CH}.)

\begin{lemma}
\label{lem3form}
For any $a\in k^\times$, the alternating $3$-form
\[e_1^\vee\wedge e_4^\vee\wedge e_7^\vee+e_1^\vee\wedge e_5^\vee\wedge e_6^\vee+e_2^\vee\wedge e_3^\vee\wedge e_7^\vee-e_2^\vee\wedge e_4^\vee\wedge e_5^\vee+ae_3^\vee\wedge e_4^\vee\wedge e_5^\vee\]
is generic.
\end{lemma}

\begin{proof}
From \eqref{eqfirstgen3form}, make the permutation $(2635)(47)$ on the indices, then compose with
\[(e_1,e_2,e_3,e_4,e_5,e_6,e_7)\mapsto(e_1,e_2,ae_3,-a^{-1}e_4,-ae_5,a^{-1}e_6,a^{-1}e_7)\]
to obtain the form in the lemma.
\end{proof}

For $a\in k^\times$, let $G_{2,a}$ be the subgroup of $GL_7$ preserving the form in the lemma. Then $G_{2,a}\cong G_2$, and $G_{2,a}$ is conjugate in $GL_7$ to the image of $G_2$ under the standard representation $R_7$ discussed in the previous subsection.

\begin{lemma}
The subgroup $T_a$ of $G_{2,a}$, given on $k$-algebras $A$ by
\[T_a(A)=\sset{\diag(t_1,t_2,t_1t_2^{-1},1,t_1^{-1}t_2,t_2^{-1},t_1^{-1})}{t_1,t_2\in\mb{G}_m(A)}\]
is a maximal torus in $G_a$.
\end{lemma}

\begin{proof}
The group $T_a$ is clearly a torus of rank $2$, and it is easy to check that it preserves the form of Lemma \ref{lem3form}. Since $G_{2,a}$ is of rank $2$, the lemma follows.
\end{proof}

We now study the root system of $G_{2,a}$ in the basis $e_1,\dotsc,e_7$. Write
\[[t_1,t_2]=\diag(t_1,t_2,t_1t_2^{-1},1,t_1^{-1}t_2,t_2^{-1},t_1^{-1})\in T_a.\]
Abusing notation (at least a priori) we write
\[\alpha([t_1,t_2])=t_1^{-1}t_2^2,\qquad\beta([t_1,t_2])=t_1t_2^{-1}.\]
We define various $1$-parameter subgroups of $G_{2,a}$ by defining them on $A$-points for $k$-algebras $A$ as follows. For $x\in A$, let
\[g(\alpha,x)=\pmat{
1&&&&&&\\
&1&x&&&&\\
&&1&&&&\\
&&&1&&&\\
&&&&1&a^{-1}x&\\
&&&&&1&\\
&&&&&&1},\qquad
g(\beta,x)=\pmat{
1&x&&&&&\\
&1&&&&&\\
&&1&-2x&-x^2&&\\
&&&1&x&&\\
&&&&1&&\\
&&&&&1&-x\\
&&&&&&1},\]
\[g(\alpha+\beta,x)=\pmat{
1&&x&&&&\\
&1&&2x&&a^{-1}x^2&\\
&&1&&&&\\
&&&1&&a^{-1}x&\\
&&&&1&&a^{-1}x\\
&&&&&1&\\
&&&&&&1},\]
\[g(\alpha+2\beta,x)=\pmat{
1&&&2x&&&a^{-1}x^2\\
&1&&&-x&&\\
&&1&&&-a^{-1}x&\\
&&&1&&&a^{-1}x\\
&&&&1&&\\
&&&&&1&\\
&&&&&&1},\]
\[g(\alpha+3\beta,x)=\pmat{
1&&&&x&&\\
&1&&&&&\\
&&1&&&&a^{-1}x\\
&&&1&&&\\
&&&&1&&\\
&&&&&1&\\
&&&&&&1},\qquad
g(2\alpha+3\beta,x)=\pmat{
1&&&&&x&\\
&1&&&&&-x\\
&&1&&&&\\
&&&1&&&\\
&&&&1&&\\
&&&&&1&\\
&&&&&&1}.\]
Then for $\gamma\in\{\alpha,\beta,\alpha+\beta,\alpha+2\beta,\alpha+3\beta,2\alpha+3\beta\}$, one checks easily the relations given by
\[[t_1,t_2]g(\gamma,x)[t_1,t_2]^{-1}=g(\gamma,\gamma([t_1,t_2])x).\]
One also checks that $g(\gamma,\cdot)\subset G_{2,a}$ by checking that these elements preserve the given generic alternating $3$-form, and it follows that
\[\{\alpha,\beta,\alpha+\beta,\alpha+2\beta,\alpha+3\beta,2\alpha+3\beta\}\]
forms a system of positive roots for $T_a$ in $G_{2,a}$.

Now we denote by $P_{\beta,a}$ the parabolic subgroup of $G_{2,a}$ containing $T_a$ along with all the positive roots for $T_a$ in $G_{2,a}$ and $-\beta$. One checks easily that if we let
\[g(-\beta,x)=\pmat{
1&&&&&&\\
x&1&&&&&\\
&&1&&&&\\
&&-x&1&&&\\
&&-x^2&2x&1&&\\
&&&&&1&\\
&&&&&-x&1},\]
then $g(-\beta,x)\in G_a$ and
\[[t_1,t_2]g(-\beta,x)[t_1,t_2]^{-1}=g(-\beta,\beta([t_1,t_2])^{-1}x).\]
Therefore for $\gamma$ any positive root or $\gamma=-\beta$, the root subgroups corresponding to $\gamma$ are the one-parameter subgroups $g(\gamma,\cdot)$ given above.

\begin{proposition}
\label{propPabeta}
Let $P_{232}$ be the standard parabolic subgroup of $GL_7$ of the form
\[P_{232}=\pmat{
*&*&*&*&*&*&*\\
*&*&*&*&*&*&*\\
&&*&*&*&*&*\\
&&*&*&*&*&*\\
&&*&*&*&*&*\\
&&&&&*&*\\
&&&&&*&*}.\]
Then $P_{232}\cap G_{2,a}=P_{\beta,a}$.
\end{proposition}

\begin{proof}
Clearly $T_a\subset P_{232}$ and $g(\gamma,\cdot)\subset P_{232}$ for any positive root $\gamma$ or $\gamma=-\beta$. Therefore $P_{\beta,a}\subset P_{232}\cap G_{2,a}$.

To show the opposite inclusion, we use the Bruhat decomposition. Let
\[\tilde{w}_\alpha=\pmat{
1&&&&&&\\
&&-1&&&&\\
&1&&&&&\\
&&&1&&&\\
&&&&&-a^{-1}&\\
&&&&a&&\\
&&&&&&1},\qquad\tilde{w}_\beta=\pmat{
&1&&&&&\\
1&&&&&&\\
&&&&1&&\\
&&&-1&&&\\
&&1&&&&\\
&&&&&&1\\
&&&&&1&}.\]
Then one checks easily that $\tilde{w}_\alpha,\tilde{w}_\beta\in G_{2,a}$. Also, $\tilde{w}_\alpha,\tilde{w}_\beta$ normalize the torus $T_a$, and they normalize the standard diagonal maximal torus in $GL_7$, which we denote $T_7$, and thus these elements are representatives for the Weyl groups of both $G_{2,a}$ and $GL_7$.

Like in the previous section, we use amalgamated notation and let, for example, $\tilde{w}_{\alpha\beta}=\tilde{w}_\alpha\tilde{w}_\beta$. Let $s_\alpha=(23)(56)\in S_7$ be the permutation corresponding to $\tilde{w}_\alpha$ when viewing the Weyl group of $GL_7$ as the symmetric group on $7$ elements. Similarly define $s_\beta=(12)(35)(67)\in S_7$, as well as $s_{\alpha\beta}$, and so on. Then one checks
\begin{gather*}
s_{\alpha\beta}=(125763),\qquad s_{\beta\alpha}=(367521),\qquad s_{\alpha\beta\alpha}=(31)(26)(57),\qquad s_{\beta\alpha\beta}=(15)(37),\\
s_{\alpha\beta\alpha\beta}=(156)(273),\qquad s_{\beta\alpha\beta\alpha}=(165)(237),\qquad s_{\alpha\beta\alpha\beta\alpha}=(16)(27)(35),\\
s_{\beta\alpha\beta\alpha\beta}=(17)(25)(36),\qquad s_{\alpha\beta\alpha\beta\alpha\beta}=(17)(26)(35),
\end{gather*}
and this defines a homomorphism from the Weyl group of $G_{2,a}$ to the Weyl group $S_7$ of $GL_7$ which is visibly injective. The Weyl group $W_{232}$ of the Levi of $P_{232}$ is the subgroup of $S_7$ which acts separately on the sets $\{1,2\}$, $\{3,4,5\}$, $\{6,7\}$. One sees from the list given above that the only elements of the Weyl group of $G_{2,a}$ which are in $W_{232}$ are $1$ and $s_{\beta}$.

Let $M_{\beta,a}$ be the Levi of $P_{\beta,a}$. Writing $W_a$ for the Weyl group of $G_{2,a}$ and $W_{\beta,a}$ for that of $M_\beta$, we thus get an injective map
\[W_{\beta,a}\backslash W_a\hookrightarrow W_{232}\backslash S_7.\]

Let us identify $W_{\beta,a}\backslash W_a$ with the set $W^{P_{\beta,a}}$ of minimal length representatives of this quotient, so
\[W_{\beta,a}\backslash W_a\cong W^{P_{\beta,a}}=\{1,s_\alpha,s_{\alpha\beta},s_{\alpha\beta\alpha},s_{\alpha\beta\alpha\beta},s_{\alpha\beta\alpha\beta\alpha}\}.\]
Write $W^{P_{232}}$ for the set of minimal length representatives for the quotient $W_{232}\backslash S_7$. Then we have an inclusion
\[W^{P_{a,\beta}}\hookrightarrow W^{P_{232}}.\]

Now the Bruhat decomposition gives a decomposition into disjoint sets,
\[GL_7=\coprod_{s\in W^{P_{232}}}P_{232}sP_{232}.\]
Similarly,
\[G_{2,a}=\coprod_{s\in W^{P_{\beta,a}}}P_{\beta,a}sP_{\beta,a}.\]
Because $W^{P_{\beta,a}}$ injects into $W^{P_{232}}$, a subdecomposition of the first decomposition above is given by
\[GL_7\supset\coprod_{s\in W^{P_{\beta,a}}}P_{232}sP_{232}.\]
Since $P_{\beta,a}\subset P_{232}$, we have
\[P_{\beta,a}sP_{\beta,a}\subset P_{232}s P_{232},\]
for any $s\in W^{P_{\beta,a}}$. Since for $s\ne 1$, $P_{232}s P_{232}$ is disjoint from $P_{232}$, this proves that $P_{\beta,a}sP_{\beta,a}\cap P_{232}=\emptyset$. Therefore we must have $P_{a,\beta}=P_{232}\cap G_{2,a}$, as desired.
\end{proof}

The following lemma will be a key step for us in checking that the cocycle we construct later on will lie in the correct Bloch--Kato Selmer group.

\begin{lemma}
\label{lemmatcoeffreln}
Let $h\in P_{\beta,a}$, and write $h_{ij}$ for the $(i,j)$-entry of the matrix $h$. Then we have the relations
\[2h_{22}h_{13}-2h_{12}h_{23}+h_{21}h_{14}-h_{11}h_{24}=0\]
and
\[h_{22}h_{14}-h_{12}h_{24}-2h_{21}h_{15}+2h_{11}h_{25}=0.\]
\end{lemma}

\begin{proof}
First we check this for $h$ in the unipotent radical of $P_{\beta,a}$. Any such $h$ can be written as
\[h=g(\alpha,x_1)g(\alpha+\beta,x_2)g(\alpha+2\beta,x_3)g(\alpha+3\beta,x_4)g(2\alpha+3\beta,x_5).\]
Then one can compute using the expressions for the $1$-parameter subgroups given above that
\[h=\pmat{
1&&x_2&2x_3&x_4&*&*\\
&1&x_1&2x_2&-x_3&*&*\\
&&1&&&*&*\\
&&&1&&*&*\\
&&&&1&*&*\\
&&&&&1&\\
&&&&&&1}\]
where the asterisks are certain polynomial combinations of $x_1,\dotsc,x_5$. The matrix entries of this element clearly satisfy the relations listed in the statement of the lemma.

Now let $h$ be any element of $P_{a,\beta}$ satisfying the relations given in the lemma. Then one computes easily that the entries of the matrices
\[g(\beta,x)h,\qquad g(-\beta,x)h,\qquad[t_1,t_2]h\]
also satisfy the relations stated in the lemma. Since the Levi subgroup of $P_{\beta,a}$ is generated by elements of the form $g(\beta,x)$, $g(-\beta,x)$ and $[t_1,t_2]$, the lemma follows.
\end{proof}

\section{Eisenstein cohomology}
\label{seceiscoh}
After introducing some general background on automorphic forms and cohomology, in this section we will define the automorphic representation $\mc{L}_\alpha(\pi_F,1/10)$ we will be interested in throughout this paper. We will locate it up to near equivalence in cohomology.

\subsection{Background on the Franke--Schwermer decomposition}
\label{subsecfsdecomp}
Fix throughout this subsection a reductive group $G$ over $\Q$. We start with a cuspidal automorphic representation $\pi$ of $M(\A)$, where $M$ is a Levi of a parabolic $\Q$-subgroup $P$ in $G$. Let $\chi$ be the central character of $\pi$, and assume $\chi$ is trivial on $A_G(\R)^\circ$, where $A_G$ is the split center of $G$. So if 
\[L^2(M(\Q)A_G(\R)^\circ\backslash M(\A),\chi)\]
denotes the space of functions on $M(\Q)A_G(\R)^\circ\backslash M(\A)$ which are square integrable modulo center and which transform under the center with respect to $\chi$, then $\pi$ occurs in the cuspidal spectrum
\[L_{\cusp}^2(M(\Q)A_G(\R)^\circ\backslash M(\A),\chi)\subset L^2(M(\Q)A_G(\R)^\circ\backslash M(\A),\chi).\]

Write $d\chi\in\mf{a}_{P,0}^\vee$ for the differential of the restriction of $\chi$ to $A_P(\R)^\circ/A_G(\R)^\circ$. Then we consider the unitary automorphic representation
\[\tilde{\pi}=\pi\otimes e^{-\langle H_P(\cdot), d\chi\rangle}.\]
(See the introduction for notation.) If $\pi$ is realized on a space of functions 
\[V_\pi\subset L_{\cusp}^2(M(\Q)A_G(\R)^\circ\backslash M(\A),\chi),\]
then $\tilde{\pi}$ is realized on the space
\[V_{\tilde{\pi}}=\{e^{-\langle H_P(\cdot), d\chi_\pi\rangle}f\,\,|\,\,f\in V_\pi\},\]
which is a subspace of $L_{\cusp}^2(M(\Q)A_P(\R)^\circ\backslash M(\A))$.

Let $W_{P,\tilde\pi}$ be the space of smooth, $K$-finite, $\C$-valued functions $\phi$ on
\[M(\Q)N(\A)A_P(\R)^\circ\backslash G(\A)\]
such that, for all $g\in G(\A)$, the function
\[m\mapsto \phi(mg),\qquad m\in M(\A),\]
lies in the $\tilde\pi$-isotypic subspace
\[L_{\cusp}^2(M(\Q)A_P(\R)^\circ\backslash M(\A))[\tilde\pi].\]
The space $W_{P,\tilde\pi}$ lets us build Eisenstein series. In fact, let $\phi\in W_{P,\tilde\pi}$. We define, for $\lambda\in\mf{a}_{P,0}^\vee$ and $g\in G(\A)$, the Eisenstein series $E(\phi,\lambda)$ by
\[E(\phi,\lambda)(g)=\sum_{\gamma\in P(\Q)\backslash G(\Q)}\phi(\gamma g) e^{\langle H_P(g),d\chi_\pi+\rho_P\rangle}.\]
This series only converges for $\lambda$ sufficiently far inside a positive Weyl chamber, but it defines a holomorphic function there in the variable $\lambda$ which continues meromorphically to all of $\mf{a}_{P,0}^\vee$; see \cite{Langlands}, \cite{MW}, or alternatively \cite{BL}, where a different and much simpler proof is given.

Now let $E$ be a complex, irreducible, finite dimensional representation of $G(\C)$. Then the annihilator of $E$ in the center of the universal enveloping algebra of $\mf{g}$ is an ideal, and we denote it by $\mc{J}_E$. Denote by $\mc{A}_E(G)$ the space of automorphic forms on $G(\A)$ which are annihilated by a power of $\mc{J}_E$, and which transform trivially under $A_G(\R)^\circ$. The forms in $\mc{A}_E(G)$ are the ones that can possibly contribute to the cohomology of $E$, as we will discuss later.

Given two parabolic subgroups of $G$ defined over $\Q$, we say that they are \textit{associate} if their Levis are conjugate by an element of $G(\Q)$. Let $\mc{C}$ be the finite set of equivalence classes for this relation. Let $[P]$ denote the equivalence class in $\mc{C}$.

Now we say a function $f\in\mc{A}_E(G)$ is \textit{negligible along} $P$ if for any $g\in G(\A)$, the function given by
\[m\mapsto f(mg),\qquad m\in M(\Q)A_G(\R)^\circ\backslash M(\A),\]
is orthogonal to the space of cuspidal functions on $M(\Q)A_G(\R)^\circ\backslash M(\A)$. Let $\mc{A}_{E,[P]}(G)$ be the subspace of all functions in $\mc{A}_E(G)$ which are negligible along any parabolic subgroup $Q\notin[P]$. It is a theorem of Langlands that
\begin{equation}
\label{eqnparadecomp}
\mc{A}_E(G)=\bigoplus_{C\in\mc{C}}\mc{A}_{E,C}(G)
\end{equation}
as $G(\A_f)\times(\mf{g}_0,K_\infty)$-modules. The summand $\mc{A}_{E,[G]}(G)$ is the space of cusp forms in $\mc{A}_E(G)$.

The Franke--Schwermer decomposition refines this even further using cuspidal automorphic representations of the Levis of the parabolics in each class $C\in\mc{C}$. We briefly recall how.

Let $\varphi$ be an \textit{associate class of cuspidal automorphic representations of} $M$. We do not recall here the exact definition of this notion, referring instead to \cite[\S 1.2]{FS} or \cite[\S 1.3]{LS}. Each $\varphi$ is a collection of automorphic representations of the groups $M_{P'}(\A)$ for each $P'\in[P]$ with Levi decomposition $P'=M_{P'}N_{P'}$, finitely many for each such $P'$, and each such representation $\pi'$ must occur in $L_{\cusp}^2(M_{P'}(\Q)\backslash M_{P'}(\A),\chi')$, where $\chi'$ is the central character of $\pi'$. Conversely, any automorphic representation $\pi$ of $M(\A)$ with central character $\chi$ occurring in $L_{\cusp}^2(M(\Q)\backslash M(\A),\chi)$ determines a unique $\varphi$. We let $\Phi_{E,[P]}$ denote the set of all associate classes of cuspidal automorphic representations of $M$.

Now given a $\varphi\in\Phi_{E,[P]}$, let $\pi'$ be one of the representations comprising $\varphi$; say $\pi'$ is an automorphic representation of $M_{P'}(\A)$, where $M_{P'}$ is a Levi of a parabolic subgroup $P'$ associate to $P$. Form the space $W_{P',\tilde\pi'}$ and let $d\chi'$ be the differential of the central character $\chi'$ of $\pi'$ at the archimedean place. Then for any $\phi\in W_{P',\tilde\pi'}$ we can form the Eisenstein series $E(\phi,\lambda)$, $\lambda\in\mf{a}_{P',0}^\vee$.

Depending on the choice of $\phi$, the Eisenstein series $E(\phi,\lambda)$ may have a pole at $\lambda=d\chi'$. Nevertheless, one can still take residues of $E(\phi,\lambda)$ at $\lambda=d\chi'$ to obtain residual Eisenstein series. We let $\mc{A}_{E,[P],\varphi}(G)$ be the collection of all possible Eisenstein series, residual Eisenstein series, and partial derivatives of such with respect to $\lambda$, evaluated at $\lambda=d\chi'$, built from any $\phi\in W_{P',\tilde\pi'}$. (For a more precise description of this space, see \cite[\S 1.3]{FS} or \cite[\S 1.4]{LS}. There is also a more intrinsic definition of this space, defined without reference to Eisenstein series, in \cite[\S 1.2]{FS} or \cite[\S 1.4]{LS}, which is proved to be equivalent to this description in \cite{FS}.) One can use the functional equation of Eisenstein series to show that the space $\mc{A}_{E,[P],\varphi}(G)$ is independent of the $\pi'$ in $\varphi$ used to define it.

We can now state the Franke--Schwermer decomposition of $\mc{A}_E(G)$.

\begin{theorem}[Franke--Schwermer \cite{FS}]
\label{thmfsdecomp}
There is a direct sum decomposition of $G(\A_f)\times(\mf{g}_0,K_\infty)$-modules
\[\mc{A}_E(G)=\bigoplus_{C\in\mc{C}}\bigoplus_{\varphi\in\Phi_{E,C}}\mc{A}_{E,C,\varphi}(G).\]
\end{theorem}

We now introduce certain explicit $G(\A_f)\times(\mf{g}_0,K_\infty)$-modules and explain how they can be related to the pieces of the Franke--Schwermer decomposition. Almost everything in the rest of this section is done in Franke's paper \cite[218, 234]{franke}, but without taking into consideration the associate classes $\varphi$.

With $\pi$ as above, for brevity, let us write $V[\tilde\pi]$ for the smooth, $K$-finite vectors in the $\tilde\pi$-isotypic component of $L_{\cusp}^2(M(\A)A_P(\R)^\circ\backslash M(\A))$. Then $V[\tilde\pi]$ is a $M(\A_f)\times(\mf{m}_0,K_\infty\cap P(\R))$-module, and we extend this structure to one of a $P(\A_f)\times(\mf{p}_0,K_\infty\cap P(\R))$-module by letting $\mf{a}_{P,0}$ and $\mf{n}$ act trivially, as well as $A_P(\A_f)$ and $N(\A_f)$. Here $N$ is the unipotent radical of $P$.

Fix for the rest of this subsection a point $\mu\in\mf{a}_{P,0}^\vee$. Let $\Sym(\mf{a}_{P,0})_\mu$ be the symmetric algebra on the vector space $\mf{a}_{P,0}$; we view this space as the space of differential operators on $\mf{a}_{P,0}^\vee$ at the point $\mu$. So if $H(\lambda)$ is a holomorphic function on $\mf{a}_{P,0}^\vee$, then $D\in\Sym(\mf{a}_{P,0})_\mu$ acts on $H$ by taking a sum of iterated partial derivatives of $H$ and evaluating the result at the point $\mu$. In this way, every $D\in\Sym(\mf{a}_{P,0})_\mu$ can be viewed as a distribution on holomorphic functions on $\mf{a}_{P,0}^\vee$ supported at the point $\mu$.

With this point of view, these distributions can be multiplied by holomorphic functions on $\mf{a}_{P,0}^\vee$; just multiply the test function by the given holomorphic function before evaluating the distribution. With this in mind, we can define an action of $\mf{a}_{P,0}^\vee$ on $\Sym(\mf{a}_{P,0})_\mu$ by
\[(XD)(f)=D(\langle X,\cdot\rangle f),\qquad X\in\mf{a}_{P,0},\,\, D\in\Sym(\mf{a}_{P,0})_\mu.\]
We also let $\mf{m}_0$ and $\mf{n}$ act trivially on $\Sym(\mf{a}_{P,0})_\mu$, which gives us an action of $\mf{p}_0$ on $\Sym(\mf{a}_{P,0})_\mu$. In addition, let $K_\infty\cap P(\R)$ act trivially on $\Sym(\mf{a}_{P,0})_\mu$. Since the Lie algebra of $K_\infty\cap P(\R)$ lies in $\mf{m}_0$, this is consistent with the $\mf{p}_0$ action just defined and makes $\Sym(\mf{a}_{P,0})_\mu$ a $(\mf{p}_0,K_\infty\cap P(\R))$-module. Finally, let $P(\A_f)$ act on $\Sym(\mf{a}_{P,0})_\mu$ by the formula
\[(pD)(f)=D(e^{\langle H_P(p),\cdot\rangle}f),\qquad p\in P(\A_f),\,\, D\in\Sym(\mf{a}_{P,0})_\mu.\]
Then with the actions just defined, $\Sym(\mf{a}_{P,0})_\mu$ gets the structure of a $P(\A_f)\times(\mf{p}_0,K_\infty\cap P(\R))$-module.

Now we form the tensor product $V[\tilde\pi]\otimes\Sym(\mf{a}_{P,0})_\mu$, which carries a natural $P(\A_f)\times(\mf{p},K_\infty\cap P(\R))$-module structure coming from those on the two factors. We will consider in what follows the induced $G(\A_f)\times(\mf{g}_0,K_\infty)$-module
\[\Ind_{P(\A)}^{G(\A)}(V[\tilde\pi]\otimes\Sym(\mf{a}_{P,0})_\mu).\]
This space turns out to be isomorphic to the tensor product
\[W_{P,\tilde\pi}\otimes\Sym(\mf{a}_{P,0})_\mu.\]
While the first factor in this tensor product is a $G(\A_f)\times(\mf{g}_0,K_\infty)$-module, the second is only a $P(\A_f)\times(\mf{p}_0,K_\infty\cap P(\R))$-module, and so we do not immediately get a $G(\A_f)\times(\mf{g}_0,K_\infty)$-module structure on the tensor product. However, one can endow this space with a $G(\A_f)\times(\mf{g}_0,K_\infty)$-module structure by viewing it as a space of distributions as described in \cite[p. 218]{franke}. The point is the following proposition, whose proof we omit for sake of space.

\begin{proposition}
\label{propwpind}
There is an isomorphism of $G(\A_f)\times(\mf{g}_0,K_\infty)$-modules
\[W_{P,\tilde\pi}\otimes\Sym(\mf{a}_{P,0})_\mu\cong\Ind_{P(\A)}^{G(\A)}(V[\tilde\pi]\otimes\Sym(\mf{a}_{P,0})_\mu).\]
More generally, if $E$ is a finite dimensional representation of $G(\C)$, then we also have an isomorphism
\[W_{P,\tilde\pi}\otimes\Sym(\mf{a}_{P,0})_\mu\otimes E\cong\Ind_{P(\A)}^{G(\A)}(V[\tilde\pi]\otimes\Sym(\mf{a}_{P,0})_\mu\otimes E),\]
where on the left hand side, $E$ is being viewed as a $(\mf{g}_0,K_\infty)$-module, and on the right, it is viewed as a $(\mf{p}_0,K_\infty\cap P(\R))$-module by restriction.
\end{proposition}

Now we come back to Eisenstein series. Assume $\pi$ is such that there is an irreducible finite dimensional representation $E$ of $G(\C)$ such that the associate class $\varphi$ containing $\pi$ is in $\Phi_{E,[P]}$. Then we can construct elements of the piece $\mc{A}_{E,[P],\varphi}(G)$ of the Franke--Schwermer decomposition from elements of $W_{P,\tilde\pi}\otimes\Sym(\mf{a}_{P,0})_\mu$ using Eisenstein series as follows.

Write
\[\iota_{P(\A)}^{G(\A)}(V[\tilde\pi],\lambda)=\Ind_{P(\A)}^{G(\A)}(V[\tilde\pi]\otimes e^{H_P(\cdot),\lambda+\rho_P}),\qquad \lambda\in\mf{a}_{P,0}^\vee,\]
for the unitary induction of $V[\tilde\pi]$, so that we have
\[W_{P,\tilde\pi}\cong\iota_{P(\A)}^{G(\A)}(V[\tilde\pi],-\rho_P).\]
Elements $\phi\in\iota_{P(\A)}^{G(\A)}(V[\tilde\pi],-\rho_P)$ fit into flat sections $\phi_\lambda\in\iota_{P(\A)}^{G(\A)}(V[\tilde\pi],\lambda)$ where $\lambda$ varies in $\mf{a}_{P,0}^\vee$. Then for such $\phi$ we have $\phi=\phi_{-\rho_P}$. In what follows, we will identify elements of $W_{P,\tilde\pi}$ with elements of $\iota_{P(\A)}^{G(\A)}(V[\tilde\pi],-\rho_P)$, and then use this notation to vary them in flat sections.

With $d\chi$ as above, let $h_0$ be a holomorphic function on $\mf{a}_{P,0}^\vee$ such that, for any $\phi\in W_{P,\tilde\pi}$, the product $h_0(\lambda)E(\phi,\lambda)$ is holomorphic near $\lambda=d\chi+\rho_P$. Then we define a map
\[\mc{E}_{h_0}:W_{P,\tilde\pi}\otimes\Sym(\mf{a}_{P,0})_{d\chi+\rho_P}\to\mc{A}_{E,[P],\varphi}(G)\]
by
\[\phi\otimes D\mapsto D(h_0(\lambda)E(\phi,\lambda)).\]
The map $\mc{E}_{h_0}$ is surjective by our definition of $\mc{A}_{E,[P],\varphi}(G)$. If all the Eisenstein series $E(\phi,\lambda)$, for $\phi\in W_{P,\tilde\pi}$, are holomorphic at $\lambda=d\chi$, then we write $\mc{E}=\mc{E}_1$ for the map just defined with $h_0(\lambda)=1$.

\begin{proposition}
\label{propEh0}
The map $\mc{E}_{h_0}:W_{P,\tilde\pi}\otimes\Sym(\mf{a}_{P,0})_{d\chi+\rho_P}\to\mc{A}_{E,[P],\varphi}(G)$ defined just above is a surjective map of $G(\A_f)\times(\mf{g}_0,K_\infty)$-modules. Furthermore, if all the Eisenstein series $E(\phi,\lambda)$ arising from $\phi\in W_{P,\tilde\pi}$ are holomorphic at $\lambda=d\chi+\rho_P$, then the map $\mc{E}$ is an isomorphism.
\end{proposition}

\begin{proof}
To check that $\mc{E}_{h_0}$ is a map of $G(\A_f)\times(\mf{g}_0,K_\infty)$-modules, one just needs to use the formulas defining the $G(\A_f)\times(\mf{g}_0,K_\infty)$-module structure on $W_{P,\tilde\pi}\otimes\Sym(\mf{a}_{P,0})_\lambda$ and show they are preserved when forming Eisenstein series and taking derivatives; this can be checked when $\lambda$ is in the region of convergence for the Eisenstein series, and then this extends to all $\lambda$ by analytic continuation. We omit the precise details of this check.

For the second claim in the proposition, that $\mc{E}$ is an isomorphism, this follows from \cite[Theorem 14]{franke}; this theorem implies that $\mc{E}$ injective, since it equals the restriction of Franke's mean value map $\mathbf{MW}$ to $W_{P,\tilde\pi}\otimes\Sym(\mf{a}_{P,0})_{d\chi+\rho_P}$. Whence by surjectivity and the first part of the proposition, we are done.
\end{proof}

The space $\mc{A}_{E,[P],\varphi}(G)$ carries a filtration by $G(\A_f)\times(\mf{g}_0,K_\infty)$-modules which is due to Franke. For our purposes, we will not need the precise definition of this filtration, but just a rough description of its graded pieces. This is described in the following theorem.

\begin{theorem}
\label{thmfrfil}
Let $C\in\mc{C}$. There is a decreasing filtration
\[\dotsb\supset\Fil^i\mc{A}_{E,C,\varphi}(G)\supset\Fil^{i+1}\mc{A}_{E,C,\varphi}(G)\supset\dotsb\]
of $G(\A_f)\times(\mf{g}_0,K_\infty)$-modules on $\mc{A}_{E,C,\varphi}(G)$, for which we have
\[\Fil^0\mc{A}_{E,C,\varphi}(G)=\mc{A}_{E,C,\varphi}(G)\]
and
\[\Fil^m\mc{A}_{E,C,\varphi}(G)=0\]
for some $m>0$ (depending on $\varphi$) and whose graded pieces have the property described below.

Fix $\pi$ in $\varphi$, and say $\pi$ is an automorphic representation of $M(\A)$ with $M$ a Levi of a parabolic $P$ in $C$. Let $d\chi$ be the differential of the archimedean component of the central character $\chi$ of $\pi$. Let $\mc{M}$ be the set of quadruples $(Q,\nu,\Pi,\mu)$ where:
\begin{itemize}
\item $Q$ is a parabolic subgroup of $G$ which contains $P$;
\item $\nu$ is an element of $(\mf{a}_P\cap\mf{m}_{Q,0})^\vee$;
\item $\Pi$ is an automorphic representation of $M(\A)$ occurring in
\[L_{\disc}^2(M_Q(\Q)A_Q(\R)^\circ\backslash M_Q(\A))\]
and which is spanned by values at, or residues at, the point $\nu$ of Eisenstein series parabolically induced from $(P\cap M_Q)(\A)$ to $M_Q(\A)$ by representations in $\varphi$; and
\item $\mu$ is an element of $\mf{a}_{Q,0}^\vee$ whose real part in $\Lie(A_{G}(\R)\backslash A_{M_Q}(\R))^\vee$ is in the closure of the positive chamber, and such that the following relation between $\mu$, $\nu$ and $\pi$ holds: Let $\lambda_{\tilde\pi}$ be the infinitesimal character of the archimedean component of $\tilde\pi$. Then
\[\lambda_{\tilde\pi}+\nu+\mu\]
may be viewed as a collection of weights of a Cartan subalgebra of $\mf{g}_0$, and the condition we impose is that these weights are in the support of the infinitesimal character of $E$.
\end{itemize}
For such a quadruple $(Q,\nu,\Pi,\mu)\in\mc{M}$, let $V_d[\Pi]$ denote the $\Pi$-isotypic component of the space
\[L_{\disc}^2(M_Q(\Q)A_Q(\R)^\circ\backslash M_Q(\A))\cap\mc{A}_{E,[P\cap M_Q],\varphi|_{M_Q}}(M_P).\]
Then the property of the graded pieces of the filtration above is that, for every $i$ with $0\leq i<m$, there is a subset $\mc{M}_\varphi^i\subset\mc{M}$ and an isomorphism of $G(\A_f)\times(\mf{g}_0,K_\infty)$-modules
\[\Fil^i\mc{A}_{E,C,\varphi}(G)/\Fil^{i+1}\mc{A}_{E,C,\varphi}(G)\cong\bigoplus_{(Q,\nu,\Pi,\mu)\in\mc{M}_\varphi^i}\Ind_{Q(\A)}^{G(\A)}(V_d[\Pi]\otimes\Sym(\mf{a}_{Q,0})_{\mu+\rho_Q}).\]
\end{theorem}

\begin{proof}
While this essentially follows again from the work of Franke \cite{franke}, in this form, this theorem is a consequence of \cite[Theorem 4]{grobner}; the latter takes into account the presence of the class $\varphi$ while the former does not.
\end{proof}

\begin{remark}
In the context of Proposition \ref{propEh0} and Theorem \ref{thmfrfil}, when all the Eisenstein series $E(\phi,\lambda)$ arising from $\phi\in W_{P,\tilde\pi}$ are holomorphic at $\lambda=d\chi_\pi$, what happens is that the filtration of Theorem \ref{thmfrfil} collapses to a single step. The nontrivial piece of this filtration is then given by $\Ind_{P(\A)}^{G(\A)}(V[\tilde\pi]\otimes\Sym(\mf{a}_{P,0})_{d\chi+\rho_P})$ through the map $\mc{E}$ along with the isomorphism of Proposition \ref{propwpind}.
\end{remark}

When $P$ is a maximal parabolic, the filtration of Theorem \ref{thmfrfil} becomes particularly simple. To describe it, we set some notation.

Assuming $P$ is maximal, if $\tilde\pi$ is a unitary cuspidal automorphic representation of $M(\A)$ and $s\in\C$ with $\re(s)>0$, let us write
\[\mc{L}_{P(\A)}^{G(\A)}(\tilde\pi,s)\]
for the Langlands quotient of
\[\iota_{P(\A)}^{G(\A)}(\tilde\pi,2s\rho_P).\]
Then we have

\begin{theorem}[Grbac \cite{grbac}]
\label{thmgrbac}
In the setting above, with $P$ maximal and $\re(s)>0$, assume $\tilde\pi$ defines an associate class $\varphi\in\Phi_{E,[P]}$. If any of the Eisenstein series $E(\phi,\lambda)$ coming from $\phi\in W_{\tilde\pi}$ have a pole at $\lambda=2s\rho_P$, then there is an exact sequence of $G(\A_f)\times(\mf{g}_0,K_\infty)$-modules as follows:
\[0\to\mc{L}_{P(\A)}^{G(\A)}(\tilde\pi,s)\to\mc{A}_{E,[P],\varphi}(G)\to\Ind_{P(\A)}^{G(\A)}(V[\tilde\pi]\otimes\Sym(\mf{a}_{P,0})_{(2s+1)\rho_P})\to 0.\]
\end{theorem}

\begin{proof}
This follows from \cite[Theorem 3.1]{grbac}.
\end{proof}

\subsection{Cohomology of induced representations}
\label{subsecindcoh}
We now calculate the cohomology of representations of $G$ that are parabolically induced from automorphic representations of Levi subgroups, and hence give a tool for computing the cohomology of the graded pieces of the Franke filtration described in Theorem \ref{thmfrfil}. The computations done in this section were essentially carried out by Franke in \cite[\S 7.4]{franke}, but not in so much detail. We fill in just a few of the details and give a version of Franke's result which focuses on one representation of a Levi subgroup at a time. The method is essentially that of the proof of \cite[Theorem III.3.3]{BW}. This method also appears in the computations of Grbac--Grobner \cite{GG} and Grbac--Schwermer \cite{GS}.

Let the notation be as in the previous section. Then we have our group $G$ and $P\subset G$ a parabolic subgroup defined over $\Q$ with Levi decomposition $P=MN$. Fix an automorphic representation (not necessarily cuspidal) $\pi$ of $M(\A)$ with central character $\chi$, occurring in the discrete spectrum
\[L_{\disc}^2(M(\Q)\backslash M(\A),\chi).\]
Then the unitarization $\tilde\pi$ occurs in
\[L_{\disc}^2(M(\Q)A_P(\R)^\circ\backslash M(\A)).\]
Let $d\chi$ denote the differential of the archimedean component of $\chi$. Fix also an irreducible finite dimensional representation $E$ of $G(\C)$.

Fix a compact subgroup $K_\infty'$ of $G(\R)$ such that $K_\infty^\circ\subset K_\infty'\subset K_\infty$. We will compute the $(\mf{g}_0,K_\infty')$-cohomology space
\[H^i(\mf{g}_0,K_\infty';\Ind_{P(\A)}^{G(\A)}(\tilde\pi\otimes\Sym(\mf{a}_{P,0})_{d\chi+\rho_P})\otimes E)\]
in terms of $(\mf{m}_0,K_\infty'\cap P(\R))$-cohomology spaces attached to $\pi$. We will require the following lemma.

\begin{lemma}
\label{lemcohsym}
Let $\mu,\mu'\in\mf{a}_{P,0}^\vee$. Let $\C_{\mu'}$ denote the one dimensional $\mf{a}_{P,0}$-module on which $X\in\mf{a}_{P,0}$ acts through multiplication by $\langle X,\mu'\rangle$. Then there is an isomorphism of $P(\A_f)$-modules
\[H^i(\mf{a}_{P,0},\Sym(\mf{a}_{P,0})_\mu\otimes\C_{\mu'})\cong\begin{cases}
\C(e^{\langle H_P(\cdot),\mu\rangle}) & \textrm{if }\mu'=-\mu\textrm{ and }i=0;\\
0 & \textrm{if }\mu'\ne-\mu\textrm{ or }i>0.
\end{cases}\]
Here, $\C(e^{\langle H_P(\cdot),\mu\rangle})$ is just the one dimensional representation of $P(\A_f)$ on which $p\in P(\A_f)$ acts via $e^{\langle H_P(p),\mu\rangle}$.
\end{lemma}

\begin{proof}
It will be convenient to work in coordinates. So let $\lambda=(\lambda_1,\dotsc,\lambda_r)$ be coordinates on $\mf{a}_{P,0}^\vee$; this is the same as fixing a basis of $\mf{a}_{P,0}$. Then the elements of $\Sym(\mf{a}_{P,0})_\mu$ may be viewed as polynomials in the variables $\lambda_1,\dotsc,\lambda_r$.

Let $\alpha=(\alpha_1,\dotsc,\alpha_r)$ be a multi-index. By definition, the monomial $\lambda^\alpha=\lambda_1^{\alpha_1}\dotsb\lambda_r^{\alpha_r}$ acts as a distribution on holomorphic functions $f$ on $\mf{a}_{P,0}^\vee$ via the formula
\[\lambda^\alpha f=\frac{\partial^\alpha}{\partial\lambda^\alpha}f(\lambda)|_{\lambda=\mu}.\]
Also by definition, if $X\in\mf{a}_{P,0}$, then $X\lambda^\alpha$ acts as
\[(X\lambda^\alpha)f=\frac{\partial^\alpha}{\partial\lambda^\alpha}(\langle X,\lambda\rangle f(\lambda))|_{\lambda=\mu}.\]

Let $P(\lambda)$ be a polynomial in $\lambda$. Then a quick induction using the above formulas shows that $X\in\mf{a}_{P,0}$ acts on $P(\lambda)$ as
\[X(P(\lambda))=\langle X,\mu\rangle P(\lambda)+\sum_{i=1}^r\frac{\partial}{\partial\lambda_i}P(\lambda).\]
Hence $X$ acts on the element $P(\lambda)\otimes 1$ in $\Sym(\mf{a}_{P,0})_\mu\otimes\C_{\mu'}$ by
\[X(P(\lambda)\otimes 1)=\langle X,\mu+\mu'\rangle (P(\lambda)\otimes 1)+\sum_{i=1}^r\left(\frac{\partial}{\partial\lambda_i}P(\lambda)\otimes 1\right).\]
It follows from this that if $X_1,\dotsc,X_r$ is the basis of $\mf{a}_{P,0}$ corresponding to the coordinates $\lambda_1,\dotsc,\lambda_r$, then the decomposition
\[\mf{a}_{P,0}=\C X_1\oplus\dotsb\oplus\C X_r\]
realizes $\Sym(\mf{a}_{P,0})_\mu\otimes\C_{\mu'}$ as an exterior tensor product of analogously defined single-variable symmetric powers:
\[\Sym(\mf{a}_{P,0})_\mu\otimes\C_{\mu'}\cong(\Sym(\C X_1)_{\mu_1}\otimes\C_{\mu_1'})\otimes\dotsb\otimes(\Sym(\C X_r)_{\mu_r}\otimes\C_{\mu_r'}),\]
where $\mu_i,\mu_i'\in(\C X_i)^\vee$ are the $i$th components of $\mu,\mu'$ in the dual basis of $\mf{a}_{P,0}^\vee$ to $X_1,\dotsc,X_r$. By the K\"unneth formula, if we ignore for now the $P(\A_f)$-action, we then reduce to checking the one-dimensional analog of the lemma, that
\[H^i(\C X_i,\Sym(\C X_i)_{\mu_i}\otimes\C_{\mu_i'})\cong\begin{cases}
\C & \textrm{if }\mu'=-\mu\textrm{ and }i=0;\\
0 & \textrm{if }\mu'\ne-\mu\textrm{ or }i>0.
\end{cases}\]
This can be checked just by writing down the complex that computes this cohomology. Furthermore, $H^0(\mf{a}_{P,0},\Sym(\mf{a}_{P,0})_\mu\otimes\C_{-\mu})$ can be identified with subspace of $\Sym(\mf{a}_{P,0})_\mu$ consisting of constants. By definition, this space has an action of $P(\A_f)$ given by the character $e^{\langle H_P(\cdot),\mu\rangle}$, which proves our lemma.
\end{proof}

Let $\mf{h}\subset\mf{g}$ be a Cartan subalgebra, and assume $\mf{h}\subset\mf{m}$. Fix an ordering on the roots of $\mf{h}$ in $\mf{g}$ which makes $\mf{p}$ standard. If $W(\mf{h},\mf{g})$ denotes the Weyl group of $\mf{h}$ in $\mf{g}$, then write
\[W^P=\sset{w\in W(\mf{h},\mf{g})}{w^{-1}\alpha>0\textrm{ for all positive roots }\alpha\textrm{ in }\mf{m}}.\]
Then $W^P$ is the set of representatives of minimal length for $W(\mf{h},\mf{g})$ modulo the Weyl group $W(\mf{h}\cap\mf{m}_0,\mf{m}_0)$ of $\mf{h}\cap\mf{m}_0$ in $\mf{m}_0$. Write $\rho$ for half the sum of the positive roots of $\mf{h}$ in $\mf{g}$.

If $\Lambda\in\mf{h}^\vee$ is a dominant weight, write $E_\Lambda$ for the representation of $\mf{g}$ of highest weight $\Lambda$. If $\nu\in\mf{h}^\vee$ is a weight which is dominant for $\mf{m}$ we denote by $F_\nu$ the representation of $\mf{m}$ of highest weight $\nu$. Then we have the Kostant decomposition:
\[H^i(\mf{n},E_\Lambda)\cong\bigoplus_{\substack{w\in W^P\\ \ell(w)=i}}F_{w(\Lambda+\rho)-\rho},\]
where $\ell(w)$ denotes the length of the Weyl group element $w$.

Now we are ready to state the main theorem of this subsection.

\begin{theorem}
\label{thmcohind}
Notation as above, let $\Lambda\in\mf{h}^\vee$ be a dominant weight such that $E=E_\Lambda$. Assume that the cohomology space
\begin{equation}
\label{eqcohspace}
H^i(\mf{g}_0,K_\infty';\Ind_{P(\A)}^{G(\A)}(\tilde\pi\otimes\Sym(\mf{a}_{P,0})_{d\chi+\rho_P})\otimes E)
\end{equation}
is nontrivial for some $i$. Then there is a unique $w\in W^P$ such that
\[-w(\Lambda+\rho)|_{\mf{a}_{P,0}}=d\chi\]
and such that the infinitesimal character of the archimedean component of $\tilde\pi$ contains $-w(\Lambda+\rho)|_{\mf{h}\cap\mf{m_0}}$. Furthermore, if $\ell(w)$ is the length of such an element $w$, then for any $i$ we have
\begin{multline*}
H^i(\mf{g}_0,K_\infty';\Ind_{P(\A)}^{G(\A)}(\tilde\pi\otimes\Sym(\mf{a}_{P,0})_{d\chi+\rho_P})\otimes E)\\
\cong\iota_{P(\A_f)}^{G(\A_f)}(\pi_f)\otimes H^{i-\ell(w)}(\mf{m}_0,K_\infty'\cap P(\R);\tilde\pi_\infty\otimes F_{w(\Lambda+\rho)-\rho,0}),
\end{multline*}
where $\iota$ denotes a normalized parabolic induction functor, and $F_{w(\Lambda+\rho)-\rho,0}$ denotes the restriction to $\mf{m}_0$ of the representation of $\mf{m}$ of highest weight $w(\Lambda+\rho)-\rho$.
\end{theorem}

\begin{proof}
We only give a brief sketch of the proof, as it is almost exactly the same as the proof of \cite[Theorem III.3.3]{BW}. Just as in that proof, one uses Frobenius reciprocity, the Kostant decomposition and the K\"unneth formula to write \eqref{eqcohspace} as
\[\Ind_{P(\A_f)}^{G(\A_f)}(H^{j-\ell(w)}(\mf{m}_0,K_\infty'\cap P(\R);\tilde\pi\otimes F_{w(\Lambda+\rho)-\rho,0})\otimes H^k(\mf{a}_{P,0},\Sym(\mf{a}_{P,0})_{d\chi+\rho_P}\otimes\C_{\nu(w)})),\]
where $i=j+k$, $w$ is as in the statement of the theorem and
\[\nu(w)=(w(\Lambda+\rho)-\rho)|_{\mf{a}_{P,0}}.\]
The only real difference is that now we use Lemma \ref{lemcohsym} to compute the $\mf{a}_{P,0}$-cohomology in this induction and obtain the theorem.
\end{proof}

\subsection{Location of a Langlands quotient in Eisenstein cohomology}
\label{subseclqincoh}
We now return to the $G_2$ setting. In this subsection we will study the occurrence of a particular Langlands quotient in the Eisenstein cohomology of $G_2$. This Langlands quotient will be the automorphic representation of $G_2$ which we will $p$-adically deform later. In order to precisely locate the it in Eisenstein cohomology, we will need to be able to distinguish between different representations in that cohomology. When these representations are coming from different parabolic subgroups, we can use $L$-functions to do this, as in the proposition below. To state it, we require some set up.

Recall that two automorphic representations $\pi$ and $\pi'$ of a reductive group $G$ are \textit{nearly equivalent} if for all but finitely many places $v$, the local components $\pi_v$ and $\pi_v'$ are isomorphic.

Identify the long root Levi $M_\alpha$ of $P_\alpha$ and the short root Levi $M_\beta$ of $P_\beta$ with $GL_2$ via the maps $i_\alpha$ and $i_\beta$ of Section \ref{subsecg2str}. Then the modulus characters are given by
\[\delta_{M_\alpha(\A)}(A)=\vert\det(A)\vert^5,\qquad\delta_{M_\beta(\A)}(A)=\vert\det(A)\vert^3,\]
for $A\in GL_2(\A)$. We will consider in what follows the unitary parabolic induction functors $\iota_{P(\A)}^{G_2(\A)}$ for $P\in\{B,P_\alpha,P_\beta\}$; see the section on notation in the introduction.

\begin{proposition}
\label{propdistneareq}
Let $\pi_\alpha$ and $\pi_\beta$ be unitary, tempered, cuspidal automorphic representations of $GL_2(\A)$, viewed respectively as representations of $M_\alpha(\A)$ and $M_\beta(\A)$. Let $\psi$ be a quasicharacter of $T(\Q)\backslash T(\A)$, and let $s_\alpha,s_\beta\in\C$. Then given any irreducible subquotients
\[\Pi_\alpha\quad\textrm{of}\quad\iota_{P_\alpha(\A)}^{G_2(\A)}(\pi_\alpha,s_\alpha)\]
and
\[\Pi_\beta\quad\textrm{of}\quad\iota_{P_\beta(\A)}^{G_2(\A)}(\pi_\beta,s_\beta)\]
and
\[\Pi_0\quad\textrm{of}\quad\iota_{B(\A)}^{G_2(\A)}(\psi),\]
we have that no two of $\Pi_\alpha$, $\Pi_\beta$ and $\Pi_0$ are nearly equivalent.
\end{proposition}

\begin{proof}
We first note that we may assume $\re(s_\alpha),\re(s_\beta)\geq 0$; indeed, if $\gamma\in\{\alpha,\beta\}$, then $\pi_\gamma$ is tempered. So if $\re(s)<0$, and if $v$ is finite place which is unramified for $\pi_\gamma$, then there is a nonvanishing intertwining operator
\[\iota_{P_\gamma(\Q_v)}^{G_2(\Q_v)}(\pi_{\gamma,v}^\vee,-s)\to\iota_{P_\gamma(\Q_v)}^{G_2(\Q_v)}(\pi_{\gamma,v},s),\]
whose image is isomorphic to the unique unramified quotient of the source representation. It follows that
\[\iota_{P_\gamma(\A_f)}^{G_2(\A_f)}(\pi_\gamma^\vee,-s)\qquad\textrm{and}\qquad\iota_{P_\gamma(\A_f)}^{G_2(\A_f)}(\pi_\gamma,s)\]
are nearly equivalent.

Thus assume $\re(s)\geq 0$. For $\Pi$ an automorphic representation of $G_2(\A)$, $\chi$ a finite order character of $\A^\times$ and $S$ a finite set of places of $\Q$ including the archimedean place and the ramified places for $\Pi$ and $\chi$, we will consider the partial $L$-function $L^S(s,R_7(\Pi)\times\chi)$, where $R_7$ is the standard $7$-dimensional representation of $G_2$ (see Section \ref{subsecg2str}). The definition of this $L$-function is as follows. If $v\notin S$ is a place, $\chi_v:\Q_v^\times\to\C^\times$ is the local component of $\chi$ at $v$, and $s_v\in G_2(\C)$ is the Satake parameter of $\Pi_v$, then the corresponding local $L$-factor is defined to be
\[L_v(s,R_7(\Pi)\times\chi)=\det(1-\chi_v(p_v)\ell_v^{-s}R_7(s_v)),\]
where $\ell_v$ is the prime corresponding to $v$. Then the global $L$-function is defined as
\[L^S(s,R_7(\Pi)\times\chi)=\prod_{v\notin S}L_v(s,R_7(\Pi)\times\chi).\]

Now let $\Pi_\alpha$ and $\Pi_\beta$ be as in the statement of the proposition. Then it follows from \eqref{eqr7beta} that for $S$ sufficiently large, we have
\[L^S(s,R_7(\Pi_\alpha)\times\chi)=L^S(s+s_\alpha,\pi_\alpha\otimes\chi)L^S(s-s_\alpha,\pi_\alpha^\vee\otimes\chi)L^S(s,\Ad(\pi_\alpha)\times\chi),\]
where the partial adjoint $L$-function for $GL_2$ is defined in a way analogous to the $L$-function defined above. If we write $\omega_{\pi_\beta}$ for the central character of $\pi_\beta$, the we also have, by \eqref{eqr7alpha}, that
\[L^S(s,R_7(\Pi_\beta)\times\chi)=L^S(s+s_\beta,\pi_\beta\otimes\chi)L^S(s-s_\beta,\pi_\beta^\vee\otimes\chi)L(s+s_\beta,\omega_{\pi_\beta}\chi)L(s-s_\beta,\omega_{\pi_\beta}^{-1}\chi)\zeta(s).\]

Now on the one hand, since $\pi_\alpha$ is assumed to be cuspidal and unitary, by the work of Gelbart--Jacquet \cite[Theorem 9.3.1, Remark 9.9]{GJ}, the $L$-function $L^S(s,\Ad^2(\pi_\alpha)\times\chi)$ has at worst a simple pole at $s=1$. Therefore the same is true for $L^S(s,R_7(\Pi_\alpha)\times\chi)$. On the other hand, we have:
\begin{itemize}
\item The $L$-function $L^S(s+s_\beta,\pi_\beta\otimes\chi)$ does not vanish at $s=1+s_\beta$: if $\re(s_\beta)>0$, this follows from the temperedness of $\pi_\beta$, and otherwise, it follows from the ``prime number theorem'' of Jacquet--Shalika \cite{JS};
\item The $L$-function $L^S(s-s_\beta,\pi_\beta^\vee\otimes\chi)$ does not vanish when $s=1+s_\beta$, again by Jacquet--Shalika;
\item The $L$-function $L(s+s_\beta,\omega_{\pi_\beta}\chi)$ and the zeta function $\zeta(s)$ do not vanish at $1+s_\beta$ for similar but more elementary reasons, and $\zeta(s)$ has a pole at $s=1$;
\item The $L$-function $L(s-s_\beta,\omega_{\pi_\beta}^{-1}\chi)$ has a pole when $\chi=\omega_{\pi_\beta}$ and $s=1+s_\beta$.
\end{itemize}
Therefore, the $L$-function $L^S(s,R_7(\Pi_\beta)\times\omega_{\pi_\beta})$ has a pole at $s=1+s_\beta$, which is simple if $s_\beta\ne 0$, and is at least double otherwise. Since $L^S(s,\Ad^2(\pi_\alpha)\times\chi)$ does not have this property, we cannot have that $\Pi_\alpha$ and $\Pi_\beta$ are nearly equivalent.

Now for $\Pi$ again an automorphic representation of $G_2(\A)$, and $\pi$ an automorphic representation of $GL_2(\A)$, we consider the degree $14$ $L$-function
\[L^S(s,\Pi\times\pi,R_7\otimes\std),\]
defined in the obvious way. Then on the one hand, we have,
\[L^S(s,\Pi_\alpha\times\pi_\alpha,R_7\otimes\std)=L^S(s+s_\alpha,\pi_\alpha\times\pi_\alpha)L^S(s-s_\alpha,\pi_\alpha^\vee\times\pi_\alpha)L^S(s,\pi_\alpha,\Ad^3)L^S(s,\pi_\alpha),\]
where the first two factors are Rankin--Selberg $L$-function, and where the third factor is the $\Ad^3$ $L$-function, $\Ad^3=\Sym^3\otimes\det^{-1}$. If $\re(s)>0$, then since $\pi_\alpha$ is tempered, all of the $L$-functions in the product on the right hand side are holomorphic and nonvanishing at $s=1-s_\alpha$ except for the second one, $L^S(s,\pi_\alpha\times\pi_\alpha^\vee)$, which has a pole at $s=1-s_\alpha$.

Otherwise, if $\re(s)=0$, then none of these $L$-functions vanish by another prime number theorem, this time due to Shahidi, \cite[Theorem 5.1]{shahidiLfcns}; this theorem says that $L$-functions appearing in the constant terms of Eisenstein series satisfy a prime number theorem under certain conditions. Since
\[L(s,\pi_\alpha,\Ad^3)L(2s,\omega_{\pi_\alpha})\]
appears in such a way via the long root parabolic of $G_2$ (see, for example, \cite{shahidisym3}) we see that the $\Ad^3$ $L$-function $L^S(s,\pi_\alpha,\Ad^3)$ does not vanish along $\re(s)=1$, nor do the Rankin--Selberg $L$-functions $L^S(s,\pi_\alpha\times\pi_\alpha)$ or $L^S(s,\pi_\alpha^\vee\times\pi_\alpha)$ (since these latter $L$-functions can be seen via the Langlands--Shahidi method for $GL_2\times GL_2\subset GL_4$).

Thus, in any case, $L^S(s,\Pi_\alpha\times\pi_\alpha,R_7\otimes\std)$ has a pole at $s=1-s_\alpha$. On the other hand, $L^S(s,\Pi_0\times\pi_\alpha,R_7\otimes\std)$ is a product of seven $L$-functions of various character twists of $\pi_\alpha$. Since $\pi_\alpha$ is cuspidal, these $L$-functions are entire, whence $\Pi_0$ is not nearly equivalent to $\Pi_\alpha$.

A completely analogous argument to this, using twists by $\pi_\beta$ instead of $\pi_\alpha$, distinguishes $\Pi_\beta$ from $\Pi_0$ as well; here the Prime Number Theorem for $\Ad^3$ is not needed; instead, one only needs the Prime Number Theorem of Jacquet--Shalika, or that for Rankin--Selberg $L$-functions for $GL_2\times GL_2$.
\end{proof}

\begin{remark}
An earlier version of this paper proved only a weaker version of the above result, and used Galois representations to do it. We are grateful to Sug Woo Shin for suggesting there should be a purely automorphic proof of this result along these lines.
\end{remark}

We will also need to distinguish between representations occurring in the Eisenstein cohomology of $G_2$ which come from the same maximal parabolic of $G_2$. To do this, we will appeal to strong multiplicity one for the Levi, as in the following proposition.

\begin{proposition}
\label{propneareqalpha}
Let $\pi$, $\pi'$ be tempered, unitary automorphic representations of $GL_2(\A)$, viewed as representations of $M_\alpha(\A)$. Let $s,s'>0$. If there are irreducible subquotients
\[\Pi\textrm{ of }\iota_{P_\alpha(\A)}^{G_2(\A)}(\pi,s)\]
and
\[\Pi'\textrm{ of }\iota_{P_\alpha(\A)}^{G_2(\A)}(\pi',s)\]
such that $\Pi$ and $\Pi'$ are nearly equivalent, then $s=s'$ and $\pi\cong\pi'$.
\end{proposition}

\begin{proof}
Let $v$ be a finite place which is unramified for both $\Pi$ and $\Pi'$, and where $\Pi_v\cong\Pi_v'$. Then $\pi_v$ and $\pi_v'$ are unramified. Since $\pi$ and $\pi'$ are tempered and unitary, there are unramified unitary characters $\chi_v,\chi_v'$ of $T(\Q_v)$ such that $\pi_v$ is the unique unramified subquotient of
\[\iota_{(B\cap M_\alpha)(\Q_v)}^{M_\alpha(\Q_v)}(\chi)\]
and $\pi_v'$ is the unique unramified subquotient of
\[\iota_{(B\cap M_\alpha)(\Q_v)}^{M_\alpha(\Q_v)}(\chi')\]
By induction in stages and the fact that $\Pi_v\cong\Pi_v'$, we have that the unique unramified subquotients of
\[\iota_{B(\Q_v)}^{G_2(\Q_v)}(\chi\delta_{M_\alpha(\Q_v)}^s)\quad\textrm{and}\quad\iota_{B(\Q_v)}^{G_2(\Q_v)}(\chi'\delta_{M_\alpha(\Q_v)}^{s'})\]
coincide. By the theory of Satake parameters, this implies that there is an element $w$ in the Weyl group $W(G_2,T)$ such that
\[w(\chi\delta_{M_\alpha(\Q_v)}^s)=\chi'\delta_{M_\alpha(\Q_v)}^{s'}.\]

Now since $\chi$ and $\chi'$ are unitary, taking absolute values gives
\[(w\delta_{M_\alpha(\A)})^s(t)=\delta_{M_\alpha(\A)}^{s'}(t)\]
for any $t\in T(\Q_p)$. Let $\gamma$ be a root and consider the equation above with $t=\gamma^\vee(\ell_v^{-1})$ where $\ell_v$ is the prime corresponding to the place $v$; since
\[\delta_{M_\alpha(\Q_v)}(\gamma^\vee(\ell_v^{-1}))=\ell_v^{5\langle\alpha+2\beta,\gamma^\vee\rangle},\]
this gives
\[\ell_v^{5s\langle w(\alpha+2\beta),\gamma^\vee\rangle}=\ell_v^{5s'\langle \alpha+2\beta,\gamma^\vee\rangle}.\]
Since $\gamma$ was arbitrary, since $s,s'>0$, and since $W(G_2,T)$ acts on the short roots of $G_2$ with the stabilizer of $\alpha+2\beta$ being $\{1,w_\alpha\}$, this forces $s=s'$ and also $w=1$ or $w=w_\alpha$. But the induced representations
\[\iota_{(B\cap M_\alpha)(\Q_v)}^{M_\alpha(\Q_v)}(\chi\delta_{M_\alpha(\Q_v)}^s)\quad\textrm{and}\quad\iota_{(B\cap M_\alpha)(\Q_v)}^{M_\alpha(\Q_v)}(w_\alpha(\chi\delta_{M_\alpha(\Q_v)}^s))\]
have the same unramified subquotients. Thus $\pi_v\cong\pi_v'$. Since this holds for almost all $v$, $\pi\cong\pi'$ by strong multiplicity one for $GL_2$.
\end{proof}

\begin{remark}
An analogous result as the above proposition holds with $P_\beta$ or $B$ in place of $P_\alpha$, and the proof is also completely analogous in either case. We only need this proposition for $P_\alpha$ in this paper, however.
\end{remark}

We now begin to examine cohomology spaces for $G_2$, starting with the following proposition.

\begin{proposition}
\label{propcohindalpha}
Let $E$ be an irreducible, finite dimensional representation of $G_2(\C)$, and say that $E$ has highest weight $\Lambda$. Write
\[\Lambda=c_1(2\alpha+3\beta)+c_2(\alpha+2\beta)\]
with $c_1,c_2\in\Z_{\geq 0}$. Let $F$ be a cuspidal eigenform of weight $k$ and trivial nebentypus and $\pi_F$ its associated automorphic representation, and let $s\in\C$ with $\re(s)\geq 0$. Assume
\[H^i(\mf{g}_2,K_\infty;\Ind_{P_\alpha(\A)}^{G_2(\A)}(\pi_F\otimes\Sym(\mf{a}_{P_\alpha,0})_{(2s+1)\rho_{P_\alpha}})\otimes E)\ne 0.\]
Then either:
\begin{enumerate}[label=(\roman*)]
\item We have
\[i=4,\qquad k=2c_1+c_2+4,\qquad s=\frac{c_2+1}{10},\]
and
\[H^4(\mf{g}_2,K_\infty;\Ind_{P_\alpha(\A)}^{G_2(\A)}(\pi_F\otimes\Sym(\mf{a}_{P_\alpha,0})_{(2s+1)\rho_{P_\alpha}})\otimes E)\cong\iota_{P_\alpha(\A_f)}^{G_2(\A_f)}(\pi_{F,f},(c_2+1)/10),\]
\item We have
\[i=5,\qquad k=c_1+c_2+3,\qquad s=\frac{3c_1+c_2+4}{10},\]
and
\[H^5(\mf{g}_2,K_\infty;\Ind_{P_\alpha(\A)}^{G_2(\A)}(\pi_F\otimes\Sym(\mf{a}_{P_\alpha,0})_{(2s+1)\rho_{P_\alpha}})\otimes E)\cong\iota_{P_\alpha(\A_f)}^{G_2(\A_f)}(\pi_{F,f},(3c_1+c_2+4)/10),\]
\item We have
\[i=6,\qquad k=c_1+2,\qquad s=\frac{3c_1+2c_2+5}{10},\]
and
\[H^6(\mf{g}_2,K_\infty;\Ind_{P_\alpha(\A)}^{G_2(\A)}(\pi_F\otimes\Sym(\mf{a}_{P_\alpha,0})_{(2s+1)\rho_{P_\alpha}})\otimes E)\cong\iota_{P_\alpha(\A_f)}^{G_2(\A_f)}(\pi_{F,f},(3c_1+2c_2+5)/10).\]
\end{enumerate}
\end{proposition}

\begin{proof}
Note that we have a decomposition
\[\mf{t}=(\mf{m}_{\alpha,0}\cap\mf{t})\oplus\mf{a}_{P_\alpha,0},\]
and $(\alpha+2\beta)$ acts by zero on the first component while $\alpha$ acts by zero on the second. By Theorem \ref{thmcohind}, in order for our cohomology space to be nontrivial, there needs to be a $w\in W^{P_\alpha}$ with
\[-w(\Lambda+\rho)|_{\mf{a}_{P_\alpha,0}}=2s\rho_{P_\alpha}=10s\frac{\alpha+2\beta}{2},\]
and
\[-w(\Lambda+\rho)|_{\mf{m}_{\alpha,0}}=\pm(k-1)\frac{\alpha}{2}.\]
One computes that, since $\re(s)\geq 0$, the first of these conditions is possible only when $w=w_{\beta\alpha\beta},w_{\beta\alpha\beta\alpha},w_{\beta\alpha\beta\alpha\beta}$. In case $w=w_{\beta\alpha\beta}$, we have
\[-w_{\beta\alpha\beta}(\Lambda+\rho)=-(2c_1+c_2+3)\frac{\alpha}{2}+(c_2+1)\frac{\alpha+2\beta}{2}.\]
Then Theorem \ref{thmcohind} implies (i). The other two cases are similar.
\end{proof}

The following lemma is key; it is one place where we use the vanishing hypothesis for the symmetric cube $L$-function in the course of proving our main theorem.

\begin{lemma}
\label{lemholoESg2}
Let $F$ be a cuspidal eigenform of weight $k\geq 2$ and trivial nebentypus, and $\pi_F$ its associated unitary automorphic representation. If
\[L(1/2,\pi_F,\Sym^3)=0,\]
then for any flat section $\phi_s\in\iota_{P_\alpha(\A)}^{G_2(\A)}(\pi_F,s)$, the Eisenstein series $E(\phi,2s\rho_{P_\alpha})$ is holomorphic at $s=1/10$.
\end{lemma}

\begin{proof}
It follows from the Langlands--Shahidi method and the vanishing hypothesis on the $L$-function that the constant term of such an Eisenstein series as in the proposition is holomorphic at $s=1/10$; see, for example, \cite{shahidisym3} or \cite{zampera}. 
\end{proof}

For $\pi$ a unitary cuspidal automorphic representation of $GL_2(\A)$ and $s\in\C$ with $\re(s)>0$, let us write
\[\mc{L}_\alpha(\pi,s)=\textrm{Langlands quotient of }\iota_{P_\alpha(\A)}^{G_2(\A)}(\pi,s),\]
where we view $\pi$ as an automorphic representation of $M_\alpha(\A)$, as usual.
Let $F$ be a cuspidal eigenform with trivial nebentypus and $\pi_F$ the automorphic representation attached to $F$. We will now study the appearance of $\mc{L}_\alpha(\pi_F,1/10)$ in the Eisenstein cohomology of $G_2$. To be precise about what this means, we make the following definition.

\begin{definition}
\label{defeiscoh}
Let $G$ be a reductive group over $\Q$ with complex Lie algebra $\mf{g}$ and let $K_\infty'$ be an open subgroup of the maximal compact subgroup of $G(\R)$. Let $E$ be a finite dimensional representation of $G(\C)$. Then we define the \textit{Eisenstein cohomology} of $G$ by
\[H_{\Eis}^*(\mf{g},K_\infty';E)=H^*(\mf{g},K_\infty';\bigoplus_{\substack{[P]\in\mc{C}\\ [P]\ne[G]}}\mc{A}_{E,[P]}(G)\otimes E),\]
where the notation in the sum is as in \eqref{eqnparadecomp}.
\end{definition}

Let $k$ be the weight of $F$ and assume $k\geq 4$. Let $\lambda_0$ be the weight
\[\lambda_0=\frac{k-4}{2}(2\alpha+3\beta).\]
Let $E_{\lambda_0}$ be the representation of $G_2(\C)$ of highest weight $\lambda_0$. Our goal is to locate $\mc{L}_\alpha(\pi_F,1/10)$ in $H_{\Eis}^*(\mf{g}_2,K_\infty;E_\Lambda)$. In doing so, we will require the following lemma.

\begin{lemma}
\label{lemnontempcoh}
Let $\sigma$ be an irreducible, admissible, nontempered representation of $GL_2(\R)$ with infinitesimal character $s_0\gamma_0$ with $s_0\geq 1/2$, where $\gamma_0$ is the positive root for $GL_2$. Then for any $s_1\geq 0$ and $\gamma\in\{\alpha,\beta\}$, the Langlands quotient of the induced representation
\begin{equation}
\label{eqindnontemp}
\iota_{P_\gamma(\R)}^{G_2(\R)}(\sigma,s_1)
\end{equation}
is not cohomological, i.e., for any finite dimensional representation $E$ of $G_2(\C)$, the $(\mf{g}_2,K_\infty)$-cohomology of this Langlands quotient twisted by $E$ vanishes.
\end{lemma}

\begin{proof}
Assume on the contrary that the Langlands quotient, call it $\tau$, of \eqref{eqindnontemp} were cohomological. Then the infinitesimal character of $\tau$ must be integral, forcing $2s_0,2s_1\in\Z$ with the same parity. It then follows from the classification of irreducible admissible representations for $GL_2(\R)$ that $\sigma$ is finite dimensional. Therefore $\sigma$ occurs as the unique irreducible quotient of
\begin{equation}
\label{eqindrepgl2}
\iota_{B_2(\R)}^{GL_2(\R)}(\chi,s_0-\tfrac{1}{2})
\end{equation}
for some finite order character $\chi$ of $T_2(\R)$, where $B_2$ denotes the upper triangular Borel in $GL_2$ and $T_2$ the standard maximal torus.

But then we can invoke (the twisted version of) \cite[Theorem VI.1.7 (iii)]{BW}, which in this case says that since representation $\tau$ is assumed to be cohomological, there is an irreducible representation $U$ of $G_2(\R)$ and an exact sequence
\[0\to U\to\iota_{T(\R)}^{G_2(\R)}(\chi\delta_{B_2(\R)}^{s_0-1/2}\delta_{P_\gamma(\R)}^{s_1})\to\tau\to 0.\]
Then $U$ must be induced from the discrete series subrepresentation of \eqref{eqindrepgl2}. But then by \cite[Theorem III.3.3]{BW}, the representation $U$ has cohomology in two consecutive degrees, and this is impossible by Poincar\'e duality. This is the contradiction sought.
\end{proof}

\begin{theorem}
\label{thmlqineiscoh}
Let $F$ be a cuspidal eigenform of weight $k\geq 4$ and trivial nebentypus, and let $\pi_F$ be the automorphic representation of $GL_2(\A)$ attached to it. Let
\[\lambda_0=\frac{k-4}{2}(2\alpha+3\beta)\]
and let $E_{\lambda_0}$ be the representation of $G_2(\C)$ of highest weight $\lambda_0$. Assume
\[L(1/2,\pi_F,\Sym^3)=0,\]
Then there is a unique summand isomorphic to
\[\iota_{P_\alpha(\A_f)}^{G_2(\A_f)}(\pi_{F,f},1/10)\]
in the Eisenstein cohomology
\[H_{\Eis}^*(\mf{g}_2,K_\infty;E_{\lambda_0}),\]
and all irreducible subquotients of this cohomology nearly isomorphic to $\mc{L}_\alpha(\pi_F,1/10)_f^p$ appear in this summand. Moreover, this summand appears in middle degree $4$.
\end{theorem}

\begin{proof}
Let $\varphi_F$ be the associate class of automorphic representations of $M_\alpha(\A)$ containing $\pi_F\otimes\delta_{P_\alpha(\A)}^{1/10}$. Then by Proposition \ref{propEh0} and Lemma \ref{lemholoESg2}, we have
\[\mc{A}_{E_{\lambda_0},[P_{\alpha}],\varphi_F}\cong\iota_{P_\alpha(\A)}^{G_2(\A)}(\pi_F\otimes\Sym(\mf{a}_{P_\alpha,0})_{(6/5)\rho_{P_\alpha}}).\]
By Proposition \ref{propcohindalpha}, we therefore have
\[H^4(\mf{g}_2,K_\infty;\mc{A}_{E_{\lambda_0},[P_{\alpha}],\varphi_F}(G_2)\otimes E_{\lambda_0})\cong\iota_{P_\alpha(\A_f)}^{G_2(\A_f)}(\pi_{F,f},1/10),\]
and that the this cohomology vanishes in all other degrees.

By the Franke--Schwermer decomposition, Theorem \ref{thmfsdecomp}, in order to prove our theorem, it now suffices to show that
\[H^*(\mf{g}_2,K_\infty;\mc{A}_{E_{\lambda_0},[P],\varphi}(G_2)\otimes E_{\lambda_0})\]
contains no irreducible subquotients nearly equivalent to $\mc{L}_\alpha(\pi_F,1/10)_f$ for any proper parabolic $P\subset G_2$ and any associate class $\varphi$, except for $P=P_\alpha$ and $\varphi=\varphi_F$. We do this for the maximal parabolics $P_\alpha$ and $P_\beta$ and for $B$ separately; thus the following two lemmas complete the proof of the theorem.
\end{proof}

\begin{lemma}
Let $\gamma\in\{\alpha,\beta\}$, and let $\varphi$ be an associate class for $P_\gamma$. Assume that for some $i$,
\[H^i(\mf{g}_2,K_\infty;\mc{A}_{E_{\lambda_0},[P_\gamma],\varphi}(G_2)\otimes E_{\lambda_0})\]
contains a subquotient nearly equivalent to $\mc{L}_\alpha(\pi_F,1/10)_f$. Then $\gamma=\alpha$ and $\varphi=\varphi_F$.
\end{lemma}

\begin{proof}
The class $\varphi$ contains a cuspidal automorphic representation of $M_\gamma(\A)\cong GL_2(\A)$, and which therefore must be of the form
\[\pi\otimes\delta_{M_\gamma(\A)}^{s},\]
where $\pi$ is a unitary cuspidal automorphic representation of $GL_2(\A)$ and $s\in\C$. After possibly conjugating by the longest element in the Weyl set $W^{P_\gamma}$, we may even assume $\re(s)\geq 0$.

Next we note that the infinitesimal character of $\mc{A}_{E_{\lambda_0},[P_\gamma],\varphi}(G_2)$ as a $(\mf{g}_2,K_\infty)$-module must match that of $E_{\lambda_0}$, i.e.,
\[\lambda_\pi+2s\rho_{P_\gamma}=\lambda_0+\rho,\]
where $\lambda_{\pi}$ is the infinitesimal character of $\pi$. But $\lambda_0+\rho$ is regular and real, and so since $\lambda_{\pi}$ is a multiple of the root $\gamma$ and $\rho_{P_\gamma}$ is a multiple of the positive root orthogonal to $\gamma$, it follows that $\lambda_{\pi}$ and $s$ are real and nonzero. In particular, $s>0$ since we assumed $\re(s)\geq 0$.

Now we apply Theorem \ref{thmgrbac} and Proposition \ref{propEh0} to find that the cohomology space
\[H^*(\mf{g}_2,K_\infty;\mc{A}_{E_{\lambda_0},[P_\gamma],\varphi}(G_2)\otimes E_{\lambda_0}),\]
if nontrivial, is made up of subquotients of the cohomology spaces
\begin{equation}
\label{eqcohlalphag2}
H^*(\mf{g}_2,K_\infty;\mc{L}_\gamma(\pi,s)\otimes E_{\lambda_0})
\end{equation}
and
\begin{equation}
\label{eqcohindalphag2}
H^*(\mf{g}_2,K_\infty;\Ind_{P_\gamma(\A)}^{G_2(\A)}(\pi\otimes\Sym(\mf{a}_{P_\gamma,0})_{(2s+1)\rho_{P_\gamma}})\otimes E_{\lambda_0}).
\end{equation}
We note that if \eqref{eqcohlalphag2} is nonzero, then $\pi$ is cohomological. Indeed, in this case, by Lemma \ref{lemnontempcoh}, the archimedean component $\pi_\infty$ of $\pi$ is tempered. (Of course, $\pi_\infty$ should be tempered by Selberg's conjecture, but obviously we would like to avoid a dependency on this conjecture, whence the appeal to Lemma \ref{lemnontempcoh}.) Since it has regular infinitesimal character, it is discrete series and therefore cohomological.

Next we have that if \eqref{eqcohindalphag2} is nonzero, then $\pi$ is cohomological; indeed, the cohomology in \eqref{eqcohindalphag2} is computed in terms of that of $\pi$ by Theorem \ref{thmcohind}. Therefore $\pi$ is cohomological. Thus $\pi$ is attached to a cuspidal eigenform of weight at least $2$ and is therefore tempered, and so by Proposition \ref{propdistneareq}, $\gamma=\alpha$, and then by Proposition \ref{propneareqalpha}, $\pi=\pi_F$ and $s=s'$. Whence also $\varphi=\varphi'$, as desired.
\end{proof}

\begin{lemma}
Let $\varphi$ be an associate class for $B$. Then the cohomology
\[H^*(\mf{g}_2,K_\infty;\mc{A}_{E_\Lambda,[B],\varphi}(G_2)\otimes E_{\Lambda})\]
does not contain any subquotient nearly equivalent to $\mc{L}_\alpha(\pi_F,1/10)_f$.
\end{lemma}

\begin{proof}
Note that the class $\varphi$ must contains a character of $T(\Q)\backslash T(\A)$ of the form
\[\psi e^{\langle H_B(\cdot),s_1\alpha+s_2\beta\rangle},\]
where $\psi$ is of finite order and $s_1,s_2\in\C$. We will study the piece $\mc{A}_{E_{\lambda_0},[B],\varphi}(G_2)$ of the Franke--Schwermer decomposition using the Franke filtration of Theorem \ref{thmfrfil}. By that theorem, there is a filtration on the space $\mc{A}_{E_{\lambda_0},[B],\varphi}(G_2)$ whose graded pieces are parametrized by certain quadruples $(Q,\nu,\Pi,\mu)$. For the convenience of the reader, we recall what these quadruples consist of now:
\begin{itemize}
\item $Q$ is a standard parabolic subgroup of $G_2$;
\item $\nu$ is an element of $(\mf{t}\cap\mf{m}_{Q,0})^\vee$;
\item $\Pi$ is an automorphic representation of $M_Q(\A)$ occurring in
\[L_{\disc}^2(M_Q(\Q)A_Q(\R)^\circ\backslash M_Q(\A))\]
and which is spanned by values at, or residues at, the point $\nu$ of Eisenstein series parabolically induced from $(B\cap M_Q)(\A)$ to $M_Q(\A)$ by representations in $\varphi$; and
\item $\mu$ is an element of $\mf{a}_{Q,0}^\vee$ whose real part in $\Lie(A_{M_Q}(\R))$ is in the closure of the positive cone, and such that $\nu+\mu$ lies in the Weyl orbit of $\lambda_0+\rho$.
\end{itemize}
Then the graded pieces of $\mc{A}_{E,[B],\varphi}(G_2)$ are isomorphic to direct sums of $G_2(\A_f)\times(\mf{g}_2,K_\infty)$-modules of the form
\[\Ind_{Q(\A)}^{G_2(\A)}(\Pi\otimes\Sym(\mf{a}_{Q,0})_{\mu+\rho_Q})\]
for certain quadruples $(Q,\nu,\Pi,\mu)$ of the form just described.

For each of the four possible parabolic subgroups $Q$ and any corresponding quadruple $(Q,\nu,\Pi,\mu)$ as above, we will show using Proposition \ref{propdistneareq} that the cohomology
\begin{equation}
\label{eqcohborelg2}
H^*(\mf{g}_2,K_\infty;\Ind_{Q(\A)}^{G_2(\A)}(\Pi\otimes\Sym(\mf{a}_{Q,0})_{\mu+\rho_Q}))
\end{equation}
cannot have $\mc{L}_\alpha(\pi_{F,f},1/10)$ as a subquotient, which will finish the proof.

So first assume we have a quadruple $(Q,\nu,\Pi,\mu)$ as above where $Q=B$. Then $\mf{m}_{Q,0}=0$, forcing $\nu=0$. The entry $\Pi$ is the unitarization of a representation in $\varphi$, and thus must be a character $\psi'$ of $T(\A)$ conjugate by $G_2(\A)$ to $\psi$. Finally, we have $\mu$ is Weyl conjugate to $\lambda_0+\rho$. Therefore the cohomology \eqref{eqcohborelg2} is isomorphic, by Theorem \ref{thmcohind}, to a finite sum of copies of
\[\iota_{B(\A_f)}^{G_2(\A_f)}(\psi_f',\mu).\]
By Proposition \ref{propdistneareq}, $\mc{L}_\alpha(\pi_F,1/10)_f$ cannot be nearly equivalent to a subquotient of this space, and we conclude in the case when $Q=B$.

If now we have a quadruple $(Q,\nu,\Pi,\mu)$ where $Q=P_\alpha$, then we find that $\Pi$ is a representation generated by residual Eisenstein series at the point $\nu$ and is therefore a subquotient of the normalized induction
\[\iota_{(B\cap M_\alpha)(\A)}^{M_\alpha(\A)}(\psi',\nu),\]
where $\psi'$ is as above. Then by Theorem \ref{thmcohind} and induction in stages, \eqref{eqcohborelg2} is isomorphic to a subquotient of a finite sum of copies of
\[\iota_{B(\A_f)}^{G_2(\A_f)}(\psi_f',\nu+\mu).\]
We then conclude in this case as well using Proposition \ref{propdistneareq}.

The case when $Q=P_\beta$ is completely similar, and we omit the details. When $Q=G_2$, it is once again similar, but we do not need to use induction in stages. So we are done.
\end{proof}

\section{The $p$-adic deformation}
\label{secdeform}
We now $p$-adically deform the representation $\mc{L}_\alpha(\pi_F,1/10)$ of the previous section, at least with the help of Conjecture \ref{conjmult} (b). This section proceeds as follows. In Section \ref{subsecev} we review some of the theory of Urban's eigenvariety in general. Then we return to the $G_2$ setting in Section \ref{subsecpstabns} and prove various preliminary results on principal series representations. In Section \ref{subsecchardistg2} we define the cuspidal overconvergent character distributions for $G_2$ and its Levi subgroups, and then we define in Section \ref{subsecpif} the $p$-stabilization $\Pi_F^{(p)}$ of $\mc{L}_\alpha(\pi_F,1/10)$ whose multiplicity in these distributions we would like to compute. We compute this multiplicity to be nonzero in \ref{subsecocmult} (under Conjecture \ref{conjmult} (b)) and consequently we get the desired $p$-adic deformation of $\Pi_F^{(p)}$ which we then study in Section \ref{subsecpadicfam}.

\subsection{Background on Urban's eigenvariety}
\label{subsecev}
We now recall the theory of Urban's eigenvariety as it is described in \cite{urbanev}. Besides recalling the general theory, we will also prove a general lemma (Lemma \ref{lemintsummand} below) which will be useful to us later. Since this lemma is general, we choose to work in this subsection in the setting of a general group $G$, and then specialize back to $G=G_2$ afterward.

So let $G$ be a reductive group over $\Q$ which is quasi-split over $\Q_p$, splitting over an unramified extension of $\Q_p$. Fix a Borel subgroup $B$ defined over $\Q_p$ and $T\subset B$ a maximal torus. Fix a maximal compact subgroup $K_{f,\mr{max}}\subset G(\A_f)$ which is hyperspecial at all places, and let $K_{f,\mr{max}}^p\subset G(\A_f^p)$ be the component away from $p$. We assume the component of $G$ at $p$ is given by $G(\Z_p)$ after fixing a model of $G$ over $\Z_p$.

We assume that $T$ and $B$ admit models over $\Z_p$ compatible with each other and that of $G$. We consider the Iwahori subgroup
\[I=\Sset{g\in G(\Z_p)}{(g\modulo{p})\in B(\F_p)}.\]
Writing $N$ for the unipotent radical of $B$, let 
\[T^{-}=\Sset{t\in T(\Z_p)}{tN(\Z_p)t^{-1}\subset N(\Z_p)}.\]
Also let
\[T^{--}=\Sset{t\in T(\Z_p)}{\bigcap_{n>0}t^nN(\Z_p)t^{-n}=0}.\]
Then $T^-\subset T^{--}$. Let $\mc{U}_p$ be the $\Z_p$-subalgebra of $C_c^\infty(I\backslash G(\Z_p)/I,\Z_p)$ generated by the characteristic functions of double cosets of the form $ItI$ with $t\in T^-$. This is a commutative algebra under convolution. In fact, writing
\[u_t=\frac{1}{\vol(I)}\chars(ItI)\]
for $t\in T^-$, then for any other $t'\in T^{-}$, we have
\[u_tu_{t'}=u_{tt'}=u_{t'}u_t.\]

As in \cite[\S 4.1]{urbanev}, we define the $\Q_p$-Hecke algebras
\[\mc{H}_p=C_c^\infty(G(\A_f^p),\Q_p)\otimes_{\Z_p}\mc{U}_p,\]
and
\[\mc{H}_p(K_f^p)=C_c^\infty(K_f^p\backslash G(\A_f^p)/K_f^p,\Q_p)\otimes_{\Z_p}\mc{U}_p.\]
We consider the ideals $\mc{H}_p'$ (respectively $\mc{H}_p'(K_f^p)$) in $\mc{H}_p$ (respectively $\mc{H}_p(K_f^p)$) generated by elements $f^p\otimes u_t$ where $f^p\in C_c^\infty(G(\A_f^p),\Q_p)$ (respectively, $f^p\in C_c^\infty(K_f^p\backslash G(\A_f^p)/K_f^p,\Q_p)$) and $t\in T^{--}$.

A representation $V$ of $\mc{H}_p$ (respectively, $\mc{H}_p'$) with coefficients in some finite extension $L$ of $\Q_p$ is called \textit{admissible} if every $f$ in $\mc{H}_p$ (respectively, $\mc{H}_p'$) acts by endomorphisms of finite rank. For example, if instead $V$ is a smooth admissible (in the usual sense) representation of $G(\A_f)$, then the algebra $\mc{H}_p$ acts by convolution on $V$ and makes $V$ admissible for this action. If this action is nontrivial (which is the case if $V$ has fixed vectors by $K_f^pI$ for some open compact subgroup $K_f^p\subset K_{f,\mr{max}}^p$) then this action determines the representation $V$ of $G(\A_f)$ up to semisimplification.

Now if $V$ is again an admissible $\mc{H}_p$-module (respectively, $\mc{H}_p'$-module) then for an open compact subgroup $K_f^p\subset K_{f,\mr{max}}^p$, we say $V$ is \textit{of level} $K_f^p$ if $\mc{H}_p(K_f^p)$ (respectively, $\mc{H}_p'(K_f^p)$) acts nontrivially on $V$. Then the space $V^{K_f^p}$ defined by
\[V^{K_f^p}=\mc{H}_p(K_f^p)V\]
becomes a finite rank $\mc{H}_p(K_f^p)$-module (respectively, $\mc{H}_p'(K_f^p)$-module). Then $V$ is determined by $V^{K_f^p}$. In any case, for any $f$ in $\mc{H}_p$ (respectively, $\mc{H}_p'$) the trace $\tr(f|V)$ is well defined.

If $V$ is an irreducible admissible $\mc{H}_p'$-module, then $V$ determines a character $\theta$ of $\mc{U}_p$ defined by
\[\tr(uf|V)=\theta(u)\tr(f|V),\]
for any $f\in\mc{H}_p'$ and $u\in\mc{U}_p$. We say $\theta$ is of \textit{finite slope} if for some (equivalently, every) $t\in T^{--}$ we have $\theta(u_t)\ne 0$. We say $V$ is of \textit{finite slope} if $\theta$ is and if there is a $K_f^p$ such that $V$ is of level $K_f^p$ and $V^{K_f^p}$ contains an $\mc{O}_L$-lattice which is stable under the subalgebra
\[C_c^\infty(K_f^p\backslash G(\A_f^p)/K_f^p,\Z_p)\otimes\mc{U}_p\]
of $\mc{H}_p(K_f^p)$. Moreover, we define the \textit{slope} of $V$ or of $\theta$ to be the character $\mu\in X^*(T)\otimes\Q$ such that
\[v_p(\theta(u_{\lambda^\vee(p)}))=\langle\mu,\lambda^\vee\rangle\]
for any rational cocharacter $\lambda^\vee$ of $T$ such that $\lambda^\vee(p)\in T^{--}$.

When $V$ is an irreducible admissible representation of $G(\A_f)$ having $I$-fixed vectors, then in general the $V^{I}$ is not an irreducible $\mc{H}_p$-module because at $p$, only the action of $\mc{U}_p$ is considered instead of that of the full Hecke albegra at $p$. An irreducible constituent of $V$ for the action of $\mc{H}_p$ is called a $p$-\textit{stabilization} of $V$.

We remark that in general, the notion of $p$-stabilization should involve fixed vectors by possibly deeper Iwahori subgroups $I_n$, $n\geq 1$; see \cite[\S 4.1.9]{urbanev} for the general definition. However, a standard argument involving the Iwahori decomposition shows that if a given vector $v\in V^{I_n}$ is in a finite slope $p$-stabilization of $V$, then actually $v\in V^I$; see for example the argument in \cite[Lemma 4.3.6]{urbanev}. Since all the $p$-stabilized representations we consider in this paper will be of finite slope, we will be content with this definition.

Next, we call a linear map $J:\mc{H}_p'\to L$ a \textit{finite slope character distribution} if there is a countable set $\{\sigma_i\}$ consisting finite slope absolutely irreducible $\mc{H}_p'$-representations such that:
\begin{itemize}
\item For any open compact subgroup $K_f^p\subset K_{f,\mr{max}}^p$, any $h\in\Q$, and any $t\in T^{--}$, there are only finitely many indices $i$ such that $\sigma_i^{K_f^p}\ne 0$ and $v_p(\tr(\vol(K_f^p)^{-1}1_{K_f^p}\otimes u_t|\sigma_i))\leq h$;
\item There are, for each $i$, integers $m_i$ such that for any $f\in\mc{H}_p'$, we have
\[J(f)=\sum_im_i\tr(f|\sigma_i).\]
\end{itemize}
For such a $J$, and for $\sigma$ an absolutely irreducible admissible $\mc{H}_p$-representation, we write $m_J(\sigma)$ for the integer $m_i$ if there is an $i$ such that $\sigma_i=\sigma$, and otherwise we set $m_J(\sigma)=0$. We call $m_J(\sigma)$ the \textit{multiplicity} of $\sigma$ in $J$. We say $J$ is \textit{effective} if $m_J(\sigma)\geq 0$ for any $\sigma$.

Finally, if $\mf{Z}$ is an affinoid rigid analytic space over $\Q_p$ and $\mc{O}(\mf{Z})$ is its global ring of functions, then we call a $\Q_p$-linear map $\mc{H}_p'\to\mc{O}(\mf{Z})$ a $\mf{Z}$-\textit{family of effective finite slope character distributions} if for any $\overline{\Q}_p$-point $\lambda\in\mf{Z}(\overline{\Q}_p)$, the composite $\lambda\circ J:\mc{H}_p'\to\lambda(\mc{O}(\mf{Z}))$ is an effective finite slope character distribution.

We now prove the following lemma which will be used later in this paper.

\begin{lemma}
\label{lemintsummand}
Let $\mf{Z}$ be an affinoid rigid analytic space over $\Q_p$ and let $J$ and $\tau$ be two $\mf{Z}$-families of effective finite slope character distributions. For $\lambda\in\mf{Z}(\overline\Q_p)$, write $J_\lambda=\lambda\circ J$ and $\tau_\lambda=\lambda\circ\tau$. Let $S\subset\mf{Z}(\overline\Q_p)$ be a Zariski dense subset. Assume that the following holds:
\begin{enumerate}[label=(\arabic*)]
\item For all $\lambda\in\mf{Z}(\overline\Q_p)$, the distribution $\tau_\lambda$ is the trace of an irreducible admissible $\mc{H}_p$-representation $V_{\tau_\lambda}$.
\item There is an open compact subgroup $K_f^p$ in $K_{f,\mr{max}}^p$ such that for all $\lambda\in S$, we have $V_{\tau_\lambda}^{K_f^p}\ne 0$.
\item For all $\lambda\in S$, the distribution $\tau_\lambda$ occurs as a summand of $J_\lambda$ (that is, the finite slope character distribution $J_{\lambda}-\tau_{\lambda}$ is effective).
\end{enumerate}
Then given any $\lambda_0\in\mf{Z}(\overline\Q_p)$, the distribution $\tau_{\lambda_0}$ is also a summand of $J_{\lambda_0}$.
\end{lemma}

\begin{proof}
We must invoke the theory of Fredholm series used in \cite{urbanev}. Let $K_f^p$ be as in the point (2). Choose $t_0\in T^{--}$ and let $f_0=1_{K_f^p}\otimes u_{t_0}$. For any $\lambda\in\mf{Z}(\overline\Q_p)$, write
\[J_\lambda=\sum_{\sigma}m_{J_\lambda}(\sigma)\sigma\]
with each $\sigma$ irreducible and admissible and each $m_{J_\lambda}(\sigma)>0$. Then write
\[P_{J_\lambda}(f_0,X)=\prod_{\sigma}\det(1-X\cdot f_0|\sigma)^{m_{J_\lambda}(\sigma)}\]
for the Fredholm series attached to $f_0$ and $J_\lambda$ (\cite[\S 4.1.11]{urbanev}). As $\lambda$ varies, these define a Fredholm series $P_J(f_0,X)\in\mc{O}(\mf{Z})[\![X]\!]$ whose specialization at any $\lambda$ is given by $P_{J_\lambda}(f_0,X)$. For any $\lambda$, let $V_{J_\lambda}$ be the $p$-adic Banach space completion of the representation $\bigoplus_\sigma\sigma^{m_{J_\lambda}(\sigma)}$ (\cite[\S 4.1.10]{urbanev}).

For any $\lambda$, we also let
\[Q_{0,\lambda}(X)=\det(1-X\cdot f_0|V_{\tau_\lambda})\]
be the characteristic polynomial of $f_0$ acting on $\tau_\lambda$. By (1), this is well defined of degree
\[\tr(\vol(K_f^p)^{-1}1_{K_f^p}\otimes u_1|V_{\tau_\lambda}).\]
Since this expression is an integer, by continuity it is locally constant in $\lambda$. Therefore by taking connected components, we can and will assume it is constant, and we let $d$ denote this degree.

Then the coefficients of the characteristic polynomials $Q_{0,\lambda}$ for $\lambda\in S$ are polynomials in $\tau_\lambda(f^n)$ for $0\leq n\leq d$, and so it follows that there is a polynomial $Q_0(X)\in\mc{O}(\mf{Z})[X]$ whose specializations at $\lambda\in\mf{Z}(\overline\Q_p)$ are given by $Q_{0,\lambda}(X)$. Let us factor $P_{J}(f_0,X)=Q(X)S(X)$ in such a way that $Q_{0}|Q$, and such that this factorization is admissible; this means that $Q$ is a polynomial, $Q$ and $S$ are coprime, $Q(0)=1$, and $Q^*(0)\in\mc{O}(\mf{Z})^\times$, where $Q^*(X)=X^{\deg(Q)}Q(X^{-1})$; that fact that we can find such a $Q$ follows from the hypothesis (3).

Let $\mf{Y}$ be the proper Zariski closed subset of $\mf{Z}$ consisting of points $\lambda$ where the specializations $Q_\lambda$ and $S_\lambda$ of $Q$ and $S$ at $\lambda$ become non-coprime. Let $S'$ be the points of $S$ in the complement of $\mf{Y}$, which is still Zariski dense in $\mf{Z}$. Then for any $\lambda\in S'$, the factorization $P_{J_\lambda}(f_0,X)=Q_\lambda(X)S_\lambda(X)$ is still admissible. We then invoke \cite[Theorem 2.3.8]{urbanev} (see also \cite[4.1.13, 5.3.1]{urbanev}) which tells us that there is an entire power series $R_Q(X)\in\mc{O}(\mf{Z})[\![X]\!]$ such that for any $\lambda\in S'$, there is an $\mc{H}_p'(K^p)$-equivariant decomposition
\[V_{J_\lambda}=N_{\lambda}(Q)\oplus F_{\lambda}(Q)\]
with the properties that $N_\lambda(Q)$ is finite dimensional of dimension $\deg(Q)$, that $\lambda(R_Q)(f_0)$ acts as the projection onto $N_\lambda(Q)$, and that $Q_\lambda(X)$ is the characteristic polynomial of $f_0$ acting on $N_\lambda(Q)$. Because of this last property, and the facts that $Q_0|Q$ and that $Q_\lambda$ and $S_\lambda$ are coprime for $\lambda\in S'$, it follows that $N_\lambda(Q)$ contains the $\tau_\lambda$-isotypic component $V_{J_\lambda}[\tau_\lambda]$ for $\lambda\in S'$. From this it follows that
\begin{equation}
\label{eqnrqf0proj}
\tr(f\cdot \lambda(R_Q)(f_0)|\tau_\lambda)=\tr(f|\tau_{\lambda})
\end{equation}
for any $f\in\mc{H}_p(K^p)$.

Now as in \cite[5.3.1]{urbanev}, for each $\lambda\in S$, the function $T_\lambda:\mc{H}_p(K_f^p)\to L$ defined by
\[T_\lambda(f)=\tr(f\cdot\lambda(R_Q)(f_0)|V_{J_\lambda})\]
is a pseudorepresentation of $\mc{H}_p(K_f^p)$ of dimension $\deg(Q)$. By analyticity, there is a pseudorepresentation $T$ of $\mc{H}_p(K_f^p)$ over $\mc{O}(\mf{Z})$, also of dimension $\deg(Q)$, such that $\lambda(T)=T_\lambda$ for any $\lambda\in S'$.

Now, for any $f\in\mc{H}_p(K_f^p)$, we have by \eqref{eqnrqf0proj} that
\[T_\lambda(f)-\tr(f|\tau_\lambda)=\tr(f\cdot\lambda(R_Q)(f_0)|V_{J_\lambda})-\tr(f\cdot \lambda(R_Q)(f_0)|\tau_\lambda)=\tr(f\cdot\lambda(R_Q)(f_0)|V_{J_\lambda-\tau_\lambda}),\]
where $V_{J_\lambda-\tau_\lambda}$ is the space attached to the effective (by (3)) finite slope character distribution $J_\lambda-\tau_\lambda$. It then follows as above that $T_\lambda-\tr(\tau_\lambda)$ is a pseudorepresentation of dimension $\deg(Q)-\deg(Q_0)$. By analyticity, so is $T-\tau$.

Now let $\lambda_0\in\mf{Z}(\overline\Q_p)$. By continuity, $\lambda_0(R_Q)(f_0)$ is idempotent, and therefore it defines a $\mc{H}_p'(K_f^p)$-equivariant projector from $V_{J_{\lambda_0}}$ onto some subspace $N_0$. The trace $\tr(N_0)$ of the representation of $\mc{H}_p(K_f^p)N_0$ is thus a summand of $J_{\lambda_0}$. But the specialization $T_{\lambda_0}$ of $T$ at $\lambda_0$ is itself a summand of $\tr(N_0)$ by construction. Thus it follows from above that $\tau_{\lambda_0}$ is a summand of $\tr(N_0)$, and hence of $J_{\lambda_0}$, which completes the proof.
\end{proof}

The main examples of effective finite slope character distributions come from the cohomology of certain local systems on the locally symmetric spaces attached to $G$. In \cite{urbanev}, Urban defines an affinoid rigid analytic space $\mf{X}$, called weight space, depending on $G$, and for each $\lambda\in\mf{X}(\overline{\Q}_p)$, a finite slope character distribution $I_G^\dagger(\cdot,\lambda)$. This distribution has the following classicity property: When $\lambda$ is given by a dominant regular weight of $T$, then an irreducible representation of $\mc{H}_p'$ appears with the same multiplicity in $I_G^\dagger(\cdot,\lambda)$ with which it appears in the cohomology with coefficients in $E_\lambda^\vee$ of the locally symmetric spaces attached to $G$, as long as this character has slope which is \textit{noncritical} with respect to $\lambda$. In this way, the noncritical constituents of $I_G^\dagger(\cdot,\lambda)$ can be related to classical cohomological automorphic representations.

The distributions $I_G^\dagger(\cdot,\lambda)$ are not in general effective. But when $G$ has discrete series, one can combine in a certain way the distributions $I_M^\dagger(\cdot,\mu)$, and particular twists of these, for certain Levi subgroups $M$ in $G$, and certain weights $\mu$ depending on $\lambda$, to produce a new distribution $I_{G,0}^\dagger(\cdot,\lambda)$. This distribution has the classicity mentioned above but with respect to \textit{cuspidal} cohomological automorphic representations for $G$. It is also effective, and as $\lambda$ varies, the distributions $I_{G,0}^\dagger(\cdot,\lambda)$ form an $\mf{X}$-family of effective finite slope character distributions. It is in this way that finite slope, cuspidal, cohomological automorphic representations for $G$ may be $p$-adically interpolated.

At this point, there is a problem in Urban's paper, however. These facts about the cuspidal distributions $I_{G,0}^\dagger(\cdot,\lambda)$ depend on a formula of Franke from \cite{franke}, and this formula is interpreted erroneously in \cite{urbanev}. But, after rewriting Franke's formula in a particular way, the classicity and effectivity properties of $I_{G,0}^\dagger(\cdot,\lambda)$ mentioned above can be still proved, at the expense of altering the combination of the distributions $I_M^\dagger(\cdot,\mu)$ used to define it. This is done, following work of Gulotta \cite{gulotta}, in forthcoming work of Urban and the author of the present article. The consequence of this is that the main results of \cite{urbanev} still hold, but the definition of $I_{G,0}^\dagger(\cdot,\lambda)$ given there needs to be modified in general.

Fortunately for us, when $G$ is $G_2$ or any of its Levi subgroups, the interpretation of Franke's formula in \cite{urbanev} is actually still correct. Because of this, in the case of $G_2$ and its Levis, the definition of $I_{G,0}^\dagger(\cdot,\lambda)$ given in \textit{loc. cit.} is correct and coincides with the corrected definition from the aforementioned forthcoming work. We will write down this definition in Section \ref{subsecchardistg2} explicitly in terms of the overconvergent character distributions attached to the Levi subgroups of $G_2$.

\subsection{Principal series and their $p$-stabilizations}
\label{subsecpstabns}
We now specialize back to the case when $G=G_2$. Since we are interested in this paper in automorphic representations of $G_2(\A)$ coming via parabolic induction from particular representations of the long root Levi which are unramified at $p$, we must now, for the purposes of $p$-adic deformation, study such induced representations at $p$ and their $p$-stabilizations. A first step toward this is the proposition below.

For what follows, we recall from Section \ref{subsecg2str} that for $\gamma\in\{\alpha,\beta\}$, we write $\det_\gamma$ for the corresponding determinant homomorphism of $M_\gamma\cong GL_2$.

\begin{proposition}
\label{proplqofPi}
Let $\chi$ be a unitary (possibly ramified) character of $\Q_p^\times$. Then the principal series representation
\begin{equation}
\label{eqredps}
\iota_{B(\Q_p)}^{G_2(\Q_p)}((\chi\circ\alpha)\cdot\vert\det_\alpha\vert^{1/2})
\end{equation}
is reducible, and its Langlands quotient is isomorphic to
\[\iota_{P_\beta(\Q_p)}^{G_2(\Q_p)}(\chi\circ\det_\beta),\]
which is irreducible.
\end{proposition}

\begin{proof}
Let $w=w_\beta w_\alpha$, so that $w\alpha=2\alpha+3\beta$ and $w(\alpha+2\beta)=\beta$. Rewrite our principal series representation \eqref{eqredps} as
\[\iota_{B(\Q_p)}^{G_2(\Q_p)}((\chi\circ\alpha)\cdot\vert\det_\alpha\vert^{1/2})=\iota_{B(\Q_p)}^{G_2(\Q_p)}((\chi\circ\alpha)\cdot(\vert\cdot\vert^{1/2}\circ(\alpha+2\beta))).\]
Because $\alpha+2\beta$ is dominant, the Langlands quotient of this is the unique irreducible subrepresentation of the twist by $w_\alpha$,
\[\iota_{B(\Q_p)}^{G_2(\Q_p)}(w_\alpha((\chi\circ\alpha)\cdot(\vert\cdot\vert^{1/2}\circ(\alpha+2\beta)))),\]
because these are intertwined by the $w_\alpha$-intertwining operator. Similarly, the Langlands quotient of the twist by $w$ of \eqref{eqredps} is intertwined with the same subrepresentation via the $w_\beta$-intertwining operator. Thus these two Langlands quotients are the same.

Now the twist by $w$ of \eqref{eqredps} is
\begin{align*}
\iota_{B(\Q_p)}^{G_2(\Q_p)}((\chi\circ(w\alpha))\cdot(\vert\cdot\vert^{1/2}\circ(w(\alpha+2\beta))))&=\iota_{B(\Q_p)}^{G_2(\Q_p)}((\chi\circ(2\alpha+3\beta))\cdot(\vert\cdot\vert^{1/2}\circ\beta))\\
&=\iota_{B(\Q_p)}^{G_2(\Q_p)}(\delta_{(B\cap M_\beta)(\Q_p)}^{1/2}\cdot(\chi\circ\det_\beta)).
\end{align*}
By induction in stages, this is
\[\iota_{P_\beta(\Q_p)}^{G_2(\Q_p)}(\iota_{(B\cap M_\beta)(\Q_p)}^{M_\beta(\Q_p)}(\delta_{(B\cap M_\beta)(\Q_p)}^{1/2}\cdot(\chi\circ\det_\beta))).\]
The inner parabolic induction is reducible, hence so is the whole one, and so is \eqref{eqredps}.

Now the unique irreducible quotient of the inner induction
\[\iota_{(B\cap M_\beta)(\Q_p)}^{M_\beta(\Q_p)}(\delta_{(B\cap M_\beta)(\Q_p)}^{1/2}\cdot(\chi\circ\det_\beta))\]
is the character $\chi\circ\det_\beta$ of $M_\beta(\Q_p)$. Therefore the Langlands quotient we are interested in is the unique irreducible quotient of
\[\iota_{P_\beta(\Q_p)}^{G_2(\Q_p)}(\chi\circ\det_\beta).\]
So to prove the proposition, it suffices to prove this is irreducible. But this was already done by Mui\'c \cite[Theorem 3.1 (i)]{muic}. So we are done.
\end{proof}

We record the following corollary of the proof of this proposition for later use.

\begin{corollary}
\label{corredstps}
In the setting of the above proposition, there is an exact sequence
\[0\to\iota_{P_\beta(\Q_p)}^{G_2(\Q_p)}(\mr{St}\otimes(\chi\circ\det_\beta))\to\iota_{B(\Q_p)}^{G_2(\Q_p)}((\chi\circ\alpha)\cdot\vert\det_\alpha\vert^{1/2})\to\iota_{P_\beta(\Q_p)}^{G_2(\Q_p)}(\chi\circ\det_\beta)\to 0,\]
where $\mr{St}$ is the Steinberg representation of $M_\beta(\Q_p)\cong GL_2(\Q_p)$.
\end{corollary}

\begin{proposition}
\label{propvppsirr}
Let $\chi$ be a character of $T(\Q_p)$. Assume that for all roots $\gamma$ for $G_2$ we have
\[\chi\circ\gamma^\vee\ne\vert\cdot\vert.\]
If $\chi$ is not unitary, then the principal series representation
\[\iota_{B(\Q_p)}^{G_2(\Q_p)}(\chi)\]
is irreducible. Otherwise, if $\chi$ is unitary, then this principal series representation can only be reducible if $\chi$ is a product of distinct characters of order $2$.
\end{proposition}

\begin{proof}
The assertion when $\chi$ is not unitary follows from \cite[Proposition 3.1]{muic}. When $\chi$ is unitary, this is \cite[Theorem $G_2$]{keys}.
\end{proof}

In the following we consider the Iwahori subgroup $I$ for $G_2(\Q_p)$ defined with respect to the standard Borel $B$ as in the previous subsection.

\begin{proposition}
\label{proppstabnofPi}
Let $\chi$ be an unramified character of $\Q_p^\times$. Then the space
\[\iota_{P_\beta(\Q_p)}^{G_2(\Q_p)}(\chi\circ\det_\beta)^I\]
of $I$-fixed vectors, which is naturally a module for $T(\Q_p)$, admits the character
\[(\chi\circ \det_\beta)\cdot\vert 3\alpha+\tfrac{11}{2}\beta\vert^{-1}\]
as a subquotient.
\end{proposition}

\begin{proof}
We combine the results in \cite[\S 2]{MY} and \cite[Propositions 6.3.1-3]{casselman}. While these results are completely general, in our case the former expresses the space of Iwahori invariants we are interested in as the twist by $\delta_{B(\Q_p)}^{-1}$ of the Jacquet module
\[\iota_{P_{\beta}(\Q_p)}^{G_2(\Q_p)}(\chi\circ\det_\beta)_{N(\Q_p)},\]
while the latter expresses this Jacquet module as a $T(\Q_p)$-module with a filtration indexed by the set $W_{M_\beta}\backslash W$, where $W_{M_\beta}$ is the Weyl group of $T$ in $M_\beta$. The graded piece of this filtration corresponding to the identity coset in $W_{M_\beta}\backslash W$ is given by the character $\chi$ times a particular modulus character (see \cite[Proposition 6.3.3]{casselman} for the explicit formula) which is easily computed in this case to be given by $\delta_{P_\beta(\Q_p)}^{1/2}$. Putting this together yields a one-dimensional subquotient of
\[\iota_{P_\beta(\Q_p)}^{G_2(\Q_p)}(\chi\circ\det_\beta)^I\]
with $T(\Q_p)$-action given by
\[(\chi\circ\det_\beta)\delta_{B(\Q_p)}^{-1}\delta_{P_\beta(\Q_p)}^{1/2}.\]
Since $\delta_{B(\Q_p)}=\vert 6\alpha+10\beta\vert$ and $\delta_{P_\beta(\Q_p)}|_{T(\Q_p)}=\vert 6\alpha+9\beta\vert$ our proposition follows.
\end{proof}

\subsection{Character distributions for $G_2$}
\label{subsecchardistg2}

In this section we describe the overconvergent cuspidal character distribution $I_{G_2,0}^\dagger$ mentioned at the end of Section \ref{subsecev}. It will be defined in terms of certain character distributions $I_{G_2}^\dagger$ for which we refer to \cite{urbanev} for the definitions. We will, however, be able to relate these distributions to classical objects, and this is how we will compute with them.

We specialize the notation of Section \ref{subsecev} now to the case of $G=G_2$. In particular, we have the compact subgroups $K_{f,\mr{max}}$ and $K_{f,\mr{max}}^p$, the Iwahori subgroup $I\subset G_2(\Z_p)$, the semigroups $T^-,T^{--}\subset T(\Q_p)$, and the Hecke algebras $\mc{U}_p$, $\mc{H}_p$ and $\mc{H}_p'$. Let $\mf{X}$ be the \textit{weight space} for $G_2$, which is the affinoid rigid analytic space over $\Q_p$ defined by the functor
\[\mf{X}(L)=\hom_{\mr{cont}}(T(\Z_p),L^\times),\]
for finite extensions $L$ of $\Q_p$. Then for $\lambda\in\mf{X}(\overline{\Q}_p)$ and any standard Levi subgroup $M$ of $G_2$, we have the overconvergent character distributions $I_M^\dagger(\cdot,\lambda)$ considered in \cite{urbanev}.

Let $W=W(G_2,T)$ be the Weyl group of $G_2$. Then for any standard parabolic subgroup $P=MN$ of $G_2$, we have the set $W^M$ of minimal length representatives for $W$ modulo multiplication of the left by elements of the Weyl group $W_M$ of $M$. We consider the subsets
\[W_{\Eis}^{M_\alpha}=\{1,w_\beta,w_{\beta\alpha}\}\subset W^{M_\alpha},\qquad W_{\Eis}^{M_\beta}=\{1,w_\alpha,w_{\alpha\beta}\}\subset W^{M_\beta},\qquad W_{\Eis}^{T}=\{1\}\subset W^T,\]
which will be used below.

Let us consider the Hecke algebra $\mc{H}_{p,M}$ defined in the same way as $\mc{H}_p$, but for $M$ in place of $G_2$. For any $w\in W^M$, we define a homomorphism $\mc{H}_p\to\mc{H}_{p,M}$, $f\mapsto f_{M,w}$ as follows. First, for $f^p\in C_c^\infty(G_2(\A_f^p),\Q_p)$, let $f_M^p\in C_c^\infty(M(\A_f^p),\Q_p)$ be given by
\[f_M^p(m)=\int_{K_{f,\mr{max}}^p}\int_{N(\A_f^p)}f^p(k^{-1}mnk)\,dn\,dk,\]
where Haar measures are normalized as follows: The Haar measure on $G_2(\A_f^p)$ is such that $K_{f,\mr{max}}^p$ has volume $1$, that on $M(\A_f^p)$ is such that $K_{f,\mr{max}}^p\cap M(\A_f^p)$ has volume $1$ as well, and the Haar measure on $N(\A_f^p)$ is normalized by the Iwasawa decomposition so that
\[\int_{G_2(\A_f^p)}f^p(g)\,dg=\int_{M(\A_f^p)}\int_{N(\A_f^p)}\int_{K_{f,\mr{max}}^p}f^p(mnk)\,dk\,dn\,dm,\]
for all $f^p\in C_c^\infty(G_2(\A_f^p),\Q_p)$. We note the relation
\[\tr(f_M^p|\sigma^p)=\tr(f|\Ind_{M(\A_f^p)}^{G_2(\A_f^p)}(\sigma^p))\]
for any smooth admissible representation $\sigma^p$ of $M(\A_f^p)$, where $\Ind$ denotes a nonunitary induction.

Then for $t\in T^-$ and $f=f^p\otimes u_t\in\mc{H}_p$ we let
\[f_{M,w}=f_M^p\otimes u_{wtw^{-1},M},\]
where
\[u_{wtw^{-1},M}=\frac{1}{\vol(I_M)}\chars(I_Mwtw^{-1}I_M),\qquad I_M=I\cap M(\Z_p).\]
Note that this is well defined because $w\in W^M$. The Hecke operator $f_{M,w}$ coincides with the operator $f_{M,w}^{reg}$ from \cite{urbanev} because $G_2$ is split.

We can now define the distribution $I_{G_2,0}^\dagger$. In the general case of a group $G$, this would be done inductively on the rank, but since $G_2$ has rank $2$, it is easy to do this explicitly. Let $\lambda\in\mf{X}(\overline{\Q}_p)$. We first define
\[I_{T,0}^\dagger(f,\lambda)=I_T^\dagger(f,\lambda),\qquad f\in\mc{H}_{p,T}.\]
Then if $M\subset G_2$ is the Levi of a maximal parabolic $P$, we define
\[I_{M,0}^\dagger(f,\lambda)=I_M^\dagger(f,\lambda)+I_T
^\dagger(f_{T,1},\lambda+2\rho_M),\qquad f\in\mc{H}_{p,M},\]
where $f_{T,1}$ is defined analogously to the operator $f_{M,w}$ defined above, and $2\rho_M$ is the positive root in $M$ (so $\alpha$ if $M=M_\alpha$ and $\beta$ otherwise). Finally, we define for $f\in\mc{H}_p$,
\begin{equation}
\label{eqoccuspdef}
I_{G_2,0}^\dagger(f,\lambda)=I_{G_2}^\dagger(f,\lambda)-\sum_{P\in\{P_\alpha,P_\beta,B\}}\sum_{w\in W_{\Eis}^{M_P}}(-1)^{\dim(N_P)-\ell(w)}I_{M_P,0}^\dagger(f_{M_P,w},w*\lambda+2\rho_P).
\end{equation}
Here $\rho_P$ is half the sum of the roots of $T$ in the unipotent radical $N_P$ of $P$, and $w*\lambda$ is defined by
\[w*\lambda=w(\lambda+\rho)-\rho.\]
Then \cite[Corollary 4.7.4]{urbanev} (see also the remarks at the end of Section \ref{subsecev} of this paper) says in this case that $I_{G_2,0}^\dagger$ is an $\mf{X}$-family of effective finite slope character distributions.

We now relate the distributions $I_M^\dagger$ more classical objects. For $K_{f,M}\subset M(\A_f)$ an open compact subgroup, we let
\[X_{K_{f,M}}=M(\Q)\backslash M(\A)/K_{f,M}K_{\infty,M}A_M(\R)^\circ,\]
where $K_{\infty,M}=K_\infty\cap M(\R)$. For $\lambda\in X^*(T)$ a weight of $T$ which is dominant for $M$, let $E_\lambda(\Q_p)$ be the algebraic representation of $M$ of highest weight $\lambda$, viewed with $\Q_p$-coefficients, and let $E_\lambda^\vee(\Q_p)$ be its dual. Then $E_\lambda(\Q_p)$ and $E_\lambda^\vee(\Q_p)$ define local systems on $X_{K_{f,M}}$ in the standard way, which we denote respectively by $\widetilde{E}_\lambda(\Q_p)$ and $\widetilde{E}_\lambda^\vee(\Q_p)$. These local systems are compatible under pullback between the spaces $X_{K_{f,M}}$ as $K_{f,M}$ vary. So we define, for any $i\geq 0$,
\[H^i(X_M,E_\lambda^\vee(\Q_p))=\varinjlim_{K_{f,M}}H^i(X_{K_{f,M}},\widetilde{E}_\lambda^\vee(\Q_p)),\]
where the direct limit is over all open compact subgroups of $M(\A_f)$. This is an $M(\A_f)$-module in the standard way. In fact, there is an isomorphism of $M(\A_f)$-modules
\[H^i(X_M,E_\lambda^\vee(\Q_p))\otimes_{\Q_p}\C\cong H^i(\mf{m}_0,K_{\infty,M};\mc{A}_{E_\lambda^\vee}(M)\otimes E_\lambda^\vee)(e^{\langle H_P(\cdot),\lambda|_{A_P}\rangle}),\]
where we have viewed $\Q_p\subset\C$ via our fixed isomorphism $\overline{\Q}_p\cong\C$, and on the right hand side, we view $E_\lambda^\vee$ with $\C$-coefficients.

We then define, for $f\in\mc{H}_{p,M}$,
\begin{equation}
\label{eqocandcl}
I_M^{\cl}(f,\lambda)=\vert\lambda(t)\vert^{-1}\sum_{i\geq 0}(-1)^i\tr(f|H^i(X_M,E_\lambda^\vee(\Q_p))).
\end{equation}
For $w\in W_M$, $\lambda\in X^*(T)$, and $f\in\mc{H}_{p,M}$ with $f=f^p\otimes u_t$ as usual, define
\[f^{w,\lambda}=\vert(\lambda-(w*\lambda))(t)\vert^{-1} f^p\otimes u_t,\]
and extend $f\mapsto f^{w,\lambda}$ to all of $\mc{H}_{p,M}$ by linearity. This defines an automorphism of $\mc{H}_{p,M}$. Then if $\lambda$ is $M$-dominant, we have, by \cite[Theorem 4.5.4]{urbanev},
\begin{equation}
\label{eqoccl}
I_M^{\cl}(f,\lambda)=\sum_{w\in W_M}(-1)^{\ell(w)}I_M^\dagger(f^{w,\lambda},w*\lambda).
\end{equation}

Let $\sigma$ be an irreducible admissible $\mc{H}_{p,M}$-representation with slope $\mu_\sigma$. Then we say the slope $\mu_\sigma$ of $\sigma$ is \textit{critical} with respect to an $M$-dominant weight $\lambda\in X^*(T)$ if for all $w\in W_M$ with $w\ne 1$, we have that
\[\mu_\sigma-\lambda+w*\lambda\]
is a nonnegative rational combination of positive roots. Otherwise we say it is \textit{noncritical}. This weight is the slope of the representation $\sigma^{w,\lambda}$ given by
\begin{equation}
\label{eqsigmawlambda}
\sigma^{w,\lambda}(f^{w,\lambda})=\sigma(f),\qquad f\in\mc{H}_{p,M}.
\end{equation}
We note that this is slightly different than the definition of $\sigma^{w,\lambda}$ that appears in \cite{urbanev}; it seems that the definition given in \textit{loc. cit.} is not quite right and that this is the one that was intended. This is because we want the appearance of $\sigma$ in $I_M^{\cl}(\cdot,\lambda)$ to be equivalent to the appearance of $\sigma^{w,\lambda}$ in $I_M^\dagger(\cdot,w*\lambda)$ for some $w$.

The local systems whose cohomology is used to define $I_M^\dagger(\cdot,\lambda)$ have an integral structure and therefore have integral $\mc{U}_p$-eigenvalues. Thus, if the slope of $\sigma$ is noncritical with respect to $\lambda$, then the multiplicity of $\sigma$ in $I_M^{\cl}(\cdot,\lambda)$ is the same as that in $I_M^{\dagger}(\cdot,\lambda)$.

A similar fact to this is true for cuspidal distributions. Let $H_{\cusp}^i(X_M,E_\lambda^\vee(\Q_p))$ be the cuspidal subspace of $H^i(X_M,E_\lambda^\vee(\Q_p))$; over $\C$, this is the summand of the corresponding $(\mf{m}_0,K_{\infty,M})$-cohomology space coming from the space $\mc{A}_{E_\lambda^\vee,[M]}(M)$ of cusp forms in $\mc{A}_{E_\lambda^\vee}(M)$. Then we define, for $f\in\mc{H}_{p,M}$,
\[I_{M,0}^{\cl}(f,\lambda)=\vert\lambda(t)\vert^{-1}\sum_{i\geq 0}(-1)^i\tr(f|H_{\cusp}^i(X_M,E_\lambda^\vee(\Q_p))).\]
If $\lambda$ is moreover regular as well as dominant, then for any noncritical $\sigma$, the multiplicity of $\sigma$ in $I_{M,0}^{\dagger}(\cdot,\lambda)$ is the same as that in $I_{M,0}^{\cl}(\cdot,\lambda)$; this is \cite[Corollary 4.6.5]{urbanev}. We will use this fact in what follows.

Let us give a name to all the relevant multiplicities. Let, respectively,
\[m_M^{\cl}(\sigma,\lambda),\qquad m_{M,0}^{\cl}(\sigma,\lambda),\qquad m_M^\dagger(\sigma,\lambda),\qquad m_{M,0}^\dagger(\sigma,\lambda)\]
be the multiplicities of an irreducible admissible $\mc{H}_{p,M}$-representation $\sigma$ in
\[I_M^{\cl}(\cdot,\lambda),\qquad I_{M,0}^{\cl}(\cdot,\lambda),\qquad I_M^\dagger(\cdot,\lambda),\qquad I_{M,0}^\dagger(\cdot,\lambda).\]
Thus if $m_{M,0}^\dagger(\sigma,\lambda)\ne 0$, then $\sigma$ has a $p$-adic deformation in a generically cuspidal family of $p$-stabilizations of cohomological automorphic representations for $M$ by \cite{urbanev}.

\subsection{The $\mc{H}_p$-representation $\Pi_F^{(p)}$}
\label{subsecpif}

Let us now introduce the $\mc{H}_p$-module which we would like to $p$-adically deform. Let $F$ be a cuspidal eigenform with trivial nebentypus and weight $k\geq 4$ and level $N$, and assume $p\nmid N$. Let $\pi_F$ be the unitary cuspidal automorphic representation attached to $F$, viewed as an automorphic representation of $M_\alpha(\A)$. We also assume
\[\epsilon(1/2,\pi_F,\Sym^3)=-1.\]
In particular we have $L(1/2,\pi_F,\Sym^3)=0$.

Now since $\pi_F$ is unramified at $p$, there is an unramified unitary character $\chi:\Q_p^\times\to\C^\times\cong\overline\Q_p^\times$ such that
\[\pi_{F,p}\cong\iota_{(B\cap M_\alpha)(\Q_p)}^{M_\alpha(\Q_p)}(\chi\circ\alpha).\]
Assume $F$ is normalized and let $a_p$ be the $p$th Fourier coefficient of $F$, so that
\[X^2-a_pX+p^{k-1}\]
is the Hecke polynomial of $F$ at $p$. We choose a root $\alpha_p$ of this polynomial. All our future considerations will depend on this choice, but we note that either choice of $\alpha_p$ will work until the proof of the main theorem at the end, where we will use a trick that shows that if the constructions we make based on one choice of $\alpha_p$ do not work for that proof, then the same constructions based on the other choice will work instead.

In any case, the character $\chi$ is then determined by
\[\chi(p)=p^{-(k-1)/2}\alpha_p.\]
(Note that the other choice of root would give the character $\chi^{-1}$, but
\[\iota_{(B\cap M_\alpha)(\Q_p)}^{M_\alpha(\Q_p)}(\chi\circ\alpha)\qquad\textrm{and}\qquad\iota_{(B\cap M_\alpha)(\Q_p)}^{M_\alpha(\Q_p)}(\chi^{-1}\circ\alpha)\]
are isomorphic.)

Now let $\mc{L}_{\alpha}(\pi_F,1/10)$ be the Langlands quotient of the induced representation
\[\iota_{M_\alpha(\A)}^{G_2(\A)}(\pi_F\otimes\delta_{P_\alpha(\A)}^{1/10}),\]
and let $\mc{L}_{\alpha}(\pi_F,1/10)_f$ be its finite part. Then the component $\mc{L}_{\alpha}(\pi_F,1/10)_p$ at $p$ is the Langlands quotient of
\[\iota_{B(\Q_p)}^{G_2(\Q_p)}((\chi\circ\alpha)\cdot\vert\det_\alpha\vert^{1/2}),\]
and therefore is isomorphic to
\[\iota_{P_\beta(\Q_p)}^{G_2(\Q_p)}(\chi\circ\det_\beta),\]
by Proposition \ref{proplqofPi}. Then by Proposition \ref{proppstabnofPi}, the representation $\mc{L}_{\alpha}(\pi_F,1/10)_f$ has a $p$-stabilization, which we call $\widetilde{\Pi}_F^{(p)}$, whose $\mc{U}_p$-eigenvalues are given by the character
\[u_t\mapsto\chi(\det_\beta(t))\vert(3\alpha+\tfrac{11}{2}\beta)(t)\vert_p^{-1},\qquad t\in T^-.\]

From now on, we let
\[\lambda_0=\frac{k-4}{2}(2\alpha-3\beta).\]
Let $\Pi_F^{(p)}$ be the twist of $\widetilde{\Pi}_F^{(p)}$ by the character of $\mc{U}_p$ given by
\[u_t\mapsto\vert\lambda_0(t)\vert^{-1},\qquad t\in T^-.\]
This is the $\mc{H}_p$-representation which we would like to $p$-adically deform.

Let $s_p=v_p(\alpha_p)$, which, in the classical sense, is the slope of the $p$-stabilization of $F$ corresponding to $\alpha_p$. By its construction, the representation $\Pi_F^{(p)}$ has slope
\[v_p(\chi(p))(2\alpha+3\beta)+(3\alpha+\tfrac{11}{2}\beta)+\lambda_0=s_p(2\alpha+3\beta)+\beta.\]
We record the following fact for later use.

\begin{proposition}
\label{propcritslope}
The slope $s_p(2\alpha+3\beta)+\beta$ is critical with respect to $\lambda_0$. More precisely, we have
\[s_p(2\alpha+3\beta)+\beta-\lambda_0+w_\beta*\lambda_0=s_p(2\alpha+3\beta),\]
which is a nonnegative rational combination of positive roots. Furthermore, if $k>4s_p+4$, then for any $w\in W$, $w\ne w_\beta,1$, the weight
\[s_p(2\alpha+3\beta)+\beta-\lambda_0+w*\lambda_0\]
is not a nonnegative rational combination of positive roots.
\end{proposition}

\begin{proof}
We only need to prove the last statement. First note that
\begin{equation}
\label{eqwslope}
s_p(2\alpha+3\beta)+\beta-\lambda_0+w*\lambda_0=s_p(2\alpha+3\beta)+\beta+(w\rho-\rho)+\frac{k-4}{2}(w(2\alpha+3\beta)-(2\alpha+3\beta)).
\end{equation}
We note that the only elements $w\in W$ which fix $2\alpha+3\beta$ are $w=1,w_\beta$. Thus for $w\ne 1,w_\beta$, we have that the $\alpha$-component of
\[w(2\alpha+3\beta)-(2\alpha+3\beta)\]
in the basis $\alpha,\beta$, is negative. The $\alpha$-component of $w\rho-\rho$ is also nonpositive for $w\ne 1$. Since we assumed $k>4s_p+4$, it follows that \eqref{eqwslope} has negative $\alpha$-component for $w\ne 1,w_\beta$.
\end{proof}

\subsection{The overconvergent multiplicity}
\label{subsecocmult}

We would like to show now that $m_{G_2,0}^\dagger(\Pi_F^{(p)},\lambda_0)$ is nonzero. We will be able to do this under Conjecture \ref{conjmult} (b), which we show in the appendix is a consequence of general conjectures of Arthur. (We note that we do not actually need the conclusion given by part (a) of this conjecture for the main body of this paper, although this is a consequence of the same conjectures of Arthur.) The strategy will be to decompose the multiplicity $m_{G_2,0}^\dagger(\Pi_F^{(p)},\lambda_0)$ into a sum of several multiplicities, each weighted with different signs. We will be able to compute that most of these signed multiplicities are zero, at least if the weight $k$ of $F$ is sufficiently large with respect to the slope $s_p$, and the remaining signed multiplicities will be shown to be nonnegative, with at least one positive. We will then be able to deduce the positivity of $m_{G_2,0}^\dagger(\Pi_F^{(p)},\lambda_0)$ no matter the slope $s_p$ by deforming $F$ in a Coleman family.

We use freely the notations from the last two subsections. As a first step, the following proposition computes the classical multiplicity of $\Pi_F^{(p)}$ assuming Conjecture \ref{conjmult} (b), using the results of Section \ref{seceiscoh}.

\begin{proposition}
\label{propclmult}
Assume Conjecture \ref{conjmult} (b). Then we have
\[m_{G_2}^{\cl}(\Pi_F^{(p)},\lambda_0)\geq 2.\]
\end{proposition}

\begin{proof}
Consider the Eisenstein multiplicity
\[m_{G_2}^{\cl}(\Pi_F^{(p)},\lambda_0)-m_{G_2,0}^{\cl}(\Pi_F^{(p)},\lambda_0).\]
After unraveling the definitions and passing through normalizations, we see that, by Theorem \ref{thmlqineiscoh}, this number is the number of $p$-stabilizations of
\[\mc{L}_\alpha(\pi_F,1/10)_f^p\otimes\iota_{P_\alpha(\Q_p)}^{G_2(\Q_p)}(\pi_{F,p},1/10)\]
whose components at $p$ are given by the $p$-component of the $\mc{H}_p$-representation $\widetilde{\Pi}_F^{(p)}$ defined above. Since there is at least one such $p$-stabilization, namely $\widetilde{\Pi}_F^{(p)}$ itself, this number is at least $1$.

Thus we must show $m_{G_2,0}^{\cl}(\Pi_F^{(p)},\lambda_0)\geq 1$. But this follows from Conjecture \ref{conjmult} (b), applied with $v=p$, since the discrete series representation there is cohomological of weight $\lambda_0$.
\end{proof}

By Proposition \ref{propcritslope}, Formula \eqref{eqocandcl}, and the integrality of eigenvalues of operators in $\mc{U}_p$, we have
\begin{equation}
\label{eqcuspmult0}
m_{G_2}^{\cl}(\Pi_F^{(p)},\lambda_0)=m_{G_2}^{\dagger}(\Pi_F^{(p)},\lambda_0)-m_{G_2}^{\dagger}(\Pi_F^{(p),w_\beta,\lambda_0},w_\beta*\lambda_0),\qquad\textrm{if }k>4s_p+4.
\end{equation}
(Here the notation $\Pi_F^{(p),w_\beta,\lambda_0}$ is as in \eqref{eqsigmawlambda}.) For $P\subset G_2$ a proper parabolic subgroup with Levi $M$, $w\in W_{\Eis}^M$, and $\lambda\in\mf{X}(\overline{\Q}_p)$, let us write $m_{G_2,M,w}^\dagger(\sigma,\lambda)$ for the multiplicity of an irreducible admissible $\mc{H}_p$-representation $\sigma$ in the character distribution
\[f\mapsto I_{M,0}^\dagger(f_{M,w},w*\lambda+2\rho_P)\]
Then by \eqref{eqoccuspdef}, we have
\begin{equation}
\label{eqcuspmult1}
m_{G_2}^\dagger(\Pi_F^{(p)},\lambda_0)=m_{G_2,0}^\dagger(\Pi_F^{(p)},\lambda_0)+\sum_{P\in\{P_\alpha,P_\beta,B\}}\sum_{w\in W_{\Eis}^{M_P}}(-1)^{\dim(N_P)-\ell(w)}m_{G_2,M_P,w}^\dagger(\Pi_F^{(p)},\lambda_0),
\end{equation}
and similarly
\begin{multline}
\label{eqcuspmult2}
m_{G_2}^\dagger(\Pi_F^{(p),w_\beta,\lambda_0},\lambda_0-\beta)=m_{G_2,0}^\dagger(\Pi_F^{(p),w_\beta,\lambda_0},\lambda_0-\beta)\\
+\sum_{P\in\{P_\alpha,P_\beta,B\}}\sum_{w\in W_{\Eis}^{M_P}}(-1)^{\dim(N_P)-\ell(w)}m_{G_2,M_P,w}^\dagger(\Pi_F^{(p),w_\beta,\lambda_0},\lambda_0-\beta).
\end{multline}
Note here that
\[w_\beta*\lambda_0=\lambda_0-\beta.\]

We must now compute the cuspidal multiplicities appearing in the formulas above. The following two lemmas will aid us in this endeavor.

\begin{lemma}
\label{lemmultslope}
Let $\gamma\in\{\alpha,\beta\}$, and let $\lambda\in X^*(T)$. Let $w\in W^{M_\gamma}$. Then any constituent of the character distribution
\begin{equation}
\label{eqlmsdist}
f\mapsto I_{M_\gamma,0}^\dagger(f_{M_\gamma,w},\lambda)
\end{equation}
has slope a rational multiple of the root $w^{-1}\gamma$.
\end{lemma}

\begin{proof}
Let $\sigma$ be a constituent of $I_{M_\gamma,0}^\dagger(\cdot,\lambda)$. Then by the main results of \cite{urbanev}, there are weights $\lambda_n\in X^*(T)$ which are dominant regular for $M_\gamma$, converging $p$-adically to $\lambda$, and there are constituents $\sigma_n$ of $I_{M_\gamma,0}^\dagger(\cdot,\lambda_n)$ whose trace functions converge $p$-adically to the trace function of $\sigma$. Therefore, to prove the lemma, it suffices to assume $\lambda$ is sufficiently dominant so that the slope of $\sigma$ is noncritical with respect to $\lambda$.

Under this assumption, we then have that $\sigma$ is classical and comes from a cuspidal eigenform for $M_\gamma\cong GL_2$. Let $\gamma^\perp$ be the positive root orthogonal to $\gamma$, so $\gamma^\perp=\alpha+2\beta$ if $\gamma=\alpha$, or $\gamma^\perp=2\alpha+3\beta$ if $\gamma=\beta$. Then $\gamma^\perp$ acts on $T$ as $\det_\gamma$. Let $t\in T^{-}$ be given by $t=(\gamma^\perp)^\vee(p)$. Then $t$ is in the center of $M_\gamma(\Q_p)$. Write $\lambda=a\gamma+b\gamma^\perp$, for some $a$ and $b$. Then as $M_\gamma(\A_f)$-modules, we have
\begin{equation}
\label{eqlmshcusp}
H_{\cusp}^*(\widetilde{X}_{M_\gamma},E_{\lambda}^\vee(\Q_p))\cong H_{\cusp}^*(\widetilde{X}_{M_\gamma},E_{a\gamma}^\vee(\Q_p))\otimes\vert\det\vert^{b},
\end{equation}
and therefore $t$ acts on any constituent of this space by $p^{-2b}$.

Now let $f^p\in C_c^\infty(G_2(\A_f))$ and $t'\in T^-$, and let $f=f^p\otimes u_{t'}$. Then
\begin{align*}
\tr((fu_{w^{-1}tw})_{M_\gamma,w}|\sigma)&=\tr(f_{M_\gamma}^p\otimes(u_{t,M}\cdot u_{wt'w^{-1},M})|\sigma)\\
&=p^{-2b}\vert\lambda(t)\vert^{-1}\tr(f_{M_\gamma}^p\otimes u_{wt'w^{-1},M}|\sigma)\\
&=\tr(f_{M_\gamma,w}|\sigma),
\end{align*}
where the second equality follows from \eqref{eqlmshcusp} and the definition of $I_{M_\gamma,0}^{\cl}$. This shows exactly that the slope of $\sigma$ is orthogonal to $w\gamma^\perp$, which proves the lemma.
\end{proof}

\begin{lemma}
\label{lemmultclass}
Let $M\in\{M_\beta,T\}$, $w\in W^M$, and $\lambda\in X^*(T)$. Assume that if $M=M_\beta$, then $\lambda$ is $M_\beta$-dominant. Then the multiplicities of $\Pi_F^{(p)}$ and $\Pi_F^{(p),w_\beta,\lambda_0}$ in any of the character distributions
\begin{equation}
\label{eqlmcdist}
f\mapsto I_{M,0}^{\cl}(f_{M,w},\lambda),\qquad f\mapsto I_M^{\cl}(f_{M,w},\lambda),\qquad f\in\mc{H}_p,
\end{equation}
is zero.
\end{lemma}

\begin{proof}
By definition, any constituent of one of the distributions in \eqref{eqlmcdist} is of the form
\[f^p\otimes u_t\mapsto\theta_p(u_t)\tr(f^p|\sigma^p),\qquad f^p\otimes u_t\in\mc{H}_p,\]
where $\theta_p$ is a character of $\mc{U}_p$, and where $\sigma^p$ is a constituent of either
\[\Ind_{P_\beta(\A_f^p)}^{G_2(\A_f^p)}(\pi)\qquad\textrm{or}\qquad\Ind_{B(\A_f^p)}^{G_2(\A_f^p)}(\psi)\]
with $\pi$ some cuspidal, cohomological automorphic representation of $GL_2(\A)$ or $\psi$ some quasicharacter of $T(\Q)\backslash T(\A)$. By Proposition \ref{propdistneareq}, such a constituent must therefore be different from $\Pi_F^{(p)}$.
\end{proof}

We now start computing the various cuspidal multiplicities from the right hand sides of \eqref{eqcuspmult1} and \eqref{eqcuspmult2}. The $T$-multiplicities are immediately handled by the previous lemma as follows.

\begin{lemma}
\label{lemTmults}
We have
\[m_{G_2,T,1}^\dagger(\Pi_F^{(p)},\lambda_0)=m_{G_2,T,1}^\dagger(\Pi_F^{(p),w_\beta,\lambda_0},\lambda_0-\beta)=0.\]
\end{lemma}

\begin{proof}
By definition, we have $I_{T,0}^\dagger=I_T^\dagger$, and for $\lambda\in X^*(T)$, we have $I_T^\dagger(\cdot,\lambda)=I_T^{\cl}(\cdot,\lambda)$. Therefore the lemma at hand follows from Lemma \ref{lemmultclass}.
\end{proof}

The next two lemmas examine the $M_\beta$-multiplicities.

\begin{lemma}
\label{lembetamults1}
We have
\[m_{G_2,M_\beta,1}^\dagger(\Pi_F^{(p)},\lambda_0)-m_{G_2,M_\beta,1}^\dagger(\Pi_F^{(p),w_\beta,\lambda_0},\lambda_0-\beta)=0.\]
\end{lemma}

\begin{proof}
We have by definition,
\[I_{M_\beta,0}^\dagger(\cdot,\lambda_0-\beta+2\rho_{P_\beta})=I_{M_\beta}^\dagger(\cdot,\lambda_0-\beta+2\rho_{P_\beta})+I_{T}^\dagger(\cdot,\lambda_0+2\rho_{P_\beta}),\]
and by \eqref{eqoccl},
\[I_{M_\beta}^\dagger(\cdot,\lambda_0-\beta+2\rho_{P_\beta})=I_{M_\beta}^\dagger((\cdot)^{w_\beta,\lambda_0-\beta+2\rho_{P_\beta}},\lambda_0+2\rho_{P_\beta})+I_{M_\beta}^{\cl}((\cdot)^{w_\beta,\lambda_0-\beta+2\rho_{P_\beta}},\lambda_0+2\rho_{P_\beta}).\]
Here we have used that
\[w_\beta(\lambda_0-\beta+2\rho_{P_\beta}+\rho_{M_\beta})-\rho_{M_\beta}=\lambda_0+2\rho_{P_\beta},\]
and that $f\mapsto f^{w_\beta,\lambda_0-\beta+2\rho_{P_\beta}}$ is an involution. Since
\[\lambda_0-(w_\beta*\lambda_0)=\beta=\lambda_0-\beta+2\rho_{P_\beta}-
(w_\beta*(\lambda_0-\beta+2\rho_{P_\beta})),\]
we have
\[f^{w_\beta,\lambda_0}=f^{w_\beta,\lambda_0-\beta+2\rho_{P_\beta}}.\]
Thus,
\begin{align*}
I_{M_\beta,0}^\dagger(\cdot,\lambda_0-\beta+2\rho_{P_\beta})=&I_{T}^\dagger(\cdot,\lambda_0+2\rho_{P_\beta})+I_{M_\beta}^\dagger((\cdot)^{w_\beta,\lambda_0},\lambda_0+2\rho_{P_\beta})\\
&+I_{M_\beta}^{\cl}((\cdot)^{w_\beta,\lambda_0-\beta+2\rho_{P_\beta}},\lambda_0+2\rho_{P_\beta})\\
=&I_{T}^\dagger(\cdot,\lambda_0+2\rho_{P_\beta})+I_{M_\beta,0}^\dagger((\cdot)^{w_\beta,\lambda_0},\lambda_0+2\rho_{P_\beta})\\
&-I_{T}^\dagger((\cdot)^{w_\beta,\lambda_0},\lambda_0+2\rho)+I_{M_\beta}^{\cl}((\cdot)^{w_\beta,\lambda_0},\lambda_0-\beta+2\rho_{P_\beta})
\end{align*}
Therefore, we have
\begin{multline*}
m_{G_2,M_\beta,1}^\dagger(\Pi_F^{(p),w_\beta,\lambda_0},\lambda_0-\beta)-m_{G_2,M_\beta,1}^\dagger(\Pi_F^{(p)},\lambda_0)=\\
m_{G_2,T,1}^\dagger(\Pi_F^{(p)},\lambda_0-\beta)-m_{G_2,T,1}^\dagger(\Pi_F^{(p),w_\beta,\lambda_0},\lambda_0)-m_{G_2,M_\beta,1}^{\cl}(\Pi_F^{(p),w_\beta,\lambda_0},\lambda_0).
\end{multline*}
Since $I_{T}^\dagger(\cdot,\lambda_0)=I_{T}^{\cl}(\cdot,\lambda_0)$ for $\lambda\in X^*(T)$, all three multiplicities on the right hand side are zero by Lemma \ref{lemmultclass}. This proves the lemma.
\end{proof}

\begin{lemma}
\label{lembetamults2}
Let $w\in\{w_\alpha,w_{\alpha\beta}\}$. Then
\[m_{G_2,M_\beta,w}^\dagger(\Pi_F^{(p)},\lambda_0)=m_{G_2,M_\beta,w}^\dagger(\Pi_F^{(p),w_\beta,\lambda_0},\lambda_0-\beta)=0.\]
\end{lemma}

\begin{proof}
The slope of $\Pi_F^{(p)}$ is $s_p(2\alpha+3\beta)+\beta$ and that of $\Pi_F^{(p),w_\beta,\lambda_0}$ is $s_p(2\alpha+3\beta)$. Since $s_p\geq 0$, the slope $s_p(2\alpha+3\beta)+\beta$ can never be a rational multiple of $w_{\alpha}^{-1}\beta=\alpha+\beta$, nor could it be a rational multiple of $w_{\alpha\beta}^{-1}\beta=\alpha+2\beta$ unless $s_p=1$, in which case $s_p(2\alpha+3\beta)+\beta=2(\alpha+2\beta)$. Also, the slope $s_p(2\alpha+3\beta)$ can never be a rational multiple of a short root unless $s_p=0$. Therefore by Lemma \ref{lemmultslope}, we reduce to showing
\begin{equation}
\label{eqbetamultexcep1}
m_{G_2,M_\beta,w_{\alpha\beta}}^\dagger(\Pi_F^{(p)},\lambda_0)=0,\qquad\textrm{if }s_p=1,
\end{equation}
and
\begin{equation}
\label{eqbetamultexcep2}
m_{G_2,M_\beta,w_\alpha}^\dagger(\Pi_F^{(p),w_\beta,\lambda_0},\lambda_0-\beta)=m_{G_2,M_\beta,w_{\alpha\beta}}^\dagger(\Pi_F^{(p),w_\beta,\lambda_0},\lambda_0-\beta)=0\qquad\textrm{if }s_p=0.
\end{equation}

To show \eqref{eqbetamultexcep1}, we note that by definition, the multiplicity in question is the multiplicity of $\Pi_F^{(p)}$ in
\[f\mapsto I_{M_\beta,0}^\dagger(f_{M_\beta,w_{\alpha\beta}},w_{\alpha\beta}*\lambda_0+2\rho_{P_\beta}).\]
We have
\[w_{\alpha\beta}*\lambda_0+2\rho_{P_\beta}=\frac{k-4}{2}(\alpha+3\beta)-2\alpha-\beta+3(2\alpha+3\beta)=\frac{3k-4}{4}\beta+\frac{k+4}{4}(2\alpha+3\beta),\]
and so
\[2(\alpha+2\beta)-(w_{\alpha\beta}*\lambda_0+2\rho_{P_\beta})+w_\beta(w_{\alpha\beta}*\lambda_0+2\rho_\beta+\rho_{M_\beta})-\rho_{M_\beta}\\
=2\alpha+3\beta-\frac{3k-4}{2}\beta,\]
which is in the negative chamber for $M_\beta$. Thus the slope of $\Pi_F^{(p)}$ is noncritical (for $M_\beta$) with respect to $w_{\alpha\beta}*\lambda_0+2\rho_{P_\beta}$, which shows that the multiplicity in \eqref{eqbetamultexcep1} is the same as the multiplicity of $\Pi_F^{(p)}$ in
\[f\mapsto I_{M_\beta,0}^{\cl}(f_{M_\beta,w_{\alpha\beta}},w_{\alpha\beta}*\lambda_0+2\rho_{P_\beta}).\]
But this multiplicity is zero by Lemma \ref{lemmultclass}.

The multiplicities in \eqref{eqbetamultexcep2} are handled completely similarly, using noncriticality to reduce to a classical multiplicity and appealing to Lemma \ref{lemmultclass}. We omit the details.
\end{proof}

The next two lemmas examine the $M_\alpha$-multiplicities.

\begin{lemma}
\label{lemalphamults1}
We have
\[m_{G_2,M_\alpha,1}^\dagger(\Pi_F^{(p)},\lambda_0)=m_{G_2,M_\alpha,w_\beta}^\dagger(\Pi_F^{(p)},\lambda_0)=m_{G_2,M_\alpha,w_{\beta\alpha}}^\dagger(\Pi_F^{(p)},\lambda_0)=0,\]
Furthermore, if $s_p\ne 0$, then
\[m_{G_2,M_\alpha,1}^\dagger(\Pi_F^{(p),w_\beta,\lambda_0},\lambda_0-\beta)=m_{G_2,M_\alpha,w_\beta}^\dagger(\Pi_F^{(p),w_\beta,\lambda_0},\lambda_0-\beta)=0.\]
\end{lemma}

\begin{proof}
The slope $s_p(2\alpha+3\beta)+\beta$ is never a rational multiple of a long root since $s_p\in\tfrac{1}{2}\Z$. Also, the slope $s_p(2\alpha+3\beta)$ is not a rational multiple of $\alpha$ or $w_\beta^{-1}\alpha=\alpha+3\beta$ unless $s_p=0$. Thus an appeal to Lemma \ref{lemmultslope} finishes the proof.
\end{proof}

If $s_p=0$, then the last two multiplicities in the lemma above must be handled differently. We do this in the next lemma.

\begin{lemma}
\label{lemalphamults2}
If $s_p=0$, we have
\begin{equation}
\label{eqalphamultszero}
m_{G_2,M_\alpha,1}^\dagger(\Pi_F^{(p),w_\beta,\lambda_0},\lambda_0-\beta)=m_{G_2,M_\alpha,w_\beta}^\dagger(\Pi_F^{(p),w_\beta,\lambda_0},\lambda_0-\beta)=0.
\end{equation}
\end{lemma}

\begin{proof}
We have
\[\lambda_0-\beta+2\rho_{P_\alpha}=\frac{k-2}{4}\alpha+\frac{3k+16}{4}(\alpha+2\beta),\]
and
\[w_\beta*(\lambda_0-\beta)+2\rho_{P_\alpha}=\frac{k-4}{4}\alpha+\frac{3k+18}{4}(\alpha+2\beta).\]
By noncriticality, any slope zero constituent of
\[f\mapsto I_{M_\alpha,0}^{\dagger}(f,\lambda_0+2\rho_{P_\alpha})\qquad\textrm{or}\qquad f\mapsto I_{M_\alpha,0}^{\dagger}(f,w_{\beta}*\lambda_0+2\rho_{P_\alpha}).\]
comes from a classical eigenform of weight $2+\tfrac{k-2}{2}$ or $2+\tfrac{k-4}{2}$, respectively. Since $k\geq 4$, both of these weights are smaller than $k$.

Now since $s_p=0$, the slope of $\Pi_F^{(p),w_\beta,\lambda_0}$ is zero. So if either multiplicity in \eqref{eqalphamultszero} were nonzero, we would have then that $\mc{L}_\alpha(\pi_F,1/10)_f^p$ would be isomorphic to a constituent of a representation induced from $M_\alpha(\A_f^p)$ from a classical eigenform $F'$ of weight smaller than $k$; more precisely, the representation $\mc{L}_\alpha(\pi_F,1/10)_f^p$ would occur as a constituent of either
\[\Ind_{P_\alpha(\A_f^p)}^{G_2(\A_f^p)}(\pi_{F',f}^p\otimes\vert\det\vert^{\tfrac{3k+16}{4}})=\iota_{P_\alpha(\A_f^p)}^{G_2(\A_f^p)}(\pi_{F',f}^p,\tfrac{3k+6}{4}),\]
or
\[\Ind_{P_\alpha(\A_f^p)}^{G_2(\A_f^p)}(\pi_{F',f}^p\otimes\vert\det\vert^{\tfrac{3k+18}{4}})=\iota_{P_\alpha(\A_f^p)}^{G_2(\A_f^p)}(\pi_{F',f}^p,\tfrac{3k+8}{4})\]
where $\pi_{F'}$ is the unitary automorphic representation of $GL_2(\A)\cong M_\alpha(\A)$ attached to $F'$ and $\pi_{F',f}^p$ is its component away from $p$ and $\infty$. But Proposition \ref{propneareqalpha} then implies that $F=F'$ (and even that $\tfrac{3k+6}{4}$ or $\tfrac{3k+8}{4}$ equals $1/10$). This is a contradiction, and the proof is complete.
\end{proof}

Now we combine these results and put a lower bound on $m_{G_2,0}^\dagger(\Pi_F^{(p)},\lambda_0)$, with an assumption on the weight of $F$ relative to the slope.

\begin{lemma}
\label{lemmultnonzero}
Assume Conjecture \ref{conjmult} (b), and also assume that $k>4s_p+4$. Then $m_{G_2,0}^\dagger(\Pi_F^{(p)},\lambda_0)>0$.
\end{lemma}

\begin{proof}
Combining Lemmas \ref{lemTmults}, \ref{lembetamults1}, \ref{lembetamults2}, \ref{lemalphamults1} and \ref{lemalphamults2}, the formulas \eqref{eqcuspmult0}, \eqref{eqcuspmult1} and \eqref{eqcuspmult2} become
\[m_{G_2}^{\cl}(\Pi_F^{(p)},\lambda_0)+m_{G_2,0}^{\dagger}(\Pi_F^{(p),w_\beta,\lambda_0},w_\beta*\lambda_0)+m_{G_2,M_\alpha,w_{\beta\alpha}}^{\dagger}(\Pi_F^{(p),w_\beta,\lambda_0},w_\beta*\lambda_0)=m_{G_2,0}^{\dagger}(\Pi_F^{(p)},\lambda_0).\]
By Proposition \ref{propclmult} and the nonnegativity of cuspidal multiplicities, the left hand side of this formula is at least $2$, which suffices.
\end{proof}

\begin{remark}
Although we do not need this for our applications, one can show without much difficulty that we have
\[m_{G_2,M_\alpha,w_{\beta\alpha}}^{\dagger}(\Pi_F^{(p),w_\beta,\lambda_0},w_\beta*\lambda_0)\geq 1,\]
and so $m_{G_2,0}^{\dagger}(\Pi_F^{(p)},\lambda_0)\geq 3$, assuming Conjecture \ref{conjmult} (b). We should expect that, at least generically, this multiplicity is exactly $3$. This is because the generically cuspidal $p$-adic family into which $\Pi_F^{(p)}$ deforms as a consequence of this multiplicity being nonzero, specializes at sufficiently regular weights to cuspidal automorphic representations of $G_2(\A)$ which are discrete series at the archimedean place, and the discrete series $L$-packets for $G_2(\R)$ all have three elements.

Now if we assume instead that $\epsilon(1/2,\pi_F,\Sym^3)=1$, but still that $L(1/2,\pi_F,\Sym^3)=0$, then Conjecture \ref{conjmult} (a) implies that $\mc{L}_\alpha(\pi_F,1/10)_f$ appears in cohomology in degrees $3$ and $5$, instead of in degree $4$, giving in general that the classical cuspidal multiplicity $m_{G_2,0}^{\cl}(\Pi_F^{(p)},\lambda_0)$ equals $-2$. Therefore we expect the multiplicity $m_{G_2,0}^{\dagger}(\Pi_F^{(p)},\lambda_0)$ to be zero in this case. Thus, even when $L(1/2,\pi_F,\Sym^3)=0$, as long as $\epsilon(1/2,\pi_F,\Sym^3)=1$, a different approach is needed to construct a $p$-adic deformation of $\Pi_F^{(p)}$.
\end{remark}

Our goal now is to remove the hypothesis that $k>4s_p+4$ in Lemma \ref{lemmultnonzero}. We will do this by deforming $F$ in a Coleman family. Given a cuspidal eigenform $F'$ of weight $k'$, a Coleman family containing $F'$ is, in particular, a $\mf{Z}$-family of effective finite slope character distributions $\mc{H}_{p,GL_2}\to\mc{O}(\mf{Z})$, where $\mf{Z}$ is a neighborhood of the weight $k'$ in the weight space for $GL_2$. Here, an integer weight $l$ is viewed as the weight $\tfrac{l-2}{2}\gamma$, where $\gamma$ is the simple positive root for $GL_2$.

We will require the following lemma.

\begin{lemma}
\label{lemmcfvariation}
Let $F'$ be a cuspidal eigenform of weight $k'$ and trivial nebentypus, and let $\mc{F}'$ be a Coleman family deforming $F'$. For any integer weight $l$ in the weight space for $\mc{F}'$, let $\mc{F}_l'$ be the specialization of $\mc{F}'$ at $l$. Let $\pi_l$ be the unitary automorphic representation of $GL_2(\A)$ with $p$-stabilization $\mc{F}_l'$.
\begin{enumerate}[label=(\alph*)]
\item Let $v\ne p$ be a finite place where $\pi_{k',v}$ is supercuspidal or a twist of the Steinberg representation. Then there is a neighborhood $\mf{U}$ of $k'$ in $\mf{Z}$ such that for all classical integer weights $l$ in $\mf{U}(\Q_p)$, we have
\[\pi_{l,v}\cong\pi_{k',v}.\]
\item Let $v\ne p$ be a finite place where $\pi_{k',v}$ is a ramified principal series representation. Write
\[\pi_{k',v}\cong\iota_{B_2(\Q_v)}^{GL_2(\Q_v)}(\chi_{v}\boxtimes\chi_{v}^{-1}),\]
where $B_2$ is the standard Borel of $GL_2$, $\chi_v$ is a ramified quasicharacter of $\Q_v^\times$, and by definition $\chi_{v}\boxtimes\chi_{v}^{-1}$ is the character of the diagonal torus given by
\[(\chi_{v}\boxtimes\chi_{v}^{-1})(\diag(t_1,t_2))=\chi_v(t_1)\chi_v(t_2)^{-1}.\]
(The representation $\pi_{k',v}$ necessarily has this form if it is a principal series representation because the nebentypus of $F'$ is trivial.) Then there is a neighborhood $\mf{U}$ of $k'$ in $\mf{Z}$, a finite extension $L$ of $\Q_p$ and an unramified character $\psi:\Q_v^\times\to(\mc{O}(\mf{U})\otimes L)^\times$ such that for all classical integer weights $l$ in $\mf{U}(\Q_p)$, we have
\[\pi_{l,v}\cong\iota_{B_2(\Q_v)}^{GL_2(\Q_v)}((\chi_{v}\psi_l)\boxtimes(\chi_{v}\psi_l)^{-1}),\]
where $\psi_l$ is the specialization of $\psi$ at the weight $l$.
\end{enumerate}
\end{lemma}

\begin{proof}
We use the theory of Bushnell--Kutzko types \cite{BKtypes} for the group $GL_2(\Q_v)$. Let $v\ne p$ be a finite place and assume first that $\pi_{k',v}$ is supercuspidal. Then by the theory of types there is an idempotent $e$ in the Hecke algebra of $GL_2(\Q_v)$ which acts as the identity on unramified twists of $\pi_{k',v}$, and as zero on any other irreducible admissible representation of $GL_2(\Q_v)$. Hence $\tr(e|\pi_{k',v})\ne 0$, which implies that $\tr(\rho|\pi_{l,v})\ne 0$ for classical integer weights $l$ in a neighborhood of $k$, which then implies that $\pi_{l,v}$ is isomorphic to $\pi_{k',v}$ up to unramified twist. Since $\pi_{l,v}$ has trivial central character for such $l$, we have that $\pi_{l,v}$ is an unramified quadratic twist of $\pi_{k',v}$. Since there are only finitely many quadratic characters of $\Q_p^\times$, there are only finitely many options for $\pi_{l,v}$, and by continuity we conclude that $\pi_{l,v}\cong\pi_{k',v}$ for such $l$.

If $\pi_{k',v}$ is instead a twist of the Steinberg representation, say by a character $\eta_v$ of $\Q_v^\times$ (which is necessarily quadratic) then we first choose a Dirichlet character $\eta$ which is unramified at $p$ and whose component at $v$ is $\eta_v$. We consider the twisted family $\mc{F}'\otimes\eta^{-1}$, whose $v$-component at a classical integer weight $l$ is $\pi_{l,v}\otimes\eta_v^{-1}$.

Then the theory of types again gives an idempotent $e$ in the Hecke algebra of $GL_2(\Q_v)$, which this time acts nontrivially on any constituent of an unramified principal series, and as zero on any other irreducible admissible representation. (In fact, we have
\[\rho=\vol(I_v(GL_2))^{-1}1_{I_v(GL_2)},\]
where $I_v(GL_2)$ is the standard Iwahori in $GL_2(\Z_v)$.) By analogous logic as in the supercuspidal case, there is a neighborhood $\mf{U}$ of $k'$ such that for any classical integer weight $l$ in that neighborhood, the representation $\pi_{l,v}\otimes\eta_v^{-1}$ is a constituent of an unramified principal series.

Now let
\[e'=\vol(GL_2(\Z_v))^{-1}1_{GL_2(\Z_v)}.\]
Then the trace of $e'$ on any unramified representation of $GL_2(\Q_v)$ is $1$, but the trace of $e'$ on any Steinberg representation is $0$. Thus by continuity, after possibly shrinking $\mf{U}$, for any classical integer weight $l$ in $\mf{U}$, we have that $\pi_{l,v}\otimes\eta_v^{-1}$ is an unramified twist of the Steinberg representation, and moreover that the twist must be by a quadratic character. Therefore there are only finitely many possibilities for such $\pi_{l,v}$, and again by continuity, we conclude that $\pi_{l,v}\cong\pi_{k',v}$. This proves (a).

In the setting of (b), let $\eta$ now be a Dirichlet character which is unramified at $p$ and for which $\chi_v\eta_v^{-1}$ is unramified. We again consider the twisted family $\mc{F}'\otimes\eta^{-1}$. The $v$-component of its specialization at $k'$ is given by
\[\pi_{k',v}\otimes\eta_v^{-1}\cong\iota_{B_2(\Q_v)}^{GL_2(\Q_v)}((\chi_v\eta_v)\boxtimes(\chi_v^{-1}\eta_v^{-1})).\]
Like in the Steinberg case, we now invoke the theory of types, the consequence of which this time is that there is a neighborhood $\mf{U}$ of $k'$ such that for all classical integer weights $l$ in $\mc{U}$, we have
\[\pi_{l,v}\otimes\eta_v^{-1}\textrm{ is a constituent of }\iota_{B_2(\Q_v)}^{GL_2(\Q_v)}(\phi_{l,v}\boxtimes(\phi_{l,v}^{-1}\eta_v^{-2}))\]
for some unramified character $\phi_{l,v}$ of $\Q_v^\times$. Moreover, if for some $l$ we had that $\pi_{l,v}$ were a twist of Steinberg, then by (a) the same would be true for $\pi_{l',v}$ for $l'$ in a neighborhood of $l$. Thus by shrinking $\mf{U}$ if necessary, we may assume that
\[\pi_{l,v}\otimes\eta_v^{-1}\cong\iota_{B_2(\Q_v)}^{GL_2(\Q_v)}(\phi_{l,v}\boxtimes(\phi_{l,v}^{-1}\eta_v^{-2})),\]
and that this is irreducible.

There are then two cases: Either $\eta_v^{-2}$ is ramified, or not. Let $\ell_v$ be the rational prime associated with $v$, and let $\mbf{a}_{\ell_v}\in\mc{O}(\mf{U})\otimes L$, for $L$ a large enough finite extension of $\Q_p$, be the $\ell_v$th coefficient of the $q$-expansion of $\mc{F}'\otimes\eta^{-1}$. Let $s\in\mc{O}(\mf{U})^\times$ be the analytic function such that
\[s(l)=\ell_v^{(l-1)/2}.\]
If $\eta_v^{-2}$ is ramified, let $U_v=U_{\ell_v}$ be the usual operator. Then in this case,
\[\phi_{l,v}(\ell_v)=\tr(U_v|\pi_{l,v}\otimes\eta_v^{-1})=\mbf{a}_{\ell_v}(l)s(l)^{-1};\]
see for example \cite[\S 2.2]{LW}. Then we may take the character $\psi_v$ in the statement of the lemma to be given by
\[\psi_v(\ell_v)=\mbf{a}_{\ell_v}s^{-1}\cdot(\chi_v\eta_v^{-1})(\ell_v)^{-1}.\]
(Note that since $\pi_{k',v}$ is not supercuspidal, we have $\mbf{a}_{\ell_v}(k')\ne 0$, so after possibly shrinking $\mf{U}$, we have indeed that $\psi_v(\ell_v)\in(\mc{O}(\mf{U})\otimes L)^\times$.)

If instead $\eta_v^{-2}$ is unramified, then we consider the polynomial
\[P(X)=X^2-\mbf{a}_{\ell_v}X+s^2.\]
If $L$ is large enough, this has a root, say $\mbf{b}\in(\mc{O}\mf{U})\otimes L$, at least after possibly shrinking $\mf{U}$. Moreover, for any classical integer weight $l$, we have that $\mbf{b}(l)$ is a $p$-adic unit. Thus actually $\mbf{b}\in(\mc{O}(\mf{U})\otimes L)^\times$. It follows that we may now take $\psi_v$ to be the unramified character given by
\[\psi_v(\ell_v)=\mbf{b}s^{-1}\cdot(\chi_v\eta_v^{-1})(\ell_v)^{-1}.\]

Finally, if $v$ does not divide the tame level of $\mc{F}'$, then $\pi_{l,v}$ is unramified for all classical integer weights $l$ in $\mf{Z}$. Then the construction of $\mbf{b}$ and $\psi_v$ just made is defined over all of $\mf{Z}$. This proves the lemma.
\end{proof}

As a corollary, we deduce that the sign of the symmetric cube functional equation is locally constant in Coleman families.

\begin{corollary}
\label{corepsilonlc}
With the setting as in the lemma, there is a neighborhood $\mf{U}$ of $k'$ in $\mf{Z}$ such that for all classical integer weights $l$ in $\mf{U}(\Q_p)$, we have
\[\epsilon(1/2,\pi_{k'},\Sym^3)=\epsilon(1/2,\pi_l,\Sym^3).\]
\end{corollary}

\begin{proof}
By part (a) of the lemma, the local signs $\epsilon_v(1/2,\pi_{l,v},\Sym^3)$ are constant in a neighborhood of $k'$ if $\pi_{k',v}$ is a twist of Steinberg or supercuspidal. Moreover the same is true if $\pi_{k',v}$ is a ramified principal series; one way to see this is to note that under the local Langlands correspondence, we have that $\Sym^3\pi_{l,v}$ corresponds to a sum of four characters, and these characters are unramified twists of the four characters to which $\Sym^3\pi_{k',v}$ corresponds by part (b) of the lemma. But the local signs are preserved under unramified twists.

Thus we just need to show that the local sign at $\infty$ is locally constant. One way to see this is to note that
\[\epsilon_\infty(1/2,\pi_{l,\infty}\times\pi_{l,\infty}\times\pi_{l,\infty})=\epsilon_\infty(1/2,\pi_{l,\infty},\Sym^3)\epsilon_\infty(1/2,\pi_{l,\infty})^2,\]
and the triple product sign $\epsilon_\infty(1/2,\pi_{l,\infty}\times\pi_{l,\infty}\times\pi_{l,\infty})$ is always $-1$ as the triple of weights $(l,l,l)$ is in the so-called balanced range. This completes the proof.
\end{proof}

Now let $\mc{F}$ be a Coleman family deforming the $p$-stabilization of $F$ corresponding to the fixed root $\alpha_p$ of the Hecke polynomial of $F$ at $p$. Let $\mf{Z}$ be its weight space, and let $\mf{U}\subset\mf{Z}$ be a small enough neighborhood of the weight $k$ so that the conclusions of Lemma \ref{lemmcfvariation} are satisfied for all places finite $v\ne p$.

The space $\mf{U}$ has a linear embedding into $\mf{X}$; on classical integer weights $l\geq 4$ in $\mf{U}$, this embedding sends $l$ to $\tfrac{l-4}{2}(2\alpha+3\beta)$. We view $\mf{U}$ as a subset of $\mf{X}$ via this embedding.

We now build a $\mf{U}$-family of effective finite slope character distributions for $G_2$, which we will call $\Pi_{\mf{U},\mc{F}}$, using the family $\mc{F}$. We do this by specifying the distribution on $C_c^\infty(G_2(\Q_v),\Q_p)$ for each $v\ne p$, and on the algebra $\mc{U}_p$, as follows.

\begin{itemize}
\item If $v\ne p$ is a finite place where $\pi_{F,v}$ is Steinberg or supercuspidal, then we let $\Pi_{\mf{U},\mc{F}}$ be given at $v$ by the distribution, which is constant over $\mf{U}$,
\[f\mapsto\tr(f|\mc{L}_{\alpha}(\pi_{F,v}\otimes\delta_{P_\alpha(\Q_v)}^{1/10})),\qquad f\in C_c^\infty(G_2(\Q_v),\Q_p),\]
where, as usual, $\mc{L}_{\alpha}$ denotes the Langlands quotient of the corresponding unitary parabolic induction.
\item If $v\ne p$ is a finite place where $\pi_{F,v}$ is a ramified principal series, say
\[\pi_{F,v}\cong\iota_{B_2(\Q_v)}^{GL_2(\Q_v)}(\chi_{v}\boxtimes\chi_{v}^{-1}),\]
then let $\psi:\Q_v^\times\to(\mc{O}(\mf{U})\otimes L)^\times$ be an unramified character as in the conclusion of Lemma \ref{lemmcfvariation} (b). Then we let $\Pi_{\mf{U},\mc{F}}$ be given at $v$ by the distribution
\[f\mapsto\tr(f|\iota_{P_\beta(\Q_v)}^{G_2(\Q_v)}(\psi\circ\det_\beta)),\qquad f\in C_c^\infty(G_2(\Q_v),\Q_p).\]
\item If $v\ne p$ is a place where $\pi_F$ is unramified, then we do the following. Let $\ell_v$ be the rational prime corresponding to $v$, and let $\mbf{a}_{\ell_v}\in\mc{O}(\mf{U})$ be the $\ell_v$th coefficient in the $q$-expansion of $\mc{F}$. Let $s\in\mc{O}(\mf{U})^\times$ be the function interpolating $l\mapsto\ell_v^{(l-1)/2}$ for integer weights $l$ in $\mf{U}$. Then let
\[\varphi_\alpha:C_c^\infty(M_\alpha(\Z_v)\backslash M_\alpha(\Q_v)/M_\alpha(\Z_v),\Q_p)\to\mc{O}(\mf{U})\]
be given by
\[\varphi_\alpha(T_v)=\mbf{a}_{\ell_v}s^{-1}\ell_v^{-1/2},\qquad\varphi_{\alpha}(R_v)=\ell_v^{-1},\]
where $T_v$ and $R_v$ are the characteristic functions of, respectively,
\[M_\alpha(\Z_v)\pmat{\ell_v&0\\ 0&1}M_\alpha(\Z_v)\qquad\textrm{and}\qquad \pmat{\ell_v&0\\ 0&\ell_v}M_\alpha(\Z_v).\]
Now let
\[\mc{S}_\alpha:C_c^\infty(G_2(\Z_v)\backslash G_2(\Q_v)/G_2(\Z_v),\Q_p)\to C_c^\infty(M_\alpha(\Z_v)\backslash M_\alpha(\Q_v)/M_\alpha(\Z_v),\Q_p)\]
be the relative Satake homomorphism:
\[\mc{S}_\alpha(f)(m)=\delta_{P_\alpha(\Q_v)}^{1/2}(m)\int_{N_\alpha(\Q_v)}f(mn)\,dn.\]
Then we let $\Pi_{\mf{U},\mc{F}}$ be given at $v$ by the spherical representation over $\mc{O}(\mf{U})$ whose trace on the spherical Hecke algebra $C_c^\infty(G_2(\Z_v)\backslash G_2(\Q_v)/G_2(\Z_v))$ is given by $\varphi_\alpha\circ\mc{S}_\alpha$.
\item At $p$, we let $\Pi_{\mf{U},\mc{F}}$ be given at $p$ by the character of $\mc{U}_p$ determined by
\[u_t\mapsto\mbf{a}_p^{v_p((2\alpha+3\beta)(t))}p^{v_p(\beta(t))},\qquad t\in T^-,\]
where $\mbf{a}_p$ is the $p$th Fourier coefficient of $\mc{F}$.
\end{itemize}

By construction, the family $\Pi_{\mf{U},\mc{F}}$ is a $\mf{U}$-family of effective finite slope character distributions which interpolates the construction $F\mapsto\Pi_F^{(p)}$ over classical integer weights $l\geq k$ in $\mf{U}$; this follows from Lemma \ref{lemmcfvariation} for the construction at ramified places, and moreover from Proposition \ref{proplqofPi} at the places where $\pi_F$ is a ramified principal series. We use the family $\Pi_{\mf{U},\mc{F}}$ to prove the following proposition, which is the culmination of all the work done in this section.

\begin{proposition}
\label{propmultnonzero}
Assume Conjecture \ref{conjmult} (b). Then $m_{G_2,0}^\dagger(\Pi_F^{(p)},\lambda_0)>0$.
\end{proposition}

\begin{proof}
Consider the $\mf{U}$-family of effective finite slope character distributions $\Pi_{\mf{U},\mc{F}}$ just constructed. For any classical integer weight $l\geq k$ in $\mf{U}$, let $\Pi_{l,\mc{F}}$ be the specialization of $\Pi_{\mf{U},\mc{F}}$ at the weight $\tfrac{l-4}{2}(2\alpha+3\beta)$. Also let $\pi_l$ be the automorphic representation of $M_\alpha(\A)$ associated with the specialization of $\mc{F}$ at $l$. By Corollary \ref{corepsilonlc}, after possibly shrinking $\mf{U}$, we have 
\[\epsilon(1/2,\pi_l,\Sym^3)=-1,\]
for such $l$. If moreover $l>4s_p+4$, then Lemma \ref{lemmultnonzero} applies and we have for such $l$ that
\[m_{G_2,0}^\dagger(\Pi_{l,\mc{F}},\tfrac{l-4}{2}(2\alpha+3\beta))>0,\]
under Conjecture \ref{conjmult} (b). We then invoke Lemma \ref{lemintsummand} with $J=I_{G_2,0}^{\dagger}$, $\tau=\Pi_{\mf{U},\mc{F}}$ and $S$ the set of classical integer weights $l$ with $l>4s_p+4$; the conclusion is exactly the conclusion of the proposition, so we are done.
\end{proof}

\subsection{The $p$-adic family}
\label{subsecpadicfam}
We keep the notation of the previous three subsections. In particular, we have our $\mc{H}_p$-representation $\Pi_F^{(p)}$ which we are now in a position to $p$-adically deform. To do this, we first set a bit of notation.

Let $S$ be a finite set of finite places of $\Q$. Let $\A_f^{S,p}$ be the ring of finite adeles away from $S$ and $p$, and let $K_{f,\mr{max}}^{S,p}\subset G_2(\A_f^{S,p})$ be the projection of the maximal compact subgroup $K_{f,\mr{max}}$ to $G_2(\A_f^{S,p})$. We let
\[R_{S,p}=C_c^\infty(K_{f,\mr{max}}^{S,p}\backslash G_2(\A_f^{S,p})/K_{f,\mr{max}}^{S,p},\Q_p)\otimes_{\Z_p}\mc{U}_p.\]
Then $R_{S,p}$ is a commutative $\Q_p$-algebra with identity.

Recall we have written $N$ for the level of $F$. Let $S(N)$ denote the set of finite places of $\Q$ dividing $N$. Then since $\Pi_F^{(p)}$ is spherical away from $S(N)$, it determines a character
\begin{equation}
\label{eqtheta0}
\theta_0:R_{S(N),p}\to L,
\end{equation}
where $L$ is some finite extension of $\Q_p$.

Let $K_f^p$ be an open compact subgroup of $K_{f,\mr{max}}^p$ such that $\Pi_F^{(p)}$ is of level $K_f^p$, and such that the projection of $K_f^p$ to $G_2(\A_f^{S,p})$ is $K_{f,\mr{max}}^{S,p}$. Then we have the following theorem.

\begin{theorem}
\label{thmpadicdef}
Assume Conjecture \ref{conjmult} (b). Then there are
\begin{itemize}
\item an open affinoid subdomain $\mf{U}\subset\mf{X}$,
\item a finite cover $\mbf{w}:\mf{V}\to\mf{U}$,
\item a point $y_0\in\mf{V}(\overline\Q_p)$ with $\mbf{w}(y_0)=\lambda_0$,
\item a Zariski dense subset $\Sigma\subset\mf{V}(\overline\Q_p)$ with $\mbf{w}(y)$ algebraic and regular for every $y\in\Sigma$,
\item for each $y\in\Sigma$, a nonempty finite set $\Pi_y$ of finite slope $p$-stabilizations of irreducible, cohomological, cuspidal automorphic representations of $G_2$ of weight $\mbf{w}(y)$ of level $K_f^p$,
\item a $\Q_p$-algebra homomorphism $\theta_{\mf{V}}:R_{S(N),p}\to\mc{O}(\mf{V})$,
\item a nontrivial $\Q_p$-linear map $I_{\mf{V}}:\mc{H}_p(K_f^p)\to\mc{O}(\mf{V})$,
\end{itemize}
satisfying the following properties:
\begin{enumerate}[label=(\arabic*)]
\item The set $\mbf{w}(\Sigma)$ contains all sufficiently regular dominant algebraic weights in $\mf{X}(\Q_p)$;
\item There is a proper Zariski closed subset of $\mf{U}$ such that for $y\in\Sigma$ with $\mbf{w}(y)$ not in this closed subset, $\Pi_y$ only contains one representation;
\item The specialization of $\theta_{\mf{V}}$ at the point $y_0$ is the character $\theta_0$ of \eqref{eqtheta0};
\item The representation $\Pi_F^{(p)}$ is an irreducible component of the specialization of $I_{\mf{V}}$ at $y_0$;
\item For each $y\in\Sigma$ and each $\sigma\in\Pi_y$, the specialization of $\theta_{\mf{V}}$ is the character of $R_{S(N),p}$ determined by $\sigma$;
\item For each $y\in\Sigma$, the specialization $I_y$ of $I_{\mf{V}}$ at $y$ satisfies
\[I_y(f)=\sum_{\sigma\in\Pi_y}m_{G_2,0}^{\cl}(\sigma,\mbf{w}(y))\tr(f|\sigma),\]
for $f\in\mc{H}_p(K_f^p)$.
\end{enumerate}
\end{theorem}

\begin{proof}
By Proposition \ref{propmultnonzero}, under Conjecture \ref{conjmult} (b) we have $m_{G_2,0}^\dagger(\Pi_F^{(p)},\lambda_0)>0$. Then this theorem follows immediately from the main result of Urban's paper, \cite[Theorem 5.4.4]{urbanev}.
\end{proof}

We now examine the variation of the local components of the representations in the $p$-adic family $\mf{V}$. In the following, we view $\overline\Q_p$ as $\C$ via our fixed isomorphism.

\begin{proposition}
\label{propunratp}
In the context of Theorem \ref{thmpadicdef}, there is a neighborhood $\mf{U}'$ of $y_0$ in $\mf{U}$ such that if $y\in\Sigma$ is in $\mf{U}'$ is such that $\mbf{w}(y)$ is sufficiently regular (that is, greater than a sufficiently large multiple of the weight $\rho$) then any $\sigma\in\Pi_y$ is a $p$-stabilization of an irreducible, cohomological, cuspidal automorphic representation which is unramified at $p$.
\end{proposition}

\begin{proof}
The content of this proposition is the unramifiedness assertion. Let $y\in\Sigma$ with $\mbf{w}(y)$ regular and $\sigma\in\Pi_y$. Write $\lambda=\mbf{w}(y)$. Let $\pi_\sigma$ be the automorphic representation of which $\sigma$ is a $p$-stabilization. By definition, the component $\pi_{\sigma,p}$ of $\pi_\sigma$ at $p$ has $I$-fixed vectors, and the space $\pi_{\sigma,p}^I$ of such, as a $\mc{U}_p$-module, has an irreducible subquotient on which $\mc{U}_p$ acts by a character $\theta$. Let $\mu$ be the slope of this character. Then $\theta$ has slope $\mu+\lambda$ (recall that the distribution $I_{G_2}^{\cl}(\cdot,\lambda)$ is normalized by a factor of $\vert\lambda(t)\vert^{-1}$ as in \eqref{eqocandcl}). 

Now if moreover $y$ is in a sufficiently small neighborhood $\mf{U}'$ of $y_0$, then the slope of $\sigma$ is equal to that of $\Pi_F^{(p)}$ by continuity. Thus we have
\[\mu+\lambda=s_p(2\alpha+3\beta)+\beta.\]
So if $\lambda$ is sufficiently regular, then $\mu$ is far from the wall of any chamber, and so in particular we have that
\[\vert v_p(\langle\mu,\gamma^\vee\rangle)\vert\]
is sufficiently large for any root $\gamma$ for $G_2$.

Now since $\pi_{\sigma,p}^I\ne 0$, we know $\pi_{\sigma,p}$ is a constituent of a principal series representation, say
\[\iota_{B(\Q_p)}^{G_2(\Q_p)}(\chi)\]
for some unramified character $\chi$ of $T(\Q_p)$. But, by the same arguments as in Proposition \ref{proppstabnofPi}, the slopes of the characters of $\mc{U}_p$ acting on such a principal series representation are bounded in terms of the values $v_p(\chi(\gamma^\vee(p)))$ for roots $\gamma$. So if $\lambda$ sufficiently regular, this forces $\chi$ to satisfy in particular that $v_p(\chi(\gamma^\vee(p)))\ne\pm 1$. Then we may appeal to Proposition \ref{propvppsirr} to see that
\[\pi_{\sigma,p}\cong\iota_{B(\Q_p)}^{G_2(\Q_p)}(\chi),\]
since this principal series representation is irreducible. This completes the proof.
\end{proof}

\begin{proposition}
\label{propscvariation}
Assume $v$ is a place of $\Q$, different from $2$ or $3$, where $\pi_{F,v}$ is supercuspidal. Then, in the context of Theorem \ref{thmpadicdef}, there is a neighborhood $\mf{U}'$ of $y_0$ in $\mf{U}$ with the following property: If $y\in\Sigma$ with $\mbf{w}(y)$ in $\mf{U}'$ and $\Pi_y=\{\sigma\}$ is a singleton, then the component $\sigma_v$ of $\sigma$ at $v$ is either given by an irreducible principal series representation of the form
\[\sigma_v\cong\iota_{P_\alpha(\Q_v)}^{G_2(\Q_v)}(\pi_{F,v},s)\]
for some $s\in\C$ with $\re(s)\geq 0$, or it is isomorphic to the component $\Pi_{F,v}^{(p)}$ of $\Pi_F^{(p)}$ at $v$.
\end{proposition}

\begin{proof}
Fintzen \cite{fintzen} has established the theory of types for reductive groups $G$ over nonarchimedean fields $F$ which split over a tamely ramified extension of $F$ and for which the order of the Weyl group of $G$ is coprime to the residue characteristic of $F$. We apply this theory for $G=G_2$ and $F=\Q_v$, which we can do as the Weyl group of $G_2$ is $12$ and we assumed $v\ne 2,3$. What this gives us is an idempotent $e$ in the Hecke algebra of $G_2(\Q_v)$ such that $e$ acts nontrivially on constituents of representations of the form
\begin{equation}
\label{eqsupercuspps}
\iota_{P_\alpha(\Q_v)}^{G_2(\Q_v)}(\pi_{F,v}\otimes\delta_{P_\alpha(\Q_v)}^s),\qquad s\in\C,
\end{equation}
and such that $e$ acts as zero on all other irreducible admissible representations of $G_2(\Q_v)$. By the self-duality of $\pi_{F,v}$, we may assume $\re(s)\geq 0$.

Now since $\Pi_{F,v}^{(p)}$ is a constituent of
\[\iota_{P_\alpha(\Q_v)}^{G_2(\Q_v)}(\pi_{F,v}\otimes\delta_{P_\alpha(\Q_v)}^{1/10}),\]
we have
\[\tr(e|\Pi_{F,v}^{(p)})\ne 0.\]
Let $K_f^{p,v}$ be the component of the tame level $K_f^p$ away from $v$, let $t\in T^{--}$, and consider the Hecke operator
\[f=e\otimes 1_{K_f^{p,v}}\otimes u_t.\]
We may assume the tame level $K_f^p$ is such that $f\in\mc{H}_p(K_f^p)$. Then
\[\tr(f|\Pi_F^{(p)})\ne 0.\]
Then since $I_{\mf{V}}(f)$ is analytic on $\mf{V}$, after possibly restricting to a neighborhood $\mf{U}'$ of $y_0$, then for any $y\in\Sigma$ with $\mbf{w}(y)$ in $\mf{U}'$ with $\Pi_y=\{\sigma\}$ a singleton, we have
\[\tr(f|\sigma)\ne 0.\]
In particular,
\[\tr(f|\sigma_v)\ne 0\]
for such $\sigma$, and therefore $\sigma_v$ is a constituent of \eqref{eqsupercuspps}. By the results of \cite[\S 8]{shahidics}, there are only finitely many values of $s$ where \eqref{eqsupercuspps} is reducible. By possibly shrinking $\mf{U}'$ again, we may exclude the possibility that $\sigma_v$ is a constituent of such a reducible principal series, besides possibly being isomorphic to $\Pi_{F,v}^{(p)}$. This completes the proof.
\end{proof}

\begin{proposition}
\label{proprampsvariation}
Assume $v$ is a place of $\Q$, different from $2$ or $3$, where $\pi_{F,v}$ is a ramified principal series representation of $GL_2(\Q_v)$, say
\[\pi_{F,v}\cong\iota_{B_2(\Q_v)}^{GL_2(\Q_v)}(\chi\boxtimes\chi^{-1}),\]
for some ramified, unitary character $\chi$ of $\Q_v^\times$; see the notation of Lemma \ref{lemmcfvariation} (b). Then, in the context of Theorem \ref{thmpadicdef}, there is a neighborhood $\mf{U}'$ of $y_0$ in $\mf{U}$ with the following property: If $y\in\Sigma$ with $\mbf{w}(y)$ in $\mf{U}'$ and $\Pi_y=\{\sigma\}$ is a singleton, then the component $\sigma_v$ of $\sigma$ at $v$ is given by an irreducible principal series representation either of the form
\[\iota_{P_\beta(\Q_v)}^{G_2(\Q_v)}((\chi\vert\cdot\vert^s)\circ\det_\beta)\qquad\textrm{or}\qquad\iota_{P_\alpha(\Q_v)}^{G_2(\Q_v)}((\chi\vert\cdot\vert^s)\circ\det_\alpha),\]
for some $s\in\C$ with $\re(s)\geq 0$. Furthermore, if $\chi|_{\Z_v^\times}$ is not of order $2$, then it is always the former possibility that occurs.
\end{proposition}

\begin{proof}
As in the previous proposition, we invoke the theory of types as developed in \cite{fintzen} to conclude that there is a neighborhood $\mf{U}'$ of $y_0$ in $\mf{U}$ such that, this time, for every $y\in\Sigma$ with $\mbf{w}(y)$ in $\mf{U}'$, and $\Pi_y=\{\sigma\}$ a singleton, we have that $\sigma_v$ is a constituent of a principal series representation of the form
\begin{equation}
\label{eqramps}
\iota_{B(\Q_v)}^{G_2(\Q_v)}((\chi\circ\alpha)\cdot e^{\langle H_B(\cdot),\Lambda\rangle})
\end{equation}
for some $\Lambda\in X^*(T)\otimes\C$. Then one of the following three cases holds:
\begin{enumerate}[label=(\arabic*)]
\item We have that $e^{\langle H_B(\cdot),\Lambda\rangle}$ is a character of order $2$;
\item We have that $e^{\langle H_B\circ\gamma^\vee,\Lambda\rangle}=\vert\cdot\vert$ for some root $\gamma$ of $G_2$;
\item Neither of the above two cases holds.
\end{enumerate}
By Proposition \ref{propvppsirr}, in case (3) the representation \eqref{eqramps} is irreducible. Moreover, by the same proposition, it is also irreducible in case (1) if and only if $\chi$ is not of order $2$. In case (2) it may reduce. In fact it is not too difficult to check using Proposition \ref{propvppsirr} that it reduces if and only if one of the following subcases holds:
\begin{enumerate}[label=(\roman*)]
\item We have $e^{\langle H_B\circ(\alpha+2\beta)^\vee,\Lambda\rangle}=\vert\cdot\vert^{\pm 1}$;
\item We have $(\chi\circ\alpha^\vee)e^{\langle H_B\circ\alpha^\vee,\Lambda\rangle}=\vert\cdot\vert^{\pm 1}$ and $\chi|_{\Z_v^\times}$ is of order $2$;
\item We have $e^{\langle H_B\circ(\alpha+\beta)^\vee,\Lambda\rangle}=\vert\cdot\vert$ and $\chi$ is of order $3$.
\end{enumerate}
By Corollary \ref{corredstps}, in case (i) above, up to semisimplification \eqref{eqramps} reduces into
\begin{equation}
\label{eqrednofps}
\iota_{P_\beta(\Q_p)}^{G_2(\Q_p)}(\mr{St}\otimes((\chi\vert\cdot\vert^s)\circ\det_\beta))\qquad\textrm{and}\qquad\iota_{P_\beta(\Q_p)}^{G_2(\Q_p)}((\chi\vert\cdot\vert^s)\circ\det_\beta),
\end{equation}
for some $s\in\C$ depending on $\Lambda$ (explicitly given by $2s=\langle\alpha^\vee,\Lambda\rangle$). By \cite[Theorem 3.1]{muic}, there are only finitely many values of $\Lambda$ satisfying (i) and such that either of these representations reduce.

In case (iii), we have $\chi\circ\alpha\circ(\alpha+\beta)^\vee=1$ because $\langle\alpha,(\alpha+\beta)^\vee\rangle=3$ and $\chi$ is order $3$. Then $\chi\circ\alpha=\chi^2\circ(\alpha+3\beta)$. The pair $(\alpha+3\beta,\alpha+\beta)$ is related by a Weyl group element to the pair $(\alpha,\alpha+2\beta)$, and so \eqref{eqramps} again reduces into \eqref{eqrednofps} by Corollary \ref{corredstps}. As above, there are only finitely many values of $\Lambda$ satisfying (iii) and such that either of these representations reduce.

In case (ii), first let $\chi'$ be the character such that $\chi'|_{\Z_v^\times}=\chi|_{\Z_v^\times}$ and $\chi(\ell_v)=1$, where $\ell_v$ is the rational prime associated with $v$. Then similar arguments as above give us that \eqref{eqramps} is, up to semisimplification, given by
\[\iota_{B(\Q_v)}^{G_2(\Q_v)}(((\chi'\vert\cdot\vert^s)\circ\beta)\cdot\vert\det_\beta\vert^{1/2}),\]
for some $s\in\C$. An analogous argument as in the proof of Proposition \ref{propvppsirr}, with the roles of $\alpha$ and $\beta$ switched, shows that this reduces into
\[\iota_{P_\alpha(\Q_p)}^{G_2(\Q_p)}(\mr{St}\otimes((\chi'\vert\cdot\vert^s)\circ\det_\alpha))\qquad\textrm{and}\qquad\iota_{P_\alpha(\Q_p)}^{G_2(\Q_p)}((\chi'\vert\cdot\vert^s)\circ\det_\alpha).\]
Again by \cite[Theorem 3.1]{muic}, there are only finitely many values of $\Lambda$ satisfying (ii) and such that either of these representations reduce.

Now we note that by definition, and by Proposition \ref{propvppsirr}, the component of $\Pi_v^{(p)}$ at $v$ is given by the irreducible principal series representation
\[\iota_{P_\beta(\Q_p)}^{G_2(\Q_p)}(\chi\circ\det_\beta).\]
By the arguments above and the continuity of Hecke traces on $\mf{V}$, we may shrink $\mf{U}'$ so that we can (and will) assume that for all $y$ in $\Sigma$ with $y\in\mf{U}'$ and $\Pi_y=\{\sigma\}$, we have that $\sigma_v$ is of one of the following five forms:
\begin{equation}
\label{eqsigmavposs1}
\sigma_v\cong\iota_{P_\beta(\Q_p)}^{G_2(\Q_p)}((\chi\vert\cdot\vert^s)\circ\det_\beta),\qquad\textrm{for some }s\in\C;
\end{equation}
\begin{equation}
\label{eqsigmavposs2}
\sigma_v\cong\iota_{P_\beta(\Q_p)}^{G_2(\Q_p)}(\mr{St}\otimes((\chi\vert\cdot\vert^s)\circ\det_\beta)),\qquad\textrm{for some }s\in\C;
\end{equation}
\begin{equation}
\label{eqsigmavposs3}
\sigma_v\cong\iota_{P_\alpha(\Q_p)}^{G_2(\Q_p)}((\chi\vert\cdot\vert^s)\circ\det_\alpha),\qquad\textrm{for some }s\in\C\textrm{ and }\chi^2|_{\Z_v^\times}=1;
\end{equation}
\begin{equation}
\label{eqsigmavposs4}
\sigma_v\cong\iota_{P_\alpha(\Q_p)}^{G_2(\Q_p)}(\mr{St}\otimes((\chi\vert\cdot\vert^s)\circ\det_\alpha)),\qquad\textrm{for some }s\in\C\textrm{ and }\chi^2|_{\Z_v^\times}=1;
\end{equation}
\begin{equation}
\label{eqsigmavposs5}
\sigma_v\cong\iota_{B(\Q_v)}^{G_2(\Q_v)}((\chi\circ\alpha)\cdot e^{\langle H_B(\cdot),\Lambda\rangle}),\qquad\textrm{for some }\Lambda\in X^*(T)\otimes\C.
\end{equation}
Indeed, there are only finitely many other possibilities for $\sigma_v$, and since $\Pi_v^{(p)}$ is not among them, we may make $\mf{U}'$ shrink $\mf{U}'$ to exclude them.

We now exclude the possibility that \eqref{eqsigmavposs2}, \eqref{eqsigmavposs4} or \eqref{eqsigmavposs5} can occur. To do this, let
\[I_{m,v}'=\Sset{g\in G_2(\Z_v)}{(g\modulo{\ell_v^m})\in N(\Z/\ell_v^m\Z)},\]
which is a subgroup of the Iwahori subgroup of $G_2(\Z_v)$. Also let
\[I_{m,v,\gamma}'=I_{m,v}'\cap M_\gamma(\Z_v),\qquad\gamma\in\{\alpha,\beta\},\]
which an analogously defined subgroup of $M_\gamma(\Z_v)$. If $m>1$ is sufficiently large, then
\[\chi(\det_\gamma(g))\vert\det_\gamma(g)\vert^s=1,\qquad\textrm{for all }g\in I_{m,v,\gamma}.\]
Thus for any unramified character $\psi$ of $T(\Q_v)$, the representations
\[\iota_{B_\gamma(\Q_v)}^{M_\gamma(\Q_v)}(\psi)\otimes((\chi\vert\cdot\vert^s)\circ\det_\gamma),\]
where $B_\gamma=B\cap M_\gamma$, have fixed vectors by $I_{m,v,\gamma}'$; in fact the number of independent such fixed vectors is the number $M_1$,
\[M_1=\#(B_\gamma(\Q_v)\backslash M_\gamma(\Q_v)/I_{m,v,\gamma}').\]
This follows since any function in $\iota_{B_\gamma(\Q_v)}^{M_\gamma(\Q_v)}(\psi)^{I_{m,v,\gamma}'}$ is determined by its values on a fixed set of representatives of the double coset space above, and vice-versa.

Similarly, we have
\[\dim_{\C}\iota_{P_\gamma(\Q_v)}^{G_2(\Q_v)}((\chi\vert\cdot\vert^s)\circ\det_\gamma)^{I_{m,v}'}=M_{2,\gamma},\]
where
\[M_{2,\gamma}=\#(B(\Q_v)\backslash G_2(\Q_v)/I_{m,v}'),\]
and we have also
\[\dim_{\C}\iota_{P_\gamma(\Q_v)}^{G_2(\Q_v)}(\mr{St}\otimes((\chi\vert\cdot\vert^s)\circ\det_\gamma))^{I_{m,v}'}=(M_1-1)M_{2,\gamma},\]
and
\[\dim_{\C}\iota_{B(\Q_v)}^{G_2(\Q_v)}((\chi\circ\alpha)\cdot e^{\langle H_B(\cdot),\Lambda\rangle})^{I_{m,v}'}=M_1M_{2,\gamma}.\]
(In particular $M_{2,\gamma}$ is actually independent of $\gamma$.) Note $M_1-1>1$ as $m>1$. Now since the trace of the Hecke operator $\vol(I_{m,v}')1_{I_{m,v}'}$ on any of these spaces of fixed vectors gives the dimension, the continuity of such traces in the family over $\mf{U}'$ excludes the possibility that \eqref{eqsigmavposs2}, \eqref{eqsigmavposs4} or \eqref{eqsigmavposs5} can occur. Thus one of \eqref{eqsigmavposs1} or \eqref{eqsigmavposs3} occurs, and we may assume that $\re(s)\geq 0$ in these formulas by possibly dualizing. This completes the proof.
\end{proof}

\begin{proposition}
\label{propunrstvariation}
Assume $v$ is a place of $\Q$ where $\pi_{F,v}$ is an unramified twist of the Steinberg representation $\mr{St}$ of $GL_2(\Q_v)$, say
\[\pi_{F,v}\cong\mr{St}\otimes\chi,\]
for some (necessarily quadratic) unramified character $\chi$ of $\Q_v^\times$. Then, in the context of Theorem \ref{thmpadicdef}, there is a neighborhood $\mf{U}'$ of $y_0$ in $\mf{U}$ with the following property: If $y\in\Sigma$ with $\mbf{w}(y)$ in $\mf{U}'$ and $\Pi_y=\{\sigma\}$ is a singleton, then the component $\sigma_v$ of $\sigma$ at $v$ is isomorphic to $\Pi_{F,v}^{(p)}$.
\end{proposition}

\begin{proof}
Recall that $\Pi_{F,v}^{(p)}$ is the Langlands quotient of
\begin{equation}
\label{eqsteinbergps}
\iota_{P_\alpha(\Q_v)}^{G_2(\Q_v)}(\pi_{F,v}\otimes\vert\det_\alpha\vert^{1/2})\cong\iota_{P_\alpha(\Q_v)}^{G_2(\Q_v)}(\mr{St}\otimes((\chi\vert\cdot\vert^{1/2})\circ\det_\alpha)).
\end{equation}
By induction in stages, this is a constituent of an unramified principal series induced from the Borel, and therefore has a nonzero fixed vector by the Iwahori subgroup $I_v$ of $G_2(\Z_v)$. By the continuity of the trace of the Hecke operator $\vol(I_v)^{-1}1_{I_v}$ on $\mf{V}$, there is thus a neighborhood $\mf{U}'$ of $y_0$ in $\mf{U}$ such that if $y\in\Sigma$ with $\mbf{w}(y)$ in $\mf{U}'$ and $\Pi_y=\{\sigma\}$, then $\sigma_v$ is a constituent of an unramified principal series representation induced from $T$. By the results of \cite[\S 3]{muic}, with finitely many exceptions, such constituents are of the form
\begin{equation}
\label{eqvariousps}
\iota_{P_\gamma(\Q_v)}^{G_2(\Q_v)}(\psi\circ\det_\gamma),\qquad\textrm{or}\qquad\iota_{P_\gamma(\Q_v)}^{G_2(\Q_v)}(\mr{St}\otimes(\psi\circ\det_\gamma)),\qquad\textrm{or}\qquad\iota_{T(\Q_v)}^{G_2(\Q_v)}(\psi'),
\end{equation}
for $\gamma\in\{\alpha,\beta\}$, and for some unramified characters $\psi$ of $\Q_v^\times$ and $\psi'$ of $T(\Q_v)$. Now the first two of these principal series representations have
\[\#(P_\gamma(\Q_v)\backslash G_2(\Q_v)/I_v)=6\]
independent $I_v$-fixed vectors, while the last one has
\[\#(B(\Q_v)\backslash G_2(\Q_v)/I_v)=12\]
independent $I_v$-fixed vectors.

Now assume $\chi$ is quadratic. Then by \cite[Theorem 3.1]{muic}, the principal series \eqref{eqsteinbergps} is reducible and therefore has fewer than $6$ independent $I_v$-fixed vectors. Thus it cannot be isomorphic to any of the principal series representations in \eqref{eqvariousps}, and therefore the same is true, after possibly shrinking $\mf{U}'$, for the $\sigma_v$ as above. Such $\sigma_v$ then fall into a finite list, and after possibly shrinking $\mf{U}'$ again, the proposition follows.
\end{proof}

\begin{proposition}
\label{propramstvariation}
Assume $v$ is a place of $\Q$, different from $2$ or $3$, where $\pi_{F,v}$ is a ramified twist of the Steinberg representation $\mr{St}$ of $GL_2(\Q_v)$, say
\[\pi_{F,v}\cong\mr{St}\otimes\chi,\]
for some (necessarily quadratic) ramified character $\chi$ of $\Q_v^\times$. Then, in the context of Theorem \ref{thmpadicdef}, there is a neighborhood $\mf{U}'$ of $y_0$ in $\mf{U}$ with the following property: If $y\in\Sigma$ with $\mbf{w}(y)$ in $\mf{U}'$ and $\Pi_y=\{\sigma\}$ is a singleton, then the component $\sigma_v$ of $\sigma$ at $v$ is isomorphic to $\Pi_{F,v}^{(p)}$.
\end{proposition}

\begin{proof}
The proof is very similar to that of Proposition \ref{propunrstvariation}, but we have to use a deeper Iwahori subgroup to distinguish $\Pi_{F,v}^{(p)}$ from various principal series representations like in the proof of Proposition \ref{proprampsvariation}; we also must appeal to \cite{fintzen} in order to know that for $y\in\Sigma$ with $\mbf{w}(y)$ in a neighborhood of $y_0$ with $\Pi_y=\{\sigma\}$, we have that $\sigma_v$ is a constituent of a representation induced from $B$ from an unramified twist of the character $\chi\circ\det_\alpha$. This is why we impose the restriction that $v$ is neither $2$ nor $3$.
\end{proof}

\section{Galois representations}
\label{secgalreps}
We now come to the Galois side of the main argument. In Section \ref{subsecgalreps} we obtain the $G_2$-Galois representations we need assuming Conjecture \ref{conjliftings} on liftings to $GL_7$. Section \ref{subsecpadichodge} proves various preliminary results in $p$-adic Hodge theory. In Section \ref{subsecpsch}, we glue the Galois representations coming the members of the $p$-adic family from Theorem \ref{thmpadicdef} using the theory of pseudocharacters. Then in Section \ref{subseclattice} we construct the lattice in this family of Galois representations whose specialization at a certain point, which is studied in detail in Section \ref{subsecmainthm}, gives the sought extension in the symmetric cube Selmer group.

\subsection{Automorphic Galois representations into $G_2$}
\label{subsecgalreps}
The purpose of this subsection is to prove the following theorem, which will be used later to attach Galois representations to members of the $p$-adic family $\mf{V}$ from Theorem \ref{thmpadicdef}. Once again, we view $\overline\Q_p$ as $\C$ via our fixed isomorphism between them.

\begin{theorem}
\label{thmgalreps}
Assume Conjecture \ref{conjliftings} (a)-(c). Let $\Pi$ be a cuspidal automorphic representation of $G_2(\A)$ with archimedean component $\Pi_\infty$ which is cohomological of regular weight
\[\lambda=c_1(\alpha+2\beta)+c_2(2\alpha+3\beta),\qquad c_1,c_2>0.\]
Then there is a continuous, semisimple Galois representation $\rho_\Pi:G_\Q\to G_2(\overline\Q_p)$ with the following properties:
\begin{enumerate}[label=(\alph*)]
\item If $\ell\ne p$ is a prime where $\Pi_\ell$ is unramified, then $\rho_\Pi$ is unramified at $\ell$ and
\[\rho_\Pi(\Frob_\ell^{-1})\in s(\Pi_\ell),\]
where $s(\Pi_\ell)$ is the conjugacy class in $G_2(\overline\Q_p)$ of the Satake parameter of $\Pi_\ell$.
\item If $\ell\ne p$ is a place where $\Pi_\ell$ is ramified, then we have
\[\mr{WD}(R_7(\rho_\Pi|_{G_{\Q_\ell}}))\cong\mr{LL}_\ell(\widetilde{\Pi}_\ell\otimes\vert\det\vert^{-3}),\]
where $R_7$ is the seven dimensional standard representation of $G_2$, $\mr{WD}(R_7(\rho_\Pi|_{G_{\Q_\ell}}))$ denotes the Frobenius semisimple Weil--Deligne representation associated with $R_7(\rho_\Pi|_{G_{\Q_\ell}})$, $\mr{LL}_\ell$ denotes the local Langlands correspondence at $\ell$ of Harris--Taylor and Henniart, and $\widetilde{\Pi}$ is the automorphic representation of $GL_7(\A)$ lifted from $\Pi$ by Conjecture \ref{conjliftings}.
\item If $\Pi_p$ is unramified, then $R_7\circ\rho_\Pi$ is crystalline at $p$ with Hodge--Tate weights
\[(2c_1+c_2+3,c_1+c_2+2,c_1+1,0,-c_1-1,-c_1-c_2-2,-2c_1-c_2-3),\]
and the crystalline Frobenius $\phi$ on $D_{\crys}(R_7(\rho_{\Pi}|_{G_{\Q_p}}))$ satisfies
\[\det(1-\phi^{-1}X)=\det(1-R_7(s(\Pi_p))X).\]
\end{enumerate}
\end{theorem}

\begin{proof}
The existence of such a representation into $GL_7(\overline\Q_p)$ attached to $\widetilde{\Pi}$ satisfying the properties (a) and (c) of the theorem is a direct consequence of the theorems of Shin \cite{shin} and Chenevier--Harris \cite[Theorem 4.2]{chehar}. Since $\widetilde{\Pi}$ may be an isobaric sum of cuspidal automorphic representations of a Levi subgroup, we may need to invoke the main result of Caraiani \cite{caraiani} to see (b) if one of the factors of this Levi subgroup has even rank. Finally, to show that this representation actually factors through $G_2(\overline\Q_p)$, one can invoke \cite[Theorem 4.18]{chen} and use the regularity of the Hodge--Tate weights to show that this representation does not fall into one of the exceptional cases of that theorem; see the proof of \cite[Theorem 6.4]{chen} for all the details.
\end{proof}

\subsection{Some $p$-adic Hodge theory}
\label{subsecpadichodge}
We now give several preliminary results in $p$-adic Hodge theory that will be used in studying the Galois representations we will construct at $p$. The first such result a technical lemma which generalizes a result due to Skinner and Urban, see \cite{urbanerr}. It is proved in much the same way.

\begin{lemma}
\label{lemEisdR}
Let $V$ and $W$ be de Rham representations of $G_{\Q_p}$. Let $E$ be an extension,
\[0\to V\to E\to W\to 0.\]
Assume that all the Hodge--Tate weights of $V$ are strictly negative. Writing $g:D_{\mr{dR}}(E)\to D_{\mr{dR}}(W)$ for the natural map, assume there is a subspace $D\subset D_{\mr{dR}}(E)$ such that
\[D_{\mr{dR}}(W)=g(D)\oplus \Fil^0(D_{\mr{dR}}(W)).\]
Then $E$ is de Rham.
\end{lemma}

\begin{proof}
We first claim that $H^0(\Q,V\otimes B_{\dR}^+)=H^1(\Q,V\otimes B_{\dR}^+)=0$. In fact, let $B=\bigoplus_{i=0}^\infty t^iB_{\dR}^
+$ where $t\in B_{\dR}^+$ is the usual uniformizer. Then $B/tB\cong B_{\mr{HT}}^+$. Since $V$ is de Rham with strictly negative Hodge--Tate weights, we have
\[H^0(\Q,V\otimes B_{\mr{HT}}^+)=H^1(\Q,V\otimes B_{\mr{HT}}^+)=0\]
Then tensoring the exact sequence
\[0\to B\overset{t}{\to} B\to B_{\mr{HT}}^+\to 0\]
with $V$ gives that multiplication by $t$ on $B$ induces an isomorphism $H^0(\Q,V\otimes B)\to H^0(\Q,V\otimes B)$ and an isomorphism $H^1(\Q,V\otimes B)\to H^1(\Q,V\otimes B)$. Thus, composing these isomorphisms with themselves enough times shows that these groups are zero. In particular, the summands $H^0(\Q,V\otimes B_{\dR}^+)$ and $H^1(\Q,V\otimes B_{\dR}^+)$ are zero, which proves the claim.

Now we consider the diagram
\[\xymatrix{
&0\ar[d] & 0\ar[d] & \\
0\ar[r]\ar[d] & \Fil^0 D_{\dR}(E) \ar[r]\ar[d] & \Fil^0 D_{\dR}(W) \ar[r]\ar[d] & 0\ar[d]\\
D_{\dR}(V) \ar[r]\ar[d] & D_{\dR}(E) \ar[r]^g\ar[d]^{f_E} & D_{\dR}(W) \ar[r]^{\delta}\ar[d]^{f_W} & H^1(\Q,V\otimes D_{\dR}) \ar[d]\\
(V\otimes \frac{B_{\dR}}{B_{dR}^+})^{G_{\Q_p}} \ar[r] & (E\otimes \frac{B_{\dR}}{B_{dR}^+})^{G_{\Q_p}}\ar[r]^{g'} & (W\otimes \frac{B_{\dR}}{B_{dR}^+})^{G_{\Q_p}} \ar[r]^{\delta'}\ar[d]& H^1(\Q,V\otimes \frac{B_{\dR}}{B_{dR}^+}),\\
&&0&
}\]
which is commutative and has exact rows and columns. The fact that the groups in the top corners are zero follows from the claim above, and exactness in the third column is due to $W$ being de Rham. Now by the hypothesis that $D_{\mr{dR}}(W)=g(D)\oplus \Fil^0(D_{\mr{dR}}(W))$, we see that $f_W(g(D))=(W\otimes \frac{B_{\dR}}{B_{dR}^+})^{G_{\Q_p}}$, so that $f_W\circ g$ is surjective. Thus $g'\circ f_E$ is surjective, and so therefore is $g'$. Thus $\delta'=0$, which in turn implies $\delta=0$. Since $V$ and $W$ are de Rham, so therefore is $E$.
\end{proof}

Next we state a lemma about the interpolation of crystalline periods.

\begin{lemma}[Kisin]
\label{lemkisin}
Let $\mf{W}$ be a reduced affinoid rigid analytic space and $\rho:G_{\Q_p}\to\GL_n(\mc{O}(\mf{W}))$ a continuous representation. Assume there is a Zariski dense subset $T\subset\mf{W}(\overline\Q_p)$ such that for all $x\in T$, the specialization $\rho_x$ of $\rho$ at $x$ is Hodge--Tate with Hodge--Tate weights $k_{1,x},\dotsc,k_{n,x}$, in increasing order. Assume furthermore that for any $n$, the subset $T_n$ of $x\in T$ such that $k_{i+1,x}-k_{i,x}\geq n$ for all $i$, is Zariski dense. Finally, for $n$ sufficiently large and for any $x\in \Sigma_n$, assume that $\rho_x$ is crystalline and that the eigenvalues of the crystalline Frobenius for $\rho_x$ are given by $\phi_{i}(x)p^{k_{i,x}}$ for some $\phi_i\in\mc{O}(\mf{W})$. Then for any $x\in\Sigma$,
\[D_{\crys}(\rho_x)^{\phi=\phi_1(x)p^{k_{1,x}}}\ne 0\]
More generally, for any $1\leq k\leq n$, we have
\[D_{\crys}(\wedge^k\rho_x)^{\phi=\prod_{i=1}^k\phi_i(x)p^{k_{i,x}}}\ne 0.\]
\end{lemma}

\begin{proof}
The last statement follows from the one before it, and that statement is Proposition 4.2.2 (i) in \cite{SUunitary} which, in turn, is derived from Corollary 5.15 in \cite{kisin}.
\end{proof}

We now consider the modular form $F$ of weight $k$ and level $N$ we fixed in Section \ref{subsecpif}. We make the following assumption on $F$ for the rest of the paper.

\begin{assumption}
\label{assnodblroot}
We assume that the root $\alpha_p$ of the Hecke polynomial of $F$ at $p$ is not a double root, i.e., $\alpha_p\ne p^{k-1}\alpha_p^{-1}$.
\end{assumption}

This is expected, but not known, to always be true. Under this assumption we have
\begin{equation}
\label{eqdistinctnums}
\alpha_p,\quad p^{k-1}\alpha_p^{-1},\quad p\alpha_p,\quad p^k\alpha_p^{-1}\qquad\textrm{are all distinct.}
\end{equation}
To check this, just note that $\alpha_p\ne p^k\alpha_p^{-1}$. Indeed, otherwise, we would have $\alpha_p^2=p^k$, which implies
\[\alpha_p+p^{k-1}\alpha_p^{-1}=\pm p^{(k-1)/2}(p^{1/2}+p^{-1/2}),\]
violating the temperedness of $\pi_{F,p}$.

Now let $\rho_F:G_\Q\to GL_2(\overline\Q_p)$ be the usual $p$-adic Galois representation attached to $F$. We will need to prove for later the fact that any extension $E$ of the form
\[0\to\rho_F(1)\to E\to\rho_F\to 0\]
is semistable. The argument we will use is inspired by one of Perrin-Riou in \cite{PRord}. We will calculate the dimension of the space of all such semistable extensions using filtered $(\phi,N)$-modules, and then that of the space of all such extensions, semistable or not, using Galois cohomology, and show that the results match.

We consider the associated filtered $(\phi,N)$-module of $\rho_F$, $D_{\st}(\rho_F)$. Since $\rho_F$ is crystalline at $p$, the nilpotent operator $N$ on $D_{\st}(\rho_F)$ is zero. The Frobenius operator $\phi$ is invertible and acts with eigenvalues $\alpha_p^{-1}$ and $p^{-(k-1)}\alpha_p$. The filtration $\Fil^i D_{\st}(V)$ has two steps: We have $\Fil^i D_{\st}(\rho_F)=D_{\st}(\rho_F)$ for $i\leq 0$; for $1\leq i\leq k$, we have $\Fil^i D_{\st}(\rho_F)$ is one-dimensional; and for $i>k$, we have $\Fil^i D_{\st}(\rho_F)=0$.

Let $\MF(\phi,N)$ denote the category of filtered $(\phi,N)$ modules over $\overline\Q_p$. Recall that a filtered $(\phi,N)$-module $D$ is \textit{admissible} if the Newton polygon and Hodge polygon of $D$ meet at their endpoints, and for every filtered $(\phi,N)$-submodule $D'\subset D$, the Newton polygon of $D'$ lies above its Hodge polygon. It is a theorem of Colmez and Fontaine that the admissible filtered $(\phi,N)$-modules are precisely those coming from semistable representations of $G_{\Q_p}$.

\begin{lemma}
\label{lemextMF5dim}
We have
\[\dim_{\overline\Q_p}\ext_{\MF(\phi,N)}^1(D_{\st}(\rho_F),D_{\st}(\rho_F(1)))=5.\]
\end{lemma}

\begin{proof}
Let us write $D=D_{\st}(\rho_F)$ and $D[1]=D_{\st}(\rho_F(1))$. The underlying vector spaces of $D$ and $D[1]$ can be considered the same, with the filtration of $D[1]$ equal to that of $D$ but shifted up by $1$, and the Frobenius $\phi[1]$ on $D[1]$ given in terms of the Frobenius $\phi$ on $D$ by $\phi[1](v)=p^{-1}\phi(v)$ for any $v$ in the underlying space of $D$ or $D[1]$.

Let $w\in D$ be any vector spanning $\Fil^1(D)$, and $w'$ any other vector spanning $D$ along with $w$. Write $A$ for the matrix of $\phi$ in the basis $w,w'$. Let $w[1],w'[1]$ be the corresponding basis of $D[1]$.

We now define a map
\[M_2(\overline\Q_p)\times\overline\Q_p\to\ext_{\MF(\phi,N)}^1(D,D[1]),\qquad (B,c)\mapsto E(B,c)\]
where $M_2(\overline\Q_p)$ is the space of $2$ by $2$ matrices over $\overline\Q_p$, as follows. We declare $E(B,c)$ to be the linear span of a basis of four vectors, denoted $v_1,v_2,v_3,v_4$, with Frobenius $\phi_{B,c}$ defined in this basis by
\[\phi_{B,c}=\pmat{p^{-1}A & 0 \\ 0 & A},\]
nilpotent endomorphism $N_{B,c}$ defined in this basis by
\[N_{B,c}=\pmat{0 & B\\0 & 0},\]
and filtration defined by
\begin{equation}
\label{eqfilEBc}
\Fil^i(E(B,c))=\begin{cases}
E(B,c)&\textrm{if }i\leq 0;\\
\overline\Q_p v_1+\overline\Q_p v_2+\overline\Q_p v_3&\textrm{if }i=1;\\
\overline\Q_p v_1+\overline\Q_p (cv_2+v_3)&\textrm{if }2\leq i\leq k-1;\\
\overline\Q_p v_1&\textrm{if }i=k;\\
0&\textrm{if }i\geq k+1.
\end{cases}
\end{equation}

It is easy to compute that $\phi_{B,c}^{-1}N_{B,c}\phi_{B,c}=pN_{B,c}$, and therefore $E(B,c)$ is a filtered $(\phi,N)$-module. Also, the maps
\[D[1]\to E(B,c),\qquad a_1w[1]+a_2w'[1]\mapsto a_1v_1+a_2v_2\]
and
\[E(B,c)\to D,\qquad a_1v_1+a_2v_2+a_3v_3+a_4v_4\mapsto a_3w+a_4w'\]
make $E(B,c)$ an extension of $D$ by $D[1]$. Furthermore, it is straightforward to check that the map $(B,c)\mapsto E(B,c)$ is $\overline\Q_p$-linear. Thus we just need to check that it is injective and surjective.

\textit{Injectivity}. Let $(B,c)$ and $(B',c')$ be elements of $M_2(\overline\Q_p)\times\overline\Q_p$. Assume we have a commutative diagram of filtered $(\phi,N)$-modules
\[\xymatrix{0\ar[r] & D[1] \ar[r]\ar[d]^{=} & E(B,c) \ar[r]\ar[d]^{\sim} & D \ar[r]\ar[d]^{=} & 0\\
0 \ar[r] & D[1] \ar[r] & E(B',c') \ar[r] & D \ar[r] & 0}\]
with exact rows. Call the middle vertical map $\psi$. Let $v_1,v_2,v_3,v_4$ be the basis for $E(B,c)$ as above, and let $v_1',v_2',v_3',v_4'$ be a similar basis for $E(B',c')$. Then by the diagram,
\[\psi(v_1)=v_1',\quad \psi(v_2)=v_2',\]
and there is a matrix $M_\psi\in M_2(\overline\Q_p)$, say
\[M_\psi=\pmat{m_{11} & m_{12}\\ m_{21} & m_{22}},\]
such that
\[\psi(v_3)=m_{11}v_1'+m_{12}v_2'+v_3',\]
and
\[\psi(v_4)=m_{21}v_1'+m_{22}v_2'+v_4'.\]
Because $\psi$ must preserve $\Fil^1$, this implies $m_{21}=m_{22}=0$. Also, becuase it must preserve $\Fil^2$, we have $m_{12}=0$, and because it must preserve $\Fil^{k-1}$, we have
\[\psi(cv_2+v_3)=cv_2'+m_{21}v_2'+v_3'\in\overline\Q_p(c'v_2'+v_3'),\]
which implies $m_{11}=c'-c$.

Now equivariance for the action of the Frobenius operator gives
\[\psi(Av_3)=p^{-1}Am_{11}v_2'+Av_3'\in\overline\Q_p v_3+\overline\Q_p v_4.\]
Thus since $A$ is invertible, we have $m_{11}=0$ and hence $c=c'$. Therefore $\psi(v_i)=v_i'$ for $i=1,2,3,4$, and so $N$-equivariance gives also $B=B'$. Thus $(B,c)\mapsto E(B',c')$ is injective.

\textit{Surjectivity}. Consider an extension $E$ in $\ext_{\MF(\phi,N)}^1(D,D[1])$, so $E$ sits in exact sequence
\[0\to D[1]\to E\to D\to 0.\]
Let $w_1,w_2$ be, respectively, the images of $w[1],w'[1]$ in $E$. Let $w_3$ be a vector in $E$ mapping to the vector $w$ in $D$, and similarly let $w_4$ be a vector in $E$ mapping to $w'$. Then $w_1,w_2,w_3,w_4$ is a basis for $E$. Let $\phi_E$ be the Frobenius for $E$. Then by construction there is a matrix $M$ such that, in the basis $\{w_1,w_2,w_3,w_4\}$, the operator $\phi_E$ is given by the block upper triangular matrix
\[\pmat{p^{-1}A & M\\ 0 & A}.\]
By \eqref{eqdistinctnums}, all the eigenvalues of the above matrix are distinct. Therefore there is a unique choice of $w_3,w_4$ as above such that
\[\phi_E=\pmat{p^{-1}A & 0 \\ 0 & A},\]
and we assume $w_3$ and $w_4$ are chosen this way.

Then, since the nilpotent operators for $D[1]$ and $D$ are zero (as $V$ is crystalline) we must have,
\[N_E=\pmat{0 & B\\ 0 & 0},\]
for some $B\in M_2(\Q_p)$. By compatibility of the filtration of $E$ with those of $D[1]$ and $D$, it is not too hard to see that the filtration on $E$ must satisfy
\[\Fil^i(E)=\begin{cases}
E(B,c)&\textrm{if }i\leq 0;\\
\overline\Q_p v_1+\overline\Q_p v_2+\overline\Q_p v_3&\textrm{if }i=1;\\
\overline\Q_p v_1&\textrm{if }i=k;\\
0&\textrm{if }i\geq k+1.
\end{cases}\]
But since
\[\Fil^1(E)=\overline\Q_p v_1+\overline\Q_p v_2+\overline\Q_p v_3,\]
we know that there are $c,d\in\Q_p$ such that
\[\Fil^i(E)=\overline\Q_p v_1+\overline\Q_p (dv_1+cv_2+v_3)\quad\textrm{if }2\leq i\leq k-1.\]
But of course, this just equals
\[\overline\Q_p v_1+\overline\Q_p (cv_2+v_3).\]
This proves that $E\cong E(B,c)$, and the proof of surjectivity is complete.
\end{proof}

\begin{lemma}
\label{lemextgal5dim}
We have
\[\dim_{\overline\Q_p}\ext_{G_{\Q_p}}^1(\rho_F,\rho_F(1))=5.\]
\end{lemma}

\begin{proof}
First we note
\[\ext_{G_{\Q_p}}^1(\rho_F,\rho_F(1))\cong H^1(\Q_p,\rho_F^\vee\otimes\rho_F(1)).\]
The group on the right and side breaks up as
\[H^1(\Q_p,\Ad^2\rho_F(1)\oplus\overline\Q_p(1))\cong H^1(\Q_p,\Ad^2\rho_F(1))\oplus H^1(\Q_p,\overline\Q_p(1)).\]
The piece $H^1(\Q_p,\overline\Q_p(1))$ is $2$-dimensional by a standard computation in Kummer theory, and the piece $H^1(\Q_p,\Ad^2\rho_F(1))$ is $3$ dimensional by a now classical computation using local Tate duality.
\end{proof}

We combine the previous two Lemmas to get the following result.

\begin{proposition}
\label{propvv1st}
Any extension $E$ of Galois representations,
\[0\to \rho_F(1)\to E\to \rho_F\to 0\]
is semistable at $p$.
\end{proposition}

\begin{proof}
We first recall that any extension of admissible filtered $(\phi,N)$-modules is again admissible. So, because of the equivalence of categories between semistable representations and admissible filtered $(\phi,N)$-modules, we get an injection
\[\ext_{\MF(\phi,N)}^1(D_{\st}(\rho_F),D_{\st}(\rho_F(1)))\hookrightarrow\ext_{G_{\Q_p}}^1(\rho_F,\rho_F(1)),\]
whose image is the group of semistable extensions of $\rho_F$ by $\rho_F(1)$. By Lemmas \ref{lemextMF5dim} and \ref{lemextgal5dim}, both the source and target have dimension $5$, so this injection is an isomorphism. The proposition follows.
\end{proof}

\subsection{Construction of a pseudocharacter}
\label{subsecpsch}
The combination of Theorems \ref{thmpadicdef} and \ref{thmgalreps} gives us a family of Galois representations into $G_2(\overline{\Q}_p)$ which should interpolate over an affinoid rigid space. The now classical theory of pseudorepresentations would allow us to show that the corresponding family of Galois representations into $GL_7(\overline\Q_p)$ interpolates in this way, but it will not allow us to show that the interpolated family factors through $G_2$. However, there is a tool which will allow us to do this, namely the theory of pseudocharacters, as was introduced by V. Lafforgue \cite{laff}.

We will now recall some of the theory of pseudocharacters, and then construct a $G_2$-pseudocharacter of $G_\Q$ which interpolates the aforementioned Galois representations. We will follow B\"ockle--Harris--Khare--Thorne \cite{BHKT} in our exposition. We begin with the definition of pseudocharacter.

\begin{definition}
Let $G$ be a split reductive group over $\Z$, $A$ a ring and $\Gamma$ a group. Let $G$ act on itself by conjugation, and let $\Z[G]$ be the ring of regular functions on $G$, on which $G$ therefore acts as well. Let $\Z[G]^G$ be the subring of invariants for this action. Then a $G$-\textit{pseudocharacter} of $\Gamma$ over $A$ is a collection $\Theta$ of ring maps
\[\Theta_n:\Z[G^n]^{G}\to\mr{Fun}(\Gamma^n,A),\]
where $\mr{Fun}(\Gamma,A)$ is the ring of $A$-valued functions on $\Gamma$, such that the maps $\Theta_n$ satisfy the following properties:
\begin{enumerate}[label=(\arabic*)]
\item Given any positive integers $m,n$, any function $\zeta:\{1,\dotsc,m\}\to\{1,\dotsc,n\}$, any $f\in\Z[G^m]^G$, and any $\gamma_1,\dotsc,\gamma_n\in\Gamma$, we have
\[\Theta_n(f^\zeta)(\gamma_1,\dotsc,\gamma_n)=\Theta_m(f)(\gamma_{\zeta(1)},\dotsc,\gamma_{\zeta(m)}),\]
where $f^\zeta$ is defined by
\[f^\zeta(g_1,\dotsc,g_n)=f(g_{\zeta(1)},\dotsc,g_{\zeta(m)})\]
for all $g_1,\dotsc,g_n\in\Gamma$;
\item Given any positive integer $n$, any $\gamma_1,\dotsc,\gamma_{n+1}\in\Gamma$, and any $f\in\Z[G^n]^G$, we have
\[\Theta_{n+1}(\hat{f})(\gamma_1,\dotsc,\gamma_{n+1})=\Theta_n(f)(\gamma_1,\dotsc,\gamma_{n-1},\gamma_n\gamma_{n+1}),\]
where $\hat{f}$ is defined by
\[\hat{f}(g_1,\dotsc,g_{n+1})=f(g_1,\dotsc,g_{n-1},g_n g_{n+1}),\]
for all $g_1,\dotsc g_{n+1}\in\Gamma$.
\end{enumerate}

If $\Gamma$ and $A$ have topologies, then we say $\Theta$ is \textit{continuous} if for any $f\in\Z[G]^G$ and any $n$, the map $\Theta_n(f)$ is a continuous function $\Gamma^n\to A$.
\end{definition}

As noted in \cite[Lemma 4.3]{BHKT}, a representation $\rho:\Gamma\to G(A)$ gives a $G$-pseudocharacter of $\Gamma$ over $A$, which is denoted $\tr\rho$ and is defined by
\[(\tr\rho)_n(f)=f(\rho(\gamma_1),\dotsc,\rho(\gamma_n)).\]
The pseudocharacter $\tr\rho$ only depends on $\rho$ up to conjugation in $G(A)$. We also note that when $G=\GL_n$, this recovers the notion of pseudorepresentation of Taylor \cite{taylor}; the function $\Theta_1(\tr)$, where $\tr\in\Z[\GL_n]^{\GL_n}$ denotes the usual trace, will be a pseudorepresentation in this case, and the rest of the pseudocharacter will be determined by this.

Also, as noted in \cite[Lemma 4.4]{BHKT}, is that we can change the ring. More precisely, if $\Theta$ is a $G$-pseudocharacter over a ring $A$ and $\phi:A\to B$ is a ring homomorphism, then $\phi_*\Theta$, defined by
\[(\phi_*\Theta)_n(f)=\phi\circ\Theta_n(f),\qquad f\in\Z[G]^G,\]
is a $G$-pseudocharacter of $\Gamma$ over $B$.

For such $\Theta$, we can also change the group; if $\psi:\Gamma'\to\Gamma$ is a group homomorphism, then $\psi^*\Theta$, defined by
\[(\psi^*\Theta)_n(f)=\Theta_n(f)\circ\psi,\qquad f\in\Z[G]^G,\]
is a $G$-pseudocharacter of $\Gamma'$ over $A$. The changes of groups and rings just described are also compatible with continuity as long as the maps $\phi$ and $\psi$ as above are continuous.

Finally, if $A=k$ is an algebraically closed field, then given a $G$-pseudocharacter $\Theta$ of $\Gamma$ over $k$, there is a unique \textit{completely reducible} (See \cite[Definitions 3.3, 3.5]{BHKT}) representation $\rho:\Gamma\to G(k)$ such that $\Theta=\tr\rho$. This is \cite[Theorem 4.5]{BHKT}.

We now use this setup to construct a $G_2$-pseudocharacter of $G_\Q$ over an affinoid $\Q_p$-algebra. Recall from Section \ref{subsecpif} the eigenform $F$ of weight $k$ and level $N$, and that we have the $\mc{H}_p$-representation $\Pi_F^{(p)}$ which is a $p$ stabilization of the automorphic representation $\mc{L}_\alpha(\pi_F,1/10)$. Then Theorem \ref{thmpadicdef} gives, under Conjecture \ref{conjmult} (b), a $p$-adic family of automorphic representations parametrized by an affinoid rigid space $\mf{V}$. Let $\mf{V}'$ be the nilreduction of the irreducible component of $\mf{V}$ containing $y_0$.

Let $\mc{O}(\mf{V}')^\circ$ be the subring of functions in $\mc{O}(\mf{V}')$ whose evaluations at every point have $p$-adic absolute value bounded above by $1$. This ring defines a formal scheme whose rigid generic fiber is $\mf{V}'$. It follows from Noether normalization for $\mc{O}(\mf{V}')$ that $\mc{O}(\mf{V}')^\circ$ is finite over a power series ring over $\Z_p$ in two variables. The ring $\mc{O}(\mf{V}')^\circ$ is therefore profinite. We will use this in the proof of the following proposition.

\begin{proposition}
\label{propTheta}
Assume Conjectures \ref{conjmult} (b) and \ref{conjliftings} (a)-(c). Then, in the setting of Theorem \ref{thmpadicdef}, there is a continuous $G_2$-pseudocharacter $\Theta^\circ$ of $G_\Q$ over $\mc{O}(\mf{V}')^\circ$ satisfying the following properties:

Let $\Theta$ be the pseudocharacter of $G_\Q$ obtained from $\Theta^\circ$ by changing the ring from $\mc{O}(\mf{V}')^\circ$ to $\mc{O}(\mf{V}')$. Then
\begin{enumerate}[label=(\alph*)]
\item The pseudocharacter $\Theta$ is continuous;
\item The pseudocharacters $\Theta^\circ$ and $\Theta$ are unramified at all finite primes $\ell\nmid Np$. That is, the restrictions of $\Theta^\circ$ and $\Theta$ to the inertia group $I_\ell$ at $\ell$ are the $G_2$-pseudocharacters attached to the trivial representation of $I_\ell$;
\item Let $j_\beta$ be the inclusion of $M_\beta$ into $G_2$, and view $\rho_F$ as a representation of $G_\Q$ into $M_\beta(\overline\Q_p)$. Then the pseudocharacter $\Theta_{y_0}$ obtained from $\Theta$ by changing the ring via the point $y_0\in\mf{V}(\overline\Q_p)$ is given by
\[\Theta_{y_0}=\tr(j_\beta\circ\rho_F(-\tfrac{k-2}{2}));\]
\item For any $y\in\Sigma$ such that $\Pi_y$ is a singleton, let $\sigma_y$ be the element of $\Pi_y$ and $\Pi(\sigma_y)$ be the automorphic representation of $G_2(\A)$ of which $\sigma_y$ is a $p$-stabilization. Let $\rho_y$ be the Galois representation obtained from $\Pi(\sigma_y)$ via Theorem \ref{thmgalreps}. Then the pseudocharacter $\Theta_y$ obtained from $\Theta$ by changing the ring via $y$ is given by
\[\Theta_y=\tr(\rho_y).\]
\end{enumerate}
\end{proposition}

\begin{proof}
We first define $\Theta$ and $\Theta^\circ$ on Frobenius elements, and then show that $\Theta$ is defined over $\mc{O}(\mf{V}')^\circ$. Then we explain how to extend the definitions of $\Theta$ and $\Theta^\circ$ to all of $G_\Q$.

Consider the character $\theta_{\mf{V}}$ of $R_{S(N),p}$ from Theorem \ref{thmpadicdef}. Let $\theta_{\mf{V}'}$ be the composition of $\theta_{\mf{V}}$ with the natural map $\mc{O}(\mf{V})\to\mc{O}(\mf{V}')$. If $\ell\nmid Np$, then $\theta_{\mf{V}'}$ gives rise to a parameter $s_{\mf{V}',\ell}\in T(F(\mf{V}'))/W$, where $F(\mf{V}')$ is the field of fractions of $\mc{O}(\mf{V}')$ and $W$ is the Weyl group of $G_2$.

Let $n\geq 1$ be an integer, and let $\ell_1,\dotsc,\ell_n$ be primes not dividing $Np$. Let $f\in\Z[G_2]^{G_2}$. Then we define
\[\Theta_n(f)(\Frob_{\ell_1}^{-1},\dotsc,\Frob_{\ell_n}^{-1})=\Theta_n^\circ(f)(\Frob_{\ell_1}^{-1},\dotsc,\Frob_{\ell_n}^{-1})=f(s_{\mf{V}',\ell_1},\dotsc,s_{\mf{V}',\ell_n}).\]
A priori, this is an element of $F(\mf{V}')$, but we will show in fact that it lies in $\mc{O}(\mf{V}')^\circ$.

To see this, let $f_1,\dotsc,f_r\in\mc{O}(\mf{V}')$ be such that $s_{\mf{V},\ell_1},\dotsc,s_{\mf{V}',\ell_n}\in\mc{O}(\mf{V}')[f_1^{-1},\dotsc,f_r^{-1}]$. Let $\Sigma'\subset\Sigma$ be the set of all $y\in\Sigma$ such that $\Pi_y$ is a singleton and such that $y$ is not in the divisor of any of the $f_i$'s. Then the specialization of $\Theta_n(f)(\Frob_{\ell_1}^{-1},\dotsc,\Frob_{\ell_n}^{-1})$ at $y$ for $y\in\Sigma'$ makes sense, and we have
\[\Theta_n(f)(\Frob_{\ell_1}^{-1},\dotsc,\Frob_{\ell_n}^{-1})=f(\rho_y(\Frob_{\ell_1}^{-1}),\dotsc,\rho_y(\Frob_{\ell_n}^{-1})),\]
by construction.

Now a standard argument using the Baire category theorem implies that there is a finite extension $\mc{O}_E$ of $\Q_p$ such that, up to conjugation by an element of $G_2(\overline\Q_p)$, we have that $\rho_y$ takes values in $G_2(\mc{O}_E)$. Therefore, the specialization at $y$,
\[y(\Theta_n(f)(\Frob_{\ell_1}^{-1},\dotsc,\Frob_{\ell_n}^{-1})),\]
is in $\mc{O}_E$. By the density of $\Sigma$, this implies then that
\[\Theta_n(f)(\Frob_{\ell_1}^{-1},\dotsc,\Frob_{\ell_n}^{-1})\in\mc{O}(\mf{V}')^\circ,\]
as desired.

The rest of the proof is more or less standard. Let $g_1,\dotsc,g_n$ be elements of $G_\Q$, and for each $i$ let $\ell_{i,j}$, $j>0$, be primes such that $\Frob_{\ell_{i,j}}\to g_i$ as $j\to\infty$. The ring $\mc{O}(\mf{V}')^\circ$ is profinite, so we can define $\Theta_n^\circ(f)(g_1,\dotsc,g_n)$ by defining it on the finite quotients of $\mc{O}(\mf{V}')^\circ$ using the definition above for Frobenius elements and taking the limit as $j\to\infty$. By continuity, the resulting assignment $f\mapsto\Theta_n(f)$ is a pseudocharacter, as $f\mapsto y(\Theta_n(f))$ is the one attached to $\rho_y$. Then by construction, the pseudocharacters $\Theta$ and $\Theta^\circ$ must satisfy the properties claimed in the proposition. (We note for point (c) that the Galois representation $j_\beta\circ\rho_F(-\tfrac{k-2}{2})$ is the one attached to $\mc{L}_\alpha(\pi_F,1/10)$.)
\end{proof}

We now restrict the pseudocharacter $\Theta$ constructed in this proposition to a curve. Let $c=(c_1,c_2)$ be a pair of positive integers with $c_1\ne c_2$. With $\mf{U}$ as in Theorem \ref{thmpadicdef}, let $\mf{L}=\mf{L}_c$ be the Zariski closure in $\mf{U}$ of the set of weights $\lambda\in\mbf{w}(\Sigma)$ of the form
\begin{equation}
\label{eqweightsinL}
\lambda=\lambda_0+n(c_1(2\alpha+3\beta)+c_2(\alpha+2\beta)),\qquad n\in\Z_{>0},
\end{equation}
where $\lambda_0=\tfrac{k-4}{2}(2\alpha+3\beta)$ as before. We can and do choose $c$ so that the intersection $\mf{L}(\overline\Q_p)\cap\mbf{w}(\Sigma)$ is infinite, so that $c_1\ne c_2$, and so that the Zariski closed subset of Theorem \ref{thmpadicdef} (2) intersects $\mf{L}$ at only finitely many points.

Then $\mf{L}$ is a line in $\mf{U}$ containing $\lambda_0$. Let $\mf{Z}'$ be the curve in $\mf{V}'$ lying over $\mf{L}$, and $\mf{Z}$ the nilreduction of irreducible component of $\mf{Z}$ containing $y_0$. We also let $\Sigma_c=\mf{Z}(\overline\Q_p)\cap\Sigma$. Changing the ring of $\Theta$ to $\mf{Z}$ gives us a continuous pseudocharacter which we denote $\Theta_\mf{Z}$.

Let $R_7$ be the seven dimensional representation of $G_2$. Then $R_7$ induces a ring map $\Z[GL_7]^{GL_7}\to\Z[G_2]^{G_2}$. Composing $\Theta_\mf{Z}$ with this map gives us a $GL_7$-pseudocharacter of $G_\Q$ over $\mc{O}(\mf{Z})$, and hence a pseudorepresentation of $G_\Q$ into $\mc{O}(\mf{Z})$ of dimension $7$ which we denote by $T$.

We now make the following assumption which we keep for the rest of the main body of this paper.

\begin{assumption}
\label{assfnotcm}
From here on, we assume $F$ is not CM.
\end{assumption}

Under this assumption, by results of Momose \cite{momose} and Loeffler \cite{loeimage}, there is a form $G$ of $GL_{2/\Q}$ such that the image of $\rho_F$ in $GL_2(\overline\Q_p)$ can be conjugated to be an open subgroup of $G(\Q_p)$. In particular, the image of $\rho_F$ is big enough for $\Ad(\rho_F)$ to be irreducible.

\begin{lemma}
\label{lemTsumoftwo}
Let $T$ be as above. Let $\tilde{\mf{Z}}$ be any reduced, irreducible, affinoid curve which is finite over a Zariski open neighborhood of $y_0$. Then, over $\tilde{\mf{Z}}$, the pseudorepresentation $T$ is either irreducible, or it is the sum of two pseudorepresentations, say $T=T_1+T_2$. In this latter case, let $T_{1,x_0}$ and $T_{2,x_0}$ denote, respectively, the specializations of $T_1$ and $T_2$ at a point $x_0\in\tilde{\mf{Z}}(\overline\Q_p)$ above $y_0$. Then (up to swapping $T_1$ and $T_2$) we have
\[T_{1,x_0}=\tr(\Ad\rho_F),\qquad T_{2,x_0}=\tr(\rho_F(-\tfrac{k-2}{2}))+\tr(\rho_F(-\tfrac{k}{2})).\]
\end{lemma}

\begin{remark}
We will eventually show that the second possibility in the conclusion of this lemma cannot hold. This requires studying a lattice in the Galois representation attached to $T$, as in the next section, and seems to be genuinely more difficult than the techniques used to prove this lemma alone.
\end{remark}

\begin{proof}[Proof (of Lemma \ref{lemTsumoftwo})]
Let $x_0\in\tilde{\mf{Z}}(\overline\Q_p)$ a point above $y_0$. For any $x\in\tilde{\mf{Z}}(\overline\Q_p)$, let $T_x$ be the specialization of $T$ at $x$. Then by Proposition \ref{propTheta} (c) and \eqref{eqr7beta}, we have
\[T_{x_0}=\tr(\rho_F(-\tfrac{k-2}{2}))+\tr(\Ad(\rho_F))+\tr(\rho_F(-\tfrac{k}{2})).\]
These three pseudorepresentations are irreducible by above.

Assume that $T$ is reducible over $\widetilde{\mf{Z}}$. Then we therefore have four cases, up to swapping factors. Namely:

\textit{Case 1.} Over $\tilde{\mf{Z}}$, we have $T=T_1+T_2+T_3$, where
\[T_{1,x_0}=\tr(\Ad(\rho_F)),\qquad T_{2,x_0}=\tr(\rho_F(-\tfrac{k-2}{2})),\qquad T_{3,x_0}=\tr(\rho_F(-\tfrac{k}{2}).\]

\textit{Case 2.} Over $\tilde{\mf{Z}}$, we have $T=T_1+T_2$, where
\[T_{1,x_0}=\tr(\Ad(\rho_F))+\tr(\rho_F(-\tfrac{k}{2}),\qquad T_{2,x_0}=\tr(\rho_F(-\tfrac{k-2}{2})).\]

\textit{Case 3.} Over $\tilde{\mf{Z}}$, we have $T=T_1+T_2$, where
\[T_{1,x_0}=\tr(\Ad(\rho_F))+\tr(\rho_F(-\tfrac{k-2}{2}),\qquad T_{2,x_0}=\tr(\rho_F(-\tfrac{k}{2})).\]

\textit{Case 4.} Over $\tilde{\mf{Z}}$, we have $T=T_1+T_2$, where
\[T_{1,x_0}=\tr(\Ad(\rho_F)),\qquad T_{2,x_0}=\tr(\rho_F(-\tfrac{k-2}{2})+\tr(\rho_F(-\tfrac{k}{2})).\]
We have to show Cases 1-3 cannot hold.

But first, let $\widetilde{\Sigma}_c$ be the preimage of $\Sigma_c$ in $\tilde{\mf{Z}}$. Let $x\in\widetilde{\Sigma}_c$ be a point, and $y\in\Sigma_c$ the point over which $x$ lies. There there is a positive integer $n$ so that $\mbf{w}(y)$ is given by \eqref{eqweightsinL}. So by Theorem \ref{thmgalreps} (c) we have that (the Galois representation attached to) $T_x$ has Hodge--Tate weights
\[\pm (k-1+2nc_1+nc_2),\qquad\pm (\tfrac{k}{2}+nc_1+nc_2),\qquad\pm (\tfrac{k-2}{2}+nc_1),\qquad 0.\]
If $x$ is close to $x_0$, then $y$ corresponds to a $p$-stabilization of an automorphic representation which is unramified at $p$ by Proposition \ref{propunratp}. By the continuity of the operators in $\mc{U}_p$, the slopes of $T_x$ are therefore given, as in Theorem \ref{thmgalreps} (c), by
\begin{equation}
\label{eqslopesofT}
\pm(2s_p-(k-1+2nc_1+nc_2)),\qquad\pm(s_p-(\tfrac{k}{2}+nc_1+nc_2)),\qquad\pm(s_p-(\tfrac{k-2}{2}+nc_1)),\qquad 0.
\end{equation}
It is worthwhile to note that these seven numbers are those given by
\[\langle s_p(2\alpha+3\beta)+\beta-\lambda-\rho,\mu^\vee\rangle,\]
where $\lambda$ is the weight of $y$ and $\mu^\vee$ runs over the seven coweights which support the character of the seven dimensional representation of the dual group $G_2^\vee$; the weight $s_p(2\alpha+3\beta)+\beta$ is the slope of $\Pi_F^{(p)}$. We also note that $T_{x_0}$ has Hodge--Tate weights
\[\pm (k-1),\qquad\pm\tfrac{k}{2},\qquad\pm\tfrac{k-2}{2},\qquad 0.\]

Now assume we are in Case 1 above. Then by looking at the Hodge--Tate--Sen weights of the semisimple Galois representation attached to $T$, we see that if $x\in\widetilde\Sigma_c$ is sufficiently close to $x_0$, then we have that $T_{3,x}$ has Hodge--Tate weights
\[\tfrac{k}{2}+nc_1+nc_2\qquad\textrm{and}\qquad -(\tfrac{k-2}{2}+nc_1).\]
Its slopes must be among those in the list \eqref{eqslopesofT}. But if $n$ is sufficiently large, a case-by-case check reveals that this makes it impossible for the Newton polygon and the Hodge polygon for $T_3$ to meet at their endpoints. Therefore we have ruled out Case 1.

A completely similar analysis of $T_2$ rules out both Case 2 and Case 3. Therefore the proof is complete.
\end{proof}

\subsection{Construction of a lattice}
\label{subseclattice}
The goal of this subsection is to prove the following proposition; after this, the main theorem will more or less follow from (albeit perhaps complicated) linear algebra.

\begin{proposition}
\label{propmcLbar}
Assume Conjectures \ref{conjmult} (a) and \ref{conjliftings}. Then there is a $7$-dimensional $\overline\Q_p$-vector space, which we call $\overline{\mc{L}}$, with a basis $v_1,\dotsc,v_7$, a continuous $G_\Q$-action $\rho_{\overline{\mc{L}}}$, and an alternating trilinear form $\langle\cdot,\cdot,\cdot\rangle$ which is preserved by $\rho_{\overline{\mc{L}}}$, such that the following properties are satisfied.
\begin{enumerate}[label=(\alph*)]
\item For $1\leq i,j\leq 7$ and $g\in G_\Q$, let $g_{ij}\in\overline\Q_p$ be the $(i,j)$-coefficient of the matrix by which $\rho_{\overline{\mc{L}}}$ acts on $\overline{\mc{L}}$ in the basis $v_1,\dotsc,v_7$. Then the map $\rho_3$ defined by
\[g\mapsto\pmat{g_{66}& g_{67}\\ g_{76}& g_{77}}\]
is a representation $\rho_3:G_\Q\to GL_2(\overline\Q_p)$ conjugate to $\rho_F(-\tfrac{k}{2})$.
\item Let $d:G_\Q\to\overline\Q_p^\times$ be defined by
\[d(g)=\det\rho_3(g)=g_{66}g_{77}-g_{67}g_{76}.\]
Define two maps $\rho_1,\rho_2$ by
\[\rho_1(g)=\frac{1}{d(g)}\pmat{g_{66}&-g_{67}\\ -g_{76}&g_{77}},\qquad\rho_2(g)=\frac{1}{d(g)}\pmat{g_{66}^2 & 2g_{66}g_{67} & -g_{67}^2\\
g_{66}g_{76} & g_{66}g_{77}+g_{67}g_{76} & -g_{67}g_{77}\\
-g_{76}^2 & -2g_{76}g_{77} & g_{77}^2}.\]
Then $\rho_1:G_\Q\to\GL_2(\overline\Q_p)$ and $\rho_2:G_\Q\to\GL_3(\overline\Q_p)$ are representations conjugate, respectively, to $\rho_F(-\tfrac{k-2}{2})$ and $\Ad(\rho_F)$. Moreover, in the basis $v_1,\dotsc,v_7$, the representation $\rho_{\overline{\mc{L}}}$ has one of the following three forms: We either have
\begin{equation}
\label{eqLbar12}
\rho_{\overline{\mc{L}}}=\pmat{\rho_1 & *_3 & *_2\\
0 & \rho_2 & *_1\\
0 & 0 & \rho_3},
\end{equation}
with $*_1$ and $*_2$ nontrivial, or
\begin{equation}
\label{eqLbar22}
\rho_{\overline{\mc{L}}}=\pmat{\rho_1 & 0 & *_2\\
*_3 & \rho_2 & *_1\\
0 & 0 & \rho_3},
\end{equation}
again with $*_1$ and $*_2$ nontrivial, or
\begin{equation}
\label{eqLbar32}
\rho_{\overline{\mc{L}}}=\pmat{\rho_1 & 0 & *_2\\
0 & \rho_2 & 0\\
0 & 0 & \rho_3},
\end{equation}
with $*_2$ nontrivial.
\item Let $v_1^\vee,\dotsc,v_7^\vee\in\overline{\mc{L}}^\vee$ be the basis dual to $v_1,\dotsc,v_7$. Then there are numbers
\[a_{147},a_{156},a_{237},a_{246},a_{345}\in\overline\Q_p,\]
not all zero, such that
\begin{multline*}
\langle\cdot,\cdot,\cdot\rangle=a_{147}(v_1^\vee\wedge v_4^\vee\wedge v_7^\vee)+a_{156}(v_1^\vee\wedge v_5^\vee\wedge v_6^\vee)+a_{237}(v_2^\vee\wedge v_3^\vee\wedge v_7^\vee)\\
+a_{246}(v_2^\vee\wedge v_4^\vee\wedge v_6^\vee)+a_{345}(v_3^\vee\wedge v_4^\vee\wedge v_5^\vee).
\end{multline*}
Furthermore, if $\rho_{\overline{\mc{L}}}$ has the form \eqref{eqLbar32}, then one of $a_{147},a_{156},a_{237},a_{246}$ is nonzero.
\item For any prime $\ell\nmid Np$, the representation $\rho_{\overline{\mc{L}}}$ is unramified at $\ell$.
\item Let $\ell$ be a prime with $\ell|N$, and assume that if $4|N$ (respectively, if $9|N$) then $\ell\ne 2$ (respectively, $\ell\ne 3$). Then there is an finite index open subgroup $U$ in the inertia group $I_\ell$ at $\ell$ such that, for any of the functions $*_1,*_2,*_3$ from any of the formulas \eqref{eqLbar12}, \eqref{eqLbar22}, \eqref{eqLbar32} above, we have
\[*_1|_U=*_2|_U=*_3|_U=0,\]
identically.
\item Write $b_1,\dotsc,b_7$ for the numbers
\[p^{-(k-1)}\alpha_p^2,\,\,p^{-(k-2)/2}\alpha_p,\,\,p^{-k/2}\alpha_p,\,\,1,\,\,p^{k/2}\alpha_p^{-1},\,\,p^{(k-2)/2}\alpha_p^{-1},\,\,p^{k-1}\alpha_p^{-2},\]
respectively. Then for any $i$ with $1\leq i\leq 7$, we have
\[D_{\crys}(\wedge^i\overline{\mc{L}})^{\phi=\prod_{j=1}^i b_j}\ne 0.\]
\end{enumerate}
\end{proposition}

\begin{proof}
The proof will proceed in several steps; we will use the pseudorepresentation $T$ and the pseudocharacter $\Theta_{\mf{Z}}$ from the previous subsection to construct a lattice $\mc{L}$ in a $7$-dimensional Galois representation, factoring through $G_2$, over the field of meromorphic functions on an affinoid domain, whose specialization at a certain special point will be $\overline{\mc{L}}$ (whence the notation). We will then pass to the completed local ring $\mc{A}$ at this special point, and we will construct a basis $\tilde{v}_1,\dotsc,\tilde{v}_7$ in $\mc{L}\otimes\mc{A}$ which reduces modulo the maximal ideal to the basis $v_1,\dotsc,v_7$ from the proposition. Then we will construct the form $\langle\cdot,\cdot,\cdot\rangle$ using the fact that $\mc{L}$ factors through $G_2$. Finally, we will proceed to prove each of the points (a)-(f).

\textit{Step 1. Auxiliary constructions: }$\mc{P}$, $\mc{A}$,\textit{ and }$g_0$. Let $\Theta_\mf{Z}$ be the $G_2$-pseudocharacter of $G_\Q$ over $\mc{O}(\mf{Z})$ constructed in the previous subsection. Let $F(\mf{Z})$ be the fraction field of $\mc{O}(\mf{Z})$ and let $\rho_\Theta:G_\Q\to G_2(\overline{F(\mf{Z})})$ be the representation associated with $\Theta_\mf{Z}$ by \cite[Theorem 4.5]{BHKT}. Then $R_7\circ\rho_\Theta$ is the semisimple Galois representation associated with $T$. Since $T$ takes values in $\mc{O}(\mf{Z})$, standard arguments in the theory of pseudorepresentations show that there is a reduced, irreducible, affinoid curve $\tilde{\mf{Z}}$ which is finite over a Zariski open neighborhood of $y_0$ in $\mf{Z}$ and such that $R_7\circ\rho_{\Theta}$ takes values in $GL_7(\mc{O}(\tilde{\mf{Z}}))$. By replacing $\mf{Z}$ with $\tilde{\mf{Z}}$ and replacing $y_0$ with a chosen point above it in $\tilde{\mf{Z}}$, we may assume $\tilde{\mf{Z}}=\mf{Z}$.

We then pass to the normalization of $\mf{Z}$ and therefore may assume $\mc{O}(\mf{Z})$ is Dedekind. The continuity of $T$ then implies that $R_7\circ\rho_\Theta$ preserves a projective $\mc{O}(\mf{Z})$-module $\mc{P}$ of rank $7$ in $F(\mf{Z})^7$.

We let $\mc{A}$ be the completed local ring of $\mc{O}(\mf{Z})$ at the point $y_0$. Then $\mc{A}$ is a discrete valuation ring since $\mc{O}(\mf{Z})$ is Dedekind.

Now choose any element $g_0\in G_\Q$ with $\rho_F(-\tfrac{k}{2})(g_0)$ having eigenvalues $\gamma_6,\gamma_7\in\overline\Q_p$ such that the numbers
\begin{equation}
\label{eqgammai}
\gamma_1=\gamma_7^{-1},\quad\gamma_2=\gamma_6^{-1},\quad\gamma_3=\gamma_6\gamma_7^{-1},\quad\gamma_4=1,\quad\gamma_5=\gamma_7\gamma_6^{-1},\quad\gamma_6,\quad\gamma_7
\end{equation}
are all distinct; this is possible by Assumption \ref{assfnotcm}.

\textit{Step 2. Construction of $\mc{L}$ when $T$ is absolutely irreducible.} Assume first that the pseudorepresentation $T$ is irreducible over any reduced, irreducible, affinoid curve $\tilde{\mf{Z}}$ which is finite over a Zariski open neighborhood of $y_0$ in $\mf{Z}$. Then by Proposition \ref{propTheta} (c) and \eqref{eqr7beta}, we have
\[\mc{P}_{y_0}^{\sss}\cong\rho_F(-\tfrac{k-2}{2})\oplus\Ad(\rho_F)\oplus\rho_F(-\tfrac{k}{2}),\]
where $\mc{P}_{y_0}$ is the specialization of $\mc{P}$ at $y_0$ and $\mc{P}_{y_0}^{\sss}$ is its semisimplification as a Galois representation. Therefore there is a vector $w_7\in\mc{P}_{y_0}$ which is an eigenvector for $g_0$ acting on $\mc{P}_{y_0}$ with eigenvalue $\gamma_7$. Let $\widetilde{w}_7\in\mc{P}$ be any vector mapping to $w_7$ in $\mc{P}_{y_0}$. Then we let
\[\mc{L}=\mc{O}(\mf{Z})[G_\Q]\widetilde{w}_7,\]
which is another projective module over $\mc{O}(\mf{Z})$, and it is of rank $7$ by irreducibility of $T$.

\textit{Step 3. Construction of $\mc{L}$ when $T$ is absolutely reducible.} Now assume there is a reduced, irreducible, affinoid curve which is finite over a Zariski open neighborhood of $y_0$ in $\mf{Z}$ over which $T$ becomes reducible. We may assume, by passing to this curve, that $T$ is reducible over $\mf{Z}$. Then this time we have that there are vectors $w_5,w_7\in\mc{P}_{y_0}$ which are, respectively, eigenvectors for $g_0$ acting on $\mc{P}_{y_0}$ with respective eigenvalues $\gamma_5$ and $\gamma_7$. Let $\widetilde{w}_5,\widetilde{w}_7\in\mc{P}$ be any vectors mapping, respectively, to $w_5,w_7$ in $\mc{P}_{y_0}$. Then we let
\[\mc{L}_1=\mc{O}(\mf{Z})[G_\Q]\widetilde{w}_5,\qquad\mc{L}_2=\mc{O}(\mf{Z})[G_\Q]\widetilde{w}_7,\qquad\mc{L}=\mc{L}_1\oplus\mc{L}_2.\]
Then by Lemma \ref{lemTsumoftwo}, we have that $\mc{L}_1$ and, respectively, $\mc{L}_2$ are projective modules over $\mc{O}(\mf{Z})$ of ranks $3$ and, respectively, $4$, and that $\mc{L}$ is a projective $\mc{O}(\mf{Z})$-submodule of $\mc{P}$ of rank $7$.

In either of the cases of Step 2 or Step 3, we let $\overline{\mc{L}}$ be the specialization of $\mc{L}$ at $y_0$, and $\rho_{\overline{\mc{L}}}$ the associated representation of $G_\Q$.

\textit{Step 4. Construction of the trilinear form }$\langle\cdot,\cdot,\cdot\rangle$\textit{ and the basis }$v_1,\dotsc,v_7$. To construct the basis $v_1,\dotsc,v_7$, we just let $v_i$ for $i=1,\dotsc,7$ be the eigenvector for $g_0$ in $\overline{\mc{L}}$ with eigenvalue $\gamma_i$. Then since the $\gamma_i$'s are distinct, these are unique up to scaling. By Hensel's lemma there are then vectors $\tilde{v}_1,\dotsc,\tilde{v}_7$ in $\mc{L}\otimes_{\mc{O}(\mf{Z})}\mc{A}$, mapping to, respectively, $v_1,\dotsc,v_7$ modulo the maximal ideal $\mf{m}$ of $\mc{A}$, which are eigenvectors for $g_0$ with distinct, respective eigenvalues $\tilde\gamma_1,\dotsc,\tilde\gamma_7$. We then have
\[\tilde\gamma_i\equiv\gamma_i\modulo{\mf{m}},\qquad i=1,\dotsc,7.\]

Now since the representation $\rho_\Theta$ we started with in Step 1 factors through $G_2(F(\mf{Z}))$, it therefore also factors through $G_2(\overline{\Frac(\mf{A})})$. By the results of Section \ref{subsecalt3forms}, it therefore preserves a trilinear form on $\overline{\Frac(\mc{A})}^7$ which we denote by $\langle\cdot,\cdot,\cdot\rangle^{\sim}$. The Galois group $G_\Q$ acts on $\mc{L}\otimes_{\mc{O}(\mf{Z})}\overline{\Frac(\mc{A})}$ through the base change of the representation $\rho_\Theta$ to $\overline{\Frac(\mc{A})}$, and so we may view $\langle\cdot,\cdot,\cdot\rangle^{\sim}$ as a trilinear form on $\mc{L}\otimes_{\mc{O}(\mf{Z})}\Frac(\mc{A})$ by restriction.

Let $\tilde{v}_1^\vee,\dotsc,\tilde{v}_7^\vee$ be the basis of $(\mc{L}\otimes_{\mc{O}(\mf{Z})}\overline{\Frac(\mc{A})})^\vee$ dual to $\tilde{v}_1,\dotsc,\tilde{v}_7$, and write
\[\langle\cdot,\cdot,\cdot\rangle^\sim=\sum_{1\leq i<j<k\leq 7}\tilde{a}_{ijk}\tilde{v}_i^\vee\wedge\tilde{v}_j^\vee\wedge\tilde{v}_k^\vee,\]
for some $\tilde{a}_{ijk}\in\overline{\Frac(\mc{A})}$. Then for any triple $(i,j,k)$ with $1\leq i<j<k\leq 7$, we have
\[\tilde{a}_{ijk}=\langle \tilde{v}_i,\tilde{v}_j,\tilde{v}_k\rangle^\sim=\langle g_0\tilde{v}_i,g_0\tilde{v}_j,g_0\tilde{v}_k\rangle^\sim=\tilde{\gamma}_i\tilde{\gamma}_j\tilde{\gamma}_k\tilde{a}_{ijk}.\]
Thus $\tilde{a}_{ijk}=0$ for all such triples $(i,j,k)$ except possibly
\[(i,j,k)=(1,4,7),\,\,(1,5,6),\,\,(2,4,6),\,\,(2,3,7),\,\,(3,4,5),\]
and $\tilde{a}_{ijk}$ is nonzero for at least one of these triples $(i,j,k)$. By adjoining to it the elements $\tilde{a}_{ijk}$ for these five triples $(i,j,k)$, we may simply assume $\mc{A}$ contains these elements. Furthermore, by scaling $\langle\cdot,\cdot,\cdot\rangle^\sim$, we may assume that $\tilde{a}_{ijk}$ is integral for these five triples $(i,j,k)$, and that $\tilde{a}_{ijk}\in\mc{A}^\times$ for at least one such $(i,j,k)$. If we write $a_{ijk}\in\overline\Q_p$ for the reduction of $\tilde{a}_{ijk}$ modulo $\mf{m}$, then this means
\[a_{ijk}\ne 0\textrm{ for at least one }(i,j,k)=(1,4,7),\,\,(1,5,6),\,\,(2,4,6),\,\,(2,3,7),\,\,(3,4,5).\]
Furthermore, if $T$ is reducible as in Step $3$, then by replacing $\mc{L}_2$ with $c\mc{L}_2$ for a suitable $c\in\mc{A}$, we may assume the valuation of $\tilde{a}_{345}$ is small enough so that we instead have
\[a_{ijk}\ne 0\textrm{ for at least one }(i,j,k)=(1,4,7),\,\,(1,5,6),\,\,(2,4,6),\,\,(2,3,7).\]

Finally, we let $\langle\cdot,\cdot,\cdot\rangle$ be the reduction of $\langle\cdot,\cdot,\cdot\rangle^\sim$ modulo $\mf{m}$.

\textit{Step 5. Proof of (a)-(d).} By construction, $\overline{\mc{L}}$ has unique irreducible quotient $\rho_F(-\tfrac{k}{2})$ unless $T$ reduces as in Step 3, in which case $\overline{\mc{L}}_2$ has unique irreducible quotient $\rho_F(-\tfrac{k}{2})$ instead. Then (a) follows. Then the representations $\rho_1$ and $\rho_2$ of (b) are, respectively, just $\rho_3^\vee$ and $\Ad(\rho_3)$. So after possibly scaling the vectors $v_1,\dotsc,v_5$ separately, (b) follows by construction. The point (c) is true directly by construction, and (d) holds because $\rho_\Theta$ was unramified to begin with.

\textit{Step 6. Proof of (e).} By construction, the affinoid curve $\mf{Z}$ is finite over a linear subspace in the weight space $\mf{X}$ of Section \ref{subsecchardistg2}. By passing to an open neighborhood of $y_0$ in $\mf{Z}$, we may assume that the image of $\mf{Z}$ in $\mf{X}$ is contained in any of the neighborhoods $\mf{U}'$ from Propositions \ref{propscvariation}, \ref{proprampsvariation}, \ref{propunrstvariation}, or \ref{propramstvariation}.

Let us first assume that $\ell$ divides $N$ exactly once, so that $\pi_F$ is an unramified twist of Steinberg at $\ell$. Let us further shrink $\mf{Z}$ by considering an open neighborhood where $\mc{L}$ is free. Let $t:I_\ell\to\Z_p$ be the tame quotient map for $p$. By Grothendieck's theorem on quasiunipotent monodromy, there is a nilpotent operator $N_{\mf{Z}}:\mc{L}\to\mc{L}$ and an open subgroup $U_{\mf{Z}}$ of $I_\ell$ such that for any $g\in U_{\mf{Z}}$, we have
\[\rho_{\mc{L}}(g)=\exp(t(g)N_{\mf{Z}}).\]
Similarly, for any point $z\in\mf{Z}(\overline\Q_p)$ lying over a classical regular weight, or for $z=z_0$, there is a nilpotent operator $N_z$ on the specialization $\mc{L}_z$ of $\mc{L}$ at $z$, and an open subgroup $U_z$ of $I_\ell$ such that for any $g\in U_z$, we have
\[\rho_{\mc{L}_z}(g)=\exp(t(g)N_z),\]
where, of course, $\rho_{\mc{L}_z}$ is the representation of $G_\Q$ on $\mc{L}_z$.

Now the Zariski closure of the $GL_7(\overline{F(\mf{Z})})$-conjugacy class $N_z$ is contained in the that of $N_{\mf{Z}}$, with equality for all but finitely many $z$. Indeed, one can conjugate $N_{\mf{Z}}$ by a matrix $M\in GL_7(\overline{F(\mf{Z})})$ to put it in Jordan form, and then the finitely many exceptional $z$ are contained in the set of points where the entries of $M$ have poles or otherwise where the specialization $M_z$ of $M$ at $z$ is not invertible. Moreover, for any $z$ where these conjugacy classes are equal, we have for any $g\in U_z\cap U_{\mf{Z}}$ that
\[\rho_{\mc{L}_z}(g)=(\rho_{\mc{L}}(g))_z,\]
the subscript $z$ on the right once again signifying specialization at $z$. But $\rho_{\mc{L}_z}^{\sss}$ is the Galois representation attached by Theorem \ref{thmgalreps} to an automorphic representation $\sigma_z$ of $G_2(\A)$ which, by Proposition \ref{propunrstvariation} has component at $\ell$ given by the Langlands quotient of
\[\iota_{P_\alpha(\Q_\ell)}^{G_2(\Q_\ell)}(\pi_{F,\ell}).\]
By Conjecture \ref{conjliftings} (d) and Theorem \ref{thmgalreps} (b), it follows that
\begin{equation}
\label{eqWDLz}
\mr{WD}(\rho_{\mc{L}_z}^{\sss}|_{G_{\Q_\ell}})
\end{equation}
is, up to twist, given by
\[\mr{LL}_\ell(\pi_{F,\ell})^\vee\oplus\Ad(\mr{LL}_\ell(\pi_{F,\ell}))\oplus\mr{LL}_\ell(\pi_{F,\ell}),\]
where $\mr{LL}_\ell$ is the local Langlands correspondence at $\ell$. Therefore the monodromy operator for \eqref{eqWDLz} is visibly conjugate to that of $\rho_{\overline{\mc{L}}}^{\sss}$, and so the Zariski closure of the conjugacy class of $N_z$, and hence also that of $N_{\mf{Z}}$, contains that of the monodromy operator $\rho_{\overline{\mc{L}}}^{\sss}$. If one of the functions $*_1,*_2,*_3$ from the statement of (e) were nontrivial on $U_{z_0}$, this would force the Zariski closure of the conjugacy class of $N_{z_0}$ to be bigger than that of $N_{\mf{Z}}$, which we have already shown is impossible.

This proves (e) when $\ell$ divides $N$ exactly once. Otherwise, if $\ell^2|N$, then $\pi_{F,\ell}$ is either supercuspidal (respectively, a ramified principal series representation, or a ramified twist of the Steinberg representation) and the proof is completely analogous using Proposition \ref{propscvariation} (respectively, \ref{proprampsvariation} or \ref{propramstvariation}) in place of Proposition \ref{propunrstvariation}; just note that these three propositions assume the prime of interest is different from $2$ or $3$, whence the hypothesis in (e). Also, in the case that $\pi_{F,\ell}$ is a ramified principal series, then Proposition \ref{proprampsvariation} shows that the representations $\sigma_z$ obtained analogously to those above may be Langlands quotients of a parabolic induction from $P_\beta(\Q_\ell)$ as  well as from $P_\alpha(\Q_\ell)$. Therefore, in this case, one needs to use Conjecture \ref{conjliftings} (e), as well as (d), to conclude.

\textit{Step 7. Proof of (f).} Like in Step 6, we may assume that the image of $\mf{Z}$ in $\mf{X}$ is contained in the neighborhood $\mf{U}'$ from Proposition \ref{propunratp}. Then (f) is just an application of Kisin's Lemma (Lemma \ref{lemkisin}) and Theorem \ref{thmgalreps} (c).
\end{proof}

\begin{remark}
In the course of proving the main theorem in the next subsection, we will rule out the possibility that \eqref{eqLbar22} and \eqref{eqLbar32} occur, and therefore also ruling out the reducibility of $T$ by Step 3 of the proof above.
\end{remark}

\subsection{Extensions and the symmetric cube Selmer group}
\label{subsecmainthm}
We work throughout this subsection in the setting of Proposition \ref{propmcLbar}. In particular, we have the $\overline{\Q}_p$-vector space $\overline{\mc{L}}$ on which $G_\Q$ acts by a representation $\rho_{\overline{\mc{L}}}$. The space $\overline{\mc{L}}$ has a basis $v_1,\dotsc,v_7$, the representation $\rho_{\overline{\mc{L}}}$ has matrix coefficients $g_{ij}$ in that basis, and $\overline{\mc{L}}$ comes equipped with an alternating trilinear form $\langle\cdot,\cdot,\cdot\rangle$ which is preserved by $\rho_{\overline{\mc{L}}}$ and determined by the numbers $a_{147}$, $a_{156}$, $a_{237}$, $a_{246}$, $a_{345}$ in $\overline{\Q}_p$

The main theorem of this paper will be proved in this subsection, and it will follow from several computations, which we separate into various lemmas below, which mostly involve the matrix coefficients $g_{ij}$ and the numbers $a_{147}$, $a_{156}$, $a_{237}$, $a_{246}$, $a_{345}$, along with some linear algebra coming from $p$-adic Hodge theory. The ultimate goal will be to show that \eqref{eqLbar12} holds, that the function $*_3$ there is nonzero, and that $\rho_{\overline{\mc{L}}}$ factors through $G_2(\overline\Q_p)$. Then $*_3$ will give the desired Selmer class.

\begin{lemma}
One of $a_{147},a_{156},a_{246},a_{237}$ is nonzero.
\end{lemma}

\begin{proof}
We assumed the conclusion of the lemma to be true when \eqref{eqLbar32} holds. Thus we may assume either \eqref{eqLbar12} or \eqref{eqLbar22} holds. Then
\begin{equation}
\label{eqnonzerogs}
\textrm{one of }g_{36},g_{46},g_{56},g_{37},g_{47},g_{57}\textrm{ is nonzero}
\end{equation}
for some $g\in G_\Q$.

Now assume for sake of contradiction that all of $a_{147},a_{156},a_{246},a_{237}$ are zero. Then $a_{345}\ne 0$. The following type of computation will be done here in detail exactly once; many analogous computations will be carried out more quickly in the proof of this lemma and the ones that follow. We have
\begin{align*}
0=a_{156}(g^{-1})_{13} &=\langle g^{-1}v_3,v_5,v_6\rangle\\
&=\langle v_3,gv_5,gv_6\rangle\\
&=a_{345}(g_{45}g_{56}-g_{55}g_{46})+a_{237}(g_{75}g_{26}-g_{25}g_{76})\\
&=a_{345}(g_{45}g_{56}-g_{55}g_{46}),\\
0=a_{246}(g^{-1})_{23} &=\langle g^{-1}v_3,v_4,v_6\rangle\\
&=\langle v_3,gv_4,gv_6\rangle\\
&=a_{345}(g_{44}g_{56}-g_{54}g_{46})+a_{237}(g_{74}g_{26}-g_{24}g_{76})\\
&=a_{345}(g_{44}g_{56}-g_{54}g_{46}).
\end{align*}
We thus get
\[a_{345}\pmat{g_{44}&g_{54}\\ g_{45}&g_{55}}
\pmat{g_{56}\\ -g_{46}}=0\]
Since $a_{345}\ne 0$, the definition of $\rho_2$ in Proposition \ref{propmcLbar} gives
\[\pmat{g_{66}g_{77}+g_{67}g_{76} & -g_{67}g_{77}\\
-2g_{76}g_{77} & g_{77}^2}
\pmat{g_{56}\\ -g_{46}}=0.\]
Now we have
\[\det\pmat{g_{66}g_{77}+g_{67}g_{76} & -g_{67}g_{77}\\
-2g_{76}g_{77} & g_{77}^2}=g_{77}^2(g_{66}g_{77}-g_{67}g_{76})=d(g)g_{77}^2,\]
which is zero only when $g_{77}$ is zero. Since $F$ is not CM and hence $\rho_F$ has big image, $g_{77}=0$ only for $g$ in a measure zero subset of $G_\Q$. For such $g$, we can invert the matrix above and we find that $g_{46}=g_{56}=0$ for all $g$ outside a subset of $G_\Q$ of measure zero. By continuity of $\rho_{\overline{\mc{L}}}$, we then have $g_{46}=g_{56}=0$ for all $g$.

One can then repeat this argument using the quantities
\[a_{147}(g^{-1})_{13},\qquad a_{237}(g^{-1})_{23},\qquad a_{156}(g^{-1})_{15},\qquad a_{237}(g^{-1})_{15}.\]
One similarly obtains vanishing of $g_{47}$ and $g_{57}$ from the computation of the first two of these quantities, and also of $g_{36}$ and $g_{37}$ from the last two. This gives the desired contradiction by \eqref{eqnonzerogs}.
\end{proof}

\begin{lemma}
We have
\[a_{147}=a_{156}=a_{237}=-a_{246}.\]
\end{lemma}

\begin{proof}
It follows from Proposition \ref{propmcLbar} (b) that
\[\pmat{(g^{-1})_{11}&(g^{-1})_{12}\\ (g^{-1})_{21}&(g^{-1})_{22}}=\pmat{g_{77}&g_{67}\\ g_{76}&g_{66}}.\]
Using this and the form $\langle\cdot,\cdot,\cdot\rangle$, one computes
\[a_{156}g_{77}=a_{156}(g_{55}g_{66}-g_{65}g_{56})+a_{147}(g_{45}g_{76}-g_{75}g_{46}).\]
Since $g_{65}=g_{75}=0$, using the definition of $\rho_2$ gives
\[a_{156}g_{77}d(g)=a_{156}g_{77}^2g_{66}-a_{147}g_{77}g_{67}g_{76},\]
or, after rearranging,
\[a_{156}g_{77}g_{67}g_{76}=a_{147}g_{77}g_{67}g_{76}.\]
Since $\rho_F$ has big image, $g_{77}g_{67}g_{76}\ne 0$ for some $g$. Thus $a_{156}=a_{147}$.

Similar computations using $a_{237}g_{66}$ and $a_{156}g_{67}$ respectively give $a_{237}=-a_{246}$ and $a_{156}=a_{237}$, which suffices to complete the proof; details are omitted.
\end{proof}

By the two previous lemmas, by scaling $\langle\cdot,\cdot,\cdot\rangle$, we can (and will) assume
\begin{equation}
\label{eqasarepm1}
a_{147}=a_{156}=a_{237}=1,\qquad a_{246}=-1.
\end{equation}

Assume now that
\[g_{13}=g_{14}=g_{15}=g_{23}=g_{24}=g_{25}=0.\]
Then $\rho_2$ is a subrepresentation of $\overline{\mc{L}}$. The quotient $\overline{\mc{L}}/\rho_2$ is the extension $E$ of the form
\[\pmat{\rho_1&*_2\\ 0&\rho_3}\]
with $*_2$ nontrivial.

\begin{lemma}
\label{lemextqpqp1}
Assume that
\[g_{13}=g_{14}=g_{15}=g_{23}=g_{24}=g_{25}=0\]
for all $g$, so that we have the extension $E$ as above. Then the exterior square $\wedge^2 E$ has a subrepresentation a nontrivial extension of the form
\[\pmat{\chi_{\cyc}& *\\ 0&1}.\]
\end{lemma}

\begin{proof}
One can compute,
\begin{equation*}
\pmat{(g^{-1})_{16}\\(g^{-1})_{26}}
=-\frac{1}{d(g)}\pmat{g_{77}^2g_{16}+g_{77}g_{67}g_{26}-g_{77}g_{76}g_{17}-g_{67}g_{76}g_{27}\\
g_{77}g_{76}g_{16}+g_{77}g_{66}g_{26}-g_{76}^2g_{17}-g_{76}g_{66}g_{27}}.
\end{equation*}
Since we have assumed $a_{156}=1$, the usual computations using $\langle\cdot,\cdot,\cdot\rangle$ yield
\begin{equation}
\label{eqginv16}
(g^{-1})_{16}=-\frac{1}{d(g)}(g_{67}g_{77}g_{26}-g_{77}^2g_{16}),
\end{equation}
Hence
\[g_{77}g_{76}g_{17}+g_{67}g_{76}g_{27}=0\]
and it follows (again from $\rho_F$ having big image) that
\begin{equation}
\label{eqgggg1}
g_{77}g_{17}+g_{67}g_{27}=0
\end{equation}

We similarly compute (using $a_{147}=1$ instead) that
\[(g^{-1})_{16}=-\frac{1}{d(g)}((g_{66}g_{77}+g_{76}g_{67})g_{27}+2g_{77}g_{76}g_{17}).\]
Combining this with \eqref{eqginv16} gives
\[g_{76}(g_{77}g_{17}+g_{67}g_{27})+g_{77}(g_{77}g_{16}+g_{76}g_{17}+g_{67}g_{26}+g_{66}g_{27})=0,\]
and by \eqref{eqgggg1}, this becomes
\[g_{77}(g_{77}g_{16}+g_{76}g_{17}+g_{67}g_{26}+g_{66}g_{27})=0.\]
Hence
\begin{equation}
\label{eqgggg2}
g_{77}g_{16}+g_{76}g_{17}+g_{67}g_{26}+g_{66}g_{27}=0.
\end{equation}

Next we compute
\begin{multline*}
-\frac{1}{d(g)}(g_{77}g_{76}g_{16}+g_{77}g_{66}g_{26}-g_{76}^2g_{17}-g_{76}g_{66}g_{27})=a_{237}(g^{-1})_{26}
=\frac{1}{d(g)}(g_{66}g_{76}g_{27}+g_{76}^2g_{17}).
\end{multline*}
Therefore,
\[g_{77}g_{76}g_{16}+g_{77}g_{66}g_{26}=0,\]
and so
\begin{equation}
\label{eqgggg3}
g_{76}g_{16}+g_{66}g_{26}=0.
\end{equation}

With this preparation we can now proceed to prove the lemma. The extension $E$ has basis $v_1,v_2,v_6,v_7$. Therefore, the exterior square $\wedge^2 E$ has the basis
\[v_2\wedge v_1,\,\,(v_7\wedge v_1+v_2\wedge v_6),\,\,v_6\wedge v_1,\,\,(v_7\wedge v_1-v_2\wedge v_6),\,\,v_7\wedge v_2,\,\,v_7\wedge v_6.\]
We now compute part of the matrix of $g$ in this basis. We have
\[g(v_2\wedge v_1)=\frac{1}{d(g)}(g_{77}v_2-g_{67}v_1)\wedge(-g_{76}v_2+g_{66}v_1)=\frac{1}{d(g)}v_2\wedge v_1,\]
and
\begin{align*}
g(v_7\wedge v_1+v_2\wedge v_6)=&\frac{1}{d(g)}(g_{77}v_7+g_{67}v_6+g_{27}v_2+g_{17}v_1)\wedge(-g_{76}v_2+g_{66}v_1)\\
&+\frac{1}{d(g)}(g_{76}v_7+g_{66}v_6+g_{26}v_2+g_{16}v_1)\wedge(-g_{77}v_2+g_{67}v_1)\\
=&\frac{1}{d(g)}(g_{67}g_{26}+g_{77}g_{16})v_2\wedge v_1+(v_7\wedge v_1+v_2\wedge v_6),
\end{align*}
where we used \eqref{eqgggg2}.

This computes the first two columns of $g$ in the basis chosen above; these two columns begin with
\[\pmat{d(g)^{-1} &*\\ 0&1},\]
and the rest of the entries are zero, showing that this extension is a subrepresentation of $\wedge^2 E$. Here the asterisk denotes
\begin{equation}
\label{eqasteriskgggg}
*=d(g)^{-1}(g_{67}g_{26}+g_{77}g_{16}).
\end{equation}
Note that $d(g)=\det\rho_F(-\tfrac{k}{2})(g)=\chi_{\cyc}^{-1}(g)$, so as long as $*$ is nontrivial, this is the desired extension.

So assume for sake of contradiction that $*$ is trivial. Then
\[g_{67}g_{26}+g_{77}g_{16}=0.\]
Combining this with \eqref{eqgggg3} gives
\[\pmat{g_{77}&g_{67}\\ g_{76}&g_{66}}\pmat{g_{16}\\ g_{26}}=0,\]
hence $g_{16}=g_{26}=0$. But then \eqref{eqgggg2} becomes
\[g_{76}g_{17}+g_{66}g_{27}=0.\]
Along with \eqref{eqgggg1}, this gives
\[\pmat{g_{77}&g_{67}\\ g_{76}&g_{66}}\pmat{g_{17}\\ g_{27}}=0.\]
This implies $g_{17}=g_{27}=0$ as well, which implies that $E$ is a trivial extension. This is a contradiction, and the lemma is proved.
\end{proof}

We are now in a position to use $p$-adic Hodge theory to rule out the possibility that $\Ad\rho_F$ is a subrepresentation of $\overline{\mc{L}}$.

\begin{lemma}
\label{lemgijnonzero}
Assume $4\nmid N$ and $9\nmid N$. Then one of
\[g_{13},g_{14},g_{15},g_{23},g_{24},g_{25}\]
is nonzero for some $g\in G_\Q$.
\end{lemma}

\begin{proof}
Assume for sake of contradiction that
\[g_{13}=g_{14}=g_{15}=g_{23}=g_{24}=g_{25}=0.\]
Then $\rho_2$ is a subrepresentation of $\overline{\mc{L}}$ with quotient given by the extension $E$ from above. By Lemma \ref{lemextqpqp1}, the exterior square $\wedge^2 E$ contains as a subrepresentation a nontrivial extension
\[\pmat{\chi_{\cyc}& *\\ 0&1}.\]
Let us call this extension $E'$. It is unramified at all primes $\ell\nmid Np$ by Proposition \ref{propmcLbar} (d), and it is even unramified at $\ell|N$ by Proposition \ref{propmcLbar} (e) and \eqref{eqasteriskgggg}. We will show now that $E'$ has to also be crystalline at $p$. This will be a contradiction since $E'$ will represent a nontrivial class in the Bloch--Kato Selmer group
\[H_f^1(\Q_p,\overline{\Q}_p(1)),\]
which itself is trivial.

To get started, by Proposition \ref{propvv1st} $E$ is semistable. Write $D=D_{\crys}(\rho_3)$, and let $u_3,u_4\in D$ be eigenvectors for the crystalline Frobenius with respective eigenvalues $p^{k/2}\alpha_p^{-1}$ and $p^{-(k-2)/2}\alpha_p$. Since these numbers are distinct by Assumption \ref{assnodblroot}, the vectors $u_3$ and $u_4$ are linearly independent. Then there is a basis $w_1,w_2,w_3,w_4$ of $D_{\st}(E)$ such that, in this basis, the crystalline Frobenius $\phi_E$ for $D_{\st}(E)$ is given by
\[\phi_E=\diag(p^{(k-2)/2}\alpha_p^{-1},p^{-k/2}\alpha_p,p^{k/2}\alpha_p^{-1},p^{-(k-2)/2}\alpha_p)\]
and the monodromy operator $N_E$ for $D_{\st}(E)$ is given by
\[N_E=\pmat{0&B\\0&0},\]
for some $B\in M_2(\overline\Q_p)$. Write
\[B=\pmat{b_{11}&b_{12}\\ b_{21}&b_{22}}.\]

We claim first that $b_{12}=b_{22}=0$. To show this, we apply Kisin's lemma (or really, Proposition \ref{propmcLbar} (f)) to $\wedge^2\overline{\mc{L}}$, in a manner inspired by \cite{urbanell}. This shows that
\[D_{\crys}(\wedge^2\overline{\mc{L}})^{\phi=p^{-(k-1)}\alpha_p^2\cdot p^{-(k-2)/2}\alpha_p}\ne 0.\]
Since $\rho_2\otimes E\cong\Ad(\rho_F)\otimes E$ is the only subquotient of $\wedge^2\overline{\mc{L}}$ that can contribute an eigenvector for the crystalline Frobenius with this eigenvalue, it follows that
\[D_{\crys}(\rho_2\otimes E)^{\phi=p^{-(k-1)}\alpha_p^2\cdot p^{-(k-2)/2}\alpha_p}\ne 0.\]
Now
\[D_{\st}(\rho_2\otimes E)=D_{\crys}(\rho_2)\otimes D_{\st}(E),\]
because $\rho_2$ is crystalline, and the monodromy operator on this space is given by $1\otimes N_E$. It follows that
\[D_{\st}(\rho_2\otimes E)^{\phi=p^{-(k-1)}\alpha_p^2\cdot p^{-(k-2)/2}\alpha_p}\ne 0.\]
This space is generated by vectors of the form $v'\otimes w_4$ where $v'\in D_{\crys}(\rho_2)$ is a crystalline Frobenius eigenvector with eigenvalue $p^{-(k-1)}\alpha_p^2$ (again we are using Assumption \ref{assnodblroot}). Therefore $D_{\crys}(\rho_2\otimes E)$ must include a vector $v'\otimes w_4$ of this type, which is thus in the kernel of $1\otimes N_E$. Hence $N_E w_4=0$, which implies $b_{12}=b_{22}=0$.

Next we claim $b_{11}=0$. To do this, we apply the same argument, but to the dual $(\wedge^2\overline{L})^\vee$ of $\wedge^2\overline{L}$ (or, equivalently, to $\wedge^5\overline{\mc{L}}$). This gives
\[D_{\crys}((\wedge^2\overline{\mc{L}})^\vee)^{\phi=p^{-(k-1)}\alpha_p^2\cdot p^{-(k-2)/2}\alpha_p}\ne 0.\]
It follows that
\[D_{\crys}(\rho_2\otimes E^\vee)^{\phi=p^{-(k-1)}\alpha_p^2\cdot p^{-(k-2)/2}\alpha_p}\ne 0.\]
We also have
\[D_{\st}(\rho_2\otimes E)=D_{\crys}(\rho_2)\otimes D_{\st}(E)^\vee.\]

Let $w_1^\vee,w_2^\vee,w_3^\vee,w_4^\vee$ be the basis dual to $w_1,w_2,w_3,w_4$. Then the monodromy operator $N_{E^\vee}$ on $D_{\st}(E)^\vee$ is given in this basis by
\[N_{E^\vee}=\pmat{0&0\\-\prescript{t}{}{B}&0}.\]
So by a similar argument as above, it follows that a vector $v'\otimes w_1^\vee$, with $v'$ as above, is in the kernel of $N_{E^\vee}$. Thus the first column of the matrix representing $N_{E^\vee}$ is zero; in particular $b_{11}=0$.

Now we come to the representation $\wedge^2 E$. We claim that
\[D_{\st}(\wedge^2 E)^{\phi=1}=D_{\crys}(\wedge^2 E)^{\phi=1}.\]
From this it will follow that the extension $E'$ is crystalline, and we will be done.

To prove the claim, we note that $D_{\st}(\wedge^2 E)^{\phi=1}$ is the span of $v_1\wedge v_4$ and $v_2\wedge v_3$. But
\[N(v_1\wedge v_4)=(N_E v_1)\wedge v_4+v_1\wedge(N_E v_4)=v_1\wedge(b_{12}v_1+b_{22}v_2)=0,\]
and
\[N(v_2\wedge v_3)=(N_E v_2)\wedge v_3+v_2\wedge(N_E v_4)=v_2\wedge(b_{11}v_1+b_{21}v_2)=b_{21}v_2\wedge v_2=0.\]
This proves the claim, and hence also the lemma.
\end{proof}

\begin{lemma}
\label{lema345nonzero}
We have $a_{345}\ne 0$.
\end{lemma}

\begin{proof}
Assume on the contrary that $a_{345}=0$. We will get a contradiction to Lemma \ref{lemgijnonzero}. One computes
\[(g^{-1})_{13} =a_{345}(g_{45}g_{56}-g_{55}g_{46})+a_{237}(g_{75}g_{26}-g_{25}g_{76})=-g_{76}g_{25},\]
and
\[(g^{-1})_{25} =a_{345}(g_{33}g_{47}-g_{43}g_{37})+a_{156}(g_{63}g_{17}-g_{13}g_{67})=-g_{67}g_{13},\]
Therefore,
\[(g^{-1})_{13}=-g_{76}g_{25}=g_{76}(g^{-1})_{67}(g^{-1})_{13}=-d(g)^{-1}g_{67}g_{76}(g^{-1})_{13}.\]
Because the image of $\rho_F$ is big, $g_{13}$ is identically zero, and by the equations above $g_{25}$ is also identically zero.

We also compute similarly
\[(g^{-1})_{15}=g_{66}g_{15},\]
Like above, this forces $g_{15}$ to be identically zero. Likewise,
\[(g^{-1})_{23}=-a_{237}g_{77}g_{23},\]
similarly forcing $g_{23}$ to be identically zero.

Finally, we have
\[(g^{-1})_{14}=a_{345}(g_{55}g_{36}-g_{35}g_{56})+a_{147}(g_{75}g_{16}-g_{15}g_{76})+a_{246}(g_{65}g_{26}-g_{25}g_{66})=0,\]
where the vanishing at the end is because $g_{15}=g_{25}=0$. Similarly,
\[a_{237}(g^{-1})_{24} =a_{345}(g_{53}g_{37}-g_{33}g_{57})+a_{147}(g_{73}g_{17}-g_{13}g_{77})+a_{246}(g_{63}g_{27}-g_{23}g_{67})=0.\]
Thus $g_{14}$ and $g_{24}$ are identically zero. This is a contradiction and finishes the proof.
\end{proof}

We now know that $\rho_{\overline{\mc{L}}}$ has the form \eqref{eqLbar12}, where $*_1$, $*_2$ and $*_3$ are all nontrivial, by Lemma \ref{lemgijnonzero}. We also have that $\rho_{\overline{\mc{L}}}$ preserves the alternating trilinear form
\[\langle\cdot,\cdot,\cdot\rangle=v_1^\vee\wedge v_4^\vee\wedge v_7^\vee+v_1^\vee\wedge v_5^\vee\wedge v_6^\vee+v_2^\vee\wedge v_3^\vee\wedge v_7^\vee-v_2^\vee\wedge v_4^\vee\wedge v_6^\vee+a_{345}v_3^\vee\wedge v_4^\vee\wedge v_5^\vee,\]
for some nonzero (by Lemma \ref{lema345nonzero}) $a_{345}\in\overline\Q_p$. From here on we write $a=a_{345}$.

By Lemma \ref{eqfirstgen3form}, the representation $\overline{\mc{L}}$ factors through the $G_2$-subgroup $G_a(\overline\Q_p)$ of $GL_7(\overline\Q_p)$. Because of its shape, by Proposition \ref{propPabeta}, $\overline{\mc{L}}$ factors though the short root parabolic $P_{a,\beta}(\overline\Q_p)$.

For $1\leq i,j\leq 7$, let us now write $g_{ij}'=d(g)g_{ij}=\chi_{\cyc}^{-1}(g)g_{ij}$. Let us also write $E$ for the extension given by
\begin{equation}
\label{eqEmatrix}
E=\frac{1}{d(g)}\pmat{
g_{66}&-g_{67}&g_{13}'&g_{14}'&g_{15}'\\ 
-g_{76}&g_{77}&g_{23}'&g_{24}'&g_{25}'\\
&&g_{66}^2 & 2g_{66}g_{67} & -g_{67}^2\\
&&g_{66}g_{76} & g_{66}g_{77}+g_{67}g_{76} & -g_{67}g_{77}\\
&&-g_{76}^2 & -2g_{76}g_{77} & g_{77}^2}.
\end{equation}
This is the upper left $5$ by $5$ block in $\rho_{\overline{\mc{L}}}$. It is an extension
\[0\to\rho_F(-\tfrac{k-2}{2})\to E\to\Ad(\rho_F)\to 0,\]
and it is a nontrivial extension by Lemma \ref{lemgijnonzero}.

We want to show $E$ is semistable at $p$. Most of our constructions in this paper up until now, including that of $E$, depended on a choice of root $\alpha_p$ of the Hecke polynomial of $F$ at $p$. However, certain choices of $\alpha_p$ may lead to problems when showing $E$ is semistable. But it turns out that if one choice of $\alpha_p$ is problematic, we can switch $\alpha_p$ for the other root $p^{k-1}\alpha_p^{-1}$ and show that this other choice is not problematic. We make this precise in the following lemma.

\begin{lemma}
\label{lemEisst}
Notation as above, there is a choice of $\alpha_p$ such that
\[D_{\crys}(\Sym^2(\rho_F)(-1))^{\phi=p^{3-2k}\alpha_p^2}\cap\Fil^0 D_{\crys}(\Sym^2(\rho_F)(-1))=0.\]
For such a choice of $\alpha_p$, the extension $E$ constructed above is semistable.
\end{lemma}

\begin{proof}
Let us begin with either choice of $\alpha_p$. We apply Kisin's lemma (or really, Proposition \ref{propmcLbar} (f)) to $\overline{\mc{L}}$. This shows that
\[D_{\crys}(\overline{\mc{L}})^{\phi=p^{-(k-1)}\alpha_p^2}\ne 0.\]
Note that
\[p^{-(k-1)}\alpha_p^2\notin\{p^{-(k-2)/2}\alpha_p,p^{k/2}\alpha_p^{-1}\},\]
because otherwise we would have
\[\alpha_p=p^{k/2},\qquad\textrm{or}\qquad\alpha_p^3=p^{(3k-2)/2};\]
We saw that this first option is impossible by Assumption \ref{assnodblroot}, and the second is impossible since $\alpha_p$ is a $p$-Weil number of weight $\tfrac{k-1}{2}$. Since $p^{-(k-2)/2}\alpha_p,p^{k/2}\alpha_p^{-1}$ are the eigenvalues of the crystalline Frobenius for $\rho_F(-\tfrac{k}{2})$, we actually have,
\[D_{\crys}(E)^{\phi=p^{-(k-1)}\alpha_p^2}\ne 0.\]

Now consider the twist $E(k-2)$. This is an extension
\[0\to\rho_F(\tfrac{k-2}{2})\to E(k-2)\to\Sym^2(\rho_F)(-1)\to 0.\]
We wish to apply Lemma \ref{lemEisdR} to $E(k-2)$. The representation $\rho_F(\tfrac{k-2}{2})$ has Hodge--Tate weights which are strictly negative, and $\Fil^0\Sym^2(\rho_F)(-1)$ is $2$-dimensional.

Let $w_1,w_2$ be a basis of $D_{\crys}(\rho_F)$ such that $\Fil^1(D_{\crys}(\rho_F))$ is generated by $w_1$. Write
\[w_{11}=(w_1\otimes w_1)[-1],\qquad w_{12}=(w_1\otimes w_2)[-1],\qquad w_{22}=(w_2\otimes w_2)[-1]\]
for the corresponding basis of
\[D_{\crys}(\Sym^2(\rho_F)(-1))=\Sym^2(D_{\crys}(\rho_F))[-1].\] Then
\[\Fil^0 D_{\crys}(\Sym^2(\rho_F)(-1))\]
is generated by $w_{11}$ and $w_{12}$.

Now we know from above that the image of $D_{\crys}(E(-1))^{\phi=p^{3-2k}\alpha_p^2}$ in
\[D_{\crys}(\Sym^2(\rho_F)(-1))=D_{\mr{dR}}(\Sym^2(\rho_F)(-1))\]
is nontrivial. Call this image $D$. It is equal to
\[D_{\crys}(\Sym^2(\rho_F)(-1))^{\phi=p^{3-2k}\alpha_p^2}.\]
Write $aw_1+bw_2$, for some $a,b\in\overline\Q_p$, for a nonzero element in $D_{\crys}(\rho_F)^{\phi=p^{-(k-1)}\alpha_p}$. Then $D$ is spanned by
\[a^2w_{11}+2abw_{12}+b^2w_{22}.\]

If $D\cap\Fil^0 D_{\crys}(\Sym^2(\rho_F)(-1))$ were nontrivial, then this would force $b=0$, and so $w_1$ would be an eigenvector for the crystalline Frobenius for $D_{\crys}(\rho_F)$ with eigenvalue $p^{-(k-1)}\alpha_p$. If this is the case, then we make the same construction of this extension $E$ but with the roots $\alpha_p$ and $p^{k-1}\alpha_p^{-1}$ switched. Then the above argument shows instead that the other eigenvector, call it $w$, for the crystalline Frobenius in $D_{\crys}(\rho_F)$ would have $w^2[-1]\in D$. But writing $w=a'w_1+b'w_2$ with $a',b'\in\overline\Q_p$, we necessarily have $b'\ne 0$ (for otherwise $w$ would be a multiple of $w_1$, a contradiction to Assumption \ref{assnodblroot}). Then in this case we now have
\[D\cap\Fil^0 D_{\crys}(\Sym^2(\rho_F)(-1))=0.\]
Thus we can apply Lemma \ref{lemEisdR} to show that $E(k-2)$, and hence $E$, is de Rham. But a de Rham extension of crystalline representations is semistable, so we are done.
\end{proof}

We now prove the main theorem of this paper.

\begin{theorem}
\label{thmmainthm}
Fix a prime $p$. Let $F$ be a cuspidal eigenform of level $N$ with $p\nmid N$, and with weight $k\geq 4$ and trivial nebentypus. Let $\rho_F$ be the $p$-adic Galois representation attached to $F$. Assume:
\begin{itemize}
\item $\epsilon(1/2,\pi_F,\Sym^3)=-1$, where $\pi_F$ is the automorphic representation of $GL_2(\A)$ attached to $F$;
\item $F$ is not CM;
\item $4\nmid N$ and $9\nmid N$;
\item The Hecke polynomial of $F$ at $p$ has simple roots.
\end{itemize}
Then under Conjectures \ref{conjmult} (b) and \ref{conjliftings}, the Bloch--Kato Selmer group
\[H_f^1(\Q,\Sym^3(\rho_F)(\tfrac{4-3k}{2}))\]
is nontrivial.
\end{theorem}

\begin{proof}
We need to construct a nontrivial class in $H_f^1(\Q,\Sym^3(\rho_F)(\tfrac{4-3k}{2}))$. This is the same as constructing a nontrivial extension $E'$ of $\overline\Q_p$-Galois representations,
\[0\to\Sym^3(\rho_F)(\tfrac{4-3k}{2})\to E'\to\overline\Q_p\to 0\]
which is unramified at all primes $\ell\ne p$ and which is crystalline at $p$. By duality, this is the same as constructing such an extension $E''$,
\[0\to\overline\Q_p\to E''\to (\Sym^3(\rho_F)(\tfrac{2-3k}{2})\to 0.\]
We will construct $E''$ as a twist of a subrepresentation of $\wedge^2 E$, where $E$ is the extension of \eqref{eqEmatrix}, and then verify that it satisfies these ramification and crystallinity properties.

For $i,j\in\{1,2,3,4,5\}$, write
\[v_{ij}=v_i\wedge v_j,\]
where the vectors $v_i$ are as in Proposition \ref{propmcLbar}. Then one can compute the action of $g$ on $\wedge^2 E$ in the basis
\[v_{12},\,\, v_{13},\,\, (v_{23}-v_{14}),\,\, (v_{24}+v_{15}),\,\, v_{25},\,\, (2v_{23}+v_{14}),\,\, (v_{24}-2v_{15}),\,\, v_{34},\,\, v_{35},\,\, v_{45}.\]
The result is that the first five columns begin with the matrix
\begin{equation}
\label{eqmatrixofE''}
\frac{1}{d(g)^2}\pmat{d(g)& c_1(g)& c_2(g)& c_3(g)& c_4(g)\\
& g_{66}^3& -3g_{66}^2g_{67}& -3g_{66}g_{67}^2& g_{67}^3\\
& -g_{66}^2g_{76}& g_{66}^2g_{67}+2g_{66}g_{67}g_{76}& g_{67}^2g_{76}+2g_{66}g_{77}g_{67}& -g_{77}g_{67}^2\\
& -g_{66}g_{76}^2& g_{67}g_{76}^2+2g_{66}g_{77}g_{76}& g_{66}g_{77}^2+2g_{77}g_{67}g_{76}& -g_{77}^2g_{67}\\
& g_{76}^3& -3g_{77}g_{76}^2& -3g_{77}g_{76}^2& g_{77}^3},
\end{equation}
and the rest of the entries of these first five columns are zero;  here we have written
\begin{align*}
c_1(g)&=g_{76}g_{13}'+g_{66}g_{23}',\\
c_2(g)&=-g_{77}g_{13}'-g_{67}g_{23}'-g_{76}g_{14}'-g_{66}g_{24}',\\
c_3(g)&=-g_{77}g_{14}'-g_{67}g_{24}'+g_{76}g_{15}'+g_{66}g_{25}',\\
c_4(g)&=-g_{77}g_{15}'-g_{67}g_{25}'.
\end{align*}
One sees by direct computation that the $4$ by $4$ block in the bottom right of this matrix is the matrix of the representation $\Sym^3(\rho_1\otimes d)$ in the basis
\[v_6^3, v_6^2v_7, -v_6v_7^2, -v_7^3,\]
where $d=\chi_{\cyc}^{-1}$ and $\rho_1$ are as in Proposition \ref{propmcLbar} (b). Since $\rho_1\cong\rho_F(-\tfrac{k}{2})$, twisting by $\chi_{\cyc}$ makes \eqref{eqmatrixofE''} into an extension $E''$ of $\overline\Q_p$ by $\Sym^3(\rho_F)(\tfrac{2-3k}{2})$. It is unramified at all $\ell\nmid Np$ by Proposition \ref{propmcLbar} (d), and it is unramified at $\ell|N$ because by Proposition \ref{propmcLbar}, the functions $c_1,c_2,c_3,c_4$ above vanish on an open subgroup of $I_\ell$ for such $\ell$. If we choose the right root $\alpha_p$ of the Hecke polynomial of $F$ at $p$ as in Lemma \ref{lemEisst}, which we assume we have done, then it is moreover semistable at $p$. We just need to verify that $E''$ is a nontrivial extension, and that it is crystalline at $p$.

To see $E''$ is nontrivial, assume otherwise. Then
\[c_1(g)=c_2(g)=c_3(g)=c_4(g)=0\]
for all $g\in G_\Q$. By Lemma \ref{lemmatcoeffreln}, we also have the relations
\[2g_{77}g_{13}'+2g_{67}g_{23}'-g_{76}g_{14}'-g_{66}g_{24}'=0\]
and
\[g_{77}g_{14}'-g_{67}g_{24}'-2g_{76}g_{15}'+2g_{66}g_{25}'=0.\]
Altogether, this gives the following linear system of relations:
\[\pmat{g_{76}&g_{66}&&&&\\
g_{77}&g_{67}&g_{76}&g_{66}&&\\
2g_{77}&2g_{67}&-g_{76}&-g_{66}&&\\
&&g_{77}&g_{67}&-g_{76}&-g_{66}\\
&&g_{77}&g_{67}&2g_{76}&2g_{66}\\
&&&&g_{67}&g_{77}}\pmat{g_{13}'\\ g_{23}'\\ g_{14}'\\ g_{24}'\\ g_{15}'\\ g_{25}'}=0.\]
A quick Gaussian elimination brings the $6$ by $6$ matrix in this relation to
\[\pmat{g_{76}&g_{66}&&&&\\
g_{77}&g_{67}&&&&\\
&&-3g_{76}&-3g_{66}&&\\
&&g_{77}&g_{67}&&\\
&&&&3g_{76}&3g_{66}\\
&&&&g_{77}&g_{67}}.\]
The determinant of the above matrix is $-9d(g)^3\ne 0$, and so it is invertible, forcing
\[g_{13}'=g_{23}'=g_{14}'=g_{24}'=g_{15}'=g_{25}'=0.\]
This contradicts Lemma \ref{lemgijnonzero}, proving that the extension $E''$ is nontrivial.

It remains to show $E''$ is crystalline. We already know it is semistable. Its crystalline Frobenius eigenvalues are
\[1,\,\,p^{(3k-2)/2}\alpha_p^{-3},\,\,p^{k/2}\alpha_p^{-1},\,\,p^{-(k-2)/2}\alpha_p,\,\,p^{-(3k-4)/2}\alpha_p^3.\]
Let $N$ be the monodromy operator for $E''$. Then the relation $N\phi=p\phi N$ shows that if $N$ is nontrivial, then we must have that one of
\[p^{(3k-2)/2}\alpha_p^{-3},\,\,p^{k/2}\alpha_p^{-1},\,\,p^{-(k-2)/2}\alpha_p,\,\,p^{-(3k-4)/2}\alpha_p^3\]
equals $p$. But this would force either
\[\alpha_p^3\in\{p^{(3k-4)/2},p^{(3k-2)/2}\}\]
or
\[\alpha_p\in\{p^{(k-2)/2},p^{k/2}\}.\]
This first case is impossible since $\alpha_p$ is a $p$-Weil number of weight $(k-1)/2$, and the second case is impossible by the assumption that the Hecke polynomial of $F$ at $p$ has simple roots (see the discussion following Assumption \ref{assnodblroot}). Thus we must have $N=0$ and $E''$ is crystalline, as desired.
\end{proof}

\appendix
\section{Consequences of Arthur's conjectures}
This appendix is devoted to explaining the conjectures used in the main body above. The point of the first five subsections of this appendix is to justify our belief in the Conjectures \ref{conjmult} and \ref{conjliftings} appearing in the last two subsections. Section \ref{subsecartconj} reviews in general the conjectures of Arthur and Section \ref{subsecadjo} reviews the results of Adams--Johnson \cite{adjo}. In Section \ref{secparamsG2} we return to $G_2$ and we make a rough classification of Arthur parameters for $G_2(\R)$. Then we explicitly describe the Arthur packets for a particular family of archimedean parameters of Adams--Johnson type in Section \ref{subsecajpacket}; this result is important for describing the archimedean component of the automorphic representations appearing in Conjecture \ref{conjmult}. Then in Section \ref{subseccohparams} we classify the cohomological Arthur parameters of $G_2(\R)$. This information is used in justifying our belief in Conjecture \ref{conjliftings}.

\subsection{Arthur's conjectures}
\label{subsecartconj}
We now recall Arthur's conjectures essentially as they were formulated in \cite{artconj}. Arthur originally formulated his conjectures in terms of parameters involving the Langlands group, and this is the point of view we take in this subsection. For simplicity, we will work over $\Q$, as this case is what is used in this paper.

Let $G$ be a quasisplit group over $\Q$, and let $L_\Q$ denote the (conjectural) Langlands group of $\Q$, which should be pro-algebraic. Denote by $\L{G}$ the $L$-group of $G$, which is a particular semidirect product $G^\vee(\C)\rtimes G_\Q$. First of all, the conjectures of Langlands predict that certain homomorphisms
\[\phi:L_\Q\to\L{G},\]
considered up to $\L{G}$-conjugacy, should classify the automorphic representations of $G(\A)$; for such a conjugacy class of $\phi$'s, there should be a set $\Pi_\phi$ of automorphic representations of $G(\A)$, and these sets should satisfy certain properties, including that all such sets $\Pi_\phi$ partition the isomorphism classes of automorphic representations. We usually denote the conjugacy class of such a map $\phi$ by the same symbol as the map itself, and call that class a \textit{Langlands parameter}, or $L$\textit{-parameter}. The set $\Pi_\phi$ is called an $L$\textit{-packet}.

Arthur then enhances part of this theory to classify the representations occurring in the discrete spectrum. Let us call a Langlands parameter \textit{tempered} if its image in $\L{G}$ is bounded (equivalently, relatively compact). Then Arthur considers conjugacy classes of homomorphisms of the form
\[\psi:L_\Q\times SL_2(\C)\to\L{G}\]
where $\psi|_{L_\Q}$ is a tempered Langlands parameter and $\psi_{SL_2(\C)}$ is algebraic. Such classes of homomorphisms will be called \textit{Arthur parameters}. Arthur's conjectures will predict how these parameters $\psi$ classify the discrete spectrum of $G$. To state these conjectures, we need to consider the local part of the theory.

Let $v$ be a place of $\Q$. Write $W_{\Q_v}$ for the Weil group of $\Q_v$. If $v$ is finite, we write $W_{\Q_v}'$ for the Weil--Deligne group of $\Q_v$, which is a certain semidirect product $W_{\Q_v}\ltimes\C$. If $v=\infty$, so that $\Q_v=\R$, we write $W_\R'=W_\R$. The Langlands group $L_\Q$ is supposed to receive maps from all the groups $W_v'$. We also have the $L$-group of $G(\Q_v)$ which we denote by $\L{G_v}$; as in the global case, it is a certain semidirect product $\L{G}_v=G^\vee(\C)\rtimes G_{\Q_v}$.

We then define a \textit{local Arthur parameter} for $G(\Q_v)$ to be a conjugacy class of homomorphisms
\[\psi:W_v'\times SL_2(\C)\to\L{G}_v,\]
where $\psi|_{W_v'}$ is a local Langlands parameter which is tempered (which, for the Weil--Deligne group, is a slightly different boundedness condition which we do not explain here) and where $\psi|_{SL_2(\C)}$ is algebraic.

The conjectures of Arthur will attach packets to both local and global Arthur parameters and give a relation between these packets, and also a relation between these packets certain $L$-packets. Given a local Arthur parameter $\psi$ for $G(\Q_v)$, we consider the Langlands parameter $\phi_\psi$ given by
\[\phi_\psi(w)=\psi(w,\sm{\vert w\vert^{1/2}&0\\ 0&\vert w\vert^{-1/2}}),\qquad w\in W_v'.\]

Given a local Arthur parameter $\psi$ for $G(\Q_v)$, let $C_\psi$ be the centralizer of the image of $\psi$ in $G^\vee(\C)$, and let
\[\mc{C}_\psi=C_\psi/C_\psi^\circ Z(G^\vee(\C))\]
where $Z(G^\vee(\C))$ is the center of $G^\vee(\C)$. We define analogously the groups $C_\phi$ and $\mc{C}_\phi$ for a local Langlands parameter $\phi$. Then it is not hard to see that $\mc{C}_\psi$ surjects onto $\mc{C}_{\phi_\psi}$, and hence we have an inclusion of character sets $\widehat{\mc{C}}_{\phi_\psi}\hookrightarrow\widehat{\mc{C}}_\psi$.

In \cite{artconj}, Arthur defines, building off work of Shelstad in the tempered case, a pairing $\langle\cdot,\cdot\rangle$ on $C_\phi\times\Pi_\phi$ for any local Langlands parameter $\phi$. He explains that this pairing defines an injection $\Pi_\phi\hookrightarrow\widehat{\mc{C}}_\phi$ by $\pi\mapsto\langle\cdot,\pi\rangle$.

We now state the following conjecture, which is due to Arthur in \textit{loc. cit.}

\begin{conjecture}
\label{conjart}
\indent \,
\begin{enumerate}[label=(\arabic*)]
\item Let $v$ be a place and $\psi:W_v'\times SL_2(\C)\to\L{G_v}$ be a local Arthur parameter. Then there is a unique set $\Pi_\psi$ of irreducible admissible representations of $G(\Q_v)$, containing the $L$-packet $\Pi_{\phi_\psi}$ of $\phi_\psi$, such that $\pi\in\Pi_\psi$ belongs to $\Pi_{\phi_\psi}$ if and only if $\langle\cdot,\pi\rangle$ is in $\widehat{\mc{C}}_{\phi_\psi}$, and such that $\Pi_\psi$ satisfies the stability and endoscopy character relations given by (ii) and (iii) in \cite[Conjecture 1.3.3]{artconj}.
\item Let now $\psi:L_\Q\times SL_2(\C)\to\L{G}$ be a global Arthur parameter. Let $C_\psi$ and $\mc{C}_\psi$ be defined analogously as in the local case, and assume $C_\psi$ is finite. For each place $v$, let $\psi_v$ be the restriction to $W_v'$ of $\psi$. Let $\Pi_\psi$ be the set of representations of the form
\[\pi=\sideset{}{'}\bigotimes_v \pi_v, \qquad\pi_v\in\Pi_{\psi_v}.\]
Then there is a character $\xi_\psi:\mc{C}_\psi\to\{\pm 1\}$ such that the multiplicity with which such a $\pi$ occurs in $L_{\mr{disc}}^2(G(\Q)\backslash G(\A))$ is
\[\frac{1}{\vert\mc{C}_\psi\vert}\sum_{s\in\mc{C}_\psi}\left(\prod_v\langle s,\pi_v\rangle\right)\xi_\psi(s).\]
Moreover, the global Arthur parameters $\psi$ partition $L_{\mr{disc}}^2(G(\Q)\backslash G(\A))$ in this way.
\end{enumerate}
\end{conjecture}

Several remarks are in order. First, Arthur has refined the formula in (2) in \cite[Section 8]{artunipconj}. We will only need the precise formula for $\xi_\psi$ in a special case, where it has already been computed explicitly by Gan and Gurevich, in Section \ref{seccuspmult}. Therefore we do not explain the precise definition of $\xi_\psi$ here.

Second, the presence of the Langlands group in (2), the global case of this conjecture, may seem to put it out of reach as it is stated. However, in \cite{artclass}, Arthur introduced a substitute for his global parameters which works for classifying the discrete spectrum of the split classical groups; the automorphic representations of $GL_n$ are treated as being understood there and everything else is classified in terms of those.

However, in the first statement (1), which is the local situation, the statement of the conjecture does not depend on the existence of anything hypothetical simply because the Weil--Deligne group is defined. Furthermore, Arthur's substitute for global parameters just mentioned do localize to local parameters in a well-defined way.

In the real case, even more is known. Adams, Barbasch and Vogan constructed in \cite{ABV} packets which satisfy part (1) of the above conjecture. But earlier, Adams and Johnson were able to construct in \cite{adjo} packets which satisfy (1) for specific kinds of real parameters $\psi$. We will discuss their construction in the next section. Moreover, Arancibia, Moeglin and Renard have shown in \cite{AMR} that these constructions coincide when they overlap, and that they coincide with Arthur's real packets from \cite{artclass} when they overlap as well.

\subsection{The construction of Adams and Johnson}
\label{subsecadjo}
We consider in this section the real case of Arthur's conjectures. So fix $G$ a reductive algebraic group defined over $\R$ with complex Lie algebra $\mf{g}$. We identify $G$ with its $\R$-points. Fix a Cartan involution $\theta$ for $G$ and let $K$ be the maximal compact subgroup of $G$ defined by $\theta$. We will assume that $G$ has discrete series, so that there is a maximal torus $T$ for $G$ contained in $K$.

We will consider the $L$-group $\L{G}=G^\vee(\C)\rtimes W_\R$. Write $j$ for the element of the Weil group $W_\R$ of $\R$ such that $j^2=-1$ and $jzj^{-1}=\bar{z}$ for $z\in\C^\times\subset W_\R$.

Let $\mf{q}\subset\mf{g}$ be a $\theta$-stable parabolic subalgebra containing the complexified Lie algebra $\mf{t}$ of $T$. The corresponding $\theta$-stable Levi subgroup of $G(\R)$ is defined to be the stabilizer of $\mf{q}$ under the adjoint action of $G(\R)$. Given such a $\theta$-stable Levi $L\subset G(\R)$ corresponding to a $\theta$ stable parabolic subalgebra $\mf{q}$ containing $\mf{t}$, there is the following natural way to embed $\L{L}$ into $\L{G}$.

Fix an ordering on the roots of the complex Lie algebra $\mf{t}$ of $T$ in $\mf{g}$ which makes $\mf{q}$ standard. Then naturally $T^\vee(\C)\subset L^\vee(\C)\subset G^\vee(\C)$. Let $n_L$ be any element of the derived group of $L^\vee(\C)$ which sends positive roots in $L^\vee(\C)$ for $T^\vee(\C)$ to negative ones, and similarly define $n_G$. Let $\rho_L$ be half the positive roots of $T(\C)$ in $L(\C)$, and similarly for $\rho_G$. Then we define the embedding $\xi_{L}:\L{L}\hookrightarrow\L{G}$ first on $\C^\times\subset W_\R$ as follows. For $z\in\C^\times$, let $\xi_L(z)=\chi(z)\rtimes z$ where $\chi:\C^\times\to T^\vee(\C)$ is the unique homomorphism such that
\begin{equation}
\label{eqnLemb1}
\lambda^\vee(\chi(z))=z^{\langle\rho_G-\rho_L,\lambda^\vee\rangle}\bar{z}^{-\langle\rho_G-\rho_L,\lambda^\vee\rangle},
\end{equation}
for any cocharacter $\lambda^\vee:\C^\times\to T^\vee(\C)$. Then we define
\begin{equation}
\label{eqnLemb2}
\xi_L(j)=n_Gn_L^{-1}\rtimes j.
\end{equation}
We consider this embedding in the following definition.

\begin{definition}
\label{defAJtype}
An Arthur parameter $\psi:W_\R\times SL_2(\C)\to \L{G}$ is said to be \textit{of Adams--Johnson type} if there is a $\theta$-stable Levi subgroup $L$ of $G$ containing the compact maximal torus $T$ such that the three points below are satisfied.
\begin{itemize}
\item The restriction $\psi|_{SL_2(\C)}$ sends $\sm{1&1\\ 0&1}$ to a principal unipotent element of $L^\vee(\C)$.
\item The image of $\psi|_{\C^\times}$ lies in $Z(L^\vee(\C))\rtimes W_\R$, where $Z$ denotes the center, considered as a subgroup of $\L{G}$ via the embedding $\xi_L$ defined above.
\end{itemize}
These two points imply that $\phi_\psi$ has image contained in $\L{L}$ and that it defines a one dimensional representation of $L$. Let $\lambda$ be the restriction of this representation to $T$. Then we also require:
\begin{itemize}
\item $L$ is the stabilizer of a $\theta$-stable parabolic subalgebra $\mf{q}$ with radical $\mf{u}$ such that, for all roots $\gamma$ in $\mf{u}$, we have $\re\langle\lambda+\rho_G,\gamma^\vee\rangle\geq 0$.
\end{itemize}
\end{definition}

Adams and Johnson in \cite{adjo} have constructed packets for any parameter $\psi$ of Adams--Johnson type. Fix such a parameter $\psi$. Let $\mf{q}$ be the $\theta$-stable parabolic subalgebra of $\mf{g}$ such that $\mf{q}=\mf{l}\oplus\mf{u}$ where
\[\mf{l}=\mf{t}\oplus\bigoplus_{\langle\lambda,\gamma^\vee\rangle=0}\mf{g}^\gamma,\qquad\mf{u}=\bigoplus_{\langle\lambda,\gamma^\vee\rangle>0}\mf{g}^\gamma,\]
where the sums are over roots $\gamma$ of $\mf{t}$ in $\mf{g}$ and $\mf{g}^\gamma$ denotes the one dimensional subspace of $\mf{g}$ corresponding to the root $\gamma$.

Adams and Johnson then construct the packet attached to $\psi$ by defining certain Weyl twists of $\pi_L$, perhaps defined on other Levis which are inner forms of $L$, and cohomologically inducing them to $G$. We now review this construction.

Let $W$ be the Weyl group of $\mf{t}$ in $\mf{g}$. For $w\in W$, we can consider the $\theta$-stable parabolic subalgebra $\mf{q}_w$ of $\mf{g}$ given as follows. Let $\Delta(\mf{q},\mf{t})$ be the set of roots of $\mf{t}$ in $\mf{q}$, so that
\[\mf{q}=\bigoplus_{\gamma\in\Delta(\mf{q},\mf{t})}\mf{g}^\gamma\oplus\mf{t}.\]
Then we define $\mf{q}_w$ by
\[\mf{q}_w=\bigoplus_{\gamma\in\Delta(\mf{q},\mf{t})}\mf{g}^{w\gamma}\oplus\mf{t}.\]
Let $L_w$ be the $\theta$-stable Levi corresponding to $\mf{q}_w$.

Adams and Johnson show that there is a one dimensional representation of $L_w$ whose restriction to $T$ is $w\lambda$. Let $\pi_{L_w}$ be any such representation.

For $i\geq 0$, we consider the cohomological induction functors $\mc{R}_{\mf{q}}^i$ from $(\mf{l},L\cap K)$-modules to $(\mf{g},K)$-modules as defined in \cite[Section V.1]{knvo}. These are normalized so that if $Z$ is an $(\mf{l},L\cap K)$-module with infinitesimal character given by $\Lambda\in\mf{t}^\vee$, then $\mc{R}_{\mf{q}}^i(Z)$ has infinitesimal character given by $\Lambda+\rho(\mf{u})$, where $\rho(\mf{u})$ half the sum of the roots of $\mf{t}$ in $\mf{u}$ (see \cite[Corollary 5.25]{knvo}).

For $w\in W$, let
\[S_w=\frac{1}{2}(\dim(K)-\dim(L_w\cap K)).\]
Let
\[A_{\mf{q}_w}(w\lambda)=\mc{R}_{\mf{q}}^{S_w}(\pi_{L_w}).\]
Let $W_L$ be the Weyl group of $\mf{t}$ in $\mf{l}$, and $W_c$ the Weyl group of $\mf{t}$ in the complex Lie algebra of $K$. Adams and Johnson show that if $w,w'\in W$ define the same double coset in
\[W_c\backslash W/W_L,\]
then $A_{\mf{q}_w}(w\lambda)\cong A_{\mf{q}_{w'}}(w'\lambda)$.

\begin{definition}
With $\psi$ a parameter of Adams--Johnson type as above, we define the corresponding Adams--Johnson packet to be
\[\Pi_\psi^{\mr{AJ}}=\sset{A_{\mf{q}_w}(w\lambda)}{w\in W_c\backslash W/W_L}.\]
\end{definition}

As mentioned before, these packets satisfy the conclusion of Arthur's conjecture in the archimedean case. In the next two sections, we will study the Arthur parameters for $G_2(\R)$ whose restriction to $SL_2(\C)$ is nontrivial. Many of these will turn out to be of Adams--Johnson type, and we will be able to compute the corresponding Adams--Johnson packets.

\subsection{Arthur parameters for $G_2(\R)$}
\label{secparamsG2}
We now study the Arthur parameters for $G_2(\R)$. Let $\psi:W_\R\times SL_2(\C)\to \L(G_2(\R))$ be an Arthur parameter which is nontrivial on the $SL_2(\C)$. By the Jacobson--Morozov theorem, the conjugacy class of the homomorphism $\psi|_{SL_2(\C)}$ is determined by the conjugacy class of the unipotent element $\psi\sm{1&1\\ 0&1}$. Such a class is called a unipotent orbit, and there are five such orbits for $G_2$, described as follows. For a root $\gamma$ of the split maximal torus $T$ in $G_2$, let $X_\gamma$ be a root vector in the Lie algebra $\mf{g}_2$ corresponding to $\gamma$. Then:
\begin{itemize}
\item There is the orbit of the identity element, for which we write $\mc{O}_0$.
\item There is the long root orbit, which is that of $\exp(X_\alpha)$. We write $\mc{O}_l$ for this orbit.
\item There is the short root orbit, which is that of $\exp(X_\beta)$. We write $\mc{O}_s$ for this orbit.
\item There is the subregular orbit, which is that of $\exp(X_\alpha+X_{\alpha+3\beta})$. We write $\mc{O}_{sr}$ for this orbit.
\item There is the regular orbit, which is that of $\exp(X_\alpha+X_\beta)$. We write $\mc{O}_r$ for this orbit.
\end{itemize}
The respective dimensions of these orbits are $0$, $6$, $8$, $10$, and $12$. The closure of each contains the previous one.

Let us write $\Psi_{?}(G_2(\R))$, for $?\in\{0,l,s,sr,r\}$, for the set of Arthur parameters $\psi$ such that $\psi|_{SL_2(\C)}$ corresponds to the orbit $\mc{O}_?$. We aim to classify the Arthur parameters $\psi$ for which $\psi|_{SL_2(\C)}$ is nontrivial; that is, we will describe $\Psi_{?}(G_2(\R))$ for $?\ne 0$.

For $\gamma$ a root of $T$, let us write $SL_{2,\gamma}(\C)$ for the $SL_2$-subgroup of $G_2(\C)$ corresponding to $\gamma$. Then if $\gamma$ and $\gamma'$ are orthogonal roots, then $SL_{2,\gamma}(\C)$ and $SL_{2,\gamma'}(\C)$ are mutual centralizers and their inclusions into $G_2(\C)$ induce a map,
\[\iota_{\gamma,\gamma'}:SL_{2,\gamma}(\C)\times SL_{2,\gamma'}(\C)\to G_2(\C)\]
with kernel $\{\pm 1\}$ embedded diagonally.

Let us fix a compact maximal torus $T_c\subset G_2(\R)$ contained in a maximal compact subgroup $K$ of $G_2(\R)$, and let $\theta$ be a Cartan involution giving $K$. Let $\mf{t}_c$ be the complex Lie algebra of $T_c$. We identify the root system of $\mf{t}_c$ in $\mf{g}_2$ with that of $T$ in $G_2$ in such a way that $\beta$ and $2\alpha+3\beta$ are the positive compact roots.

Let $\mf{q}$ be the $\theta$-stable parabolic subalgebra of $\mf{g}_2$ whose Levi factor $\mf{l}$ contains the roots $\pm\beta$, and whose radical $\mf{u}$ contains the positive roots different from $\beta$. Let $L$ be the $\theta$-stable Levi corresponding to $\mf{q}$. It is isomorphic to the unitary group $U(2)$ since the roots $\pm\beta$ are compact. On the dual side, we identify $L^\vee(\C)$ with $GL_2(\C)$.

Given an even integer $k\geq 2$, we let $\psi_{L,k}:W_\R\times SL_2(\C)\to\L{L}$ be the homomorphism given by
\begin{gather*}
\psi_{L,k}(z)=\pmat{(z/\vert z\vert)^{k-4}&0\\ 0&(z/\vert z\vert)^{k-4}}\times z\in L^\vee(\C)\rtimes W_\R,\qquad z\in\C^\times,\\
\psi_{L,k}(A)=A\times 1,\qquad A\in SL_2(\C),\\
\psi_{L,k}(j)=1\times j.
\end{gather*}
Note that this is a homomorphism because $k$ is even, so that $\psi(j)^2=\psi(-1)$. Let $\xi_L:\L{L}\to\L{(G_2(\R))}$ be the embedding considered in \eqref{eqnLemb1} and \eqref{eqnLemb2}, and define
\[\psi_k=\xi_L\circ\psi_{L,k}.\]
Then $\psi_k$ is an Arthur parameter of Adams--Johnson type. As in Definition \ref{defAJtype}, the Langlands parameter $\phi_\psi$, when viewed as having target $\L{L}$, defines a one dimensional representation of $L$ and hence a character $\lambda_k$ of $T_c$. One checks that $\lambda_k=\frac{k-4}{2}(2\alpha+3\beta)$.

\begin{proposition}
\label{propPsil}
Let $\psi\in\Psi_l(G_2(\R))$. Then $\psi$ factors as
\begin{align*}
\psi:W_\R\times SL_2(\C)&\xrightarrow{\phi\times\id_{SL_2(\C)}}W_\R\times SL_{2,\beta}(\C)\times SL_{2,2\alpha+3\beta}(\C)\\
&\xrightarrow{\id_{W_\R}\times\iota_{\beta,2\alpha+3\beta}}W_\R\times G_2(\C)=\L{(G_2(\R))},
\end{align*}
where $\phi:W_\R\to W_\R\times SL_2(\C)$ is a tempered Langlands parameter for $PGL_2(\R)$. If $\psi$ is unipotent, then $\phi|_{\C^\times}$ is trivial and $\phi(j)\in\{\pm 1,\sm{1&0\\ 0&-1}\}$. Otherwise, $\psi=\psi_k$ for some $k\geq 2$ even; in this case, $\phi$ is the Langlands parameter of the discrete series of weight $k$ for $PGL_2(\R)$.

Finally, the component groups are given as follows. If $\psi$ is unipotent, then $\mc{C}_\psi$ is trivial. Otherwise $\mc{C}_\psi$ has two elements.
\end{proposition}

\begin{proof}
Since $\psi\in\Psi_l(G_2(\R))$, by conjugating we may assume $\psi$ identifies $SL_2(\C)$ with $SL_{2,2\alpha+3\beta}(\C)$. Then since $SL_{2,\beta}(\C)$ is the centralizer of $SL_{2,2\alpha+3\beta}(\C)$ in $G_2(\C)$, $\psi|_{W_\R}$ must factor through $SL_{2,\beta}(\C)\times W_\R$. Since $\psi|_{W_\R}$ has bounded image, it therefore defines a tempered Langlands parameter, which we take to be our $\phi$, of $PGL_2(\R)$.

If $\psi$ is unipotent, then $\phi(j)$ is of order 2, and must be either $\pm 1$ or a conjugate of $\sm{1&0\\ 0&-1}$. Otherwise $\phi$ is a discrete series parameter of positive even weight $k$. Assuming this, we will now show $\psi=\psi_k$.

Recall that the discrete series of weight $k$ for $PGL_2(\R)$ has Langlands parameter given by
\[z\mapsto\pmat{(z/\vert z\vert)^{k-1}&0\\ 0&(z/\vert z\vert)^{1-k}}\times z\in SL_2(\C)\times W_\R,\qquad z\in\C^\times,\]
and
\[j\mapsto\pmat{0&-1\\1&0}\times j\in SL_2(\C)\times W_\R.\]
The two equations above therefore define $\phi=\psi|_{W_\R}$ when identifying the $SL_2(\C)$'s with $SL_{2,\beta}(\C)$. Recall also that the Levi $L$ used to define $\psi_k$ is a short root Levi, and therefore we may identify $L^\vee(\C)$ with the long root Levi in $G_2(\C)$ given by
\[T_\beta\times SL_{2,2\alpha+3\beta}(\C)/\{\pm 1\},\]
where $T_\beta$ is the maximal torus in $SL_{2,\beta}(\C)$. That $\psi=\psi_k$ then follows from a straightforward computation using the definitions; one must use \eqref{eqnLemb1} and \eqref{eqnLemb2}, noting that $\rho_G-\rho_L=\frac{3}{2}(2\alpha+3\beta)$ and that we may take
\[n_Gn_L^{-1}=\pmat{0&-1\\ 1&0}\times 1\in SL_{2,\beta}(\C)\times SL_{2,2\alpha+3\beta}(\C).\]

Finally, for the component groups, we note that since $\psi|_{SL_2(\C)}$ has image $SL_{2,2\alpha+3\beta}(\C)$, the centralizer $C_\psi$ is the subgroup of $SL_{2,\beta}(\C)$ which centralizes the image of $\psi|_{W_\R}$. If $\psi$ is unipotent and $\psi(j)$ is central, then $C_\psi=SL_{2,\beta}(\C)$; if $\psi$ is unipotent and $\psi(j)=\sm{1&0\\ 0&-1}$, then $C_{\psi}=T_\beta$; if $\psi=\psi_k$, then $C_\psi=\{\pm 1\}$. The result follows.
\end{proof}

We omit the details, but a completely analogous analysis can be made of $\Psi_s(G_2(\R))$ by switching the roles of long and short roots.

Now we look at the subregular parameters.

\begin{proposition}
\label{propPsisr}
There are two parameters in $\Psi_{sr}(G_2(\R))$ and they are both unipotent. One has component group $S_3$, the symmetric group on three elements, and the other has component group $\Z/2\Z$.
\end{proposition}

\begin{proof}
Since $\mc{O}_{sr}$ has dimension $10$, the centralizer of any of its representatives is $4$ dimensional. On the level of Lie algebras, $X_\alpha+X_{\alpha+3\beta}$ is centralized by the four independent nilpotent elements
\[X_{\alpha+\beta},\quad X_{\alpha+2\beta},\quad X_{2\alpha+3\beta},\quad X_\alpha+X_{\alpha+3\beta},\]
which therefore span the centralizer of $X_\alpha+X_{\alpha+3\beta}$. The other nilpotent element in the $\sl_2$-triple containing $X_\alpha+X_{\alpha+3\beta}$ is (up to scalar) $X_{-\alpha}+X_{-\alpha-3\beta}$, which is centralized by
\[X_{-\alpha-\beta},\quad X_{-\alpha-2\beta},\quad X_{-2\alpha-3\beta},\quad X_{-\alpha}+X_{-\alpha-3\beta}.\]
Thus the centralizer of the entire $\sl_2$-triple in $\mf{g}_2$ is trivial.

It follows that the centralizer of the image of the homomorphism $SL_2(\C)\to G_2(\C)$ corresponding to $\mc{O}_{sr}$ is discrete. Now it is a fact that the component group of the centralizer of $X_\alpha+X_{\alpha+3\beta}$ is $S_3$. It is easy to check that the Weyl group element $w_{\beta}$, along with $\beta^\vee(\zeta_3)$ with $\zeta_3$ a third root of unity, centralize this $SL_2(\C)$ and generate a group isomorphic to $S_3$ which is therefore the centralizer of this $SL_2(\C)$ in $G_2(\C)$.

Now let $\psi\in\Psi_{sr}(G_2(\C))$. Then $\psi|_{W_\R}$ must have image contained in this $S_3$ subgroup, which implies $\psi$ is unipotent and $\psi(j)=1$ or $\psi(j)$ is an order two element of this $S_3$ (all of which are conjugate, even in $S_3$ itself). In the former case, $C_\psi=S_3$, and in the latter case, $C_\psi=\{1,\psi(j)\}$.
\end{proof}

Finally, we examine $\Psi_r(G_2(\R))$. The analysis is similar to the above proposition.

\begin{proposition}
\label{propPsir}
There is only one parameter in $\Psi_r(G_2(\R))$, and it is trivial on $W_\R$. Its component group is trivial.
\end{proposition}

\begin{proof}
The component group of the centralizer of $X_\alpha+X_\beta$ in $G_2(\C)$ is known to be trivial. The centralizer of this element in $\mf{g}_2$ is $2$ dimensional (because $\mc{O}_r$ has dimension 12) and contains (and is thus spanned by) $X_{\alpha+2\beta}$ and $X_{2\alpha+3\beta}$. Since $X_{-\alpha}+X_{-\beta}$ is centralized by $X_{-\alpha-2\beta}$ and $X_{-2\alpha-3\beta}$, this implies that the centralizer of the regular $SL_2(\C)$ in $G_2(\C)$ is discrete, hence trivial.

Thus if $\psi\in\Psi_r(G_2(\R))$, then $\psi|_{W_\R}$ is trivial, and so is the centralizer of the image of $\psi$. The proposition follows.
\end{proof}

\subsection{Determination of the packet $\Pi_{\psi_k}^{\mr{AJ}}$}
\label{subsecajpacket}
For an even integer $k\geq 2$, let $\psi_k\in\Psi_{l}(G_2(\R))$ be the parameter of Adams--Johnson type defined in the previous section. It gives rise to the character $\lambda_k=\frac{k-4}{2}(2\alpha+3\beta)$, where $2\alpha+3\beta$ is viewed as a long compact root for a compact torus $T_c$. We also let $\mf{q}$ and $L$ be as in the construction of $\psi_k$. Then $L\cong U(2)$.

For the remainder of this section, we let $w\in W$ be the rotation counterclockwise by $\pi/3$ in the root system, so $w(2\alpha+3\beta)=\alpha$ and $w\beta=\alpha+2\beta$. It represents the nontrivial double coset in
\[W_c\backslash W(G_2,T_c)/W_L.\]

By definition, the Adams--Johnson packet for $\psi_k$ is
\[\Pi_{\psi_k}^{\mr{AJ}}=\{A_{\mf{q}}(\lambda_k),A_{\mf{q}_w}(w\lambda_k)\}.\]
We determine these representations when $k\geq 4$. Write $\rho=\rho_{G_2}=3\alpha+5\beta$.

\begin{proposition}
\label{propajqds}
Let $k\geq 4$. Then the representation $A_{\mf{q}}(\lambda_k)$ is the discrete series representation with Harish-Chandra parameter $\frac{k-4}{2}(2\alpha+3\beta)+\rho$. In the terminology of Gan--Gross--Savin \cite{ggs}, this is the quaternionic discrete series of weight $k/2$.
\end{proposition}

\begin{proof}
We will use the spectral sequence of \cite[Theorem 11.77]{knvo}. Let $\mf{b}$ be the standard Borel in $\mf{g}_2$ containing the complex Lie algebra $\mf{t}_c$ of $T_c$, and $\mf{n}$ its radical. Let $\mf{u}$ be the radical of $\mf{q}$ and $\mf{l}$ its Levi factor. Then in our case, this spectral sequence reads
\[\mc{R}^i(\mc{R}^j(Z\otimes\C_{-2\rho(\mf{n}\cap\mf{l})})\otimes\C_{-2\rho(\mf{u})})\Rightarrow\mc{R}^{i+j}(Z\otimes\C_{-2\rho(\mf{n})}),\]
for $(\mf{t}_c,T_c)$-modules $Z$, where the $\mc{R}$'s denote cohomological inductions, the $\rho$'s denote the obvious half sums of roots, and the modules $\C_\mu$ for weights $\mu$ are the obvious $1$ dimensional modules. (See also \cite[(11.73)]{knvo} for the discrepancy that gives rise to these half sums.)

Let $\lambda_k'$ be the character of $T_c$ given by
\[\lambda_k'=\frac{k-4}{2}(2\alpha+3\beta)+(6\alpha+10\beta).\]
We apply the spectral sequence above with $Z=\lambda_k'$. We note that
\[2\rho(\mf{n}\cap\mf{l})=\beta,\qquad 2\rho(\mf{u})=6\alpha+9\beta,\qquad 2\rho(\mf{n})=6\alpha+10\beta.\]

Now on the one hand, by the classification of discrete series via cohomological induction \cite[Theorem 11.178(a)]{knvo},
\[\mc{R}^2(\lambda_k'\otimes\C_{-2\rho(\mf{n})})=\mc{R}^2(\tfrac{k-4}{2}(2\alpha+3\beta))\]
is the discrete series of $G_2(\R)$ sought, with Harish-Chandra parameter $\frac{k-4}{2}(2\alpha+3\beta)+\rho$. On the other hand, by the same theorem,
\[\mc{R}^1(\lambda_k'\otimes\C_{-2\rho(\mf{n}\cap\mf{l})})=\mc{R}^1(\lambda_k'\otimes\C_{-\beta})\]
is the discrete series representation of $L$ with Harish-Chandra parameter $\lambda_k'-\frac{1}{2}\beta$; i.e., it is the character of $L\cong U(2)$ whose restriction to $T_c$ is $\frac{k-4}{2}+(6\alpha+9\beta)$. Thus
\[\mc{R}^1(\lambda_k'\otimes\C_{-2\rho(\mf{n}\cap\mf{l})})\otimes\C_{-2\rho(\mf{u})}\]
is the character of $L$ given by $\lambda_k=\frac{k-4}{2}(2\alpha+3\beta)$. Cohomologically inducing again gives
\[\mc{R}^1(\mc{R}^1(\lambda_k'\otimes\C_{-2\rho(\mf{n}\cap\mf{l})})\otimes\C_{-2\rho(\mf{u})})=A_\mf{q}(\lambda_k)\]
by definition.

Now since all representations considered here have infinitesimal character in the good range (this is where we use $k\geq 4$) these cohomological inductions are concentrated in one degree (see \cite[Theorem 0.50]{knvo}) and the spectral sequence collapses. Thus,
\[\mc{R}^1(\mc{R}^1(\lambda_k'\otimes\C_{-2\rho(\mf{n}\cap\mf{l})})\otimes\C_{-2\rho(\mf{u})})=\mc{R}^2(\lambda_k'\otimes\C_{-2\rho(\mf{n})}),\]
which, by the above computations, proves the proposition.
\end{proof}

Now we consider the other representation $A_{\mf{q}_w}(w\lambda_k)$ in the packet. We note that the $\theta$-stable Levi $L_w$ associated with $\mf{q}_w$ is a $U(1,1)$ because its complex Lie algebra $\mf{l}_w$ contains the root $w\beta=\alpha+2\beta$, which is noncompact. Also, we have $w\lambda_k=\frac{k-4}{2}\alpha$.

It is easy to see by our description of the parameter $\psi_k$ in Proposition \ref{propPsil} that the Langlands parameter $\phi_{\psi_k}$ is that of the Langlands quotient $\mc{L}_\alpha(\pi_k,1/10)$ of the induction of the discrete series $\pi_k$ of weight $k$ from the long root parabolic $P_\alpha$. Since the Adams--Johnson packet of $\psi_k$ should contain the $L$-packet of $\phi_{\psi_k}$, we should have that this Langlands quotient is the remaining member of our packet. However, to my knowledge, no direct proof of this is written in the literature. So we give a direct proof in this case for the sake of completeness.

\begin{proposition}
With the notation as above, we have
\[A_{\mf{q}_w}(w\lambda_k)\cong\mc{L}_\alpha(\pi_k,1/10)\]
if $k\geq 4$.
\end{proposition}

\begin{proof}
The key to this proposition is a theorem in Vogan's book, \cite[Theorem 6.6.15]{voganbook}, which links the composition of ordinary parabolic induction with cohomological induction with the composition in the opposite order. Instead of recalling the theorem in general, we explain what it means in our special case. It requires three types of data as input: We need what Vogan calls $\theta$-stable data, character data, and cuspidal data, which are defined in general in Definitions 6.5.1, 6.6.1, and 6.6.11, respectively, in Vogan's book. Moreover, there is a bijection between these first two kind of data (\cite[Proposition 6.6.2]{voganbook}) and a surjective map from pieces of character data to pieces of cuspidal data (\cite[Proposition 6.6.12]{voganbook}). Pieces of $\theta$-stable data are used to construct cohomological inductions of parabolically induced representations and in our case will be used to realize the representation $A_{\mf{q}_w}(w\lambda_k)$. On the other hand, cuspidal data are used to construct parabolic inductions of discrete series representations and will be used to realize $\mc{L}_\alpha(\pi_k,1/10)$. Then \cite[Theorem 6.6.15]{voganbook} will give that these two constructions coincide. We note that this theorem is stated in terms of Langlands subrepresentations instead of Langlands quotients, so we have to make a few minor adjustments.

To build the $\theta$-stable data we need, we first construct a certain $\theta$-stable maximal torus of $G_2(\R)$. Let $T_0$ be the center of $L_w$. Let $A$ be the $\theta$-stable maximal split torus in the derived group of $L_w$. Then $H=T_0A$ is a maximal torus in $G_2(\R)$. It is neither split nor compact. Let $\mu:T_0\to\C^\times$ be given by
\[\mu=w\lambda_k|_{T_0}=\frac{k-4}{2}\alpha|_{T_0}.\]
Fix a minimal parabolic subgroup $B_w$ in $L_w$ containing $H$, and let $\nu:A\to\C^\times$ be the character given by
\[\nu=\delta_{B_w}^{-1/2}|_{A}.\]
Then the quadruple $(\mf{q}_w,H,\mu,\nu)$ is a piece of $\theta$-stable data in the sense of \cite{voganbook}. We write $\mu\otimes\nu$ for the character of $H$ given by $\mu$ on $T_0$ and by $\nu$ on $A$, and we construct the representation (called a \textit{standard module} for our data) given by
\begin{equation}
\label{eqstdmod1}
\mc{R}^2(\Ind_{B_w}^{L_w}((\mu\otimes\nu)\otimes\delta_{B_w}^{1/2})).
\end{equation}
Of course, in the parabolic induction, the characters $\nu$ and $\delta_{B_w}^{1/2}$ cancel, and the parabolic induction thus becomes
\[\Ind_{B_w}^{L_w}(\mu\otimes 1).\]
By definition of $\mu$, this contains the one dimensional representation $\pi_w$ of $L_w$ given by $w\lambda_k$ as its unique subrepresentation. Since cohomological induction is exact in the good range (again, we use $k\geq 4$ here) we see that $\mc{R}^2(\pi_w)$ is a subrepresentation of \eqref{eqstdmod1}.

Now we construct a piece of character data from $(\mf{q}_w,H,\mu,\nu)$ as in \cite{voganbook}. For us this will be a pair $(H,\Gamma)$ where $\Gamma:H\to\C^\times$ is a character satisfying certain properties. (Actually, Vogan's definition contains also the data of a character of the complex Lie algebra $\mf{h}$ of $H$, but that character is determined from the differential of $\Gamma$.) We set $\Gamma|_A=\nu$, and we let $\Gamma|_{T_0}$ be the product of $\mu$ with the restriction to $T_0$ of the character $\det(\mf{g}_2^{\theta=-1}\cap\mf{u}_w)$, where $\mf{u}_w$ is the radical of $\mf{q}_w$. This latter character is equal to the sum of noncompact roots in $\mf{u}_w$, and is therefore given by
\[\alpha+(\alpha+\beta)-(\alpha-3\beta)=\alpha-2\beta=2\alpha-(\alpha+2\beta).\]
Its restriction to $T_0$ is therefore given by $2\alpha$, and thus
\[\Gamma|_{T_0}=\tfrac{k}{2}\alpha|_{T_0}.\]

From $(H,\Gamma)$ we construct another piece of data, which Vogan calls cuspidal data. Consider the centralizer of $A$ in $G_2(\R)$; this is a Levi subgroup of $G_2(\R)$, and we write $MA$ for its Langlands decomposition. The torus $A$ was a maximal split torus in a short root $SL_2(\R)$, and it follows that $M$ is a long root $SL_2(\R)$ in $G_2(\R)$. Therefore there is a long root parabolic $P=MAN$ in $G_2(\R)$.

A piece of cuspidal data constructed from $(H,\Gamma)$ will consist of the Levi $MA$, along with a character of $A$, which is given by $\Gamma|_A=\nu$, and also a discrete series representation $\pi_0$ of $MA$. This latter representation is given as the cohomological induction of $\mu'\otimes\nu$ from $H$ to $MA$, where $(\mf{q}',H,\mu',\nu)$ is the $\theta$-stable data for $MA$ obtained from the restriction of the character data $(H,\Gamma)$ to $MA$. In this data, the $\theta$-stable parabolic $\mf{q}'$ is the intersection of $\mf{q}_w$ with the complex Lie algebra $\mf{m}\oplus\mf{a}$ of $MA$. It contains the noncompact root $\alpha$ in its radical. The character $\mu'$ is the restriction of $\Gamma$ to $T_0$ multiplied by the inverse of the sum of the noncompact roots in the radical of $\mf{q}'$. Thus it is equal to $\frac{k-2}{2}\alpha|_{T_0}$.

Then \cite[Theorem 6.6.15]{voganbook} asserts that \eqref{eqstdmod1} is isomorphic to
\begin{equation}
\label{eqstdmod2}
\Ind_{P}^{G_2(\R)}(\mc{R}^1(\mu'\otimes\nu)\otimes\delta_P^{1/2}).
\end{equation}
The cohomological induction in this expression is, by \cite[Theorem 11.178(a)]{knvo}, the twist of the discrete series representation of $MA$ of weight $k$ by the character $\det^{-1/2}$. Since $P$ is a long root parabolic, $\det^{-1/2}=\delta_P^{-1/10}|_{MA}$, and we get that \eqref{eqstdmod2}, and also hence \eqref{eqstdmod1}, are isomorphic to the normalized induction
\[\iota_{M_\alpha(\R)}^{G_2(\R)}(\pi_k,-1/10).\]
Since $\pi_k$ is self dual, the unique irreducible subrepresentation of this is, by dualizing, isomorphic to $\mc{L}_\alpha(\pi_k,1/10)$, and this is isomorphic to $\mc{R}^2(\pi_w)$ by above. This is what we wanted to prove.
\end{proof}

We summarize the above results as a theorem.

\begin{theorem}
\label{thmAJpacket}
The Adams--Johnson packet $\Pi_{\psi_k}^{\mr{AJ}}$ consists of the quaternionic discrete series of weight $k/2$, of Harish-Chandra parameter $\frac{k-4}{2}(2\alpha+3\beta)+\rho$, and the Langlands quotient $\mc{L}_\alpha(\pi_k,1/10)$ of the discrete series of weight $k$ from the long root parabolic $P_\alpha(\R)$ of $G_2(\R)$.
\end{theorem}

\subsection{Cohomological parameters}
\label{subseccohparams}
We would like to describe all Arthur parameters for $G_2(\R)$ whose associated Arthur packets contain a representation with cohomology. This will not be so difficult from what we have set up. However, we need to specify what we mean by ``Arthur packets'' for parameters which are not of Adams--Johnson type.

In \cite{ABV}, Adams, Barbasch and Vogan define Arthur packets very generally for parameters for real groups, and prove that their packets satisfy the conclusion of Arthur's conjecture \cite[Conjecture 1.3.3]{artconj}. These packets are hard to compute in general, but in the case of unipotent parameters, they can be shown to give the unipotent representations constructed by the methods of Barbasch--Vogan \cite{BV}; see \cite[Corollary 27.13]{ABV}. They are also known to coincide with the packets constructed by Adams--Johnson for parameters of Adams--Johnson type; see \cite{aran}. We will write $\Pi_\psi=\Pi_\psi^{\mr{ABV}}$ for the Arthur packet for a real Arthur parameter $\psi$ as constructed by Adams--Barbasch--Vogan.

Let $\psi$ be an Arthur parameter for $G_2(\R)$ which is nontrivial on $SL_2(\C)$. By the results of Section \ref{secparamsG2}, any such parameter is either unipotent or of Adams--Johnson type, and therefore we can compute the representations in the packets $\Pi_\psi$ via the methods of Adams--Johnson or those of Barbasch--Vogan. If, on the other hand, $\psi$ is trivial on $SL_2(\C)$, then $\psi|_{W_\R}=\phi_\psi$. Therefore $\Pi_\psi$ is just the $L$-packet attached to the tempered Langlands parameter $\psi|_{W_\R}$.

\begin{proposition}
\label{propcohparamsG2}
Let $\psi$ be an Arthur parameter for $G_2(\R)$. Assume $\Pi_\psi$ contains a representation with cohomology. Then exactly one of the following holds.
\begin{itemize}
\item We have $\psi\in\Psi_0(G_2(\R))$. In this case $\psi|_{W_\R}$ is the Langlands parameter for a discrete series representation, and $\Pi_\psi$ is the corresponding discrete series $L$-packet.
\item We have $\psi\in\Psi_l(G_2(\R))$. In this case $\psi=\psi_k$  for some even $k\geq 4$ in the notation of Proposition \ref{propPsil}, and the representations in $\Pi_\psi$ are both cohomological for the irreducible representation of $G_2$ of highest weight $\frac{k-4}{2}(2\alpha+3\beta)$; moreover, the $L$-packet element has cohomology exactly in degrees $3$ and $5$, and the other has cohomology exactly in degree $4$.
\item We have $\psi\in\Psi_s(G_2(\R))$ and $\psi$ is obtained in the same way as $\psi_k$ with $k\geq 4$ as in Proposition \ref{propPsil}, except that $SL_{2,\beta}(\C)$ and $SL_{2,2\alpha+3\beta}(\C)$ are switched in the construction. There are two representations in $\Pi_\psi$ and they are both cohomological for the irreducible representation of $G_2$ of highest weight $\frac{k-4}{2}(\alpha+2\beta)$. Moreover, the $L$-packet element has cohomology exactly in degrees $3$ and $5$, and the other has cohomology exactly in degree $4$.
\item We have $\psi\in\Psi_r(G_2(\R))$. Then $\Pi_\psi$ contains only the trivial representation of $G_2(\R)$.
\end{itemize}
\end{proposition}

\begin{proof}
As in Section \ref{secparamsG2}, we classify the Arthur parameters according to their restriction to $SL_2(\C)$. If $\Psi|_{SL_2(\C)}$ is trivial, then $\phi_\psi=\psi|_{W_\R}$, which is tempered by definition, and moreover $\Pi_\psi=\Pi_{\phi_\psi}$. Being cohomological and tempered, $\phi_\psi$ must correspond to a discrete series $L$-packet, which proves the proposition in the case that $\psi\in\Psi_0(G_2(\R))$.

Next, let $\psi\in\Psi_l(G_2(\R))$ be cohomological. We note that $\psi$ cannot be unipotent because the representations corresponding to unipotent parameters were classified by Vogan in \cite[Theorem 18.3]{voganG2}, and they all have irregular infinitesimal character. (Note that Vogan's list contains one extra unipotent representation because he is working with the double cover of $G_2(\R)$ and one of the representations he obtains does not factor through the projection to $G_2(\R)$.)

Thus $\psi=\psi_k$ for some $k\geq 2$ even, as in Proposition \ref{propPsil}. But we cannot have $k=2$ because in this case the Adams--Johnson packets for $\psi_2$ contain only representations with irregular infinitesimal character; they are $A_{\mf{q}}(\lambda)$'s with $\lambda=-(2\alpha+3\beta)$ and therefore have infinitesimal character $\rho-(2\alpha+3\beta)=\alpha+2\beta$.

Thus $\psi=\psi_k$ for $k\geq 4$, and it follows from Theorem \ref{thmAJpacket} that the $L$-packet element of $\Pi_{\psi_k}$ is cohomological in degrees $3$ and $5$, while the other element is cohomological in degree $4$. Alternatively, one can see this directly from the presentation of these representations as $A_{\mf{q}}(\lambda)$'s using \cite[Theorem 5.5]{VZ}. This proves the proposition in case $\psi\in\psi_l(G_2(\R))$.

The case of $\psi\in\Psi_s(G_2(\R))$ is completely analogous; one uses instead \cite[Theorem 18.4]{voganG2} to handle the unipotent representations.

 Next we note that $\psi$ cannot be in $\Psi_{sr}(G_2(\R))$, for then by Proposition \ref{propPsisr}, $\psi$ is unipotent and the unipotent representations for these parameters were classified in \cite[Theorem 18.5]{voganG2}; they all have irregular infinitesimal character.

Finally, by Proposition \ref{propPsir}, there is only one parameter $\psi$ in $\Psi_{r}(G_2(\R))$ and $\psi|_{W_\R}$ is trivial. Moreover, $\Pi_\psi=\Pi_{\phi_\psi}$ because the component groups are trivial. We have that the restriction of $\psi$ to the maximal torus of $SL_2(\C)$ is $(\alpha+2\beta)^\vee+(2\alpha+3\beta)^\vee=3\beta^\vee+5\alpha^\vee$, and therefore $\phi_\psi$ is the parameter corresponding to the Langlands quotient from the Borel of the character $\rho$; that is, it is the trivial representation. Since the trivial representation is cohomological in degrees $0$ and $8$, we are done.
\end{proof}

\subsection{Occurrence of $\mc{L}_\alpha(\pi_F,1/10)$ in the cuspidal spectrum}
\label{seccuspmult}
We begin by making the following conjecture.

\begin{conjecture}
\label{conjmult}
Let $F$ be a cuspidal holomorphic eigenform of weight $k\geq 4$ and trivial nebentypus, $\pi_F$ its associated automorphic representation, and $\Pi_{F,f}=\mc{L}_\alpha(\pi_F,1/10)_f$ the finite part Langlands quotient of $\pi_F\otimes\vert\det\vert^{1/2}$ from the long root parabolic in $G_2$. Moreover, let $v$ be a place at which $\pi_F$ is unramified, and let $\Pi_{F,f}^v$ be the component of $\Pi_{F,f}$ away from $v$. Assume $L(1/2,\pi_F,\Sym^3)=0$. Then the $\Pi_{F,f}^v$-isotypic component of the cuspidal spectrum is given by
\[L_{\cusp}^2(G_2(\Q)\backslash G_2(\A))[\Pi_{F,f}^v]=\Pi_{F,f}\otimes\Pi_\infty,\]
where:
\begin{enumerate}[label=(\alph*)]
\item In the case that $\epsilon(1/2,\pi_F,\Sym^3)=1$, the representation $\Pi_\infty$ is the archimedean component $\mc{L}_\alpha(\pi_F,1/10)_\infty$ of the same Langlands quotient, or,
\item In the case that $\epsilon(1/2,\pi_F,\Sym^3)=-1$, the representation $\Pi_\infty$ is the (quaternionic) discrete series representation of $G_2(\R)$ with Harish-Chandra parameter $\frac{k-4}{2}(2\alpha+3\beta)+\rho$.
\end{enumerate}
\end{conjecture}

We will show how this conjecture follows from Arthur's multiplicity formula. There are three main assertions from which this conjecture would follow:
\begin{enumerate}[label=(\arabic*)]
\item There is the assertion that the representation $\Pi_{F,f}\otimes\Pi_\infty$ occurs in the \textit{discrete} spectrum, with $\Pi_\infty$ depending on the symmetric cube root number as described, with multiplicity one.
\item There is the assertion that the only occurrence of $\Pi_{F,f}\otimes\Pi_\infty$ in the discrete spectrum is actually \textit{cuspidal}.
\item Finally there is the assertion that $\Pi_{F,f}^v$ cannot occur in the discrete spectrum as the finite part away from $v$ of any other representation beside the one described above.
\end{enumerate}

To justify assertion (1) using Arthur's multiplicity formula, we first note that $\pi_F$ factors through $PGL_2(\A)$ and so should define a tempered Langlands parameter $\phi_F:L_\Q\to SL_2(\C)$. Following Gan and Gurevich \cite{gangurl}, we define a global Arthur parameter $\psi_F:L_\Q\times SL_2(\C)\to\L{G}_2$ as the composition
\begin{align*}
\psi_F:L_\Q\times SL_2(\C)&\xrightarrow{\phi_F\times\id_{SL_2(\C)}}L_\Q\times SL_{2,\beta}(\C)\times SL_{2,2\alpha+3\beta}(\C)\\
&\xrightarrow{\id_{W_\R}\times\iota_{\beta,2\alpha+3\beta}}L_\Q\times G_2(\C)=\L{G_2},
\end{align*}
where the notation is as in Section \ref{secparamsG2}. Note that the local component of $\psi_F$ at $\infty$ is the parameter $\psi_k$ from Section \ref{secparamsG2} by Proposition \ref{propPsil}. The implications of Arthur's multiplicity formula for this parameter $\psi_F$ were explained in \cite[Section 13.4]{gangurl}. We recall this now.

First, at places $w$ where $\pi_{F,w}$ is discrete series, the component group $\mc{C}_{\psi_w}$ has two elements; otherwise it is trivial. Thus we expect the Arthur packets $\Pi_{\psi_w}$ attached to $\psi_w$ at places $w$ where $\pi_{F,w}$ is discrete series to have two elements, and we write $\Pi_{\psi_{F,w}}=\{\pi_w^+,\pi_w^-\}$, where $\pi_w^+$ is the $L$-packet element. Otherwise we have $\Pi_{\psi_{F,w}}=\{\pi_w^+\}$, the singleton containing the $L$-packet element. The $L$-packet elements are visibly the Langlands quotients $\mc{L}_\alpha(\pi_{F,w},1/10)$ from the long root parabolic. Note in particular that for $w=v$, the packet $\Pi_{\psi_{F,v}}$ is the singleton containing $\mc{L}_\alpha(\pi_{F,v},1/10)$.

The global component group $\mc{C}_\psi$ has two elements and the character $\xi_\psi$ of Conjecture \ref{conjart} (2) can be computed to be the character which sends the nontrivial element in $\mc{C}_\psi$ to the root number $\epsilon(1/2,\pi_F,\Sym^3)$. Thus, if $\pi$ is of the form
\[\pi=\sideset{}{'}\bigotimes_w \pi_w, \qquad\pi_w\in\Pi_{\psi_w},\]
and if we let $\epsilon_{\pi_w}=1$ or $-1$ according to whether $\pi_w$ is $\pi_w^+$ or $\pi_w^-$, respectively, then $\pi$ occurs with multiplicity $1$ in $L_{\disc}^2(G_2(\Q)\backslash G_2(\A))$ if $\prod_w\epsilon_{\pi_w}=\epsilon(1/2,\pi_F,\Sym^3)$; otherwise $\pi$ occurs with multiplicity $0$. Since the Arthur packet at $\infty$ coincides with the possibilities for $\Pi_\infty$ described in our conjecture by Theorem \ref{thmAJpacket}, we have justified the assertion (1) above.

To see (2), we note that by Proposition \ref{propneareqalpha} and Lemma \ref{lemholoESg2}, no Eisenstein series with finite part nearly equivalent to $\Pi_{F,f}$ has a pole, since we assumed in our conjecture that $L(1/2,\pi_F,\Sym^3)=0$. Thus $\Pi_{F,f}$ cannot occur as the finite component of a representation in the residual spectrum. Then using the well known decomposition
\[L_{\disc}^2=L_{\cusp}^2\oplus L_{\mr{res}}^2\]
justifies (2).

To understand why (3) should hold, we have to examine other global Arthur parameters $\psi$. More precisely, let $\Pi_\infty'$ be a representation of $G_2(\R)$ and $\Pi_v'$ a representation of $G_2(\Q_v)$ such that $\Pi_{F,f}^v\otimes\Pi_v'\otimes\Pi_\infty'$ occurs in $L_{\cusp}^2(G_2(\Q)\backslash G_2(\A))$, and let $\psi$ be the global Arthur parameter corresponding to this representation. We must show that $\psi=\psi_F$.

To see this, we first break the set of global Arthur parameters for $G_2$ into five subsets, called $\Psi_?(G_{2/\Q})$, $?\in\{0,l,s,sr,r\}$, based on their restrictions to $SL_2(\C)$, just as in Section \ref{secparamsG2}; the definitions of these subsets are just as in that section. We note that the restriction of $\psi$ to $SL_2(\C)$ is the same as the restriction of any of the local components $\psi_w$ to $SL_2(\C)$.

Now if $w\ne v$ is a finite place which is unramified for $\Pi_{F,f}$ and such that the local packet $\Pi_{\psi_w}$ contains only one element, then the $L$-parameter $\phi_{\psi_w}$ is nontempered because it must coincide with the $L$-parameter $\phi_{\psi_{F,w}}$; indeed, if $\ell_w$ is the prime corresponding to $w$ and $s_w$ is the Satake parameter of $\pi_F$ at $w$, then
\[\phi_{\psi_{F,w}}(\Frob_w^{-1})=s_w\times\pmat{\ell_w^{1/2}&0\\0& \ell_w^{-1/2}}\in SL_{2,\beta}(\C)\times SL_{2,2\alpha+3\beta}(\C),\]
whose powers are unbounded. Thus we cannot have $\psi|_{W_w'}=\phi_{\psi_{F,w}}$, which implies $\psi\notin\Psi_0(G_{2/\Q})$.

Now assume for sake of contradiction that $\psi\in\Psi_{r}(G_{2/\Q})$. Then $\psi|_{L_\Q}$ is trivial (because the centralizer of the image of the corresponding homomorphism from $SL_2(\C)$ is trivial, as we saw in Proposition \ref{propPsir}). Moreover, at a finite place $w$ where $\Pi_{F,w}$ is unramified and $\Pi_{\psi_w}$ contains only one element, we have
\[\psi_{\psi_w}(\Frob_w^{-1})=(3\beta^\vee+5\alpha^\vee)(\ell_w^{1/2}),\]
which does not coincide with our expression for $\phi_{\psi_{F,w}}(\Frob_w^{-1})$ above. Thus $\psi\notin\Psi_{r}(G_{2/\Q})$. A similar argument, applied with $\Frob_w^{-6}$, shows also that $\psi\notin\Psi_{sr}(G_{2/\Q})$, since the centralizer of the image of the subregular homomorphism $SL_2(\C)\to G_2(\C)$ is a group of order $6$.

Now if $\psi\in\Psi_{s}(G_{2/\Q})$, then $\psi|_{L_\Q}$ must factor through a long root $SL_2(\C)$. Therefore $\Pi_{F,f}$ would be nearly equivalent to a Langlands quotient of a parabolic induction of a tempered representation from both the long root and short root parabolics. By Proposition \ref{propdistneareq}, this is impossible.

Thus $\psi\in\Psi_l(G_{2/\Q})$, and $\psi|_{L_\Q}$ factors through a short root $SL_2(\C)$. Moreover, by strong multiplicity one for $GL_2$, the parameter $\psi|_{L_\Q}$ corresponds to $\pi_F$, and thus $\psi=\psi_F$, as desired.

This completes our justification of Conjecture \ref{conjmult} using Arthur's multiplicity formula. We would like to remark, however, that this conjecture seems to now be within reach just using theta correspondence, rather than by extending the endoscopy classification of representations to $G_2$. Indeed, according to Wee Teck Gan, using forthcoming work of himself and Baki\'c on theta correspondence for $U(3)\times G_2$, one should be able to complete the work proposed in \cite{gangurl} and prove the multiplicity formula for long root CAP representations of $G_2$. Moreover, by \cite[Theorem 5.2]{hupasa}, this theta correspondence is functorial at the archimedean place and gives the quaternionic discrete series representation that we obtained in Proposition \ref{propajqds}.

\subsection{Liftings to $GL_7$}
\label{subsecliftings}
We make a conjecture in this section about lifting cohomological automorphic representations of $G_2(\A)$ with regular weight to ones for $GL_7$. This conjecture is used in the main body of the paper studying not the Langlands quotient $\mc{L}_\alpha(\pi_F,1/10)$, but rather certain cuspidal members of a $p$-adic deformation of it. Here is the conjecture.

\begin{conjecture}
\label{conjliftings}
Let $\Pi$ be a cuspidal automorphic representation of $G_2(\A)$ with archimedean component $\Pi_\infty$ which is cohomological of regular weight
\[\lambda=c_1(\alpha+2\beta)+c_2(2\alpha+3\beta),\qquad c_1,c_2>0.\]
Then there is a self dual automorphic representation $\widetilde{\Pi}$ of $GL_7(\A)$ with the following properties.
\begin{enumerate}[label=(\alph*)]
\item If $v$ is an unramified place for $\Pi$ with Satake parameter $s_v$, then $v$ is also unramified for $\widetilde{\Pi}$, and $\widetilde{\Pi}_v$ has Satake parameter $R_7(s_v)$, where $R_7$ is the standard $7$ dimensional representation of $G_2$.
\item The representation $\widetilde{\Pi}_\infty$ is cohomological of weight
\[(2c_1+c_2,c_1+c_2,c_1,0,-c_1,-c_1-c_2,-2c_1-c_2).\]
Here this septuple of integers is interpreted as a weight of the diagonal torus of $GL_7$ in the standard way.
\item The representation $\widetilde{\Pi}$ is the isobaric sum of cuspidal automorphic representations from $GL_{n_i}$ for some $n_i$'s with $\sum n_i=7$.
\item Assume $v$ is a finite place such that there is an $s\in\C$ with $\re(s)\geq 0$, and an irreducible admissible representation $\pi_v$ of $GL_2(\Q_v)$ with trivial central character, such that $\Pi_v$ is the Langlands quotient of the unitary parabolic induction
\begin{equation*}
\iota_{P_\alpha(\Q_v)}^{G_2(\Q_v)}(\pi_v\otimes\vert\det_\alpha\vert^s).
\end{equation*}
Then $\widetilde{\Pi}_v$ is the Langlands quotient of the parabolically induced representation
\begin{equation*}
\iota_{P_{2,3,2}(\Q_v)}^{GL_7(\Q_v)}((\pi_v\otimes\vert\det\vert^{s})\boxtimes\Sym^2(\pi_v)\boxtimes(\pi_v^\vee\otimes\vert\det\vert^{-s})).
\end{equation*}
Here, $P_{2,3,2}$ is the standard parabolic subgroup of $GL_7$ with Levi $GL_2\times GL_3\times GL_2$, and $\Sym^2$ denotes the symmetric square lift.
\item Similarly, assume $v$ is a finite place such that there is an $s\in\C$ with $\re(s)\geq 0$, and an irreducible admissible representation $\pi_v$ of $GL_2(\Q_v)$ with trivial central character, such that $\Pi_v$ is the Langlands quotient of the unitary parabolic induction
\begin{equation*}
\iota_{P_\beta(\Q_v)}^{G_2(\Q_v)}(\pi_v\otimes\vert\det_\beta\vert^s).
\end{equation*}
Then $\widetilde{\Pi}_v$ is the Langlands quotient of the parabolically induced representation
\begin{equation*}
\iota_{P_{1,2,1,2,1}(\Q_v)}^{GL_7(\Q_v)}(\vert\cdot\vert^s\boxtimes(\pi_v\otimes\vert\det\vert^{s})\boxtimes 1\boxtimes(\pi_v^\vee\otimes\vert\det\vert^{-s})\boxtimes\vert\cdot\vert^{-s}).
\end{equation*}
Similarly as above, $P_{1,2,1,2,1}$ is the standard parabolic subgroup of $GL_7$ with Levi $GL_1\times GL_2\times GL_1\times GL_2\times GL_1$.
\end{enumerate}
\end{conjecture}

We would like to explain now how this conjecture would follow from Arthur's conjectures. With $\Pi$ as in the hypotheses of the conjecture, let $\psi:L_\Q\times SL_2(\C)\to\L{G}_2$ be the global Arthur parameter corresponding to $\Pi$. Then the Arthur packet $\Pi_{\psi_\infty}$ for the archimedean component $\psi_\infty$ of $\psi$ includes, by hypothesis, a cohomological representation of regular cohomological weight $\lambda$. We classified the archimedean Arthur parameters whose packets contain cohomological representations in Proposition \ref{propcohparamsG2}, and the only such parameters which give representations with regular cohomological weight are the ones which are trivial on $SL_2(\C)$ and which give the discrete series $L$-packet of that weight. It follows that $\psi$ is trivial on $SL_2(\C)$, that $\psi|_{L_\Q}=\phi_\psi$, and that the Arthur packet for $\psi$ equals the $L$-packet of $\phi_\psi$.

Now we define a new global Arthur parameter $\tilde{\psi}$, this time for $GL_7$. Letting $R_7$ be the standard $7$ dimensional representation of $G_2$, we simply set
\[\tilde\psi=(R_7\times\id)\circ\psi:L_\Q\times\SL_2(\C)\to G_2(\C)\times L_\Q\to GL_7(\C)\times L_\Q=\L{GL_7}.\]
Thus the assignment $\psi\mapsto\tilde\psi$ should realize Langlands functoriality from $G_2$ to $GL_7$ for the representation $\Pi$.

Now $\tilde\psi$ is trivial on $SL_2(\C)$ because $\psi$ is, and so the Arthur packet for $\tilde\psi$ equals the $L$-packet of $\tilde\psi|_{L_\Q}=\psi_{\tilde\psi}$. Then $\tilde\psi$ gives a representation occurring in $L_{\disc}^2(GL_7(\Q)\backslash GL_7(\A))$. (From the point of view of Arthur's conjectures, since $\psi$ and all $\psi_v$'s are tempered, the centralizers of their images in either $G_2(\C)$ or $GL_7(\C)$ are Levi subgroups, hence are connected. Arthur's multiplicity formula then guarantees that the representation in the singleton $L$-packet for $\phi_{\tilde\psi}$ occurs in the discrete spectrum, in fact exactly once.) We take $\widetilde\Pi$ to be this representation.

Then the self duality of $\widetilde{\Pi}$ follows from the fact that $\phi_{\tilde\psi}$ factors through $\L{G}_2$ (and hence $\L{SO}_7$). Point (a) of the conjecture follows immediately from this construction as well.

To justify part (b) of the conjecture, we first describe explicitly the parameters $\phi_{\psi_\infty}$ and $\phi_{\tilde{\psi}_\infty}$. Since all the groups involved are split, we view them as having the corresponding dual groups as targets. The parameter $\phi_{\psi_\infty}$ gives the discrete series $L$-packet with infinitesimal character $\lambda+\rho$, and is thus given on $z\in\C^\times\subset W_\R$ by viewing $\lambda+\rho$ as a cocharacter of the dual torus of the maximal torus $T$ of $G_2$, and evaluating that cocharacter on $(z/\vert z\vert)$; identifying the dual group of $G_2$ with $G_2(\C)$ as usual, the element $\phi_{\psi_\infty}(j)$ is then given as any representative in $N_{G_2(\C)}(T(\C))$ of the Weyl element $w_{-1}$ that acts as inversion on $T$ and which squares to $\phi_{\psi_\infty}(-1)$. Thus,
\begin{align*}
\phi_{\psi_\infty}(z)&=(c_1(2\beta^\vee+3\alpha^\vee)+c_2(\beta^\vee+2\alpha^\vee)+3\beta^\vee+5\alpha^\vee)(z/\vert z\vert)\\
&=((3c_1+2c_2+5)\alpha^\vee+(2c_1+c_2+3)\beta^\vee)(z/\vert z\vert),
\end{align*}
for $z\in\C^\times$. Since this is an integral combination of $\alpha^\vee$ and $\beta^\vee$, we see $\phi_{\psi_\infty}(-1)=1$.

Now we consider $\phi_{\tilde{\psi}_\infty}=R_7\circ\phi_{\psi_\infty}$. Since $R_7$ has the seven weights given as in Figure \ref{figr7weights}, one computes that for $z\in\C^\times$, $\phi_{\tilde{\psi}_\infty}(z)$ is given by the evaluation on $(z/\vert z\vert)$ of the coweight of the dual torus in $GL_7(\C)$ given by the weight
\[(2c_1+c_2+3,c_1+c_2+2,c_1+1,0,-c_1-1,-c_1-c_2-2,-2c_1-c_2-3).\]
The element $\phi_{\tilde{\psi}_\infty}(j)$ is given by the Weyl group element $(17)(26)(35)\in S_7$, as in the proof of Proposition \ref{propPabeta}.

The Langlands parameter $\phi_{\tilde{\psi}_\infty}$ therefore factors through a Levi subgroup of $GL_7(\C)$ given by $GL_2(\C)^3\times GL_1(\C)$, and is given by discrete series parameters on the $GL_2$ factors of weights
\[k_1=2(2c_1+c_2+3)+1,\qquad k_2=2(c_1+c_2+2)+1,\qquad k_3=2(c_1+1)+1\]
and by the trivial character on the $GL_1$ factor. Letting $\pi_k$ denote the discrete series of $GL_2(\R)$ of weight $k$, for $k\geq 2$, we therefore have that the Langlands parameter $\phi_{\tilde{\psi}_\infty}$ corresponds to the induced representation
\[\iota_{P_{2,2,2,1}(\R)}^{GL_7(\R)}(\pi_{k_1}\otimes\pi_{k_2}\otimes\pi_{k_3}\otimes 1),\]
where $P_{2,2,2,1}$ is the standard parabolic in $GL_7$ corresponding to the partition $7=2+2+2+1$. A standard computation using \cite[Theorem III.3.3]{BW}, which we omit here, shows that this representation is cohomological with the cohomological weight claimed in part (2) of the conjecture.

We now discuss part (d) of the conjecture; we will return to (c) in a moment. The representation $\pi_v$ of (d) factors through $PGL_2(\A)$ by assumption, and therefore defines a Langlands parameter
\[\phi_v:W_v'\to SL_2(\C).\]
The parameter of $\pi_v\otimes\vert\det\vert^s$ is then given by
\[\phi_{v,s}(w')=\diag(\vert w'\vert^s,\vert w'\vert^s)\phi_v(w')\in GL_2(\C).\]
for $w'\in W_v$, where $\vert\cdot\vert$ is the usual valuation on $W_v'$.

If $\re(s)\geq 0$ and $\Pi_v$ is a Langlands quotient, then $\phi_{\psi_v}$ is given by
\[\iota_\beta\circ\phi_{v,s},\]
where $\iota_\beta$ is the inclusion of the short root $GL_2$ Levi into $G_2(\C)$. In this case, by \eqref{eqr7beta}, the parameter $\phi_{\tilde{\psi}_v}$ is given block-diagonally as
\[\phi_{\tilde{\psi}_v}(w)=\pmat{\phi_{v,s}(w)&0&0\\
0&\Ad^2\phi_{v,s}(w)&0\\
0&0&\prescript{t}{}{\phi}_{v,s}(w)^{-1}},\]
which is the parameter for the Langlands quotient which we claimed was $\widetilde\Pi_v$ in (d). A similar argument justifies (e).

Finally, we note that something much stronger than (c) should be true. By Arthur's conjectures, the representation $\widetilde\Pi$ should be in the discrete spectrum of $GL_7$. By the main result of \cite{MWres}, $\widetilde{\Pi}$ is then either cuspidal or it is the isobaric sum of seven copies of the same character of $GL_1(\A)$. But by (b) of the conjecture, this latter case is impossible, since such a sum would have to have infinitesimal character $(3,2,1,0,-1,-2,-3)$, which is smaller than that of $\widetilde{\Pi}_\infty$. However, we are able to work in the body of this paper with the version of (c) stated in the conjecture which, of course, is weaker than cuspidality. We choose to do this because of how we expect one might go about proving this conjecture, which we discuss now.

First, it follows from a general conjecture of Shahidi that the $L$-packet $\Pi_{\phi_\psi}$ should contain a generic member. If one could prove this, then as in the paper of Harris--Khare--Thorne \cite{HKT}, one could use an exceptional theta correspondence to lift the generic constituent of $\Pi_{\phi_\psi}$ to a generic automorphic representation of $Sp_6(\A)$ and then lift that representation to one of $GL_7(\A)$ using \cite{CKPSS}, as is done in \cite{HKT}. This representation would automatically satisfy points (a) and (b) of the conjecture, by the computations in \textit{loc. cit.} But the results of \cite{CKPSS} are not strong enough to guarantee that the resulting representation is cuspidal. However, they do guarantee that point (c) of our conjecture would hold instead, and this is enough for our purposes.

As for points (d) and (e), this does not seem to be directly amenable to just the methods used in \cite{HKT}. However, there are results in \cite{CKPSS} which would describe local lifts of parabolically induced representations from the group $Sp_6$ to $GL_7$. So one would have to examine explicitly the local exceptional theta correspondence in the case of representations of the types in (d) and (e).
\printbibliography
\end{document}